\documentclass[10pt]{article}

\usepackage[T1]{fontenc}
\usepackage[utf8]{inputenc}

\usepackage[left=1in,top=1in,right=1in,bottom=1in,head=.1in]{geometry}

\usepackage{mystyle}
\usepackage[font=small,labelfont=bf]{caption}

\usepackage[normalem]{ulem}

\usepackage{parskip}

\allowdisplaybreaks

\usepackage{algorithm}
\usepackage{algorithmic}

\usepackage{tikz}
\usetikzlibrary{arrows.meta,arrows}
\usepackage{enumitem}

\usepackage{soul}
\setstcolor{red}

\newcommand{\titletext}{Bagging in overparameterized learning:\\
Risk characterization and risk monotonization}

\newcommand{\removelinebreaks}[1]{%
      \def\\{\relax}#1}
\def\titleRLB{\removelinebreaks{\titletext}}

\newcommand{\footremember}[2]{%
    \footnote{#2}
    \newcounter{#1}
    \setcounter{#1}{\value{footnote}}%
}
\newcommand{\footrecall}[1]{%
    \footnotemark[\value{#1}]%
}
  
\title{\texorpdfstring{\vspace{-2em}}{}\titletext}
\author{
Pratik Patil\footremember{equal}{Equal contribution.}\footremember{berkeleystats}{Department of Statistics, University of California, Berkeley, CA 94720, USA.}\\ {\small pratikpatil@berkeley.edu}
\and Jin-Hong Du\footrecall{equal} \footremember{cmustats}{Department of Statistics and Data Science,
Carnegie Mellon University, Pittsburgh, PA 15213, USA.}\footremember{cmumld}{Machine Learning Department, Carnegie Mellon University, Pittsburgh, PA 15213, USA.} \\ {\small jinhongd@andrew.cmu.edu}
\and Arun Kumar Kuchibhotla\footrecall{cmustats} \\ {\small arunku@cmu.edu}
}
\date{\today}

\begin{document}

\maketitle

\begin{abstract}
Bagging is a commonly used ensemble technique in statistics and machine learning to improve the performance of prediction procedures. In this paper, we study the prediction risk of variants of bagged predictors under the proportional asymptotics regime, in which the ratio of the number of features to the number of observations converges to a constant. Specifically, we propose a general strategy to analyze the prediction risk under squared error loss of bagged predictors using classical results on simple random sampling. Specializing the strategy,  we derive the exact asymptotic risk of the bagged ridge and ridgeless predictors with an arbitrary number of bags under a well-specified linear model with arbitrary feature covariance matrices and signal vectors. Furthermore, we prescribe a generic cross-validation procedure to select the optimal subsample size for bagging and discuss its utility to eliminate the non-monotonic behavior of the limiting risk in the sample size (i.e., double or multiple descents). In demonstrating the proposed procedure for bagged ridge and ridgeless predictors, we thoroughly investigate the oracle properties of the optimal subsample size and provide an in-depth comparison between different bagging variants.
\end{abstract}

\setcounter{tocdepth}{2}

\tableofcontents

\section{Introduction}
\label{sec:introduction}

Modern machine learning models often use a large number of parameters relative to the number of observations.
In this regime, several commonly used procedures exhibit a peculiar risk behavior, which is referred to as double or multiple descents in the risk profile \citep{belkin2019reconciling,zhang_bengio_hardt_recht_vinyals_2016,zhang_bengio_hardt_recht_vinyals_2021}.
The precise nature of the double or multiple descent behavior in the generalization error has been studied for various procedures: e.g., linear regression \citep{belkin_hsu_xu_2020,muthukumar_vodrahalli_subramanian_sahai_2020,hastie2022surprises}, logistic regression \citep{deng2022model}, random features regression \citep{mei_montanari_2022}, kernel regression \citep{liu2021kernel}, among others.
We refer the readers to the survey papers by \citet{bartlett_montanari_rakhlin_2021,belkin_2021,dar_muthukumar_baraniuk_2021} for a more comprehensive review and other related references.
In these cases, the asymptotic prediction risk behavior is often studied as a function of the data aspect ratio (the ratio of the number of parameters/features to the number of observations). 
The double descent behavior refers to the phenomenon where the (asymptotic) risk of a sequence of predictors first increases as a function of the aspect ratio, peaks at a certain point (or diverges to infinity), and then decreases with the aspect ratio. 
From a traditional statistical point of view, the desirable behavior as a function of aspect ratio is not immediately obvious. 
We can, however, reformulate this behavior as a function of $\phi = p/n$, in terms of the observation size $n$ with a fixed $p$; imagine a large but fixed $p$ and $n$ changing from $1$ to $\infty$. 
In this reformulation, the double descent behavior translates to a pattern in which the risk first decreases as $n$ increases, then increases, peaks at a certain point, and then decreases again with $n$. 
This is a rather counter-intuitive and sub-optimal behavior for a prediction procedure. 
The least one would expect from a good prediction procedure is that it yields better performance with more information (i.e., more data).
However, the aforementioned works show that many commonly used predictors may not exhibit such ``good'' behavior.
Simply put, the non-monotonicity of the asymptotic risk as a function of the number of observations or the limiting aspect ratio implies that more data may hurt generalization \citep{nakkiran2019more}.

Several ad hoc regularization techniques have been proposed in the literature to mitigate the double/multiple descent behaviors. 
Most of these methods are trial-and-error in nature in the sense that they do not directly target monotonizing the asymptotic risk but instead try a modification and check that it yields a monotonic risk. 
The recent work of \citet{patil2022mitigating} introduces a generic cross-validation framework that directly addresses the problem and yields a modification of any given prediction procedure that provably monotonizes the risk.
In a nutshell, the method works by training the predictor on subsets of the full data (with different subset sizes) and picking the optimal subset size based on the estimated prediction risk computed using testing data. Intuitively, it is clear that this yields a prediction procedure whose risk is a decreasing function of the observation size. In the proportional asymptotic regime, where $p/n\to\phi$ as $n,p\to\infty$, the paper proves that this strategy returns a prediction procedure whose asymptotic risk is monotonically increasing in $\phi$.
The paper theoretically analyzes the case where only one subset is used for each subset size and illustrates via numerical simulations that using multiple subsets of the data of the same size (i.e., subsampling) can yield better prediction performance in addition to monotonizing the risk profile. 
Note that averaging a predictor computed on $M$ different subsets of the data of the same size is referred to in the literature as subagging, a variant of the classical bagging (bootstrap aggregation) proposed by~\cite{breiman_1996}. 
The focus of the current paper is to analyze the properties of bagged predictors in two directions (in the proportional asymptotics regime): (1) what is the asymptotic predictive risk of the bagged predictors with $M$ bags as a function of $M$, and (2) does the cross-validated bagged predictor provably yield improvements over the predictor computed on full data and does it have a monotone risk profile (i.e., the asymptotic risk is a monotonic function of $\phi$)?

In this paper, we investigate several variants of bagging, including subagging as a special case. 
The second variant of bagging, which we call splagging (that stands for \textbf{spl}it-\textbf{agg}regat\textbf{ing}), is the same as the divide-and-conquer or the data-splitting approach~\citep{rosenblatt2016optimality,banerjee2019divide}. 
The divide-and-conquer approach is widely used in distributed learning, although not commonly featured in the bagging literature \citep{dobriban_sheng_2020,dobriban_sheng_2021,mucke_reiss_rungenhagen_klein_2022}.
Formally, splagging splits the data into non-overlapping parts of equal size and averages the predictors trained on these non-overlapping parts. 
We refer to the equal size of each part of the data as subsample size. 
We use the same terminology for subagging also for the sake of simplicity. 
Using classical results from survey sampling and some simple lemmas about almost sure convergence, we are able to analyze the behavior of subagged and splagged predictors\footnote{A note on terminology for the paper: when referring to subagging and \splagging together, we use the generic term bagging. Similarly, when referring to subagged and \splagged predictors together, we simply say bagged~predictors.} with $M$ bags for arbitrary prediction procedures and general $M\ge1$.
In fact, we show that the asymptotic risk of bagged predictors for general $M\ge1$ (or simply, $M$-bagged predictor) can be written in terms of the asymptotic risks of bagged predictors with $M = 1$ and $M = 2$. 
Rather interestingly, we prove that the $ M$-bagged predictor's finite sample predictive risk is uniformly close to its asymptotic limit over all $M\ge1$. 
These results are established in a model-agnostic setting and do not require the proportional asymptotic regime.
Deriving the asymptotic risk behavior of bagged predictors with $M = 1$ and $M = 2$ has to be done on a case-by-case basis, which we perform for ridge and ridgeless prediction procedures. 
In the context of bagging for general predictors, we further analyze the cross-validation procedure with $M$-bagged predictors for arbitrary $M\ge1$ to select the ``best'' subsample size for both subagging and splagging. 
These results show that subagging and splagging for any $M\ge1$ outperform the predictor computed on the full data. 
We further present conditions under which the cross-validated predictor with $M$-bagged predictors has an asymptotic risk monotone in the aspect ratio.
Specializing these results to the ridge and ridgeless predictors leads to somewhat surprising results connecting subagging to optimal ridge regression as well as the advantages of interpolation.

Before proceeding to discuss our specific contributions, we pause to highlight the two most significant take-away messages from our work. 
These messages hold under a well-specified linear model, where the features possess an arbitrary covariance structure, and the response depends on an arbitrary signal vector, both of which are subject to certain bounded norm regularity constraints.

\begin{figure}[!t]
    \centering
    \includegraphics[width=0.85\textwidth]{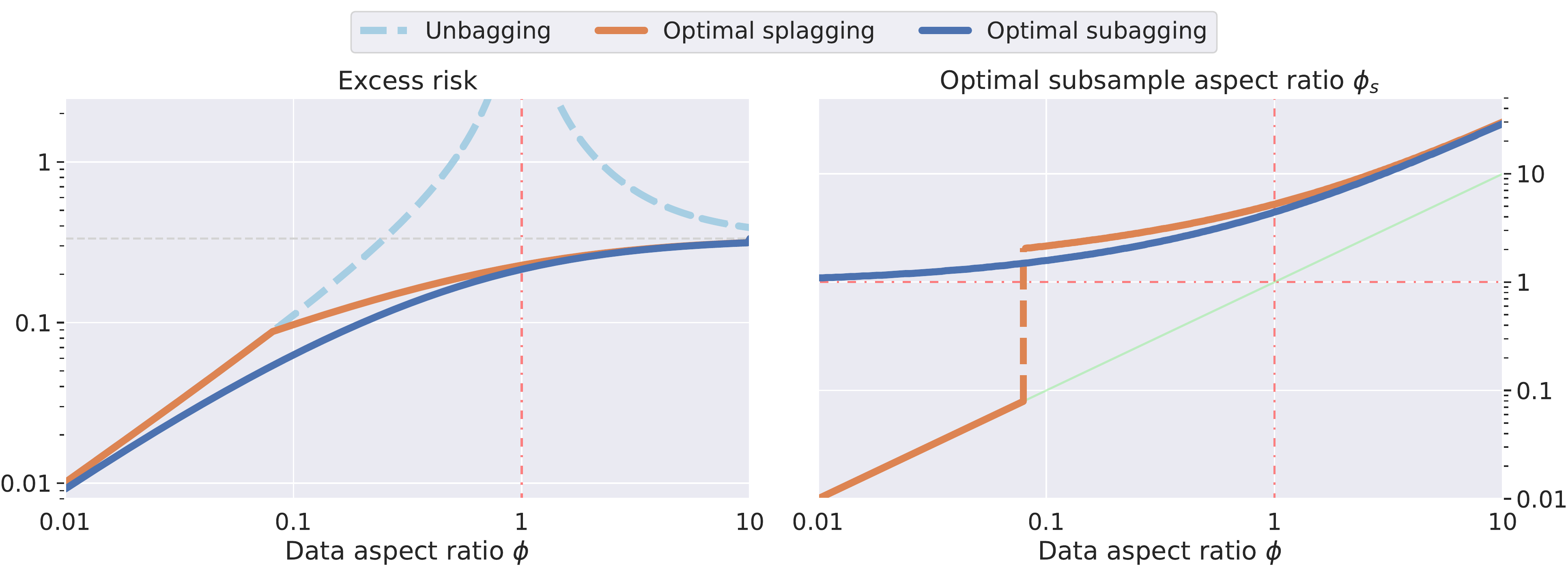}
    \caption{Overview of optimal bagging over both the subsample aspect ratio and the number of bags.
    (a) Optimal asymptotic excess risk curves for ridgeless predictors with and without bagging, under model \eqref{eq:model-ar1} when $\rhoar=0.25$ and $\sigma^2=1$.
    The excess risk is the difference between the prediction risk and the noise level $\sigma^2$.
    The risk for the unbagged ridgeless predictor is represented by a blue dashed line, and the null risk is marked as a gray dotted line.
     (b) The corresponding optimal limiting subsample aspect ratio $\phi_s=p/k$ versus the data aspect ratio $\phi=p/n$ for bagged ridgeless predictors.
     The line $\phi_s=\phi$ is colored in green.
     The optimal subsample aspect ratios are larger than one (above the horizontal red dashed line).
     See \Cref{subsec:comparison} for more details on the setup and further related discussion.
}\label{fig:overview}
\end{figure}

\begin{enumerate}
[\rm(T\arabic*),leftmargin=7mm]
    \item 
    Subagging and splagging (the data-splitting approach) of the ridge and ridgeless predictors, when properly tuned, can significantly improve the prediction risks of these standalone predictors trained on the full data.
    This improvement is most pronounced near the interpolation threshold.
    Importantly, \ul{subagging always outperforms splagging}.
    See the left panel of \Cref{fig:overview} for a numerical illustration and \Cref{prop:improve-with-without-replace} for a formal statement of this result.
    
    \item
    A model-agnostic algorithm exists to tune the subsample size for subagging. This algorithm produces a predictor whose risk matches that of the oracle-tuned subagged predictor.
    Notably, the oracle-tuned subsample size for the ridgeless predictor is always smaller than the number of features.
    As a result, \ul{subagged ridgeless \emph{interpolators} always outperform subagged least squares}, even when the full data has more observations than the number of features.
    The same observation holds true for splagging whenever it provides an improvement.
    See the right panel of \Cref{fig:overview} for numerical illustrations and \Cref{prop:opt-risk-ridgeless} for formal statements of this result.
\end{enumerate}

Intuitively, although bagging may induce bias due to subsampling, it can significantly reduce the prediction risk by reducing the variance for a suitably chosen subsample size that is smaller than the feature size.
This tradeoff arises because of the different rates at which the bias and variance of the ridgeless predictor increase near the interpolation threshold.
This advantage of \emph{interpolation} or \emph{overparameterization} is distinct from other benefits discussed in the literature, such as self-induced regularization \citep{bartlett_montanari_rakhlin_2021}.

\subsection{Summary of main results}

Below we provide a summary of the main results of this paper.

\begin{enumerate}[leftmargin=7mm]
    \item 
    \textbf{General predictors.}
    In \Cref{sec:general-predictors}, we formulate a generic strategy for analyzing the limiting squared \emph{data conditional risk} (expected squared error on a future data point, conditional on the full data) of general $M$-bagged predictors, showing that the existence of the limiting risk for $M = 1$ and $M = 2$ implies the existence of the limiting risk for every $M \ge 1$.
    Moreover, we show that the limiting risk of the $M$-bagged predictor can be written as a linear combination of the limiting risks of $M$-bagged predictors with $M = 1$ and $M = 2$.
    Interestingly, the same strategy also works for analyzing the limit of the \emph{subsample conditional risk}, which considers conditioning on both the full data and the randomly drawn subsamples.
    See \Cref{thm:risk_general_predictor} for a formal statement.
    In this general framework, \Cref{thm:risk_general_predictor} implies that both the data conditional and subsample conditional risks are asymptotically monotone in the number of bags $M$.
    Moreover, for general strongly convex and smooth loss functions, we can sandwich the risks between quantities of the form $\mathfrak{C}_1 + \mathfrak{C}_2 / M$, for some fixed random variables $\mathfrak{C}_1$ and $\mathfrak{C}_2$ (\Cref{prop:convex-sconvex-smooth}).

    \item 
    \textbf{Ridge and ridgeless predictors.}
    In \Cref{sec:ridge-ridgeless-predictors}, we specialize the aforementioned general strategy to characterize the data conditional and subsample conditional risks of $M$-bagged ridge and ridgeless predictors.
    The results are formalized in \Cref{thm:ver-with-replacement} for subagging with and without replacement, and \Cref{thm:ver-without-replacement} for splagging without replacement.
    All these results assume a well-specified linear model, with an arbitrary covariance matrix for the features and an arbitrary signal vector. Notably, we assume neither Gaussian features nor isotropic features nor a randomly generated signal.
    These results reveal that for the three aforementioned bagging strategies, the bias and variance risk components are non-increasing in the number of bags $M$.
    
    \item 
    \textbf{Cross-validation.}
    In \Cref{sec:monotonizing-risk-profiles}, we develop a generic cross-validation strategy to select the optimal subsample or split size (or equivalently, the subsample aspect ratio) and present a general result to understand the limiting risks of cross-validated predictors.
    Our theoretical results provide a way to verify the monotonicity of the limiting risk of the cross-validated predictor in terms of the limiting data aspect ratio $\phi$ (\Cref{thm:cv_general}).
    In \Cref{subsec:cross-validation}, we specialize in the cross-validated ridge and ridgeless predictors to obtain the optimal subsample aspect ratio for every $M$ (\Cref{thm:monotonicity-phi}).
    Moreover, when optimizing over both the subsample aspect ratio and the number of bags, we show that optimal subagging always outperforms 
    optimal \splagging (\Cref{prop:improve-with-without-replace}).
    Rather surprisingly, in our investigation of the oracle choice of the subsample size for optimal subagging with $M = \infty$, we find that the subsample ratio is always large than one (\Cref{prop:opt-risk-ridgeless}).
    In \Cref{sec:isotropic_features}, we also show optimally-tuned subbaged ridgeless predictor yields the same prediction risk as the optimal ridge predictor for isotropic features (\Cref{thm:comparison_optimal_ridge}).
    
\end{enumerate}

From a technical perspective, during the course of our risk analysis of the bagged ridge and ridgeless predictors, we derive novel deterministic equivalents for ridge resolvents with random Tikhonov-type regularization. We extend ideas of conditional asymptotic equivalents and related calculus, which may be of independent interest.
See \Cref{sec:calculus_asymptotic_equivalents}, and in particular \Cref{sec:asympequi-extended-ridge-resolvents}.

\subsection{Related work}
    
        The risk non-monotonicity of commonly used predictors has been well documented in the literature.
    For instance, a recent line of work by \citet{belkin2019reconciling,viering2019open,nakkiran2019more,loog2019minimizers}, among others, illustrates the non-monotonic risk behavior of several prediction procedures.
    See also the survey papers by \citet{belkin_2021,bartlett_montanari_rakhlin_2021,dar_muthukumar_baraniuk_2021,loog_viering_mey_krijthe_tax_2020} for other related references.
    As highlighted by \cite{loog_viering_mey_krijthe_tax_2020}, the phenomenon of multiple descents can be traced back to empirical findings in the 1990s, including earlier papers by \cite{vallet1989linear,hertz1989phase,opper1990ability,hansen1993stochastic,barber1995finite,duin1995small,opper1995statistical,opper1996statistical,raudys1998expected}, among others.
     
    Since non-monotonic risk leads to suboptimal use of the data, several methods have been proposed that modify a given (class of) prediction procedure(s) to construct a new prediction procedure with a monotonic risk profile.
    In particular, \citet{nakkiran2021optimal} investigates the role of optimal tuning in the context of ridge regression and demonstrates that the optimally-tuned $\ell_{2}$ regularization achieves monotonic generalization performance for a class of linear models under isotropic design.
    \citet{mhammedi2021risk} provides an algorithm to monotonize the risk profile for bounded loss functions.
    \citet{patil2022mitigating} propose a general framework to monotonize the prediction risk for general predictors under both bounded and unbounded loss functions, using cross-validation.
    The paper also empirically shows that bagging can further improve the performance of the predictors while achieving a monotonized risk profile.
    In this paper, we characterize the risk behavior of bagging, which was left as an open direction in \cite{patil2022mitigating}.
    Below we provide a brief overview of the literature pertaining to bagging and its relation to our work.
    
    Ensemble methods are widely used in machine learning and statistics and combine several weak predictors to produce one powerful predictor.
    One important class of ensemble methods is bagging \citep{breiman_1996,buhlmann2002analyzing}, and its variants, such as subagging \citep{buhlmann2002analyzing}, that operate by averaging predictors trained on independent subsamples of the data.
    Numerous empirical studies have demonstrated that bagging leads to significant improvements in predictive performance \citep{breiman_1996,strobl2009,fernandez2014we}.
    However, the theoretical analysis of bagging has primarily focused on smooth predictors (predictors that are smooth functions of the empirical data distribution); see \cite{buja2006observations,friedman_hall_2007}.
    For some work on bagging for non-parametric estimators, see \cite{hall_samworth_2005,samworth2012optimal,wu2021chebyshev,buhlmann2002analyzing,athey_tibshirani_wager_2019}.
    In addition to sample-wise bagging, bagging over linear combinations of features has also been considered in \citet{lopes2011more,srivastava2016raptt,cannings_samworth_2017}.
    This approach broadly falls under the umbrella of feature side sketching; we refer readers to \citet{wang2017sketched,derezinski2020debiasing,lopes2018error,derezinski2023algorithmic,lejeune2022asymptotics,patil2023asymptotically}, among others, for related results and further references.
    
    Bagging in the proportional asymptotic regime has also been considered in the literature.
    \cite{lejeune2020implicit} study subagging of both features and observations and derive the limiting risk of the resulting subagged predictor. \citet{dobriban_sheng_2020,dobriban_sheng_2021,mucke_reiss_rungenhagen_klein_2022} consider the divide-and-conquer approach, or splagging, and investigate their properties.
    These works are set in the context of distributed learning.
    Specifically, under proportional asymptotics, \citet{dobriban_sheng_2020} derive the limiting mean squared error of the distributed ridge estimator in the underparameterized regime. 
    On the other hand, \citet{mucke_reiss_rungenhagen_klein_2022} provide finite-sample upper bounds on the prediction risk for ridgeless regression in the overparameterized regime.

    The closest works to ours are those of \citet{lejeune2020implicit} and \cite{mucke_reiss_rungenhagen_klein_2022}.
    \citet{lejeune2020implicit} investigate bagged least squares predictor obtained by subsampling both features and observations in a Gaussian isotropic design.
    They impose a restriction on subsampling such that the final subsampled data always has more observations than the features (so that ordinary least squares are well-defined). 
    Consequently, they do not allow for overparameterized subsampled datasets.
    Similar to our work, they also study the monotonicity of the asymptotic expected squared risk with respect to the number of bags in their restricted setting.
    Further, they study the best subsampling ratios for optimal asymptotic risk, but do not consider the question of how to select the best subsample size.
    The most significant difference between their work and ours is that we subsample observations, and they effectively subsample features, which is only appropriate under isotopic covariance.
    On the other hand, \citet{mucke_reiss_rungenhagen_klein_2022} consider splagging and provide finite-sample upper bounds on the bias and variance components of the squared prediction risk under the assumption of sub-Gaussian features.
    In contrast, our results do not assume sub-Gaussianity for either the feature or response distributions and only impose minimal bounded moment assumptions.

\subsection{Organization}

Below we provide an outline for the rest of the paper.

\begin{itemize}[leftmargin=7mm]
    
    \item
    In \Cref{sec:general-predictors},
    we provide risk decompositions
    conditional on both the full dataset and subsampled datasets
    for different bagging variants for general predictors.
    Based on the form of decompositions,
    we provide a series of reductions
    and a generic strategy for analyzing 
    the squared prediction risk 
    of general bagged predictors.
    
    \item
    In \Cref{sec:ridge-ridgeless-predictors},
    we give risk characterizations
    for bagging ridge and ridgeless predictors.
    We give results for both subagging with and without replacement,
    and splagging without replacement,
    and show monotonicities of the bias and variance components
    in the number of bags.

    \item
    In \Cref{sec:monotonizing-risk-profiles},
    we prescribe a framework for monotonizing
    the risk profile of any given predictor
    based on cross-validation over subsample size.
    The result is then specialized to the ridge and ridgeless predictors.
    Furthermore, we compare the monotonized risk profiles
    of bagged ridgeless and ridge predictors.

    \item
    In \Cref{sec:isotropic_features}, we specialize our results for isotropic features and provide explicit analytic expressions for the risks of bagged ridgeless regression.
    In addition, we present the analysis of the optimal subsample size and the corresponding optimal bagged risk.
    
    \item
    In \Cref{sec:discussion},
    we conclude the paper 
    by providing related questions for future work.
\end{itemize}

In the Supplement to this paper, we give a brief background on simple random sampling, provide proof of all the results, and present additional numerical illustrations.
The organization structure for the Supplement is provided in the first section of the Supplement, which also gives an overview of the general notation employed throughout the paper.
The source code for reproducing all the experimental illustrations in this paper can be found at: \url{https://github.com/jaydu1/Overparameterized-Bagging/tree/main/bagging}.

\section{Bagging general predictors}
\label{sec:general-predictors}

In this section, we will describe different versions of subagged predictors.
But first, let us define the index sets pertinent to our study. Fix any $k\in\{1, 2, \ldots, n\}$ and any permutation $\pi:\{1,2,\ldots,n\}\to\{1, 2, \ldots, n\}$.
Define the sets $\cI_k$ and $\cI_k^\pi$ as follows:
\begin{equation}\label{eq:indices-for-bagging}
    \begin{split}
        \mathcal{I}_k &:= \{\{i_1, i_2, \ldots, i_k\}:\, 1\le i_1 < i_2 < \ldots < i_k \le n\},\\
        \mathcal{I}_k^{\pi} &:= \left\{\{\pi((j-1)k + 1), \pi((j-1)k + 2), \ldots, \pi(jk)\}:\, 1\le j\le \left\lfloor\frac{n}{k}\right\rfloor\right\}.
    \end{split}
\end{equation}
Note that both the sets $\mathcal{I}_k$ and $\mathcal{I}_k^\pi$ technically need to be indexed by $n$, but for notation convenience, we will not explicitly indicate the dependence on $n$.
The set $\mathcal{I}_k$ represents the set of all $k$ subset choices from $\{1, 2, \ldots, n\}$. 
There are $\binom{n}{k}$ many of them. 
The set $\mathcal{I}_k^{\pi}$, on the other hand, represents the set of indices in a non-overlapping split of $\{1, 2, \ldots, n\}$ into blocks of size $k$. 
If we split $\{1, 2, \ldots, n\}$ randomly into different non-overlapping blocks each of size $k$, then this corresponds to choosing a permutation $\pi$ randomly from the set of all permutations and splitting them in order. 
Finally, observe that $\mathcal{I}_k^{\pi} \subseteq \mathcal{I}_k$ for any permutation $\pi$ and $\cup_{\pi} \mathcal{I}_k^{\pi} = \mathcal{I}_k$. 

\subsection{Conditional risk decompositions}
\label{subsec:bagged_general_risk_decom}

Suppose now $\mathcal{D}_n = \{(\bx_1, y_1), \ldots, (\bx_n, y_n)\}$ represents a dataset with random vectors from $\mathbb{R}^p\times\mathbb{R}$.
A prediction procedure $\widetilde{f}(\cdot; \cdot)$ is defined as a map from $\mathbb{R}^p\times \mathscr{P}(\mathcal{D}_n) \to \mathbb{R}$, where $\mathscr{P}(A)$ for any set $A$ represents the power set of $A$. 
For any $I\in\mathcal{I}_k$ (or $I\in\mathcal{I}_k^{\pi}$), let $\mathcal{D}_I$ and the corresponding subsampled predictor be defined as $\mathcal{D}_{I} = \{(\bx_j, y_j):\, j \in I\}$ and $\hf(\bx; \mathcal{D}_I) = \hf(\bx; \{(\bx_j, y_j):\, j\in I\})$.
Given two sets of indices and two types of simple random samplings one can draw, we have four different versions of subagged predictors.
When employing simple random sampling with replacement, the corresponding predictors can be expressed as follows:
\begin{align}
    \widetilde{f}_{M,\mathcal{I}_k}^{\WR}(\bx; \{\mathcal{D}_{I_{\ell}}\}_{\ell = 1}^M) &=  \frac{1}{M}\sum_{\ell = 1}^M \hf(\bx; \mathcal{D}_{I_{\ell}}),\quad \text{where} \quad I_1, \ldots, I_M \overset{\texttt{SRSWR}}{\sim} \begin{cases}
        \mathcal{I}_k & \text{for subagging}\\
        \mathcal{I}_k^{\pi} & \text{for \splagging},
    \end{cases}\label{eq:bagged-predictor-subagging}
\end{align}
and the predictors $\widetilde{f}_{M,\mathcal{I}_k}^{\WOR}$ using simple random sampling without replacement are defined analogously.

Traditionally, bagging (\smash{as in \textbf{b}ootstrap-\textbf{agg}regat\textbf{ing}}) refers to computing predictors multiple times based on bootstrapped data~\citep{breiman_1996}, which can involve repeated observations. 
In this paper, we do not allow for repeated observations and consider only the four versions of bagging mentioned in~\eqref{eq:bagged-predictor-subagging}. \citet[Section 3.2]{buhlmann2002analyzing} call $\tf_{M,\cI_k}^{\WR}$ as subagging (\smash{as in \textbf{sub}sample-\textbf{agg}regat\textbf{ing}}). 
Given that SRSWOR mean estimator has a smaller mean squared error than SRSWR mean estimator, we also consider the variant $\tf_{M,\cI_k}^{\WOR}$ of subagging. 
Because for any fixed $M$, the expectation and variance of $\tf_{M,\cI_k}^{\WR}$ and $\tf_{M,\cI_k}^{\WOR}$ are the same as $N\to\infty$, the asymptotic risk behavior of $\tf_{M,\cI_{k}}^{\WR}$ and $\tf_{M,\cI_k}^{\WOR}$ is the same if $|\cI_k| = \binom{n}{k} \to \infty$ (which holds, for example, if $1 \le k \le n-1$ and $n\to\infty$). 
Given this equivalence and the relative prevalence of subagging (i.e., $\tf_{M,\cI_k}^{\WR}$), in \Cref{sec:bagging-with-replacement}, we focus our results on $\tf_{M,\cI_k}^{\WR}$ although we indicate the implications for $\tf_{M,\cI_k}^{\WOR}$. 
In what follows, we refer to $\tf_{M,\cI_k}^{\WR}$ and $\tf_{M,\cI_k}^{\WOR}$ as {\em subagging with and without replacement}, respectively.

In contrast, the predictors $\tf_{M,\cI_{k}^{\pi}}^{\WR}$ and $\tf_{M,\cI_k^{\pi}}^{\WOR}$ do not frequently appear in the bagging literature. 
Rather, they are more common in distributed learning, where the predictors are trained on different parts of the data and averaged to yield a final predictor. 
We call these versions as ``\splagging'' (as in \smash{\textbf{spl}it-\textbf{agg}regat\textbf{ing}}). 
Among these, the without replacement predictor $\tf_{M,\cI_k^{\pi}}^{\WOR}$ tends to be more prevalent \citep{dobriban_sheng_2020,mucke_reiss_rungenhagen_klein_2022}.
Owing to its popularity and the fact that SRSWOR is superior to SRSWR in general, in \Cref{sec:bagging-without-replacement}, we primarily focus on $\tf_{M,\cI_k^{\pi}}^{\WOR}$. In what follows, we refer to $\tf_{M,\cI_{k}^{\pi}}^{\WR}$ and $\tf_{M,\cI_k^{\pi}}^{\WOR}$ as \emph{\splagging with and without replacement.} 
For the sake of simplicity, we define $\tf_{M, \cI_{k}^\pi}^{\WOR}$ as
$\tf_{\lfloor n / k \rfloor, \cI_k^{\pi}}^{\WOR}$ if $M > \lfloor n / k \rfloor$.
In doing so, we are effectively substituting $M$ with $\min\{M,\lfloor n/k\rfloor\}$.

The results to be discussed below are general and apply to all four versions of the bagged predictors in \eqref{eq:bagged-predictor-subagging}.
Consider the finite population $\{\hf(\bx; \mathcal{D}_I):\,I\in\mathcal{I}_k\}$ or $\{\hf(\bx; \mathcal{D}_I):\, I\in\mathcal{I}_k^{\pi}\}$, but with the data $\mathcal{D}_n$ treated as fixed (non-stochastic). 
We know that $\widetilde{f}_{M,\mathcal{I}_k}^{\WR}(\bx)$ and $\widetilde{f}_{M,\mathcal{I}_k}^{\WOR}(\bx)$ has the same expectation, given by
\[
\widetilde{f}_{\infty, \mathcal{I}_k}(\bx) = \frac{1}{|\mathcal{I}_k|}\sum_{I\in\mathcal{I}_k} \widehat{f}(\bx; \mathcal{D}_I).
\]
However, the variance is smaller for $\widetilde{f}_{M,\mathcal{I}_k}^{\WOR}(\bx)$. 
Using the bias and variance formulas from \citet[Section 2.5]{chaudhuri_2014}, the following result can be derived for the subagged predictors; see \Cref{sec:setup} for a detailed explanation.

\begin{proposition}[Conditional risk decomposition]\label{prop:bagged-predictors-conditional-mse}
    Without any assumptions on the data and the prediction procedure $\hf(\cdot; \cdot)$, we have for every $(\bx, y)\in\mathbb{R}^p\times\mathbb{R}$,
    \begin{align}
        \begin{split}
            \mathbb{E}[(y-\widetilde{f}_{M,\mathcal{I}_k}^{\textup{\WR}}(\bx; \{\mathcal{D}_{I_{\ell}}\}_{\ell = 1}^M) )^2\mid\mathcal{D}_n] &= \mathscr{B}_{\mathcal{I}_k}(\bx, y) + \frac{1}{M}\mathscr{V}_{\mathcal{I}_k}(\bx, y),\\
            \mathbb{E}[(y-\widetilde{f}_{M,\mathcal{I}_k}^{\textup{\WOR}}(\bx; \{\mathcal{D}_{I_{\ell}}\}_{\ell = 1}^M))^2\mid\mathcal{D}_n] &= \mathscr{B}_{\mathcal{I}_k}(\bx, y) + \frac{|\mathcal{I}_k|-M}{|\mathcal{I}_k|-1}\frac{1}{M}\mathscr{V}_{\mathcal{I}_k}(\bx, y),
        \end{split} \label{eq:I_k-bagged-predictors}\\
        \mbox{where} \quad \mathscr{B}_{\mathcal{I}_k}(\bx, y) = (y - \widetilde{f}_{\infty, \mathcal{I}_k}(\bx))^2, \quad \mbox{and} \quad &\mathscr{V}_{\mathcal{I}_k}(\bx, y) = \frac{1}{|\mathcal{I}_k|}\sum_{I\in\mathcal{I}_k} \left(\hf(\bx; \mathcal{D}_{I}) - \widetilde{f}_{\infty,\mathcal{I}_k}(\bx)\right)^2. \label{eq:bias-and-variance-I_k-bagged-predictors}
    \end{align}
    The results still hold by replacing $\cI_k$ with $\cI_k^{\pi}$. Here in~\eqref{eq:I_k-bagged-predictors}, the expectation is with respect to the randomness of $I_1, \ldots, I_M$ only. 
\end{proposition}

In line with traditional predictive thinking, we care about the performance of our predictors computed on $\cD_n$ on future data from the same distribution $P$.
As we have access to a single dataset $\cD_n$, we consider the behavior of the predictors in terms of the conditional risk, conditional on $\cD_n$.
To be precise, for a predictor $\widehat{f}$ fitted on $\cD_n$ and its subagged predictor $\tf_{M,\cI_k}^{\WR}$ fitted on $\{\mathcal{D}_{I_{\ell}}\}_{\ell=1}^M$, with $I_1,\ldots,I_M$ being $M$ samples of size $k$ from $\cI_k$, the conditional risks (conditional on $\cD_n$) are defined as follows:
\begin{equation}
    \label{eq:unconditional-risk}
    \begin{split}
    R(\widehat{f}; \mathcal{D}_n) 
    &:= 
    \int
    (y-\widehat{f}(\bx; \cD_n))^2
    \, \mathrm{d}P(\bx, y),
    \\
    R(\tf_{M,\cI_k}^{\WR}(\cdot; \{\mathcal{D}_{I_\ell}\}_{\ell=1}^M) ; \, \cD_n) 
    &:= 
    \int
    \EE\left[\left(
    y-\tf_{M, \cI_k}^{\WR}(\bx; \{ \cD_{I_\ell} \}_{\ell = 1}^{M})\right)^2 \mathrel{\Big |} 
    \mathcal{D}_n
    \right]
    \, \mathrm{d}P(\bx, y).
    \end{split}
\end{equation}
The conditional risk of $\tf_{M, \cI_k}^{\WOR}(\cdot; \{ \cD_{I_\ell}\}_{\ell = 1}^{M})$ is defined similarly, and so are the conditional risks for the splagged predictors with and without replacement from $\cI_{k}^{\pi}$ for a fixed permutation $\pi$.
Observe that the conditional risk of the subagged predictor $\tf_{M, \cI_k}^{\WR}(\cdot; \{ \cD_{I_\ell}\}_{\ell = 1}^{M})$ integrates over the randomness of the future observation $(\bx, y)$ as well as the randomness due to the simple random sampling of $I_{\ell}$, $\ell = 1, \dots, M$.
Given that only a single dataset $\cD_n$ is observed in practice and one typically only draws one simple random sample $I_{\ell}$, $\ell = 1, \dots, M$, it is also insightful to consider an alternate version of the conditional risk that ignores the expectation over the simple random sample:
\begin{equation}
    \label{eq:conditional-risk}
    R(\tf_{M,\cI_k}^{\WR}(\cdot; \{\mathcal{D}_{I_\ell}\}_{\ell=1}^M) ; \, \cD_n,
    \{I_{\ell}\}_{\ell = 1}^{M}) 
    := 
    \int
    \left(
    y-\tf_{M, \cI_k}^{\WR}(\bx; \{ \cD_{I_\ell} \}_{\ell = 1}^{M})
    \right)^2 
    \, \mathrm{d}P(\bx, y).
\end{equation}
We call the former type of conditional risk (conditional on $\cD_n$) as \emph{data conditional risk} and the latter type of conditional risk (conditional on $\cD_n$ and $\{ I_{\ell} \}_{\ell=1}^{M}$) as \emph{subsample conditional risk}.

\Cref{prop:bagged-predictors-conditional-mse} implies that the data conditional risks of the predictors $\widetilde{f}_{M,\mathcal{I}_k}^{\WR}(\cdot)$ and $\widetilde{f}_{M,\mathcal{I}_k}^{\WOR}(\cdot)$ can be written as
\begin{equation}\label{eq:representation-conditional-risk}
    \begin{split}
    R(\widetilde{f}_M; \mathcal{D}_n)
    &= 
    \int
    \mathscr{B}_{\mathcal{I}_k}(\bx, y)
    \, \mathrm{d}P(\bx, y)
    + K_{|\mathcal{I}_k|,M}\frac{1}{M}
    \int
    \mathscr{V}_{\mathcal{I}_k}(\bx, y)
    \, \mathrm{d}P(\bx, y) \\
    &=
    R(\tf_{\infty}; \cD_n)
     + \frac{K_{|\mathcal{I}_k|,M}}{M}
    \int
    \mathscr{V}_{\mathcal{I}_k}(\bx, y)
    \, \mathrm{d}P(\bx, y),
    \end{split}
\end{equation}
where for $N \geq 1$, $K_{N, M}$ is defined as
\begin{equation}
    \label{eq:KNM-def}
    K_{N, M}
    =
    \begin{cases}
    1 & \text{if } \tf = \tf_{M, \cI_{k}}^{\WR} \\
    (N-M)_{+}/(N-1) & \text{if } \tf = \tf_{M, \cI_{k}}^{\WOR}. \\
    \end{cases}
\end{equation}
The advantage of the representation~\eqref{eq:representation-conditional-risk} for the data conditional risk of $\widetilde{f}_{M,\mathcal{I}_k}^{\WR}(\cdot)$ and $\widetilde{f}_{M,\mathcal{I}_k}^{\WOR}(\cdot)$ is that it allows us to obtain the limiting behavior of their risks for any $M \ge 1$ by just studying their limiting risk behavior for $M = 1$ and $M = 2$. This is trivially shown by solving a system of linear equations in two variables and is formalized in the following result.

\begin{proposition}
[Data conditional risk for arbitrary $M$]
\label{prop:limiting-risk-for-arbitrary-M}
    Let $R(\tf_M; \mathcal{D}_n)$ be as defined in \eqref{eq:unconditional-risk}.
    For $\tf_M\in\{\tf_{M,\mathcal{I}_k}^{\textup{\WR}},
    \tf_{M,\mathcal{I}_k}^{\textup{\WOR}},
    \tf_{M,\mathcal{I}_k^{\pi}}^{\textup{\WR}},
    \tf_{M,\mathcal{I}_k^{\pi}}^{\textup{\WOR}}\}$,
    suppose there exist non-stochastic numbers $a_1$ and $a_2$ such that as $n\to\infty$,
    \begin{align}
        |R(\tf_M; \mathcal{D}_n) - a_M| \asto 0, \quad \text{for} \quad M = 1, 2,    \label{eq:data_cond_risk_M12}
    \end{align}
    where the almost sure convergence is with respect to the randomness of $\cD_n$.
    Then, we have\footnote{For SRSWOR, supremum over $M\in\mathbb{N}$ should be understood as either $M \le |\cI_k|$ or $M \le |\cI_k^{\pi}|$ depending on whether $\tf_M$ is $\tf_{M,\cI_k}^{\WOR}$ or $\tf_{M,\cI_k^{\pi}}^{\WOR}$.
    The same convention is used for all the other results in this section.}
    \begin{equation}\label{eq:guarantees-for-WR-WOR}
       \sup_{M\in\NN}\left|R(\tf_M; \mathcal{D}_n) - \left[(2a_2 - a_1) + \frac{2(a_1 - a_2)}{M}\right]\right| \asto 0.
    \end{equation}
\end{proposition}

Note that according to \Cref{prop:bagged-predictors-conditional-mse}, we have $a_1 \ge a_2$, irrespective of the prediction procedure. In \Cref{prop:limiting-risk-for-arbitrary-M}, if $a_1>a_2$ (instead of just $a_1 \ge a_2$), then the asymptotic approximations of the conditional risk $R(\tf_M; \mathcal{D}_n)$ are strictly decreasing in $M$.
Similarly, we can also derive the asymptotic subsample conditional risk defined in \eqref{eq:conditional-risk} of subagged predictors with an arbitrary number of bags $M$ if we know the limiting risk for $M=1$ and $M=2$, as summarized in \Cref{prop:limiting-risk-for-arbitrary-M-cond} below.

\begin{proposition}
[Subsample conditional risk for arbitrary $M$]
\label{prop:limiting-risk-for-arbitrary-M-cond}
    Let $R(\tf_M; \mathcal{D}_n,\{I_{\ell}\}_{\ell=1}^M)$ be as defined in \eqref{eq:conditional-risk}.
    For $\tf_M\in\{\tf_{M,\mathcal{I}_k}^{\textup{\WR}},
    \tf_{M,\mathcal{I}_k}^{\textup{\WOR}},
    \tf_{M,\mathcal{I}_k^{\pi}}^{\textup{\WR}},$ $\tf_{M,\mathcal{I}_k^{\pi}}^{\textup{\WOR}}\}$,
    suppose there exist non-stochastic numbers $b_1$ and $b_2$
    such that
    \begin{align}
        |R(\tf_1; \mathcal{D}_n, I^{(n)}) - b_1| &\asto 0,
        \quad \text{for all } I^{(n)} \in \cI_k\text{ or }\cI_k^{\pi},
        \label{eq:subsample_cond_risk_M1} \\
        |R(\tf_2; \mathcal{D}_n, \{ I_{1}^{(n)}, I_{2}^{(n)} \}) - b_2| &\asto 0, 
        \quad \text{for random samples } I_{1}^{(n)}, I_{2}^{(n)}\footnotemark,
        \label{eq:subsample_cond_risk_M12}
    \end{align} 
    where the almost sure convergence is with respect to the randomness of both $\cD_n$ and $I$ (or $I_1$, $I_2$).
    For any $M\in\NN$, suppose $\{ I_{\ell} \}_{\ell = 1}^{M}$ is a simple random sample according to the definition of $\tf_M$.\footnotetext{According to \eqref{eq:bagged-predictor-subagging}, $I_1^{(n)}$ and $I_2^{(n)}$ are drawn using SRSWR if $\tf_{M} \in \{ \tf_{M, \cI_k}^\WR, \tf_{M, \cI_k^\pi}^\WR \}$ and SRSWOR if $\tf_{M} \in \{ \tf_{M, \cI_k}^\WOR, \tf_{M, \cI_k^\pi}^\WOR \}$. From now on, for notational simplicity, we drop the dependence on $n$ and simply write $I_1$ and $I_2$.}
    Then \begin{equation}\label{eq:guarantees-for-WR-WOR-cond}
       \sup_{M\in\NN}\left|R(\tf_M; \mathcal{D}_n, \{ I_{\ell} \}_{\ell = 1}^{M}) - \left[(2b_2 - b_1) + \frac{2(b_1 - b_2)}{M}\right]\right| 
       \pto 0.
    \end{equation}
\end{proposition}

We make a couple of remarks on 
the assumption of \Cref{prop:limiting-risk-for-arbitrary-M-cond} below.

\begin{remark}[On the requirement \eqref{eq:subsample_cond_risk_M1}]
    Requirement \eqref{eq:subsample_cond_risk_M1}
    might on surface seem stronger
    as it requires almost sure convergence to hold for all $I \in \cI_k$.
    However, recall that,
    for any fixed $I \in \cI_k$,
    $\tf_{1, \cI_k}(\cdot; \cD_{I})$ is the same as
    the prediction procedure $\hf$ computed on the subset $\cD_{I}$
    with cardinality $k$.
    This implies that if the original prediction procedure
    satisfies almost sure convergence as the sample size on
    which it is trained goes to $\infty$,
    then as $k \to \infty$,
    the requirement \eqref{eq:subsample_cond_risk_M1}
    is satisfied for every fixed $I \in \cI_k$.
\end{remark}

\begin{remark}[Role of squared loss]
    In \Cref{prop:limiting-risk-for-arbitrary-M,prop:limiting-risk-for-arbitrary-M-cond}, 
    we observed that only the limiting risks for $M = 1$ and $M = 2$ matter.
    This is because the data conditional risk can be decomposed as
    \begin{align*}
        R(\tf_{M,\cI_k} ; \, \cD_n) &= - \left(1-\frac{2}{M}\right) R(\tf_{1,\cI_k}; \, \cD_n) + 
        2
        \left(
        1 - \frac{1}{M}
        \right)
        R(\tf_{2,\cI_k} ; \, \cD_n) .
    \end{align*}
    The subsample conditional risk admits similar decomposition as well.
    See \Cref{sec:app-general-predictor} for the derivations for both of them.
    Essentially, the interaction of subsampled datasets is only up to order two.
    For other loss functions, this may not be true. 
    However, a simple monotonicity property and bounds can be obtained
    for a large class of loss functions as shown in the next proposition.
    It is also worth mentioning
    that while 
    \Cref{prop:limiting-risk-for-arbitrary-M,prop:limiting-risk-for-arbitrary-M-cond} 
    are derived under the assumption that
    the distribution of the out-of-sample test point $(\bx, y)$,
    $P(\bx, y)$, is the same as the distribution
    of the training data,
    it is not difficult to see that the same conclusions
    hold for a test point sampled from any arbitrary distribution.
    The results are thus also applicable to out-of-distribution scenarios.
\end{remark}

\begin{proposition}[Convex, strongly-convex, and smooth loss functions]
    \label{prop:convex-sconvex-smooth}
    For any loss function $L : \RR \times \RR \to \RR$, every $(\bx, y)\in\mathbb{R}^p\times\mathbb{R}$,
    and for $\tf_M\in\{\tf_{M,\mathcal{I}_k}^{\textup{\WR}},
    \tf_{M,\mathcal{I}_k}^{\textup{\WOR}},
    \tf_{M,\mathcal{I}_k^{\pi}}^{\textup{\WR}},
    \tf_{M,\mathcal{I}_k^{\pi}}^{\textup{\WOR}}\}$,
    define
    \[
        R(\tf_M; \cD_n)
        =
        \int
        \EE[L(y, \tf_M(\bx; \{ \cD_{I_{\ell}} \}_{\ell=1}^{M})) \mid \cD_n]
        \, \mathrm{d}P(\bx, y).
    \]
    If $L : \RR \times \RR \to \RR$
    is convex in the second argument\footnote{
    Recall that a function $f : \RR \to \RR$ is convex if
    $f(t x_1 + (1 - t) x_2) \le t f(x_1) + (1 - t) f(x_2)$
    for all $x_1, x_2 \in \RR$ and $t \in [0, 1]$.
    },
    then
    $R(\tf_M, \cD_n)$ is non-increasing in $M \ge 1$,
    i.e., $R(\tf_{M+1}; \cD_n) \le R(\tf_M; \cD_n)$.
    Alternatively, if 
    there exists $\underline{m}, \overline{m} \in \RR$
    such that
    $L(\cdot, \cdot)$ is $\underline{m}$-strongly convex
    and
    $\overline{m}$-smooth in the second argument\footnote{A function $g:\mathbb{R}\to\mathbb{R}$ is said to be $\lambda_1$-strongly convex if $x\mapsto f(x) - \lambda_1/2 x^2$ is convex. It is called a $\lambda_2$-smooth function if the derivative of $f$ is $\lambda_2$-Lipschitz (i.e., $|f'(x_1) - f'(x_2)| \le \lambda_2|x_1 - x_2|$ for all $x_1, x_2$).},
    then for 
    $\tf_M \in 
   \{\tf_{M,\mathcal{I}_k}^{\textup{\WR}},
    \tf_{M,\mathcal{I}_k}^{\textup{\WOR}} \}, 
    $
    \begin{align}
       \frac{\underline{m}K_{|\mathcal{I}_k|,M}}{2M}
    \int
    \mathscr{V}_{\mathcal{I}_k}(\bx, y)
    \, \mathrm{d}P(\bx, y) ~\leq~ R(\tf_{M}; \cD_n) -R(\tf_{\infty}; \cD_n) ~\leq~
     \frac{\overline{m}K_{|\mathcal{I}_k|,M}}{2M}
    \int
    \mathscr{V}_{\mathcal{I}_k}(\bx, y)
    \, \mathrm{d}P(\bx, y) \label{eq:sconvex-smooth-bounds}
    \end{align}
    with $K_{N, M}$ defined in \eqref{eq:KNM-def}.
    The inequalities in \eqref{eq:sconvex-smooth-bounds}
    continue to hold for $\tf_M \in 
     \{ \widetilde{f}_{M,\mathcal{I}_k^{\pi}}^{\textup{\WR}}, \tf_{M,\mathcal{I}_k^{\pi}}^{\textup{\WOR}}\}
    $, with $K_{|\cI_k|, M}$ and $\mathscr{V}_{\cI_k}$
    replaced with $K_{|\cI_k^\pi|, M}$ and $\mathscr{V}_{\cI_k^\pi}$, respectively.
\end{proposition}

\begin{remark}
    [Comparison with squared risk]
    In \Cref{prop:convex-sconvex-smooth}, $R(\tf_{\infty}; \cD_n)$ is defined with respect to a general loss function $L$.
    Note that the upper and lower bounds of \eqref{eq:sconvex-smooth-bounds} 
    do not depend on the loss function and are of the same form as the second term on the right-hand side of
    \eqref{eq:representation-conditional-risk},
    except for constant multiples of $\underline{m}/2$ and $\overline{m}/2$.
    Furthermore, even when the loss function is $\overline{m}$-smooth but not convex in the second argument, the data conditional risk
    of $\tf_M$ can be sandwiched between
    $c_1 - \overline{m} c_2/ M$
    and $c_1 + \overline{m} c_2 / M$
    for two data-dependent quantities $c_1$ and $c_2$.
    Beyond the squared error loss,
    several popular loss functions
    used in learning
    satisfy the conditions of \Cref{prop:convex-sconvex-smooth};
    for example, 
    the Logistic loss,
    the Huber loss, 
    among others.
\end{remark}

The next lemma connects the data conditional risk with the subsample conditional risk for $M=1,2$.
In practice, the ingredient predictor is fitted on the subsampled datasets, on which the subsample conditional risk is evaluated.
By \Cref{lem:risk_general_predictor_M12}, we are able to infer the data conditional risk based on the subsample conditional risk for the simple case when $M=1,2$.

\begin{lemma}[Transferring from subsample conditional to data conditional risk for $M = 1, 2$]\label{lem:risk_general_predictor_M12}
    Suppose the conditions in \Cref{prop:limiting-risk-for-arbitrary-M-cond} hold, then \eqref{eq:data_cond_risk_M12} holds with $a_M=b_M$ for $M=1,2$.
   Consequently,
  the conclusions of \Cref{prop:limiting-risk-for-arbitrary-M} hold.    
\end{lemma}

It is worth noting that in the proof of \Cref{lem:risk_general_predictor_M12}, 
we only use the convexity of the square loss function.
Therefore, analogous results can be obtained for other convex loss functions as long as the limiting subsample conditional risks exist for $M=1,2$.

\subsection{General reduction strategy}

Finally, combining \Cref{prop:limiting-risk-for-arbitrary-M}, \Cref{prop:limiting-risk-for-arbitrary-M-cond}, and \Cref{lem:risk_general_predictor_M12} yields a general strategy for obtaining both limiting subsample and data conditional risks for an arbitrary number $M$ of bags. 
The end-to-end result is presented in the form of \Cref{thm:risk_general_predictor}. 
This theorem establishes that it is sufficient to obtain the limiting subsample conditional risks for $M=1,2$; see \Cref{fig:risk_general}.

\begin{theorem}[Transferring from subsample conditional to data conditional for general $M$]\label{thm:risk_general_predictor}
    Suppose the conditions \eqref{eq:subsample_cond_risk_M1} and \eqref{eq:subsample_cond_risk_M12} hold, then the conclusions in
    \Cref{prop:limiting-risk-for-arbitrary-M-cond,prop:limiting-risk-for-arbitrary-M} hold.
\end{theorem}
For general predictors, both the data conditional risk and the subsample conditional risk for $M=1$ (required for \eqref{eq:subsample_cond_risk_M1} to hold) are typically available from known results.
In such cases, it remains to first derive limiting subsample conditional risk for $M=2$ (required for \eqref{eq:subsample_cond_risk_M12} to hold) depending on the sampling strategies, and then verify the properties of the limiting conditional risks required in \Cref{thm:risk_general_predictor}.
In this paper, we focus on the asymptotic risk characterization for the bagged ridge and ridgeless predictors and verify the conditions \eqref{eq:subsample_cond_risk_M1} and \eqref{eq:subsample_cond_risk_M12} in the next section.

\begin{figure}[!t]
    \centering
    \begin{tikzpicture}
        \node[draw,
        text width=4.6cm,
        align=center](subsampleM12) at (0,0) {$R(\tf_M; \mathcal{D}_n, \{ I_{\ell} \}_{\ell = 1}^{M})$, $M=1,2$};
        
        \node[right of =subsampleM12,
            text width=4.6cm, 
            align=center,
            node distance=10.0cm, draw](subsample) at (0,0) {$R(\tf_M; \mathcal{D}_n, \{ I_{\ell} \}_{\ell = 1}^{M})$, $M\in\NN$};
        
        \node[below of =subsampleM12, 
            text width=4.6cm, 
            align=center,
            node distance=2.0cm, draw](dataM12) at (0,0) {$R(\tf_M; \mathcal{D}_n)$, $M=1,2$};
        
        \node[right of =dataM12,
            text width=4.6cm,
            align=center,
            node distance=10.0cm, draw](data) at (0,-2) {$R(\tf_M; \mathcal{D}_n)$, $M\in\NN$};
        
        \draw[-triangle 45] (subsampleM12) to node [text width=2.5cm,midway,above,align=center] {\Cref{prop:limiting-risk-for-arbitrary-M-cond}} (subsample);
        
        \draw[-triangle 45] (subsampleM12) to node [text width=2.5cm,midway,right,align=center] {\Cref{lem:risk_general_predictor_M12}} (dataM12);
        
        \draw[-triangle 45] (dataM12) to node [text width=2.5cm,midway,below,align=center] {\Cref{prop:limiting-risk-for-arbitrary-M}} (data);
        
        \draw[-triangle 45] (subsampleM12.east) to node [text width=2.5cm,midway,above right,align=center] {\Cref{thm:risk_general_predictor}} (data.west);

    \end{tikzpicture}
    \caption{A general reduction strategy for obtaining limiting risks of subagged predictors with $M$ bags.}
    \label{fig:risk_general}
\end{figure}
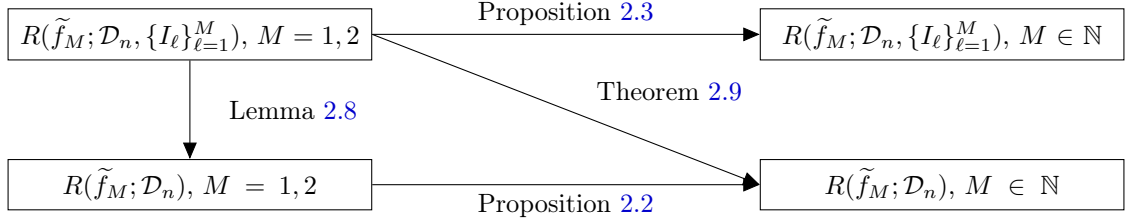
    
\section{Bagging ridge and ridgeless predictors}
\label{sec:ridge-ridgeless-predictors}

In this section, we adopt the reduction strategy proposed in \Cref{sec:general-predictors} to characterize the risk of subagged ridge and ridgeless predictors.
The formal definitions of these predictors and data assumptions imposed for our results are given in \Cref{subsec:predictor_ass}. 
Subsequently, the risk characterizations for subagging and \splagging are presented in \Cref{sec:bagging-with-replacement} and \Cref{sec:bagging-without-replacement}, respectively.

\subsection{Predictors and assumptions}\label{subsec:predictor_ass}

    Consider a dataset $\mathcal{D}_n = \{(\bx_1, y_1), \ldots, (\bx_n, y_n)\}$ consisting of random vectors in $\RR^{p} \times \RR$.
    Let $\bX\in\RR^{n\times p}$ denote the corresponding feature matrix whose $j$-th row contains $\bx_j^\top$, and let $\by\in\RR^n$ denote the corresponding response vector whose $j$-th entry contains $y_j$.
    For any index set $I\subseteq\{1, 2, \ldots, n\}$, let $\mathcal{D}_{I} = \{(\bx_j, y_j):\, j \in I\}$ be a subsampled dataset, and let $\bL \in \RR^{n \times n}$ denote a diagonal matrix such that $L_{jj} = 1$ if and only if $j \in I$. 
    
    Recall that the \emph{ridge} estimator with regularization parameter $\lambda>0$ fitted on $\cD_I$ is defined as
    \begin{align}
        \betaridge(\cD_I) 
        &= \argmin\limits_{\bbeta\in\RR^p}
        \frac{1}{ | I |}
        \sum_{j \in I} (y_j  - \bx_j^\top \bbeta)^2 
        + \lambda \| \bbeta \|_2^2 \notag \\
        &= (\bX^{\top} \bL \bX/ |\cD_I| + \lambda\bI_p)^{-1}(\bX^{\top} \bL \by/|\cD_I|) \notag.
    \end{align}
    The associated ridge predictor is given by
    $
        \hat{f}_{\lambda}(\bx;\cD_I) = \bx^{\top}\betaridge(\cD_I).
    $
    The \emph{ridgeless} estimator is the limiting estimator $\hbeta_{\lambda}(\cD_I)$ as $\lambda \to 0^{+}$.
    When $|\cD_I| \ge p$, and assuming that the $p$ feature vectors are linearly independent in $\RR^{p}$, it is simply the least squares estimator:
    \[
        \hbeta_0(\cD_I)
        = (\bX^\top \bL \bX / |\cD_I|)^{-1} (\bX^\top \bL \bY / |\cD_I|).
    \]
    When $|\cD_I| < p$, it is the minimum $\ell_2$-norm least squares estimator:
    \begin{align}
        \betamntwo(\cD_I) &= \argmin\limits_{\bbeta'\in\RR^p} \left\{\|\bbeta'\|_2
        \mathrel{\Big |}
        \bbeta' \in\argmin\limits_{\bbeta\in\RR^p}\sum_{j\in I}(y_j - \bx_j^{\top}\bbeta)^2\right\} \notag \\
        &= (\bX^{\top} \bL \bX/|\cD_I|)^{+}(\bX^{\top} \bL \by/|\cD_I|)\notag.
    \end{align}
    Here $\bA^+$ denotes the Moore-Penrose inverse of matrix $\bA$.
    Assuming that $\cD_I$ has $|\cD_I|$ linearly independent observation vectors in $\RR^{p}$, this estimator also interpolates the data, i.e., we have $y_j = \bx_j^\top \hbeta_0(\cD_I)$ for $j \in I$, and has the minimum $\ell_2$-norm among all interpolators.
    The associated ridgeless predictor is again given by
    $
        \hat{f}_0(\bx;\cD_n) = \bx^{\top}\betamntwo(\cD_n).
    $
    
    Given their relevance to the subagged predictors studied in the literature, we will primarily focus on only two of the four subagged predictors as defined in \eqref{eq:bagged-predictor-subagging}, although the implications for the other two can be trivially obtained.
    For $\lambda\geq0$, the subagged and splagged predictors 
    respectively are defined as
    \begin{align}
        \begin{split}
            \tilde{f}_{M,\cI_k}^{\textup{\WR}}(\bx;\cD_n) &= \bx^{\top}\tbetaridge{M},\qquad I_{1},\ldots,I_M\overset{\texttt{SRSWR}}{\sim} \cI_k,\\
            \tilde{f}_{M,\cI_k^{\pi}}^{\textup{\WOR}}(\bx;\cD_n) &= \bx^{\top}\tbetaridge{M},\qquad I_{1},\ldots,I_M\overset{\texttt{SRSWOR}}{\sim} \cI_k^{\pi},
        \end{split}
         \label{eq:ingredient-predictor}
    \end{align}
    where $\tbetaridge{M}=M^{-1}\sum_{\ell=1}^M\betaridge(\cD_{I_{\ell}})$.
    For $M > |\cI_k^\pi|$, the splagged predictor is defined to be the predictor with $M = |\cI_k^\pi|$.
    When $\lambda=0$, the base predictors become the ridgeless predictors.

   We impose the following Assumptions~\ref{asm:rmt-feat}-\ref{asm:spectrum-spectrumsignproj-conv} on the dataset $\cD_n$ to characterize the risk. These assumptions are standard in the study of the ridge regression under proportional asymptotics; see, e.g., \citet{hastie2022surprises}.

    \begin{assumption}
        \label{asm:rmt-feat}
        The feature vectors $\bx_i \in \RR^{p}$, $i = 1, \dots, n$,
        multiplicatively decompose as $\bx_i = \bSigma^{1/2} \bz_i$,
        where $\bSigma \in \RR^{p \times p}$ is a positive semidefinite matrix
        and $\bz_i \in \RR^{p}$ is a random vector
        containing i.i.d.\ entries with
        mean $0$, variance $1$,
        and bounded moment of order $4 + \delta$
        for some $\delta > 0$.
    \end{assumption}
    
   \begin{assumption}
        \label{asm:lin-mod}
        The response variables $y_i \in \RR$, $i = 1, \dots, n$, 
        additively decompose as $y_i = \bx_i^\top \bbeta_0 + \epsilon_i$,
        where $\bbeta_0 \in \RR^{p}$ is an unknown signal vector 
        and $\epsilon_i$ is an unobserved error
        that is assumed to be independent of $\bx_i$
        with 
        mean $0$,
        variance $\sigma^2$,
        and bounded moment of order $4 + \delta$
        for some $\delta > 0$.
   \end{assumption}
   
   \begin{assumption}
        \label{asm:signal-bounded-norm}
        The signal vector $\bbeta_0$ has bounded limiting energy, i.e., $\lim_{p \to \infty} \| \bbeta_0 \|_2^2 = \rho^2 < \infty$.
   \end{assumption}
    
    \begin{assumption}
        \label{asm:covariance-bounded-eigvals}
        There exist real numbers $r_{\min}$ and $r_{\max}$
        independent of $p$
        with $0 < r_{\min} \le r_{\max} < \infty$
        such that
        $r_{\min} \bI_{p}~\preceq~\bSigma~\preceq~r_{\max} \bI_{p}$.
    \end{assumption}
    
    \begin{assumption}
        \label{asm:spectrum-spectrumsignproj-conv}
        Let $\bSigma = \bW \bR \bW^\top$
        denote the eigenvalue decomposition of the covariance matrix $\bSigma$,
        where 
        $\bR \in \RR^{p \times p}$ is a diagonal matrix
        containing eigenvalues (in non-increasing order) 
        $r_1 \ge r_2 \ge \dots \ge r_{p} \ge 0$,
        and
        $\bW~\in~\RR^{p \times p}$ is an orthonormal matrix
        containing the associated eigenvectors 
        $\bw_1, \bw_2, \dots, \bw_{p}~\in~\RR^{p}$.
        Let $H_{p}$ denote the empirical spectral distribution
        of $\bSigma$
        (supported on $\RR_{> 0}$)
        whose value at any $r \in \RR$ is given by
        \[
            H_{p}(r)
            = \tfrac{1}{p} \sum_{i=1}^{p} \ind_{\{r_i \le r\}}.
        \]
        Let $G_{p}$ denote a certain distribution (supported on $\RR_{> 0}$)
        that encodes the components of the signal vector $\bbeta_0$ in the eigenbasis of $\bSigma$
        via the distribution of (squared) projection of $\bbeta_0$
        along the eigenvectors $\bw_j, 1 \le j \le p$,
        whose value at any $r \in \RR$ is given by
        \[
            G_{p}(r)
            = \tfrac{1}{\| \bbeta_0 \|_2^2} \sum_{i = 1}^{p} (\bbeta_0^\top \bw_i)^2 \, \ind_{\{ r_i \le r \}}.
        \]
        Assume there exist fixed distributions $H$ and $G$ such that $H_{p} \dto H$ and $G_{p} \dto G$ as $p  \to \infty$.    
    \end{assumption}

\subsection{Subagging with replacement}
\label{sec:bagging-with-replacement}

In this section, we delve into the risk asymptotics and properties for subagging.
In \Cref{subsec:bagging-with-replacement-main}, we provide exact risk characterization of subagged ridge and ridgeless predictors.
The monotonicity properties of the asymptotic bias and variance components of the risk are presented in \Cref{subsec:monotonicity-M-with-replacement}.

\subsubsection{Risk characterization}
\label{subsec:bagging-with-replacement-main}

In preparation for our first result on the risk characterization of subagged ridge and ridgeless predictors, let us establish some notations.
We will analyze the subagged predictors (with $M$ bags) in the proportional asymptotics regime, in which the original data aspect ratio ($p / n$) converges to $\phi \in (0, \infty)$ as $n, p \to \infty$, and the subsample data aspect ratio ($p/k$) converges to $\phi_s$ as $k, p \to \infty$. 
Because $k \le n$, $\phi_s$ is always no less than $\phi$.
    
A fixed-point equation defines one of the key quantities that recurs throughout our analysis of subagged ridge predictors.
Such fixed point equations have appeared in the literature before in the context of risk analysis of regularized estimators under proportional asymptotics regime. 
For instance, see \citet{dobriban_wager_2018,hastie2022surprises,mei_montanari_2022} in the context of ridge regression. 
For other $M$-estimators, see \cite{thrampoulidis_oymak_hassibi_2015,thrampoulidis_abbasi_hassibi_2018}, \cite{sur_chen_candes_2019}, \cite{karoui_2013,elkaroui_2018}, \cite{miolane_montanari_2021}, among others.

For any $\lambda > 0$ and $\theta > 0$, define $v(-\lambda; \theta)$ as the unique nonnegative solution to the fixed-point equation:
\begin{align}
    \frac{1}{v(-\lambda;\theta)} =\lambda+ \theta\int\frac{r}{1 + v(-\lambda;\theta)r} \, \rd H(r), \label{eq:fixed-point-ridge}
\end{align}
and for $\lambda = 0$, $\theta > 1$, we define:
\begin{align}
    v(0; \theta) = 
    \begin{cases}
        \lim\limits_{\lambda \to 0^{+}} v(-\lambda; \theta), & \text{if }\theta>1\\
        +\infty, & \text{if }\theta\in(0,1] .
    \end{cases} \label{eq:v0-def}
\end{align}
    
The fact that the fixed-point equation \eqref{eq:fixed-point-ridge} has a unique nonnegative solution is well known in the random matrix theory literature. See, e.g., \citet{bai2010spectral,couillet2011random}.
For completeness, we also provide a proof in \Cref{append:det-equi-resol}.
The existence of the limit of $v(-\lambda;\theta)$ as $\lambda\rightarrow0^+$ is due to the fact that $v(-\lambda;\theta)$ is monotonically decreasing in $\lambda>0$ \citep[Lemma S.6.15 (4)]{patil2022mitigating}.
Additionally, we define non-negative constants $\tv(-\lambda;\vartheta,\theta)$ and $\tc(-\lambda;\theta)$ via the following equations:
\begin{align}
    \tv(-\lambda;\vartheta,\theta) = \frac{\vartheta \int{r^2}{(1+ v(-\lambda; \theta)r)^{-2}}\,\rd H(r)}{v(-\lambda; \theta)^{-2}-\vartheta \int{r^2}{(1+ v(-\lambda; \theta)r)^{-2}}\,\rd H(r)},\quad 
    \tc(-\lambda;\theta)& =\int \frac{r}{(1 + v(-\lambda; \theta)r)^2} \,\rd G(r). \label{eq:tv_tc}
\end{align}

    \begin{theorem}[Risk characterization of subagged ridge and ridgeless predictors]\label{thm:ver-with-replacement}
        Let $\tfWR{M}{\cI_k}$ be the predictor as defined in \eqref{eq:ingredient-predictor} for $\lambda \ge 0$.
        Suppose 
        Assumptions \ref{asm:rmt-feat}-\ref{asm:spectrum-spectrumsignproj-conv} hold for the dataset $\cD_n$.
        Then,
        as $k,n,p\rightarrow\infty$
        such that
        $p/n\rightarrow\phi\in(0,\infty)$ and $p/k\rightarrow\phi_s\in[\phi,\infty]$
        (and $\phi_s \neq 1$ if $\lambda = 0$),
        there exist deterministic functions $\RlamM{M}{\phi}$ for $M \in \NN$, such that for $I_1, \ldots, I_M \overset{\texttt{\textup{SRSWR}}}{\sim} \mathcal{I}_k$, 
        \begin{equation}
            \label{eq:ridge-wr-guarantee}
            \begin{split}
            \sup_{M\in\NN}| R(\tfWR{M}{\cI_k}; \cD_n,\{I_{\ell}\}_{\ell = 1}^{M}) - \RlamM{M}{\phi}| &\pto 0,\\ 
           \sup_{M\in\NN}
            | R(\tfWR{M}{\cI_k}; \cD_n) - \RlamM{M}{\phi}| &\asto 0.
            \end{split}
        \end{equation}
        The guarantee \eqref{eq:ridge-wr-guarantee} also holds true if $\tf_{M, \cI_k}^{\WR}$ is replaced by $\tf_{M, \cI_k}^{\WOR}$.
        Furthermore, the function $\RlamM{M}{\phi}$ decomposes as
        \begin{align}
            \RlamM{M}{\phi} = \sigma^2 + \BlamM{M}{\phi}   + \VlamM{M}{\phi} , \label{eq:risk-det-with-replacement}
        \end{align}
        where the bias and variance terms are given by
        \begin{align}
            \BlamM{M}{\phi}
            &= M^{-1} B_{\lambda}(\phi_s, \phi_s)
            + (1 -  M^{-1}) B_{\lambda}(\phi,\phi_s),
            \label{eq:bias-component-decomposition}\\
            \VlamM{M}{\phi}
            &= M^{-1} V_{\lambda}(\phi_s, \phi_s)
            + (1 - M^{-1})
            V_{\lambda}(\phi,\phi_s)
            \label{eq:var-component-decomposition},
        \end{align}
        and the functions $B_{\lambda}(\cdot, \cdot)$ and $V_{\lambda}(\cdot, \cdot)$ are defined as
        \begin{align}
            B_{\lambda}(\vartheta, \theta)
            = \rho^2 (1+\tv(-\lambda; \vartheta, \theta)) \tc(-\lambda;\theta)~\mbox{and}~
            V_{\lambda}(\vartheta, \theta)
            =\sigma^2\tv (-\lambda; \vartheta, \theta)~\mbox{for}~\theta \in (0, \infty], \, \vartheta \le \theta. \label{eq:Blam_V_lam}
        \end{align}
    \end{theorem}
        
    \Cref{thm:ver-with-replacement} provides precise asymptotics for the data conditional as well as the subsample conditional risks of subagged ridge and ridgeless predictors.
    We have also derived the bias-variance decomposition for the asymptotic risk in \eqref{eq:risk-det-with-replacement}.
    Interestingly, the individual bias term is a convex combination of $B_{\lambda}(\phi_s,\phi_s)$ and $B_{\lambda}(\phi,\phi_s)$, which correspond to the biases for $M = 1$ and $M = \infty$, respectively.
    The same conclusion also holds for the variance term.
    Although the risk behavior for $M=1$ has been studied by \citet{patil2022mitigating}, the risk characterization for general (data-dependent) $M$ is new.
    As we shall see later in \Cref{sec:isotropic_features}, the risk behavior for $M=\infty$ is significantly different from that for $M=1$.
    
    When $\theta>1$, the parameter $v(0; \theta)$ defined in \eqref{eq:v0-def} can also be seen as the unique nonnegative solution to the following fixed-point equation \citep[Lemma S.6.14]{patil2022mitigating}:
        \begin{align}
           \frac{1}{v(0;\theta)} = \theta\int\frac{r}{1 + v(0;\theta)r}\,\rd H(r). \label{eq:fixed-point-ridgeless}
        \end{align}
        When $\theta\in(0,1]$, since $\lim_{\lambda\rightarrow0^+}v(-\lambda; \theta) = \infty$, we have that $\lim_{\lambda\rightarrow0^+}\tc(-\lambda; \theta) = 0$ and $\lim_{\lambda\rightarrow0^+}\tv(-\lambda; \vartheta, \theta) = \vartheta (1-\vartheta)^{-1}$. Therefore, the bias and variance functions in \eqref{eq:Blam_V_lam} 
        for $\vartheta \le \theta$
        reduce to
        \begin{align}
            B_0(\vartheta, \theta)
            &= 
            \begin{dcases}
            0 
            & \theta \in (0, 1]\\
            \rho^2 
            (1 + \tv(0;\vartheta, \theta)) \tc(0;\theta) & \theta \in (1, \infty]
            \end{dcases},
            \qquad
            V_0(\vartheta,\theta)
            =
            \begin{dcases}
                \sigma^2
                \tfrac{\vartheta}{1  - \vartheta}
                & \theta \in (0, 1)\\
                \infty & \theta = 1 \\
                \sigma^2 \tv(0;\vartheta,\theta)
                & \theta \in (1, \infty].
            \end{dcases}\label{remark:ridgeless}
        \end{align}
    As a sanity check when $\vartheta  = \theta$,
    it is easy to see that the bias and variance components
    collapse to that of the minimum $\ell_2$-norm least squares estimator
    with limiting aspect ratio $\theta$.
    
    A few additional remarks on \Cref{thm:ver-with-replacement} follow.

    \begin{remark}
        [Data conditional versus subsample conditional risks]
        \label{rem:conditional-vs-unconditional-risk}
        \Cref{thm:ver-with-replacement} shows that the data conditional risk
        and the subsample conditional risk
        both converge to the same deterministic limit.
         Intuitively, this is expected because the data conditional risk is the average subsample conditional risks
         over all subsamples.
    \end{remark}

    \begin{remark}[Extending theorem to negative regularization]
        For $\lambda<0$, the fixed-point equation \eqref{eq:fixed-point-ridge} may have more than one solution.
        However, there still exists a solution to \eqref{eq:fixed-point-ridge} with which \Cref{thm:ver-with-replacement} holds whenever $\lambda > - (1-\sqrt{\phi})^2 r_{\min}$ where $r_{\min}$ is the uniform lower bound on the smallest eigenvalue of $\bSigma$.
        In this paper, for simplicity we restrict to the case when $\lambda\geq 0$.
    \end{remark}
    
    \begin{remark}
    [The requirement of $\phi_s \neq 1$ for $\lambda = 0$]
        When $\lambda = 0$, the base predictors are ridgeless predictors.
        In this case, the variance function $\theta\mapsto \sV_{0, M}(\vartheta,\theta)$ is unbounded if $M$ is finite and $\theta\rightarrow1$ because $V_0(\theta,\theta)$ in \eqref{remark:ridgeless} diverges as $\theta\rightarrow1$.
        Empirically, this can be explained by the singularity of the empirical covariance matrices with aspect ratios close to 1.
        However, the asymptotic risk for $M = \infty$ is always bounded.
    \end{remark}

    \begin{figure}[!t]
        \centering
        \includegraphics[width=0.85\textwidth]{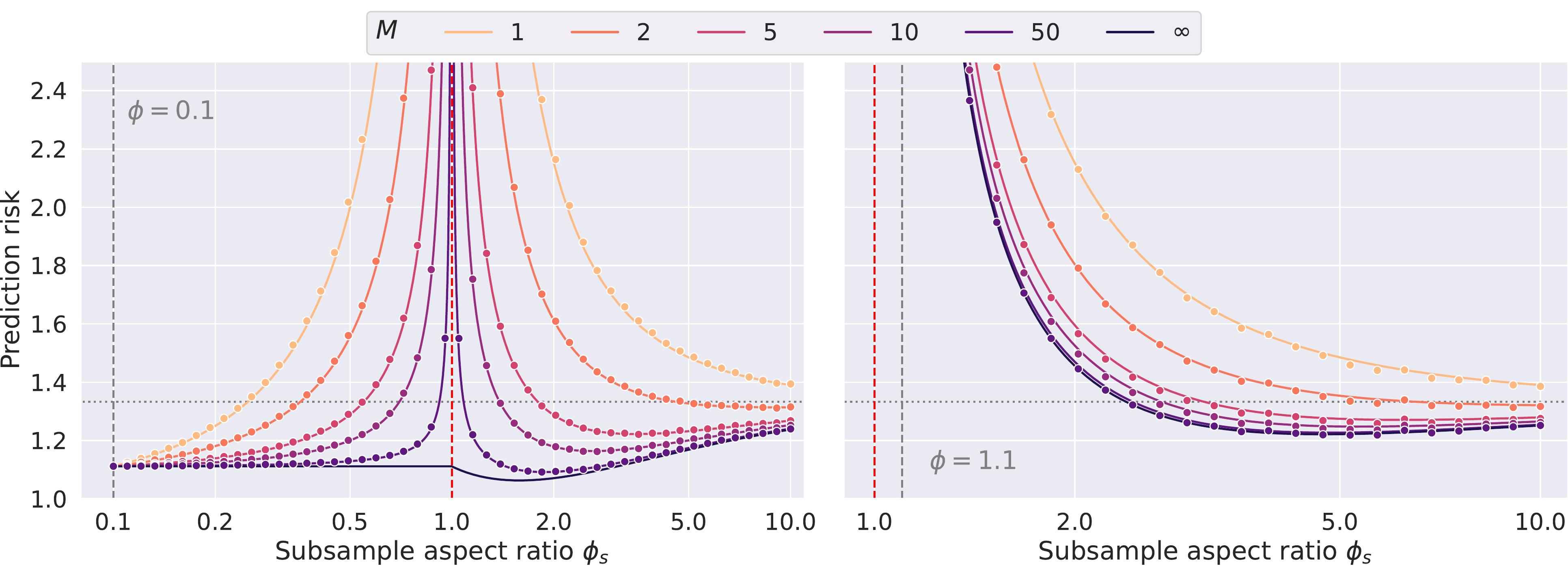}
        \caption{Asymptotic prediction risk curves in \eqref{eq:risk-det-with-replacement} for subagged ridgeless predictors ($\lambda=0$), under model \eqref{eq:model-ar1} when $\rhoar=0.25$ and $\sigma^2=1$, for varying subsample sizes $k=\lfloor p/\phi_s\rfloor$ and numbers of bags $M$.
         The null risk is marked as a dotted line. For each value of $M$, the points denote finite-sample risks averaged over 100 dataset repetitions, with $n=\lfloor p\phi\rfloor$ and $p=500$.
         The left and the right panels correspond to the cases when $p<n$ ($\phi=0.1$) and $p>n$ ($\phi=1.1$), respectively.
        }
        \label{fig:ridgeless-with-replacement-varing-M}
    \end{figure}

    \begin{figure}[!t]
        \centering
        \includegraphics[width=0.85\textwidth]{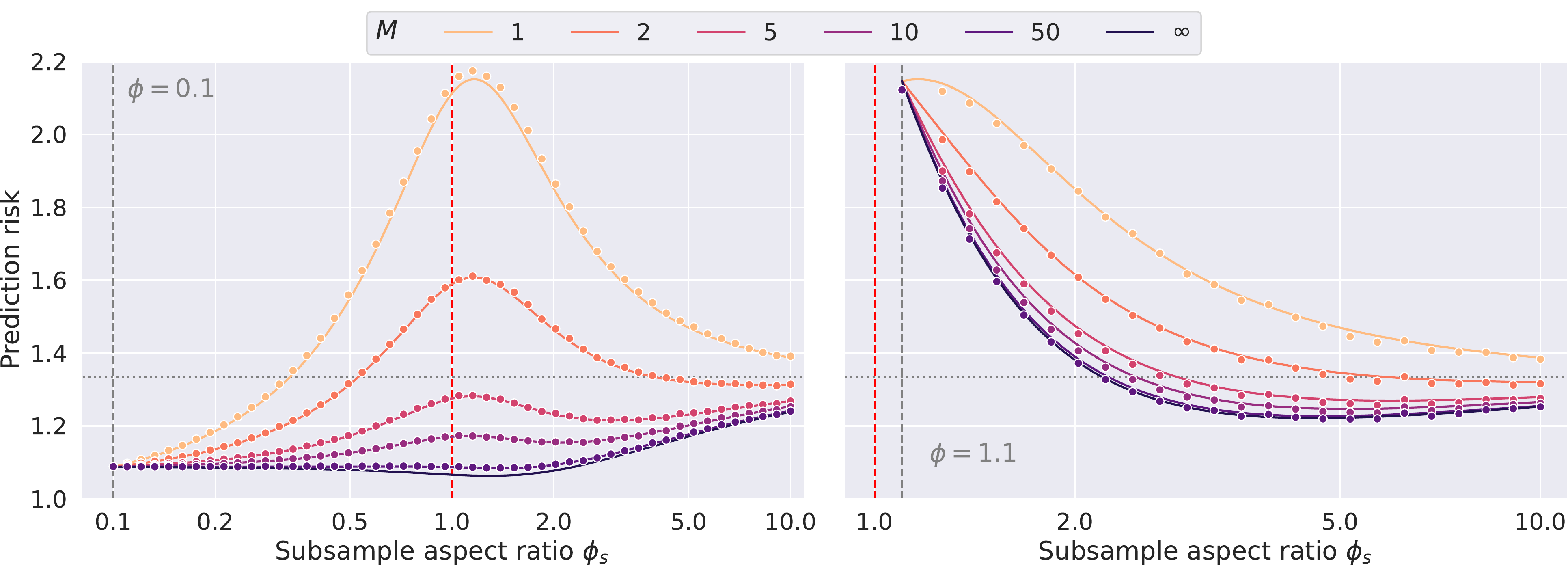}
        \caption{Asymptotic prediction risk curves in \eqref{eq:risk-det-with-replacement} for subagged ridge predictors ($\lambda=0.1$), under model \eqref{eq:model-ar1} when $\rhoar=0.25$ and $\sigma^2=1$, for varying subsample sizes $k=\lfloor p/\phi_s\rfloor$ and numbers of bags $M$.
         The null risk is marked as a dotted line. For each value of $M$, the points denote finite-sample risks averaged over 100 dataset repetitions, with $n=\lfloor p\phi\rfloor$ and $p=500$.
         The left and the right panels correspond to the cases when $p<n$ ($\phi=0.1$) and $p>n$ ($\phi=1.1$), respectively.}
        \label{fig:ridge-with-replacement-varing-M}
    \end{figure}

    \paragraph{Illustration of \Cref{thm:ver-with-replacement}.\hspace{-2mm}}
    Before we delve into the proof outline for \Cref{thm:ver-with-replacement}, we first provide some numerical illustrations under the AR(1) data model. The covariance matrix of an auto-regressive process of order 1 (AR(1)) is denoted by $\bSigma_{\mathrm{ar1}}$, where $(\bSigma_{\mathrm{ar1}})_{ij} = \rhoar^{|i-j|}$ for some parameter $\rhoar\in(0,1)$. The AR(1) data model is defined as follows:
    \begin{align}
        y_i &= \bx_i^{\top}\bbeta_0 + \epsilon_i,
    \quad \bx_i\sim\cN(0,\bSigma_{\mathrm{ar1}}),
    \quad \bbeta_0=\frac{1}{5}\sum_{j=1}^5 \bw_{(j)},
    \quad  \epsilon_i\sim \cN(0,\sigma^2), \tag{M-AR1-LI}\label{eq:model-ar1}
    \end{align}
    where $\bw_{(j)}$ is the eigenvector of $\bSigma_{\mathrm{ar1}}$ associated with the top $j$th eigenvalue $r_{(j)}$.
    From \citet[pp.\ 69-70]{grenander1958toeplitz}, the top $j$-th eigenvalue can be written as $r_{(j)}=(1-\rhoar^2)/(1-2\rhoar\cos \theta_{jp}+\rhoar^2)$ for some $\theta_{jp}\in((j-1)\pi/(p+1), j\pi/(p+1))$.
    Under model \eqref{eq:model-ar1}, the signal strength $\rho^2$ defined in \Cref{asm:signal-bounded-norm} is $5^{-1} (1-\rhoar^2)/(1-\rhoar)^2$, which is the limit of $25^{-1}\sum_{j=1}^5 r_{(j)} $.
    The \eqref{eq:model-ar1} model is thus parameterized by two parameters $\rhoar$ and $\sigma^2$ satisfies \Cref{asm:rmt-feat}-\ref{asm:spectrum-spectrumsignproj-conv}.
    
    \Cref{fig:ridgeless-with-replacement-varing-M,fig:ridge-with-replacement-varing-M} display the limiting risk for the subagged ridgeless predictor and subagged ridge predictor, respectively, with the number of bags $M$ varying from $1$ to $\infty$.
    In the figures, the limiting aspect ratio $\phi$ of the full data is fixed to be either 0.1 or 1.1, corresponding to the cases when $n>p$ and $n<p$, respectively.
    For each case, the limiting aspect ratio $\phi_s$ of each bag takes values in $(\phi, \infty)$.
    We observe that the empirical risks align with the deterministic approximations for both cases, and they are more concentrated around the deterministic approximations as $M$ increases. This is expected as the variance of the subagged predictors reduces with $M$.
    Furthermore, for any fixed $\phi_s$, the asymptotic risk decreases as $M$ increases.

    Due to the non-monotonic risk behavior of the underlying ridge and ridgeless predictors,~\Cref{fig:ridgeless-with-replacement-varing-M,fig:ridge-with-replacement-varing-M} show that the best subsample aspect ratio $(\phi_s)$ in terms of prediction risk might be strictly larger than $\phi$. This holds true for any choice of $M\ge1$. The case of $M = 1$ was already mentioned in~\cite{patil2022mitigating}. 
    This observation is intriguing as it suggests it is better to bag predictors that use even fewer observations than the original data.  
    Similar phenomena are also observed in our simulations with varying signal-to-noise ratios; see \citet[Appendix S.9.1]{patil2022bagging}. 
    We discuss an actionable algorithm for finding the optimal choice of $\phi_s$ in practice in~\Cref{sec:monotonizing-risk-profiles}.
   
   \paragraph{Proof outline of \Cref{thm:ver-with-replacement}.\hspace{-2mm}} 
      The proof of \Cref{thm:ver-with-replacement} employs the reduction strategy discussed in Section 3.
      In particular, we apply \Cref{thm:risk_general_predictor} (subsample conditional for $M = 1$ and $M = 2$ to subsample and data subsample for any $M$) to prove the theorem.
      Below we outline the main steps:
       
  \begin{enumerate}[leftmargin=7mm]
       \item
       The deterministic risk approximation to the subsample conditional risk
       for $M = 1$ can be obtained from the results of \cite{patil2022mitigating}
       that build on those of \cite{hastie2022surprises}.
    
    \item
    Under the linear model, to analyze the subsample conditional risk for $M = 2$, we first decompose it as follows:
    \begin{align}
        &R(\tf_2; \cD_n,\{I_1, I_2\}) \nonumber\\ 
        &= \sigma^2 + \frac{1}{4}\sum_{i=1}^2 (\bbeta_0 - \hbeta(\cD_{I_i}))^{\top}\bSigma (\bbeta_0 - \hbeta(\cD_{I_i})) + \frac{1}{2}
        (\bbeta_0 - \hbeta(\cD_{I_1}))^{\top}\bSigma (\bbeta_0 - \hbeta(\cD_{I_2})) \nonumber \\
        &=
        \frac{\sigma^2}{2}
        + \frac{R(\tf_{1}; \cD_n, I_1) + R(\tf_{1}; \cD_n, I_2)}{4} 
        + \frac{(\bbeta_0 - \hbeta(\cD_{I_1}))^{\top}\Sigma (\bbeta_0 - \hbeta(\cD_{I_2}))}{2}. \label{eq:risk-decomp-m2-wr}
    \end{align}
    The first term in the display above is non-random.
    The asymptotic risk approximation for the second term follows
    from the asymptotics of the subsample conditional risk for $M = 1$.
    The challenging part is the analysis of the final cross term $(\bbeta_0 - \hbeta(\cD_{I_1}))^{\top}\Sigma (\bbeta_0 - \hbeta(\cD_{I_2}))$, due to the non-trivial dependence implied by the overlap between $\cD_{I_1}$ and $\cD_{I_2}$.
    Our strategy to obtain a deterministic approximation for such a term is to write $h(\hbeta(\cD_{I_1}),\hbeta(\cD_{I_2})) = h(\hbeta(\cD_{I_1'}\cup \cD_{I_0}), \hbeta(\cD_{I_2'}\cup \cD_{I_0}))$ for any univariate function $h$.
    Here $I_0 = I_1 \cap I_2$ denotes the indices of the overlap, and $I_j'=I_j\setminus I_0$ for $j=1,2$ are the indices of non-overlapping observations.
    Observe that conditioning on $\cD_{I_0}$, $\cD_{I'_1}$ and $\cD_{I'_2}$ are independent datasets.
    This conditional independence, coupled with the closed-form expression of the ridge predictor, forms a crucial piece in our argument.
    To carry out this program, we derive conditional deterministic equivalence results for ridge resolvents. The resulting new results here are collected in \Cref{sec:asympequi-extended-ridge-resolvents}.  
    
    \item
    To prove the results for the ridgeless predictor, we essentially take the limit as $\lambda \to 0^{+}$ of the deterministic risk approximation for the ridge predictor with regularization $\lambda$.
    This process requires appealing to a uniformity argument in $\lambda$. See \Cref{sec:appendix-with-replacement-ridgeless} for more details.
  
  \end{enumerate}

\subsubsection{Monotonicity of bias and variance in number of bags}\label{subsec:monotonicity-M-with-replacement}

Monotonicity in the number of bags $M$ for both the data conditional risk
$R(\tfWR{M}{\cI_k}; \cD_n)$ and the subsample conditional risk $R(\tfWR{M}{\cI_k}; \cD_n,\{I_{\ell}\}_{\ell = 1}^{M})$
follow from \eqref{eq:representation-conditional-risk}.
In the classical literature of bagging and subagging, however, it has been of interest to better understand the effect of aggregation on not just the risk, but also on the bias and variance.
In this section, we show for the ridge and ridgeless predictors, subagging reduces both the bias and the variance.
Monotonicity of the risk proved in \Cref{thm:ver-with-replacement},
does not imply the monotonicity of asymptotic bias and variance components.
Fortunately, the risk decomposition derived in \Cref{thm:ver-with-replacement} demonstrates that both asymptotic bias and variance components are monotonic in $M$, as summarized below.

\begin{proposition}[Improvement due to subagging]\label{prop:monotonicity-M-with-replacement}
    For all $M=1,2,\ldots$ and $\lambda\in[0,\infty)$, it holds that
    \begin{align}
        \BlamM{\infty}{\phi}\leq \BlamM{M+1}{\phi}\leq \BlamM{M}{\phi} \label{eq:monotonicity-M-B}\\
        \VlamM{\infty}{\phi}\leq \VlamM{M+1}{\phi}\leq \VlamM{M}{\phi}.\label{eq:monotonicity-M-V}
    \end{align}
    The inequalities in \eqref{eq:monotonicity-M-B} are strict whenever $\rho^2>0$ and $\phi_s\in(\phi,\infty)$ (and $\phi_s\neq 1$ when $\lambda=0$), while the inequalities in \eqref{eq:monotonicity-M-V} are strict when $\sigma^2>0$ and $\phi_s\in(\phi,\infty)$ (and $\phi_s\neq 1$ when $\lambda=0$).
    Thus, the asymptotic risk is monotonically decreasing in $M$, i.e., $\RlamM{\infty}{\phi}\leq \RlamM{M+1}{\phi}\leq \RlamM{M}{\phi} $.
\end{proposition} 
       
The monotonicity property in \Cref{prop:monotonicity-M-with-replacement} does not immediately follow from the decomposition of $\BlamM{M}{\phi}$ and $\VlamM{M}{\phi}$
in \eqref{eq:bias-component-decomposition} and \eqref{eq:var-component-decomposition}.
All that is implied by \eqref{eq:bias-component-decomposition} and \eqref{eq:var-component-decomposition} is that $\BlamM{M}{\phi}$ and $\VlamM{M}{\phi}$ either monotonically increase or decrease in $M\geq1$.
However, \Cref{prop:monotonicity-M-with-replacement} confirms that they are both decreasing in $M$. 
We establish this by demonstrating that $\BlamM{1}{\phi}\geq \BlamM{\infty}{\phi}$ and $\VlamM{1}{\phi}\geq \VlamM{\infty}{\phi}$.
Moreover, the proposition explicitly distinguishes the cases of non-increasing and strict decreasing of the bias and variance components.

The monotonicity properties claimed in \Cref{prop:monotonicity-M-with-replacement} are supported by \Cref{fig:ridgeless-bias-var-with-replacement}, which shows the bias and variance components for subagged ridgeless predictors under the model \eqref{eq:model-ar1}.
For a similar illustration for subagged ridge predictors, see \citet[Figure S.7]{patil2022bagging}.

\begin{figure}[!t]
    \centering
    \includegraphics[width=0.85\textwidth]{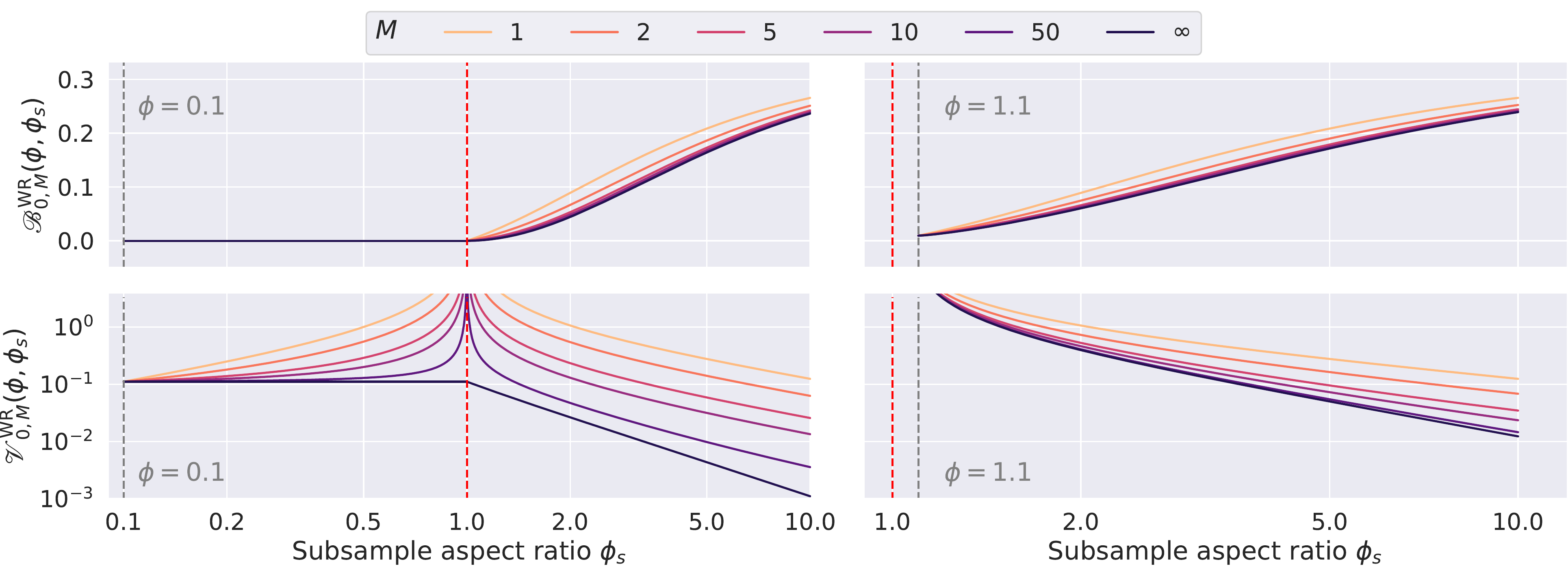}
    \caption{Asymptotic bias and variance curves in \eqref{eq:Blam_V_lam} for subagged ridgeless predictors ($\lambda=0$), under model \eqref{eq:model-ar1} when $\rhoar=0.25$ and $\sigma^2=0.25$, for varying subsample aspect ratio $\phi_s$ and numbers of bags $M$.
    The left and the right panels correspond to the cases when $p<n$ ($\phi=0.1$) and $p>n$ ($\phi=1.1$), respectively.
    The values of $\VzeroM{M}{\phi}$ are shown on a $\log$-10 scale.
    }
    \label{fig:ridgeless-bias-var-with-replacement}
\end{figure}            
    
\subsection{\Splagging without replacement}\label{sec:bagging-without-replacement}

In this section, we focus on analyzing the risk asymptotics and properties for \splagging.
More formally, we consider the risk asymptotics of the \splagged predictor obtained by averaging the predictors computed on $M$ non-overlapping subsets of the data, each of size $k$.
This is precisely the \splagged predictor $\tfWOR{M}{\cI_k^{\pi}}$.
Throughout all the asymptotics below, we consider the permutation $\pi$ to be fixed.
Because the limiting risk below does not depend on the permutation $\pi$, the conclusions continue to hold true even when the data or subsample conditional risk is averaged over all permutations $\pi$.
However, it should be emphasized that this is not the same as the data conditional risk of the \splagged predictor averaged over all permutations $\pi$.
In \Cref{subsec:bagging-without-replacement-main}, we provide exact risk characterization of \splagging without replacement for both ridge and ridgeless predictors.
The monotonicity properties of asymptotic bias and variance are then established in \Cref{subsec:monotonicity-M-without-replacement}.

\subsubsection{Risk characterization}
\label{subsec:bagging-without-replacement-main}

Recall our convention is defining the \splagged predictor $\tf_{M, \cI_k^{\pi}}^{\WOR}$ as $\tf_{\min\{M,\lfloor n/k\rfloor\}, \cI_k^{\pi}}^{\WOR}$, so that the \splagged predictor is well defined for all $M\in\NN$.
    
    \begin{theorem}[Risk characterization for splagged ridge and ridgless predictors] %
    \label{thm:ver-without-replacement}
    Let $\tfWOR{M}{\cI_k^{\pi}}$ be the predictor as defined in \eqref{eq:ingredient-predictor} for $\lambda \ge 0$.
    Suppose Assumptions \ref{asm:rmt-feat}-\ref{asm:spectrum-spectrumsignproj-conv} hold for the dataset $\cD_n$. Then as $k,n,p\rightarrow\infty$, $p/n\rightarrow\phi\in(0,\infty)$, $p/k\rightarrow\phi_s\in[\phi,\infty]$ (and $\phi_s \neq 1$ for $\lambda = 0$), there exist deterministic functions $\bRlamM{M}{\phi}$ for all
    $M \in\NN$, and $\phi_s\geq \phi$, such that for $I_1, \ldots, I_M \overset{\textup{\texttt{SRSWOR}}}{\sim} \mathcal{I}_k^{\pi}$, 
    \begin{align*}
         \sup_{M \in \NN}|R(\tfWOR{M}{\cI_k^{\pi}};\cD_n,\{I_{\ell}\}_{\ell = 1}^{M}) - \bRlamM{M}{\phi}| &\pto 0,\\
        \sup_{M \in \NN}
        |R(\tfWOR{M}{\cI_k^{\pi}};\cD_n) - \bRlamM{M}{\phi}| &\asto 0.
    \end{align*}
    Here $\bRlamM{M}{\phi}
    = \bRlamM{\lfloor \phi_s / \phi \rfloor}{\phi}$ for $M \ge \lfloor \phi_s / \phi \rfloor$,
    and for $M \le \lfloor \phi_s / \phi \rfloor$, 
    the function $\bRlamM{M}{\phi}$ decomposes as
    \begin{align}
        \bRlamM{M}{\phi} &= \sigma^2 + \bBlamM{M}{\phi} + \bVlamM{M}{\phi}, \label{eq:risk-det-without-replacement}
    \end{align}
    where $\bBlamM{M}{\phi}=M^{-1}B_{\lambda}(\phi_s, \phi_s)+(1-M^{-1})C_{\lambda}(\phi_s)$, $\bVlamM{M}{\phi} =M^{-1}V_{\lambda}(\phi_s,\phi_s)$, $C_{\lambda}(\phi_s)= \rho^2 \tc(-\lambda;\phi_s)$, and $B_{\lambda}(\phi_s,\phi_s)$ and $V_{\lambda}(\phi_s,\phi_s)$ are quantities as defined in \Cref{thm:ver-with-replacement}.
\end{theorem}

\begin{remark}
    For every pair $(\phi, \phi_s)$
    satisfying $\phi_s \ge \phi$,
    note that the splagged predictor
    and the risks are defined
    non-trivially only for $M = 1, \dots, \lfloor \phi_s / \phi \rfloor$,
    and is defined as a constant for $M > \lfloor \phi_s / \phi \rfloor$.
    In particular,
    for a fixed pair $(\phi, \phi_s)$,
    the sequence of risks as $M$ changes
    looks like: 
    \[
    \bRlamM{1}{\phi}, \, 
    \bRlamM{2}{\phi}, \, 
    \dots, \,
    \bRlamM{\lfloor \phi_s / \phi \rfloor}{\phi}, \,
    \bRlamM{\lfloor \phi_s / \phi \rfloor}{\phi}, \, 
    \dots.
    \]
\end{remark}

\begin{remark}
    [Dependence on data and subsample aspect ratios]
    Even though \splagging does not formally involve repeated observations like bootstrapping, we will still refer to $\phi_s = p / k$ as the subsample aspect ratio, where $k$ is the number of observations in each split part of the full dataset.
    In \Cref{thm:ver-with-replacement}
    for the subagged predictor with replacement,
    the asymptotic risk depends on 
    both the data aspect ratio $\phi$
    as well as the subsample aspect ratio $\phi_s$.
    In contrast, the asymptotic risk for the \splagged predictor
    without replacement in \Cref{thm:ver-without-replacement}
    does not depend on the data aspect ratio $\phi$.
    This can be seen from the expressions for $\bBlamM{M}{\phi}$ and $\bVlamM{M}{\phi}$.
    However, it is interesting to note that
    the asymptotic risk for $\tf^{\WR}_{M, \cI_k^\pi}$
    depends on both $\phi$ and $\phi_s$
    because $\limsup_{k, n \to \infty} |\cI_k^\pi|$ 
    is finite, which makes the limiting risk of $\tf^{\WR}_{M, \cI_k^\pi}$ and $\tf^{\WOR}_{M, \cI_k^\pi}$
    different.
    Because $K_{N, M}$ defined in \eqref{eq:KNM-def}
    is bounded above by 1 and $\limsup_{k, n \to \infty} K_{|\cI_k^\pi|, M} < 1$ for any $M > 1$, $\tf^{\WOR}_{M, \cI_k^\pi}$ is
    a strictly better predictor then $\tf^{\WR}_{M, \cI_k^\pi}$ in terms of the squared risk.
    In other words, $\tf_{M, \cI_k^\pi}^{\WR}$ is inadmissible, even asymptotically.
\end{remark}

\begin{remark}[Comparison with distributed learning]
    \Cref{thm:ver-with-replacement} considers the simple average of base predictors fitted on non-overlapped samples, which is also closely related to distributed learning \citep{mucke_reiss_rungenhagen_klein_2022} that utilizes multiple computing devices to reduce overall training time.
    \citet{mucke_reiss_rungenhagen_klein_2022} only provide finite-sample upper bounds for the prediction risk of distributed ridgeless predictor, while \Cref{thm:ver-with-replacement} gives exact risk characterization.
    The distributed ridge predictors are also studied in \citet{dobriban_sheng_2020}, though their goal is to obtain the optimal weight and the optimal regularization parameter.
\end{remark}

\paragraph{Illustration of \Cref{thm:ver-without-replacement}.}

In \Cref{fig:ridgeless-without-replacement-varing-M,fig:ridge-without-replacement-varing-M}, we provide numerical illustrations for \Cref{thm:ver-without-replacement} (bagged ridgeless and ridge predictors with $\lambda = 0.1$) under the model \eqref{eq:model-ar1}, with the number of bags $M$ varying from $1$ to $\infty$.
The limiting data aspect ratio is fixed at $0.1$ when $n > p$ and at $1.1$ when $n < p$.
We find that the empirical risks align remarkably well with the deterministic approximations, as stated in \Cref{thm:ver-without-replacement}, for both bagged ridge and ridgeless predictors.
Mirroring the findings in \Cref{fig:ridgeless-with-replacement-varing-M}, for any fixed $M$, the optimal $\phi_s$ may be strictly larger than $\phi$, an implication of the non-monotonic risk behavior.

\begin{figure}[!t]
        \centering
        \includegraphics[width=0.85\textwidth]{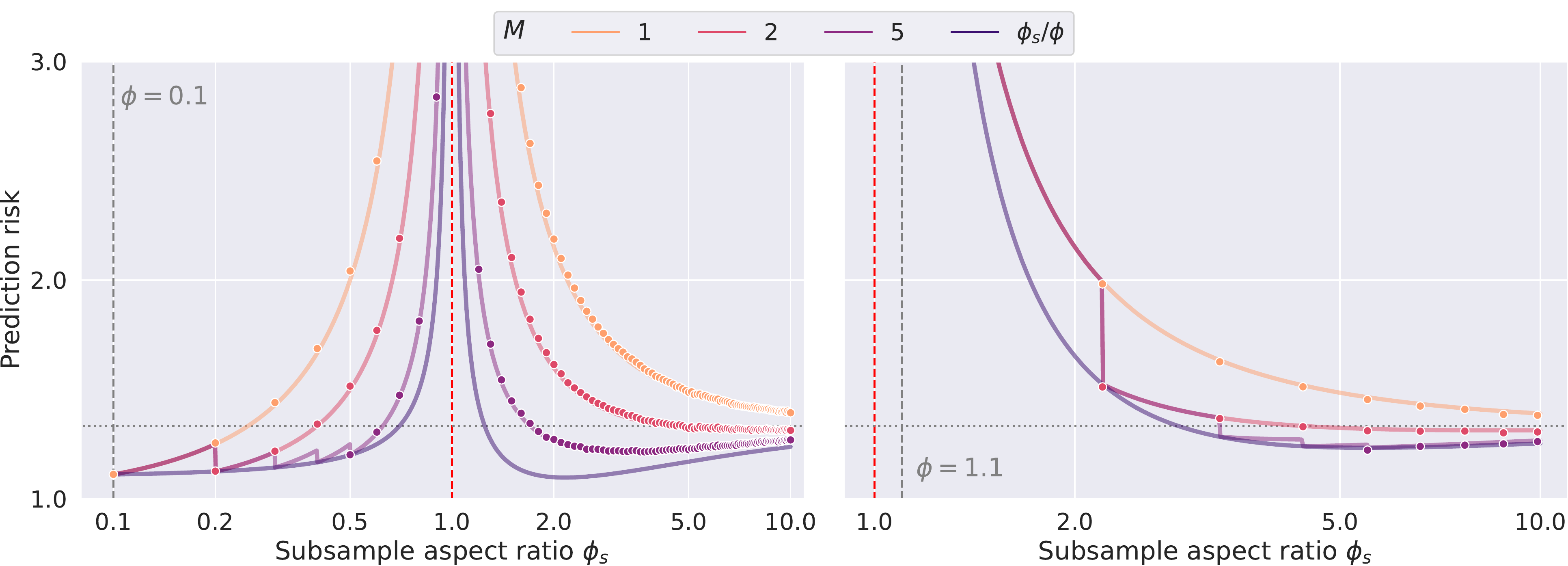}
        \caption{Asymptotic prediction risk curves in \eqref{eq:risk-det-without-replacement} for \splagged ridgeless predictors ($\lambda=0$), under model \eqref{eq:model-ar1} when $\rhoar=0.25$ and $\sigma^2=1$, for varying split sizes $k=\lfloor p/\phi_s\rfloor$ and numbers of bags $M$.
        The left and the right panels correspond to the cases when $p<n$ ($\phi=0.1$) and $p>n$ ($\phi=1.1$), respectively.
        The null risk is marked as a dotted line.
        For each value of $M$, the points denote finite-sample risks averaged over 100 dataset repetitions, with $p=500$ and $n=\lfloor p\phi\rfloor$.}
        \label{fig:ridgeless-without-replacement-varing-M}
\end{figure}

\begin{figure}[!t]
    \centering
    \includegraphics[width=0.85\textwidth]{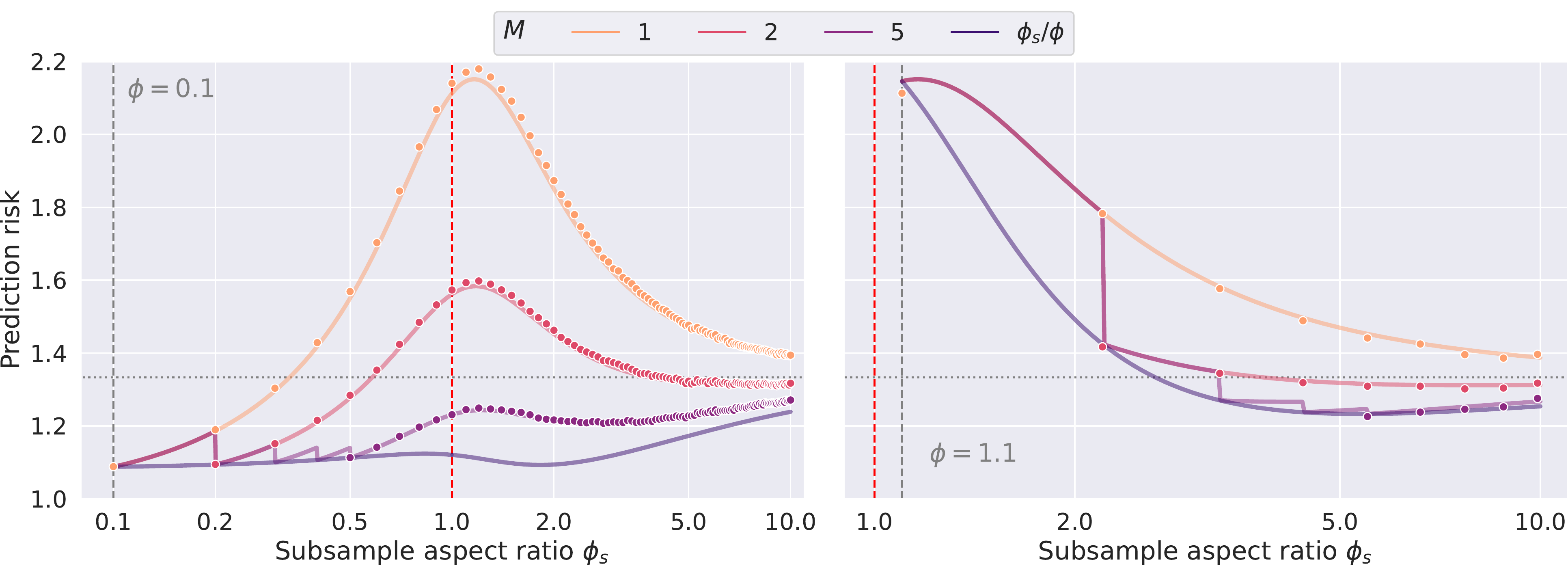}
    \caption{Asymptotic prediction risk curves in \eqref{eq:risk-det-without-replacement} for \splagged ridge predictors ($\lambda=0.1$), under model \eqref{eq:model-ar1} when $\rhoar=0.25$ and $\sigma^2=1$, for varying split sizes $k=\lfloor p/\phi_s\rfloor$ and numbers of bags $M$.
    The left and the right panels correspond to the cases when $p<n$ ($\phi=0.1$) and $p>n$ ($\phi=1.1$), respectively.
     The null risk is marked as a dotted line. For each value of $M$, the points denote finite-sample risks averaged over 100 dataset repetitions, with $p=500$ and $n=\lfloor p\phi\rfloor$.}
    \label{fig:ridge-without-replacement-varing-M}
\end{figure}

\paragraph{Proof outline of \Cref{thm:ver-without-replacement}.\hspace{-2mm}} 

The proof of \Cref{thm:ver-without-replacement} follows a similar reduction strategy as in the proof of \Cref{thm:ver-with-replacement}, where we first analyze the subsample conditional risks for $M = 1$ and $M = 2$, and appeal to \Cref{thm:risk_general_predictor}
to obtain the result for data conditional and subsample conditional
risks for any $M$. Below we briefly outline the main steps:
    
\begin{enumerate}[leftmargin=7mm]
    \item The deterministic risk approximation to the subsample conditional risk for $M = 1$ \splagging is exactly the same as that of subagging.
    
    \item
    Under the linear model, the subsample conditional risk for $M = 2$ decomposes in a similar manner as \eqref{eq:risk-decomp-m2-wr}, except in this case, the datasets $\cD_{I_1}$ and $\cD_{I_2}$ are independent of each other (conditional on $I_1, I_2$), which makes the analysis in this case slightly easier compared to the one for subagging.
    By conditioning on each of the datasets successively and utilizing the closed-form expression of the ridge estimator, we obtain the desired deterministic approximations.
    \item
    Finally, akin to what we did for \Cref{thm:ver-with-replacement}, we prove results for the ridgeless predictor in the form of the limiting risk approximations to the risk of the ridge predictor in the limit as $\lambda \to 0^{+}$, based on uniformity arguments.
\end{enumerate}

\subsubsection{Monotonicity of bias and variance in number of bags}\label{subsec:monotonicity-M-without-replacement} 

Just as with subagging, the asymptotic bias and variance components of the conditional risk for \splagging are also monotonically decreasing in the number of bags $M$. This is formalized below.

    \begin{proposition}[Improvement due to \splagging]\label{prop:monotonicity-M-without-replacement}
        Fix any pair $(\phi, \phi_s)$
        such that $\phi_s \ge \phi$.
        Then for all $M \in \{ 1, \dots, \lfloor \phi_s / \phi \rfloor \}$,
        \begin{align}
            \bBlamM{\lfloor \phi_s/\phi \rfloor}{\phi}\leq \bBlamM{M+1}{\phi}\leq \bBlamM{M}{\phi} \label{eq:monotonicity-M-bB},\\
            \bVlamM{\lfloor \phi_s/\phi \rfloor}{\phi}\leq \bVlamM{M+1}{\phi}\leq \bVlamM{M}{\phi}.\label{eq:monotonicity-M-bV}
        \end{align}
        The inequalities in \eqref{eq:monotonicity-M-bB} are strict whenever $\rho^2>0$ and $\phi_s\in(\phi,\infty)$ (and $\phi_s\neq 1$ when $\lambda=0$), while the inequalities in \eqref{eq:monotonicity-M-bV} are strict when $\sigma^2>0$ and $\phi_s\in(\phi,\infty)$ (and $\phi_s\neq 1$ when $\lambda=0$).
        Thus, the asymptotic risk is monotonically decreasing in $M$, i.e., $\bRlamM{M+1}{\phi}\leq \bRlamM{M}{\phi}$.
   \end{proposition} 
    
As a concluding remark, because the deterministic risk approximation for splagging is defined as a constant in $M$ for $M \ge \lfloor \phi_s / \phi \rfloor$, \Cref{prop:monotonicity-M-without-replacement} implies that the for every fixed pair $(\phi, \phi_s)$, the optimal splagged predictor utilizes $M = \lfloor \phi_s / \phi \rfloor$ bags.
    
\section{Risk profile monotonization}
\label{sec:monotonizing-risk-profiles}

The results presented in the previous sections provide risk characterizations for different variants of bagged predictors, per \eqref{eq:bagged-predictor-subagging}, for all possible subsample aspect ratios $\phi_s$. 
In practice, the choice of $\phi_s$ is crucial for achieving optimal prediction performance. 
Following the cross-validation strategy discussed in~\cite{patil2022mitigating}, one can apply cross-validation to choose the optimal $\phi_s$ in order to obtain the best possible prediction performance by subagging or \splagging the base predictor across different subsample sizes. 
In~\Cref{subsec:general-predictors-risk-monotonization}, we first describe the risk monotonization results for general predictors, going back to the general setting in~\Cref{sec:general-predictors}. 
In~\Cref{subsec:cross-validation}, we then specialize the general risk monotonization results to the bagged ridge and ridgeless predictors.
In~\Cref{subsec:comparison}, we provide a comparison between the best subagged and the best \splagged predictors, considering all possible choices of both $\phi_s$ and $M$, when the base predictor is either ridge or ridgeless.

\subsection{Bagged general predictors}
\label{subsec:general-predictors-risk-monotonization}

Several commonly used prediction procedures, such as min-$\ell_2$-norm least squares and ridge regression, exhibit a non-monotonic risk behavior as a function of the data aspect ratio $\phi$. 
This is referred to in the literature as double/multiple descents \citep{belkin2019reconciling,hastie2022surprises}. 
The deterministic risk approximation, as a function of the aspect ratio $\phi$, first increases, reaches a peak, and then decreases. 
This can be understood in the context of fixed dimension and changing sample size $n$ as follows:
the risk first decreases as the sample size increases up to a certain threshold, after which it starts increasing with a further increase in sample size. 
This is a counter-intuitive behavior from a conventional statistical viewpoint, as this indicates that more data may hurt performance. 
However, from a theoretical perspective, additional information should only lead to improved performance. 
The underlying issue here lies not in the theory but in the sub-optimality of the prediction procedures when applied as-is on the full data.

There are at least two ways in which one can think of improving a given predictor:
\begin{enumerate}[leftmargin=7mm]
    \item 
    Obtain a new predictor whose risk is the greatest monotone minorant of the risk of the given prediction procedure. 
    This can be achieved by computing the predictor on a smaller sample size if necessary. Such a procedure is referred to as the zero-step procedure (with $M = 1$) in~\cite{patil2022mitigating}; see~\Cref{alg:cross-validation} for details. 
    The zero-step procedure does the bare minimum to achieve monotone risk.
    
    \item 
    The zero-step procedure (with $M = 1$) is not a genuine improvement of the base predictor, as it simply computes the same predictor on a smaller dataset. 
    Building upon the positive effects of subagging or \splagging mentioned in previous sections, we can further improve on the zero-step procedure by aggregating over multiple subsets of the data. 
    This was already hinted at and illustrated in~\cite{patil2022mitigating}. 
    In this section, we delve deeper into this point.
\end{enumerate}

We note from \Cref{thm:ver-with-replacement} and
\Cref{fig:ridgeless-with-replacement-varing-M,fig:ridge-with-replacement-varing-M}
that for each $\phi$, there are essentially infinitely many risk values possible (one for each pair of subsample aspect ratio $\phi_s$ and the number of bags $M$). 
The zero-step procedure (with $M = 1$) improves on the base predictor by optimizing over $\phi_s$, while keeping $M = 1$ fixed. 
Taking a step further, based on our aforementioned results, we can consider optimizing over $\phi_s$ and $M\ge1$ (or just over $\phi_s$, while fixing $M \ge 1$). 
In the following, we present an actionable algorithm to achieve the optimum over $\phi_s$ for any fixed $M\ge1$. 
(It is worth noting that we have already established monotonicity over $M \ge 1$, and one can always choose $M$ to be as large as feasible in practice.) 
We then present \Cref{thm:cv_general}, in which we prove that the general cross-validation attains the optimum over $\phi_s$ (asymptotically). 
\Cref{thm:cv_general} provides theoretical guarantees for the cross-validation procedure for general base predictors, extending the results of \citet{patil2022mitigating} to subagging and \splagging.

\begin{algorithm}[!t]
    \caption{Cross-validation for subagging or splagging}
        \label{alg:cross-validation}
        \begin{algorithmic}[1]
        \REQUIRE 
        A dataset $\cD_n = \{ (\bx_i, y_i) \in \RR^{p} \times \RR : 1 \le i \le n \}$, 
        a positive integer $n_\test < n$ (number of test samples),
        a base prediction procedure $\hf$, 
        a real number $\nu \in (0, 1)$ (bag size unit parameter),
        a natural number $M$ (number of bags),
        a centering procedure $\CEN \in \{ \AVG, \MOM \}$,
        a real number $\eta$ when $\CEN = \MOM$.
        \smallskip
        
        \STATE \textbf{Data splitting:} Randomly split $\cD_n$ into training set $\cD_\train$ and test set $\cD_\test$ as:
        $$\cD_\train = \{ (\bx_i, y_i) : i \in \cS_\train \},\quad \text{and} \quad \cD_\test = \{ (\bx_j, y_j) : j \in \cS_\test \},$$
        where $\cS_\test\subset[n]$ with $| \cS_\test | = n_\test$, and $ \cS_\train=[n]\setminus\cS_\test$.
            
        \STATE\label{algo:2}
        \textbf{Bag sample sizes grid construction:} Let $k_{0}=\lfloor n^\nu \rfloor$ and $\cK_n = \{k_{0}, 2k_{0},\ldots, \lfloor n/k_0 \rfloor k_0\}$.
        
        \STATE \textbf{Subagging or splagging predictors:} 
        For each $k \in \cK_n$, 
        define $\tf_{M, k}$ trained on $\cD_\train$ as:
        \begin{itemize}[leftmargin=7mm]
            \item 
            For subagging,
            let $\tf_{M, k}(\cdot) = \tf_{M}(\cdot; \{ \cD_{I_{k,\ell}} \}_{\ell=1}^{M})$
            denote the subagged predictor 
            as in \eqref{eq:bagged-predictor-subagging}
            with $M$ bags.
            Here, $I_{k,1}, \dots, I_{k,M}$ represent a simple random sample
            with or without replacement from the set of all subsets of
            $\cS_\train$ of size $k$. 
            \item
            For splagging, $\tf_{M, k}(\cdot)$ is the same as above but now $I_{k,1}, \ldots, I_{k,M}$ represent a simple random sample without replacement from a random split of $\cS_{\train}$ into $\lfloor n/k\rfloor$ parts with each part containing $k$ elements.  
            As explained in \Cref{subsec:predictor_ass}, 
            for $M > \lfloor n / k \rfloor$,
            no such splitting exists.
            In this case,
            we return $\tf_{\lfloor n / k \rfloor, k}$.
            Hence in general, we have $\tf_{M, k} = \tf_{\min\{M, \lfloor n / k \rfloor\}, k}$.
        \end{itemize}

        \STATE \textbf{Risk estimation:} For each $k \in \cK_n$, estimate the conditional prediction risk on $\cD_{\test}$ of $\tf_{M,k}$ as:
        \begin{empheq}[left={\widehat{R}(\tf_{M,k}):=\empheqlbrace\,}]{alignat=3}
            &{|\cS_\test|}^{-1} \sum_{j \in \cS_\test}(y_j- \tf_{M,k}(\bx_j))^2 , 
            &&~~\text{ if \CEN=\AVG}\label{eqn:avg-risk-pre} \\
            &\mathrm{median}(\widehat{R}_{1}(\tf_{M, k}), \dots, \widehat{R}_B(\tf_{M, k})), 
            &&~~\text{ if \CEN=\MOM},\label{eqn:mom-risk-pre}
        \end{empheq}
        where $B = \lceil 8 \log(1/\eta) \rceil$,
        and $\widehat{R}_j(\tf_{M, k})$, $1 \le j \le B$ is defined similarly to \eqref{eqn:avg-risk-pre} for $B$ random splits of the test dataset $\cD_\test$.

        \STATE \textbf{Cross-validation}: Set $\hat{k} \in \cK_n$ to be the bagging sample size that minimizes the estimated prediction risk using
            \begin{align}
            \label{eqn:model-selection-xi}
                \hat{k}
                \in \argmin_{k \in \cK_n}
                \hat{R}(\tf_{M,k}).
            \end{align}
                
        \ENSURE Return the predictor 
        $\hf_{M}^\cv(\cdot; \cD_n) 
        = \tf_{M, \widehat{k}}(\cdot)
        = \tf_{M}(\cdot;  \{\cD_{I_{\hat{k},\ell}}\}_{\ell=1}^M)$.
        \end{algorithmic}
    \end{algorithm}
    
\begin{theorem}[Risk monotonization by cross-validation]\label{thm:cv_general}
    Suppose that as $n,p\rightarrow\infty$, $p/n\rightarrow\phi\in(0,\infty)$.
    Let $\cK_n$ be the set of subsample sizes defined in \Cref{alg:cross-validation} and $\cI_k$ be the set of subsets of $\cS_{\train}$ of size $k\in\cK_n$ according to the sampling scheme.
    Suppose that for any $k \in \cK_n$,
    as $n,k, p \to \infty$,
    and $p/k\rightarrow \phi_s\in[\phi,\infty)$, there exists a deterministic function $\mathscr{R} : (0, \infty]^2 \to [0, \infty]$ such that: 
    \begin{enumerate}[(i),leftmargin=7mm]
        \item \label{asm:cv_general-i}
        For any $I\in \cI_k$ and $\{I_{k,1},I_{k,2}\}$
        a simple random sample from $\cI_k$,
        \begin{align*}
                R(\tf_1; \mathcal{D}_n, \{I\}) &\asto \mathscr{R}(\phi_s,\phi_s) ,\quad \text{and} \quad R(\tf_2; \mathcal{D}_n,\{I_{k,1},I_{k,2}\}) \asto \mathscr{R}(\phi,\phi_s).
            \end{align*}

        \item \label{asm:cv_general-ii}
        For any $\phi\in(0,\infty)$, $\phi_s\mapsto \mathscr{R}(\phi,\phi_s)$ is proper and lower semi-continuous over $[\phi,\infty]$, and is continuous on the set $\argmin_{\{\psi:\psi\geq \phi\}}\mathscr{R}(\phi,\psi)$.
        
    \end{enumerate}
    Let $\hf_{M}^{\cv}$ be the cross-validated predictor returned by \Cref{alg:cross-validation} with base predictor $\hf$. 
    If the estimated risk $\widehat{R}(\tf_{M,k})$ defined in \eqref{eqn:avg-risk-pre} or \eqref{eqn:mom-risk-pre} is uniformly (in $k \in \cK_n$) close to the subsample conditional risk $R(\tf_{M, k}; \cD_n, \{ I_{k, \ell} \}_{\ell=1}^{M})$ with probability converging to $1$, then the following conclusions hold.
    For subagging with or without replacement, or splagging without replacement, for all $M \in \NN$, we have
    \begin{align*}
        \left(R(\hf_{M}^\cv;\cD_n,\{I_{\hat{k},\ell}\}_{\ell=1}^M) - \min\limits_{\phi_s\geq\phi}\mathscr{R}_M(\phi,\phi_s)\right)_+ \pto 0,
    \end{align*}
    where the function $\mathscr{R}_M(\phi,\phi_s)$ is defined as
    \begin{align*}
        \mathscr{R}_M(\phi,\phi_s) := (2\mathscr{R}(\phi,\phi_s) - \mathscr{R}(\phi_s,\phi_s)) + \frac{2}{M}(\mathscr{R}(\phi_s,\phi_s) - \mathscr{R}(\phi,\phi_s)).
    \end{align*}
    Furthermore, if for any $\phi_s\in(0,\infty)$, $\phi\mapsto \mathscr{R}(\phi,\phi_s)$ is non-decreasing over $(0,\phi_s]$, then the function $\phi \mapsto \min_{\phi_s \ge \phi} \mathscr{R}_M(\phi, \phi_s)$ is monotonically increasing for every $M$. 
\end{theorem}

\begin{remark}
    [Asymptotic risks are different for subagging and splagging.] 
    Although~\Cref{thm:cv_general} presents a unified framework for subagging and splagging, the actual limiting risks can be (and in most cases are) different. This discrepancy arises due to the distinct expressions for assumed asymptotic risks in assumption \ref{asm:cv_general-i} of~\Cref{thm:cv_general}.
\end{remark}

\begin{remark}[Exact risk characterization of the cross-validated predictor with stronger assumptions]
    Note that
    \Cref{thm:cv_general}
    does not exactly characterize the risk of cross-validated
    bagged predictor;
    it only states that the subsample conditional risk
    of $\tf^\cv_M$ is asymptotically no larger than
    $\min_{\phi_s} \mathscr{R}_M(\phi, \phi_s)$.
    Nevertheless, this is an improvement over the results
    of \cite{patil2022mitigating},
    who proved that the subsample conditional risk
    of $\tf_M^\cv$ is asymptotically no larger than
    $\min_{\phi_s} \mathscr{R}_1(\phi, \phi_s)$.
    For the exact risk characterization of $\tf_M^\cv$,
    one can make the stronger assumption that as $n,p\rightarrow\infty$ and $p/n\rightarrow\phi$,
    \begin{align*}
        \sup_{k\leq n}| R(\tf_1; \mathcal{D}_n,\{I_1\overset{\textup{\texttt{SRSWR}}}{\sim}\cI_k\}) - \mathscr{R}(p/k,p/k)| \pto 0,\quad \text{and} \quad \sup_{k\leq n}| R(\tf_2; \mathcal{D}_n,\{I_1,I_2\overset{\textup{\texttt{SRSWR}}}{\sim}\cI_k\}) - \mathscr{R}(\phi,p/k)| \pto 0,
    \end{align*}
    which can be used to conclude
   \[R(\hf_{M}^\cv;\cD_n,\{I_{\hat{k},\ell}\}_{\ell=1}^M) ~\pto~ \min_{\phi_s\geq\phi}\mathscr{R}_M(\phi,\phi_s).
   \]
    The result for bagging without replacement can be extended analogously.
\end{remark}

\begin{remark}
    [Assumption of uniform consistency of the estimated risk]
    \label{rem:uniform-consistency-estimated-risk}
    The assumption of uniform 
    (in $k \in \cK_n$)
    closeness of the 
    estimated risk $\widehat{R}(\tf_{M, k})$
    to the subsample conditional risk
    $R(\tf_{M, k}; \cD_n, \{ I_{k, \ell} \}_{\ell=1}^{M})$
    is meant to represent either
    \[
        \max_{k \in \cK_n}
        |\widehat{R}(\tf_{M, k}) - R(\tf_{M, k}; \cD_n, \{ I_{k,\ell} \}_{\ell=1}^{M})|
        = o_p(1),
        \quad
        \text{or}
        \quad
        \max_{k \in \cK_n}
        \left|
        \frac{\widehat{R}(\tf_{M, k})}{R(\tf_{M, k}; \cD_n, \{ I_{k, \ell} \}_{\ell=1}^{M})}
        - 1
        \right|
        = o_p(1).
    \]
    In Section 2 of \cite{patil2022mitigating},
    the authors have provided several assumptions
    on the data distribution and the predictors
    such that this uniform closeness assumption holds true.
    In \Cref{subsec:cross-validation}, we will apply \Cref{thm:cv_general} for bagged linear predictors which are themselves linear predictors.
    In this specific case, Theorem 2.22 in the aforementioned work shows that uniform closeness holds true under assumptions on the data distribution alone (no matter what linear predictor is, even those that have diverging risks); see \citet[Remarks 2.19 and 2.20]{patil2022mitigating}.
    We do not further discuss this uniform closeness condition but only remark that Assumptions \ref{asm:rmt-feat}-\ref{asm:spectrum-spectrumsignproj-conv} imply the assumptions of Theorem 2.22 with $\CEN = \MOM$ (the median-of-means estimator). With $\CEN = \AVG$, sub-Gaussian features imply the assumptions of Theorem 2.22.
\end{remark}

\subsection{Bagged ridge and ridgeless predictors}
\label{subsec:cross-validation}

    \Cref{thm:cv_general} provides a very general result that describes the risk behavior of cross-validated bagged predictors in general. Following our results in previous sections that verify condition \ref{asm:cv_general-i} of~\Cref{thm:cv_general} for both ridge and ridgeless predictors, we now specialize~\Cref{thm:cv_general} to these predictors under Assumptions \ref{asm:rmt-feat}-\ref{asm:spectrum-spectrumsignproj-conv}.

    \begin{theorem}
    [Risk monotonicity in aspect ratio]\label{thm:monotonicity-phi}
    Suppose that the cross-validated predictor $\hf_{M}^{\cv}$ is returned by \Cref{alg:cross-validation} with base predictor $\hf_{\lambda}$ and $M$ bags, and the conditions in \Cref{thm:ver-with-replacement} (or \Cref{thm:ver-without-replacement}) hold\footnote{The statement as stated holds for $\CEN = \MOM$ in \Cref{alg:cross-validation}. For $\CEN = \AVG$, we need to assume sub-Gaussian features as discussed in \Cref{rem:uniform-consistency-estimated-risk}.} with $\cR_{\lambda,M}(\phi,\phi_s)$ being the limiting risk $\RlamM{M}{\phi}$ (or $\bRlamM{M}{\phi}$).
    Then it holds for all $M \in \NN$,
    \begin{align}
        \left(R(\hf_{M}^\cv;\cD_n,\{I_{\hat{k},\ell}\}_{\ell=1}^M) - \min\limits_{\phi_s\geq\phi}\cR_{\lambda,M}(\phi,\phi_s)\right)_+ \pto 0.
    \end{align}
    Furthermore, $\phi \mapsto \min_{\phi_s \ge \phi} \cR_{\lambda,M}(\phi,\phi_s)$ is a monotonically increasing function of $\phi$ for every $M$.
        
    \end{theorem}
    
    In \Cref{thm:monotonicity-phi}, the monotonicity of $\phi\mapsto\min_{\phi_s \ge \phi} \cR_{\lambda,M}$ implies that for every $M$, for the optimal bagged predictor, more data (i.e, increasing $n$) cannot hurt.
    In the plot of \Cref{fig:ridgeless_cv}, we observe slight non-monotonicity of the empirical risk profile for $M = 1$. This is because of the small sample size which does not allow for the optimal cross-validated predictor to be the null predictor.
    One way to not let this happen (in this specific case) is to always include a perfect ``null'' predictor in the set of predictors tuned with cross-validation in~\Cref{alg:cross-validation}.

    For \splagging without replacement, the simulation results are shown in \Cref{fig:ridgeless_cv}(b). 
    As expected, as the limiting aspect ratio $\phi$ increases, the empirical excess risks are nearly monotone increasing and match with theoretical curves.
    Another pattern we observe in \Cref{fig:ridgeless_cv} (splagging without replacement) is that the asymptotic risk may not be monotonically decreasing in $M$ when $\phi$ is small.
    This is because the subsample aspect ratio $\phi_s$ is restricted by the number of bags $M$ in that it cannot be below $M\phi$, and the differences in the range of $\phi_s$ when using different numbers of bags result in the non-monotonicity when $\phi$ is small.
    While in the overparameterized region when $\phi$ is large enough, the cross-validated risk for bagging without replacement is guaranteed to be monotonically decreasing in $M$.
    Furthermore, the choice of $M=\phi_s/\phi$ guarantees that the risk is always optimal compared to any other value of $M$.
        
    \begin{figure}[!ht]
        \centering
        \includegraphics[width=0.85\textwidth]{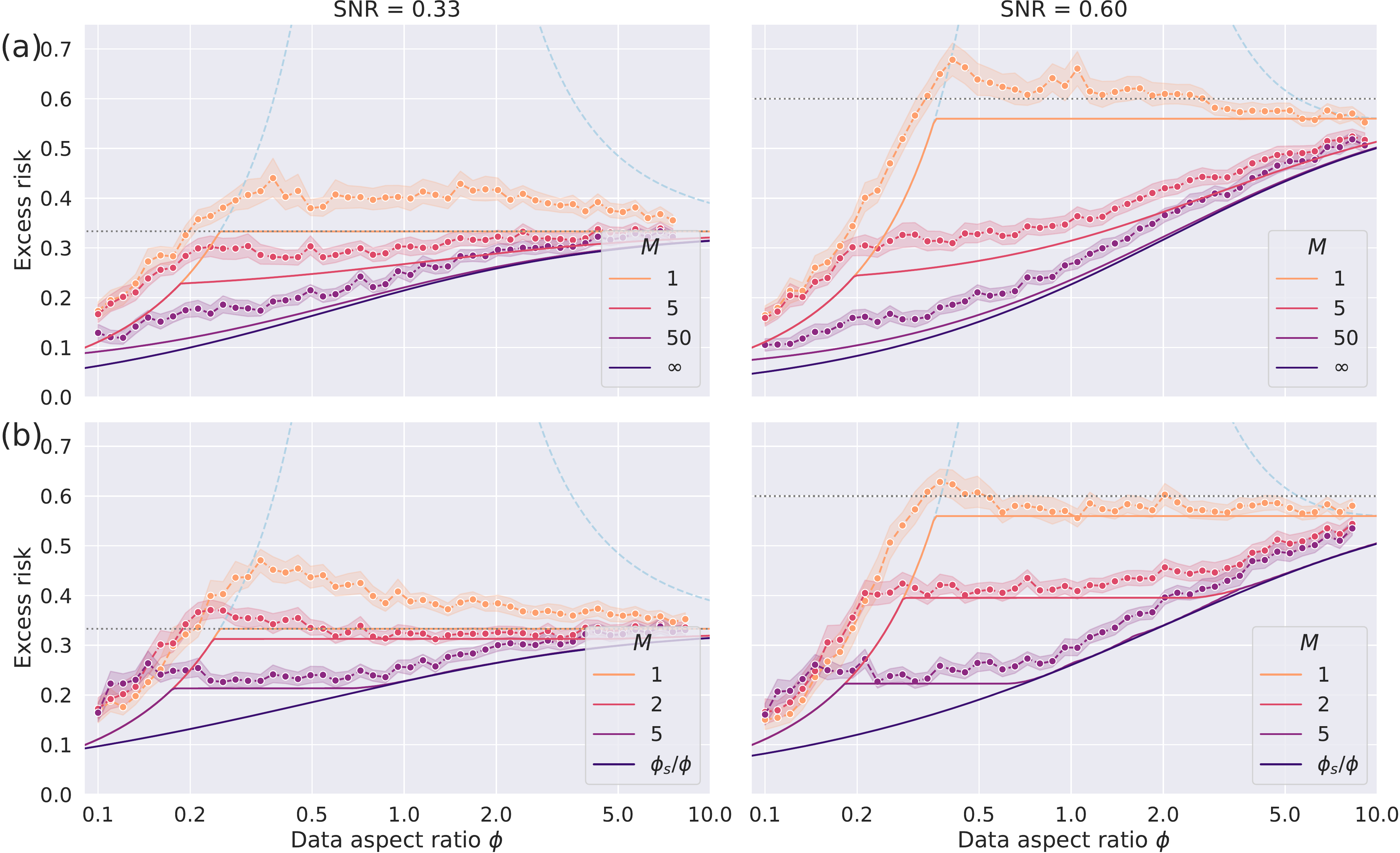}
        \caption{Asymptotic excess risk curves for cross-validated bagged ridgeless predictors ($\lambda=0$) for (a) subagging and (b) \splagging,
        under model \eqref{eq:model-ar1} when $\sigma^2=1$ 
        for varying $\SNR$,
        subsample sizes $k=\lfloor p/\phi_s\rfloor$ and numbers of bags $M$.
        The left and the right panels correspond to the cases when $\SNR = 0.33$ ($\rhoar=0.25$) and 0.6 ($\rhoar=0.5$), respectively.
        The excess null risks and the risks for the unbagged ridgeless predictors are marked as dotted lines and the dashed lines, respectively. For each value of $M$, the points denote finite-sample risks averaged over 100 dataset repetitions and the shaded regions denote the values within one standard deviation, with $n=1000$, $n_{\test}=63$, and $p=\lfloor n\phi\rfloor$.}
        \label{fig:ridgeless_cv}
    \end{figure}

\subsection{Optimal subagging versus optimal splagging}
\label{subsec:comparison}
    The cross-validated predictors discussed previously yield asymptotic optimal risks over subsample aspect ratio $\phi_s$ for every $M$.
    As a step further, we can obtain the optimal subagging and optimal \splagging by jointly optimizing over both $\phi_s$ and $M$.
    From the explicit formulas for the limiting risks for each pair of aspect ratios $(\phi, \phi_s)$ and each $M$, the optimal bagged risks 
    in the two cases can be compared.
\begin{proposition}[Comparison of the optimal risk of subagging and splagging]\label{prop:improve-with-without-replace}
    Under Assumptions \ref{asm:signal-bounded-norm}-\ref{asm:spectrum-spectrumsignproj-conv}, let $\RlamM{M}{\phi}$ and $\bRlamM{M}{\phi}$ be defined as in \Cref{thm:ver-with-replacement} and \Cref{thm:ver-without-replacement}, respectively. Then for any $\lambda\in[0,\infty)$ and $\phi\in(0,\infty)$, the following holds:
    \begin{align}
        \inf_{M\in\NN, \phi_s\in[\phi,\infty]} \RlamM{M}{\phi} \leq \inf_{M\in\NN,\phi_s\in[\phi,\infty]} \bRlamM{M}{\phi} .\label{eq:with_without_optimal}
    \end{align}
    In words, optimal subagging is at least as good as optimal splagging (without replacement) in terms of squared loss for ridge predictors.
\end{proposition}

For any dataset with fixed aspect ratio $\phi$, \Cref{prop:improve-with-without-replace} indicates that the optimal risk for bagged predictor across all possible choices of $M$ and subsample aspect ratio $\phi_s$ is always given by subagging. The optimal subagging and optimal splagging risks in~\Cref{prop:improve-with-without-replace} can be written as
\begin{equation}\label{eq:optimal-risks}
\cR_{\texttt{opt}}^{\sub}(\phi) = \mathscr{R}_{\lambda,\infty}^{\sub}(\phi, \phi_s^{\sub}(\phi)),\quad\mbox{and}\quad \cR_{\texttt{opt}}^{\spl}(\phi) = \mathscr{R}_{\lambda,\phi_s^{\spl}(\phi)/\phi}^{\spl}(\phi, \phi_s^{\spl}(\phi)),
\end{equation}
where the functions $\phi \mapsto \phi_s^{\sub}(\phi)$ and $\phi \mapsto \phi_s^{\spl}(\phi)$ are defined via
\begin{equation}
    \label{eq:optimal-subsample-aspect-ratio}
    \phi_s^{\sub}(\phi) ~:=~ \argmin_{\phi_s \ge \phi} \mathscr{R}_{\lambda,\infty}^{\sub}(\phi, \phi_s),\quad\mbox{and}\quad \phi_s^{\spl}(\phi) ~:=~ \argmin_{\phi_s \ge \phi} \mathscr{R}_{\lambda,\phi_s/\phi}^{\spl}(\phi, \phi_s).
\end{equation}
The fact that the optimal risks shown in~\Cref{prop:improve-with-without-replace} are the same as shown in~\eqref{eq:optimal-risks} follows from the fact that the risks are monotonically decreasing in $M$ for subagging and that the risk at $M = \phi_s/\phi$ is the best for splagging without replacement for any pair $(\phi, \phi_s)$.
The quantities $\phi_s^{\sub}(\cdot)$ and $\phi_s^{\spl}(\cdot)$ represent the best possible subsample aspect ratios for subagging and splagging (without replacement) for every data aspect ratio $\phi$ given.  
(Minimizers of lower semi-continuous functions over compact domains exist,
which is true for the functions in \eqref{eq:optimal-subsample-aspect-ratio} from
\Cref{thm:monotonicity-phi}.)

We calculate and present the theoretical optimal asymptotic risks~\eqref{eq:optimal-risks} for bagged ridgeless predictors in \Cref{fig:comparison_with_without_optimal}. 
The optimal risk $\min_{\phi_s\geq\phi}\RlamM{1}{\phi}=\min_{\phi_s\geq\phi}\bRlamM{1}{\phi}$ of the bagged ridgeless predictor with $M=1$ is also presented as the dashed line, which is the same as the monotone risk  of the zero-step ridgeless predictor of~\cite{patil2022mitigating} with $M = 1$.
As shown in \Cref{fig:comparison_with_without_optimal}(a), the optimal risk for the subagged ridgeless predictor 
is always smaller than the \splagged ridgeless predictor without replacement.
Both of them improve the risk for the ridgeless predictor with optimal subsample aspect ratio $\phi_s$ using only one bag ($M=1$).

\begin{figure}[!ht]
    \centering
    \includegraphics[width=0.85\textwidth]{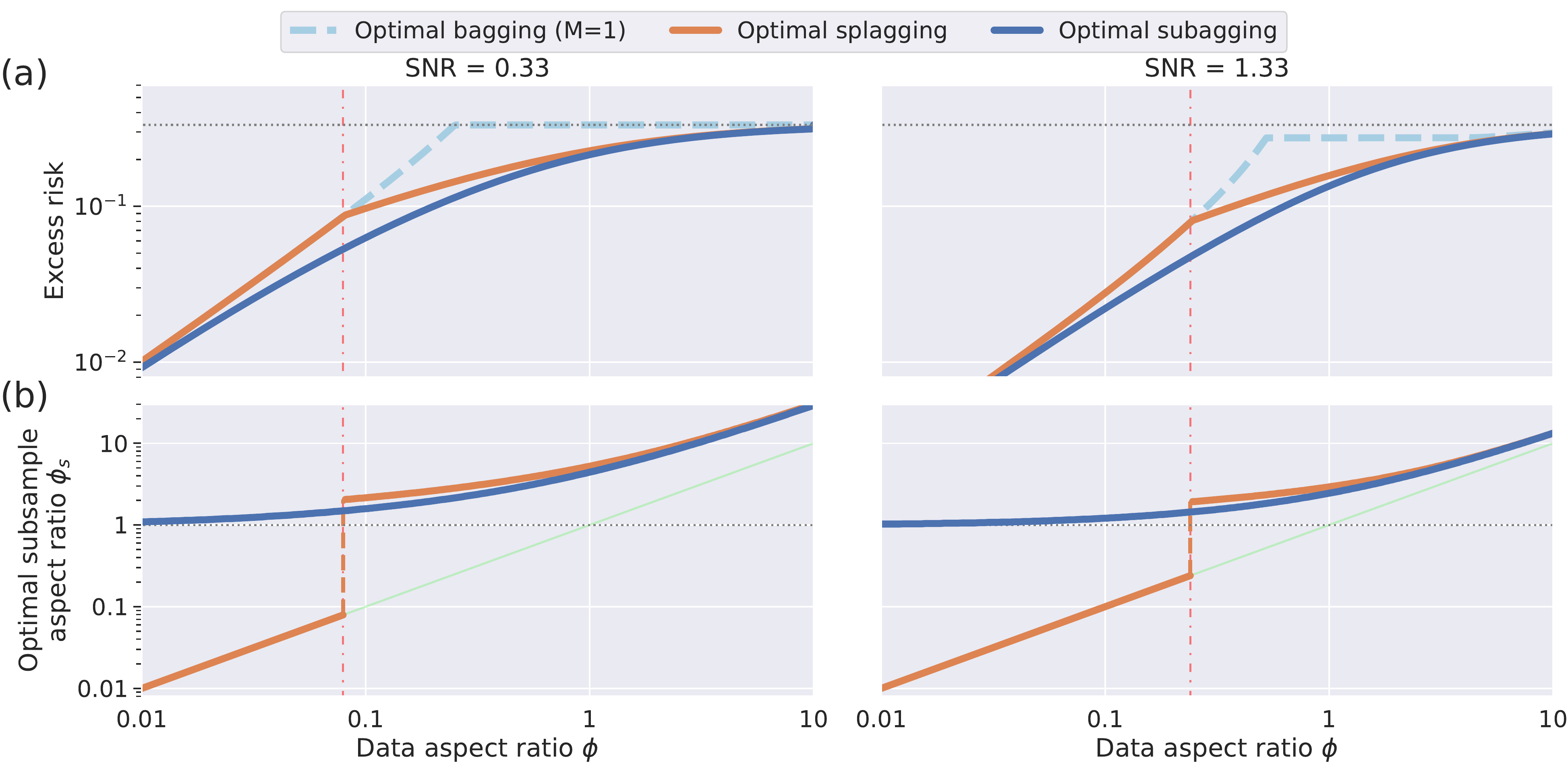}
    \caption{Comparison between optimal subagging and optimal \splagging of ridgeless predictors ($\lambda=0$) for varying limiting aspect ratios $\phi$ of $p/n$ under model \eqref{eq:model-ar1} when $\sigma^2=1$.
    The left and right panels correspond to $\SNR=0.33$ ($\rhoar=0.25$) and $\SNR=0.6$ ($\rhoar=0.5$), respectively.
    The point of phase transition for \splagging is marked as the red dash-dot line in every subplot.
    (a) Optimal asymptotic excess risk curves \eqref{eq:with_without_optimal}.
    The excess null risks are marked as gray dotted lines and the blue dashed lines represent the optimal risks of bagged ridgeless predictor with $M=1$, which are the same as the risks from the zero-step procedure of~\cite{patil2022mitigating}.
    (b) The corresponding optimal subsample aspect ratio $\phi_s$ as a function of data aspect ratio $\phi$. For subagging, the optimal subsample aspect ratio is always larger than one (above the gray dotted line).
    The line $\phi_s=\phi$ is colored in green.}
    \label{fig:comparison_with_without_optimal}
\end{figure}

    \paragraph{Oracle properties of optimal subsample aspect ratios.}
    From the previous section, we see that optimal subagged ridge or ridgeless regression always outperforms the \splagged one in terms of limiting risk.
    Due to the monotonicity in the number of bags $M$ from \Cref{prop:monotonicity-M-with-replacement}, the optimal risk for subagging must be obtained at $M=\infty$ for any given subsample aspect ratio $\phi_s$.
    One question that arises is: what is the optimal subsample aspect ratio $\phi_s$?
    We provide a partial answer to this question in \Cref{prop:opt-risk-ridgeless} specialized to ridgeless regression.
    
    \begin{proposition}[Optimal risk for bagged ridgeless predictor]\label{prop:opt-risk-ridgeless}
        Suppose the conditions in \Cref{thm:ver-with-replacement,thm:ver-without-replacement} hold, and $\sigma^2,\rho^2 \ge 0$ are the noise variance and signal strength from Assumptions \ref{asm:lin-mod} and \ref{asm:signal-bounded-norm}. Let $\SNR=\rho^2/\sigma^2$. For any $\phi\in(0,\infty)$, the properties of the optimal asymptotic risks $\mathscr{R}_{0,\infty}^{\sub}(\phi, \phi_s^{\sub}(\phi))$ and $\mathscr{R}_{0,\phi_s/\phi}^{\spl}(\phi, \phi_s^{\spl}(\phi))$ in terms of $\SNR$ and $\phi$ are characterized as follows:
        \begin{enumerate}[label={(\arabic*)}]
            \item $\SNR=0 \; (\rho^2 = 0, \sigma^2 \neq 0)$:  For all $\phi \ge 0$, the global minimum $\sigma^2$ of both $\mathscr{R}_{0,\infty}^{\sub}(\phi, \phi_s^{\sub}(\phi))$ and $\mathscr{R}_{0,\phi_s/\phi}^{\spl}(\phi, \phi_s^{\spl}(\phi))$ are obtained with $\phi_s^\sub(\phi)=\phi_s^\spl(\phi)=\infty$.
            
            \item $\SNR>0$: For all $\phi \ge 0$, the global minimum of $\phi_s \mapsto \RzeroM{\infty}{\phi}$ is obtained at $\phi_s^\sub(\phi)\in(1,\infty)$.
            For $\phi\geq 1$, the global minimum of $\phi_s \mapsto \bRzeroM{\phi_s/\phi}{\phi}$ is obtained at $\phi_s^\spl(\phi)\in (1,\infty)$; for $\phi\in(0,1)$, the global minimum of $\phi_s \mapsto \bRzeroM{\phi_s/\phi}{\phi}$ is obtained at $\phi_s^\spl(\phi)\in \{\phi\}\cup(1,\infty)$.
            
            \item $\SNR=\infty \; (\rho^2 \neq 0, \sigma^2 = 0)$: If $\phi\in(0,1]$, the global minimum $\mathscr{R}_{0,\infty}^{\sub}(\phi, \phi_s^{\sub}(\phi)) =\mathscr{R}_{0,\phi_s/\phi}^{\spl}(\phi, \phi_s^{\spl}(\phi))= 0$ is obtained with any $\phi_s^\sub(\phi),\phi_s^\spl(\phi)\in[\phi,1]$.
            If $\phi\in(1,\infty)$, then the global minimums $\mathscr{R}_{0,\infty}^{\sub}(\phi, \phi_s^{\sub}(\phi))$ and $\mathscr{R}_{0,\phi_s/\phi}^{\spl}(\phi, \phi_s^{\spl}(\phi))$ are obtained at $\phi_s^\sub(\phi),\phi_s^\spl(\phi)\in[\phi,\infty)$.
        \end{enumerate}
    \end{proposition}
    \Cref{prop:opt-risk-ridgeless} implies that the optimal subsample aspect ratio $\phi_s^\sub(\phi)$ for subagging is always in $[1,\infty]$, i.e., the overparameterized regime.
    In other words, 
    subagging interpolators with larger aspect ratios
    (larger than the full data aspect ratio $\phi$)
    helps to reduce the prediction risk, 
    even when $\phi < 1$.
    For \splagging, however, the minimum risk can be obtained either using the full data or 
    \splagging interpolators,
    depending on the data aspect ratio $\phi$ and the signal-to-noise ratio.
    
    It is interesting to note
    that the optimal subsampling aspect ratio for \splagging
    is either $\phi$ or it is in the overparameterized regime $(1, \infty)$.
    This means that either splagging does not help,
    or when it helps, one has to splag interpolators.
    Whenever $\SNR$ is positive, the optimal subsample aspect ratio is finite for any $\phi$.
    Hence we are able to visualize $\phi_s^\sub(\phi)$ and $\phi_s^\spl(\phi)$ in \Cref{fig:comparison_with_without_optimal}(b). 
    As shown in \Cref{fig:comparison_with_without_optimal}, there is a point of non-differentiability of $\phi_s^\spl(\phi)$ for optimal \splagging without replacement.
    Before this point of non-differentiability, $\phi_s^\spl(\phi)=\phi$, 
    which is the same as the optimal bagged ridgeless with $M = 1$.
    This is also the same as the ridgeless predictor trained on the full data.
    After the point of non-differentiability, the optimal risk for \splagging without replacement is obtained in the overparameterized regime, i.e., $\phi_s^{\spl}(\phi)>1$.
    In contrast to \splagging, $\phi_s^{\sub}(\phi) \ge 1$ for all $\phi > 0$, meaning that it is always better to subag interpolators (i.e., the overparameterized regime). 

    These observations indicate that, when the number of bags is sufficiently large enough, splagging without replacement only helps when the limiting aspect ratio $\phi$ of the full dataset is above some threshold, but subagging is always beneficial in reducing the prediction risk, even in the underparameterized regime.
    
        \begin{remark}
        [Guidelines for practical data analysis]
        \label{rem:practical-choice-M}
        \Cref{prop:opt-risk-ridgeless} implies that when using $M = \infty$,
        one should consider bagging interpolators to get better predictive performance, 
        at least when the linear model holds true.
        However, $M = \infty$ is practically infeasible particularly when $n, k \to \infty$. Note from~\Cref{fig:ridgeless-with-replacement-varing-M,fig:ridge-with-replacement-varing-M} that for $M$ large enough, the same phenomenon holds true, i.e., it is better to bag interpolators with a large $M$. How large such an $M$ should be depends on various unknowns related to the linear model and also on how much gap $\delta > 0$ from $M = \infty$ one is willing to allow. 
        Given the form of the limiting risk as a function of $M$, we can figure out
        the necessary value of $M$ as a function of the gap $\delta$, based on the
        cross-validation procedure~(\Cref{alg:cross-validation}). Note that this is completely data-driven and model-agnostic.
        The procedure is as follows:
        (1) Run \Cref{alg:cross-validation} with $M = 1$ and $M = 2$ to obtain the estimators $\widehat{\mathscr{R}_1}(\phi_s, \phi)$ and $\widehat{\mathscr{R}_2}(\phi_s, \phi)$ of the limiting subsample conditional risks $M = 1, 2$, respectively, for a grid of values $\phi_s \ge \phi$.
        Following \Cref{prop:limiting-risk-for-arbitrary-M-cond}, this yields an estimator of the subsample conditional risk for every $M \ge 1$, in particular, for $M = \infty$.
        (2) Find $\widehat{\phi_s^{\sub}}(\phi)$, the minimizer of $\phi_s \mapsto 2 \widehat{\mathscr{R}_2}(\phi_s, \phi) - \widehat{\mathscr{R}_1}(\phi_s, \phi)$.
        Note that this map is an estimator of the limiting risk for $M = \infty$.
        (3) Fix a tolerance level $\delta > 0$, and choose 
        \[
        M 
        = 
        \frac{2}{\delta}
       \left\{
        \widehat{\mathscr{R}_1}(\widehat{\phi_s^{\sub}}(\phi), \phi) 
        - 
        \widehat{\mathscr{R}_2}(\widehat{\phi_s^{\sub}}(\phi), \phi)
        \right\}.
        \]
        Operating at $\widehat{\phi_s^\sub}(\phi)$ with such a value of $M$ will yield an asymptotic risk that is $\delta$ close (in the additive sense) to the optimal risk.
    \end{remark}

\section[Illustrations and insights]{Illustrations and insights}
\label{sec:isotropic_features}

    The results discussed so far are derived under Assumptions \ref{asm:rmt-feat}-\ref{asm:spectrum-spectrumsignproj-conv} that, in particular, allow for features with arbitrary covariance structure $\bSigma$.
    We will shift our attention to a simpler case of isotropic features (i.e., $\bSigma = \bI_p$ in Assumption \ref{asm:rmt-feat}). 
    In this case, the spectral distribution simplifies, enabling us to compute the fixed point solutions analytically.
    Our discussion will primarily revolve around the case of ridgeless predictors for the sake of illustration. 
    While it is possible to obtain similar results for ridge predictors, the resulting expressions would be more involved.
    In \Cref{subsec:isotropic-ridge}, we provide formulas for the fixed-point solutions for $\lambda>0$. 
    From these, one can derive the risk as well as the individual bias and variance numerically for ridge predictors (with arbitrary $\lambda > 0$).
    Generally speaking, these quantities can always be computed numerically for nonisotropic models.
   
    In the case of isotropic features, the bias and variance functions presented in \Cref{thm:ver-with-replacement,thm:ver-without-replacement} take on relatively simple forms, as demonstrated in \Cref{cor:isotropic-ridgeless}.
    Furthermore, the asymptotic bias and variance can be computed for all $M\in\NN$ based on \eqref{remark:ridgeless}.
    
    \begin{corollary}
    [Bias-variance components for isotropic design]\label{cor:isotropic-ridgeless} Assume the conditions in \Cref{thm:ver-with-replacement} or \Cref{thm:ver-without-replacement} hold with $\bSigma=\bI_p$. Then we have
    \begin{align*}
        \begin{aligned}
        B_{0}(\phi,\phi_s)  = \rho^2 \tfrac{(\phi_s-1)^2}{\phi_s^2-\phi}\mathds{1}_{(1,\infty]}(\phi_s),\\
        C(\phi_s)  = \rho^2 \frac{(\phi_s-1)^2}{\phi_s^2} \mathds{1}_{(1,\infty]}(\phi_s),
        \end{aligned}\qquad \qquad
        V_{0}(\phi,\phi_s) = \begin{dcases}
         \sigma^2 \tfrac{\phi}{1-\phi},&\phi_s\in(0,1)\\
        \infty ,& \phi_s=1\\
        \sigma^2 \tfrac{\phi}{\phi_s^2-\phi}, &\phi_s\in(1,\infty].
        \end{dcases}
    \end{align*}
    \end{corollary}

    \paragraph{Subagging with replacement.\hspace{-2mm}} 
    Based on \Cref{cor:isotropic-ridgeless}, we are equipped to evaluate the closed-form asymptotic risk under model \eqref{eq:model}:
    \begin{align}
        y_i &= \bx_i^{\top}\bbeta_0 + \epsilon_i,
        \quad \bx_i\sim\cN(0,\bI_p),
        \quad \bbeta_0\sim\cN(0,p^{-1}\rho^2\bI_p),
        \quad  \epsilon_i\sim \cN(0,\sigma^2). \tag{M-ISO-LI}\label{eq:model}
    \end{align}
    Additional experimental results under model \eqref{eq:model} can be found in \citet[Appendix S.9.1]{patil2022bagging}.
    It is worth noting that while the Gaussianity of the noise $\epsilon_i$ in model \eqref{eq:model} simplifies numerical evaluation, it is not a requirement for \Cref{cor:isotropic-ridgeless}.
    It suffices to have the first and second moments match as above.
    For $M\in\NN$, the bias term is always increasing, while the variance term will blow up when the subsample aspect ratio $\phi_s$ approaches one.
    However, the variance for $M=\infty$ is different;
    it is decreasing in $\phi_s$ and continuous at $\phi_s=1$.
    Consequently, one might be interested in the optimal subsample aspect ratio $\phi_s^\sub(\phi)$, that best trades off the bias and variance, and minimizes the risk for a given value of $\phi$ and $M=\infty$.
    
    \begin{proposition}[Optimal risk for subagged ridgeless predictors with isotropic features]\label{thm:ridgeless-isotipic-optrisk}
        Suppose the conditions in \Cref{cor:isotropic-ridgeless} hold, and $\sigma^2,\rho^2\ge 0$ are the noise variance and signal strength from Assumptions \ref{asm:lin-mod} and \ref{asm:signal-bounded-norm}.
        Let $\SNR=\rho^2/\sigma^2$. For any $\phi\in(0,\infty)$, the properties of the asymptotic risk $\RzeroM{\infty}{\phi}$ as a function of $\phi_s$ are characterized as follows:
        \begin{enumerate}[label={(\arabic*)},leftmargin=7mm]
            \item $\SNR=0\; (\rho^2=0, \sigma^2\neq0)$: The global minimum $\mathscr{R}_{0,\infty}^{\sub}(\phi, \phi_s^{\sub}(\phi)) = \sigma^2$ is obtained at $\phi_s^\sub(\phi)=~\infty$.
            
            \item $\SNR>0$: The global minimum
            \begin{align} \RzeroMe{\infty}{\phi}{\phi_s^\sub(\phi)} %
               &= \frac{\sigma^2}{2}\left[ 1 + \frac{\phi-1}{\phi} \SNR + 
             \sqrt{\left(1 - \frac{\phi-1}{\phi}\SNR\right)^2 + 4 \SNR}\right] \label{eq:opt_phis}
            \end{align}
            is obtained at $\phi_s^\sub(\phi)=A + \sqrt{A^2-\phi}\in(1,\infty)$ where $A=(\phi+1 + \phi/\SNR)/2$.

            \item $\SNR=\infty\; (\rho^2\neq0, \sigma^2=0)$: If $\phi\in(0,1]$, then the global minimum is $\mathscr{R}_{0,\infty}^{\sub}(\phi, \phi_s^{\sub}(\phi))=0$ is attained at any $\phi_s\in[\phi,1]$.
            If $\phi\in(1,\infty)$, then the global minimum $\mathscr{R}_{0,\infty}^{\sub}(\phi, \phi_s^{\sub}(\phi)) = \sigma^2+\rho^2(\phi-1)/\phi$ is attained at $\phi_s^\sub(\phi)=\phi$. 
        \end{enumerate}
    \end{proposition}
    
    As a specific application of \Cref{prop:opt-risk-ridgeless}, \Cref{thm:ridgeless-isotipic-optrisk} provides the analytic expression of the optimal risk attainable through optimization over all choices of the number of bags $M$ and the subsample aspect ratio $\phi_s$.
    Additionally, it elucidates the relationship between the optimal risk and the $\SNR$, which is further visualized in \Cref{fig:risk_varying_SNR_phi}.
    Particularly, the optimal subagged risk is monotonically decreasing in $\SNR$ when $\sigma^2$ is fixed, which is an intuitive behavior as one would expect a larger $\SNR$ results in a smaller prediction risk.
    In contrast, such a property is not satisfied by the ridge or ridgeless predictor computed on the full data \citep[Figure 2]{hastie2022surprises}.
    It can be shown that the gap between the optimal risk, given in \Cref{thm:ridgeless-isotipic-optrisk}, and the underparameterized excess risk $\sigma^2\phi/(1-\phi)$, obtained with the full dataset, gets larger when $\SNR$ gets smaller.
    Most importantly, it benefits more when the $\SNR$ gets smaller, with a higher overparameterized aspect ratio $\phi_s^\sub(\phi)$.
    
    \begin{figure}[!t]
        \centering
        \includegraphics[width=0.85\textwidth]{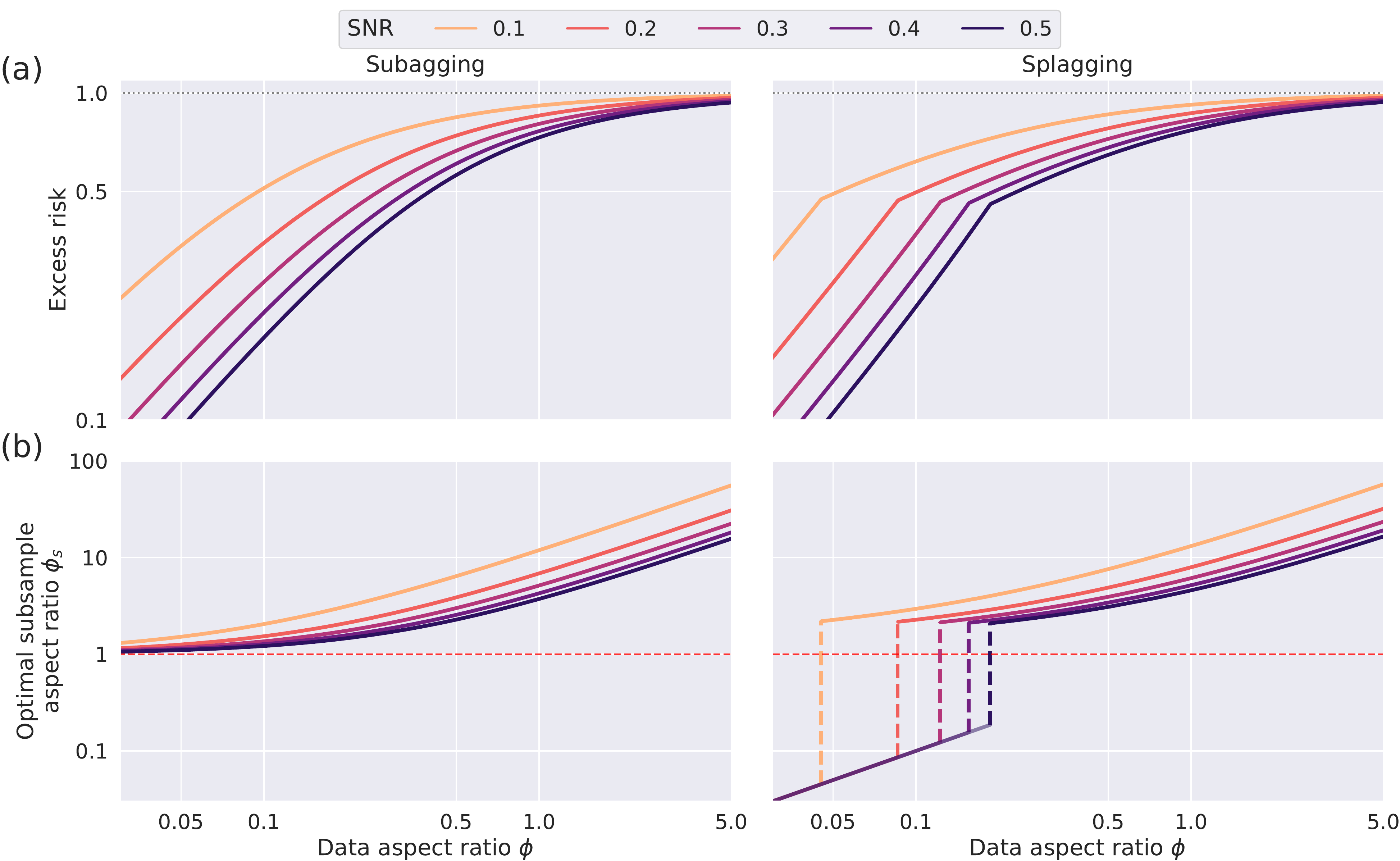}
        \caption{Properties of optimal bagged ridgeless predictors ($\lambda=0$) under model \eqref{eq:model} when $\rho^2=1$, for varying signal noise ratio $(\SNR=\rho^2/\sigma^2)$. (a) Optimal asymptotic excess risk curves of subagging (left panel) and \splagging (right panel) over the number of bags $M$ and subsample aspect ratio $\phi_s$. 
        The optimal numbers of bags are $M=\infty$ and $M=\phi_s/\phi$ for subagging and \splagging, respectively.
        The gray dotted lines represent the excess null risk.
        (b) The corresponding optimal subsample aspect ratio $\phi_s$ as a function of data aspect ratio $\phi$. For subagging, the optimal subsample aspect ratio is always larger than one (above the red dashed line).}
        \label{fig:risk_varying_SNR_phi}
    \end{figure}
    
    \begin{theorem}
    [Optimal subagged ridgeless risk versus optimal ridge risk]
    \label{thm:comparison_optimal_ridge}
        Under the conditions in \Cref{cor:isotropic-ridgeless},
        we have that for all $\phi\in(0,\infty)$,
        \begin{align*}
            \min_{\phi_s\geq \phi}\RzeroM{\infty}{\phi} ~=~ \min_{\lambda\geq 0}\RlamMe{1}{\phi}{\phi}. %
        \end{align*}
        In words,
        the optimal limiting risk of the subagged ridgeless predictors equals the optimal ridge predictors trained on the full data.
    \end{theorem}

    \Cref{thm:comparison_optimal_ridge} reveals a rather surprising connection between subagging and ridge regression. This result implies that subagging a ridge predictor with $\lambda = 0$ and optimizing over the subsample size is ``same'' as using the ridge predictor with $\lambda \ge 0$ and optimizing over $\lambda$. Consequently, this suggests that subsampling and optimizing over subsample size is a form of regularization. A similar connection between subsampling features and ridge regression was made by~\citet[Theorem 3.6]{lejeune2020implicit}.

    Compared to Theorem 3.6 of \cite{lejeune2020implicit}, our \Cref{thm:comparison_optimal_ridge} provides the following three key improvements: 
    (1) \textbf{Subsampling scope}. The former theorem focuses solely on the subsampling of features, whereas our theorem considers the sampling of observations.
    Moreover, in the approach by \citet{lejeune2020implicit}, sampling is restricted to ensure that the final optimal ensemble comprises only least squares estimators. Specifically, they maintain the number of observations in the subsample greater than the number of features, ensuring the existence of a least squares solution for the subsampled data. In contrast, our method permits arbitrary subsample sizes, which means the optimal ensemble can encompass both subsampled least squares and ridgeless interpolators. This distinction is crucial, as there can be scenarios where the optimal subsample might contain more features than observations, a phenomenon highlighted in \Cref{thm:ridgeless-isotipic-optrisk}.
    (2) \textbf{Signal constraints.} The previous theorem limits itself to isotropic random signals $\bbeta_0$.
    We broaden this scope to incorporate any arbitrary deterministic signals with bounded norms.
    (3) \textbf{Distributional assumptions.} 
    \citet{lejeune2020implicit} assumes strong distributional assumptions on the features, noise, and signal, particularly requiring all of them to follow a Gaussian distribution.
    In comparison, our results do not require such strong distributional assumptions on either the features or the noise and accommodate any deterministic signal with bounded norms.
    
    \begin{remark}
        [Difference between optimal subagged ridgeless and optimal ridge predictors]

        While Theorem \ref{thm:comparison_optimal_ridge} suggests that the two optimal limiting risks coincide under the isotropic model,
        it is important to note the difference in their risk monotonicity properties in the data aspect ratio $\phi$.
        The optimal risk of the subagged ridgeless predictor is expected to remain monotonically decreasing in $\phi$, as shown in \Cref{thm:monotonicity-phi}. In contrast, it is yet to be ascertained whether the optimal ridge predictor has the same property under general models.
        In the isotropic case, the fixed point parameter can be explicitly solved in terms of the parameters $(\lambda,\phi,\phi_s)$.
        The explicit formula enables direct analysis of the monotonicity properties of the asymptotic risk and subsequently facilitates the derivation of the optimal risks.        
        However, in the non-isotropic case, such an explicit formula is not available.
        This lack of an explicit formula calls for a different strategy to extend \Cref{thm:comparison_optimal_ridge} to non-isotropic features.\footnote{Subsequent to finishing work, \Cref{thm:comparison_optimal_ridge} has now been extended for non-isotropic cases in \citet{du2023gcv} by establishing connections between the fixed-point equations involved and utilizing their monotonicity properties.}
    \end{remark}

    \paragraph{\Splagging without replacement.\hspace{-2mm}}
    Unlike subagging, it is possible, though very cumbersome to obtain the optimal sub-sampling ratio $\phi_s^\spl(\phi)$ in this case. It involves solving a cubic equation (for a fixed $M$) or a quartic equation (for the optimal $M$).
    Consequently, we resort to numerical computation for $\phi_s^\star$ and provide a qualitative behavior for $\phi_s$ next.
    We observe that as $\SNR$ increases, the point of phase transition occurs at a larger value of $\phi$.
    This indicates that when there are much more features than samples in the full dataset and the $\SNR$ is relatively large, then \splagging does not help to reduce the prediction risk.
    However, when the $\SNR$ is small, \splagging interpolators is beneficial, even when $n$ is much larger than $p$ in the full data set.

    \paragraph{Subagging versus \splagging.\hspace{-2mm}} 
    The comparison between subagging and splagging methods shows interesting findings in terms of prediction risks.
    Next we briefly summarize these findings concerning the similarities and differences between the two types of bagging strategies for ridgeless predictors.
        From \Cref{fig:risk_varying_SNR_phi}, we observe that for any data aspect ratio $\phi$ and any $\SNR$, subagging can help to reduce the risk with a suitable subsample aspect ratio in the overparameterized regime, if we have enough bags.
        In contrast, \splagging may not help when $\phi<1$ and $\SNR$ is large, even if we optimize over all possible numbers of bags and subsample aspect ratios jointly.
        For the cases when subagging or \splagging is beneficial, the maximal gain compared to the predictor computed on the full data increases as the $\SNR$ decreases.
        When the full data aspect ratio $\phi$ is near $1$, both subagging and \splagging substantially reduce the prediction risk; see
        \Cref{fig:ridgeless-with-replacement-varing-M,fig:ridge-with-replacement-varing-M,fig:ridgeless-without-replacement-varing-M,fig:ridge-without-replacement-varing-M}.
        Most surprisingly, even if the original dataset is heavily underparameterized, overparameterized subagging always helps, as shown in \Cref{fig:comparison_with_without_optimal}(b).
        For example, recall in \Cref{fig:ridgeless-with-replacement-varing-M} when $n = 5000$ and $p = 500$ (which is a favorable case in classical statistics), subagged ridgeless predictors trained on overparameterized subsampled datasets (e.g., with $n = 50$ and $p = 500$) with $M=50$ bags have smaller prediction risk than least squares fitted on the original data.
    
\section{Discussion}
\label{sec:discussion}

In this paper, we provide a generic reduction strategy for characterizing the prediction risk of general bagged predictors (for two bagging strategies of subagging and splagging).
As a function of the number of bags $M$, we show that the asymptotic risk of the $M$-bagged predictor under squared error loss can be expressed as $M^{-1} \mathfrak{R}_1 + (1 - M^{-1}) \mathfrak{R}_{\infty}$, where $\mathfrak{R}_1$ and $\mathfrak{R}_\infty$ represent the asymptotic squared risks of the $M$-bagged predictor with $M = 1$ and $M = \infty$, respectively.
More generally, for a smooth loss function, we show that the risk of the $M$-bagged predictor is sandwiched between similar convex combinations.
In addition, we prescribe a generic cross-validation method to tune the subsample size that aims at obtaining the best subagged predictor, which also serves to monotonize the risk profile of any given prediction procedure.

Following this general strategy, along with certain novel derivations from random matrix theory (to analyze conditional resolvents), we obtain explicit risk characterization for bagged ridge and ridgeless predictors.
The risk expressions reveal bias and variance monotonicity in the number of bags. Comparing different variants of bagging for ridge and ridgeless predictors, we show that subagging (with optimal subsample size) improves upon the divide-and-conquer or the data-splitting approach of averaging the predictors computed on different non-overlapping splits of data (with optimal split size).
This is especially notable in the overparameterized regime, where the latter data-splitting has been recently observed to improve upon the ridgeless predictor computed on the entire data \citep{mucke_reiss_rungenhagen_klein_2022} under sub-Gaussian features.

Surprisingly, our results show that, under a well-specified linear model, subagging on properly chosen ridgeless interpolators always improves upon the ridgeless predictor trained on the complete data, even when the entire data has more observations than the number of features. Moreover, our generic and model-agnostic cross-validation procedure provably yields the best ridgeless interpolators for subagging. Further specializing to the case of isotropic features, we prove that the optimal subagged predictor has the asymptotic risk that matches the unbagged ridge predictor with optimally-tuned regularization parameter.

    \begin{figure}[!t]
        \centering
        \includegraphics[width=0.85\textwidth]{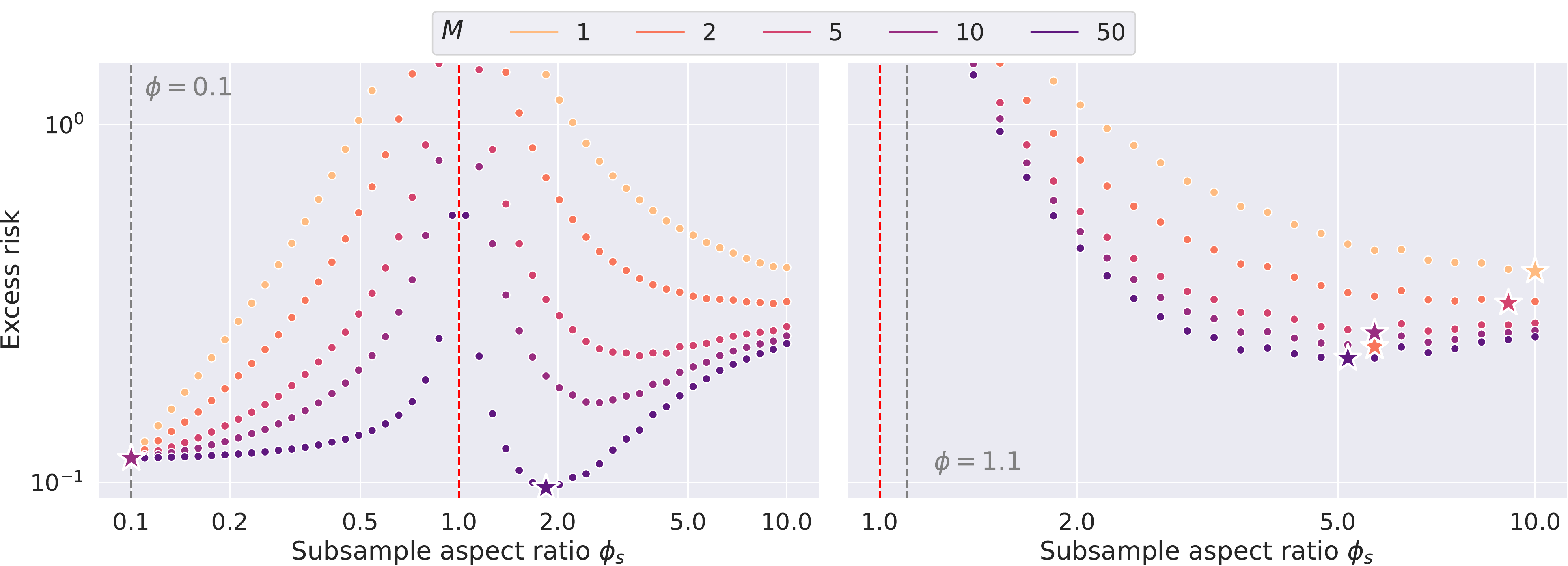}
        \caption{Finite-sample prediction risks for subagged ridgeless predictors ($\lambda=0$) under a nonlinear model, averaged over 100 dataset repetitions, for varying bag size $k=[p/\phi_s]$ and number of bags $M$ with replacement, with $n=[p/\phi]$ and $p=500$.
         The left and the right panels correspond to the cases when $p<n$ ($\phi=0.1$) and $p>n$ ($\phi=1.1$), respectively.
        We generated data from a nonlinear model where the response $y_i$ for $i\in[n]$ is generated from a nonlinear function of $\bx_i$ with additive noise:
        $y_i= \bx_i^{\top}\bbeta_0 + \tfrac{1}{p}(\|\bx_i\|_2^2-\tr(\bSigma_{\mathrm{ar1}})) + \epsilon_i$ and $\bbeta_0,\bX,\bepsilon$ are generated as in \eqref{eq:model-ar1} with $\rhoar=0.25$ and $\sigma^2=1$.
        We observe a similar pattern as in \Cref{fig:ridgeless-with-replacement-varing-M} that the risk of the subagged ridgeless predictor with $M=50$ and $\phi_s\approx 1.5$ is smaller than the risk of the ridgeless predictor fitted on the full data.
        Consequently, it is likely that the key results on subagging continue to hold under more general response models. 
        }
        \label{fig:ridgeless-with-replacement-varing-M-ar1-misspecified}
    \end{figure}
    
Several natural extensions of the current work can be considered going forward.
We briefly discuss two of them below.

First, although our proposed general strategy for analyzing bagged predictors can be helpful for other prediction procedures, we have only derived the precise bagged risk expressions for the ridge and ridgeless regression. In the context of the ridge and ridgeless predictors, we had to develop new random matrix theory tools related to conditional asymptotic equivalents.
It may be necessary to develop similar new tools to analyze other predictors based on our strategy.
A natural prediction procedure to analyze next for bagging is the lasso or lassoless regression.
An empirical investigation of the bagged lassoless predictor has already been conducted by \cite{patil2022mitigating} (see Figure 8, for example).
The traditional analysis of this predictor trained on the full data is performed via approximate message passing (AMP) techniques \citep{li2021minimum}.
It would be interesting to see if our general strategy can be combined with AMP, the convex Gaussian min-max theorem, or the leave-one-out perturbation analysis to yield a more encompassing strategy for bagging analysis.

Second, we have analyzed the bagged ridge and ridgeless predictors under a well-specified linear model. It is interesting to extend the analysis to a general data-distributional setting for two main reasons: (1) to make the results more relevant for practical data analysis, and (2) to investigate whether bagging interpolators can still improve upon the ridgeless predictor trained on the full data.
Regarding (1) above, the techniques developed by \cite{bartlett_montanari_rakhlin_2021} may prove useful in relaxing the linear model assumptions.
Regarding (2) above, we performed a simple simulation study that suggests that even in the misspecified nonlinear model, bagging properly selected interpolators can improve the unbagged ridgeless predictor.
See \Cref{fig:ridgeless-with-replacement-varing-M-ar1-misspecified} for more details. Making these empirical observations more precise presents an exciting avenue for future work.

\section*{Acknowledgements}

We are indebted to Ryan J.\ Tibshirani
for his constant encouragement and advice throughout this work.
We are also indebted to Alessandro Rinaldo, Yuting Wei, Matey Neykov,
and other members of the Operational Overparameterized Statistics (OOPS)
Working Group at Carnegie Mellon University,
for many insightful conversations.
We warmly thank the participants of the Theory of Overparameterized Learning Workshop (TOPML) 2022, in particular Daniel LeJeune, 
for many enlightening discussions.

P.\ Patil was partially supported by ONR grant N00014-20-1-2787. 
A.\ K.\ Kuchibhotla was partially supported by NSF grant DMS-2113611.

\clearpage
\bibliographystyle{apalike}
\bibliography{main}

\appendix      

\setcounter{section}{-1}
\setcounter{equation}{0}
\setcounter{figure}{0}
\renewcommand{\thesection}{S.\arabic{section}}
\renewcommand{\theequation}{E.\arabic{equation}}
\renewcommand{\thefigure}{S.\arabic{figure}}
\clearpage

\begin{center}
\Large
{\bf
Supplement for ``\titletext''
}
\end{center}

This document acts as a supplement to the paper ``\titleRLB.''
The section numbers in this supplement begin with the letter ``S'' and the equation numbers begin with the letter ``E'' to differentiate them from those appearing in the main paper.

\subsection*{Notation}
\label{sec:notation}

Below we provide an overview of some general notation used in the main paper and the supplement.
    
    We denote scalars in non-bold lower or upper case (e.g., $a$), vectors in bold lower case (e.g., $\bb$), and matrices in bold upper case (e.g., $\bC$).
    We denote sets using calligraphic letters (e.g., $\cD$) and use blackboard letters to denote some special sets: $\NN$ denotes the set of positive integers, $\RR$ denotes the set of real numbers, $\RR_{\ge 0}$ denotes the set of non-negative real numbers, $\RR_{> 0}$ denotes the set of positive real numbers, $\CC$ denotes the set of complex numbers, $\CC^{>0}$ denotes the set of complex numbers with positive imaginary part, and $\CC_{<0}$ denotes the set of complex numbers with negative imaginary part.
    For a natural number $n$, we use $[n]$ to denote the set $\{ 1, \dots, n \}$.

    For a real number $a$, $(a)_{+}$ denotes its positive part, $\lfloor a \rfloor$ its floor, and $\lceil a \rceil$ its ceiling.
    For a vector $\bb$, $\| \bb \|_{2}$ denotes its Euclidean norm.
    For a pair of vectors $\bc$ and $\bd$, $\langle \bc, \bd \rangle$ denotes their inner product.
    For an event $E$, $\ind_E$ denotes the associated indicator random variable.
    For a matrix $\bF \in \RR^{n \times p}$, $\bF^\top \in \RR^{p \times n}$ denotes its transpose, and $\bF^{+} \in \RR^{p \times n}$ denotes its Moore-Penrose inverse.
    For a square matrix $\bG \in \RR^{p \times p}$, $\tr[\bG]$ denotes its trace, and $\bG^{-1} \in \RR^{p \times p}$ denotes its inverse, provided it is invertible.
    For a positive semidefinite matrix $\bH$, $\bH^{1/2}$ denotes its principal square root.
    A $p \times p$ identity matrix is denoted $\bI_p$, or simply by $\bI$ when it is clear from the context.
    
    For a real matrix $\bA$, its operator norm with respect to $\ell_2$ vector norm is denoted by $\| \bA \|_{\mathrm{op}}$, and its trace norm is denoted by $\| \bA \|_{\mathrm{tr}}$ (recall that $\| \bA \|_{\mathrm{tr}} = \tr[(\bA^\top \bA)^{1/2}]$).
    For a positive semidefinite matrix $\bB \in \RR^{p \times p}$ with eigenvalue decomposition $\bB = \bC \bD \bC^{-1}$ (where $\bC \in \RR^{p \times p}$ is an orthonormal matrix and $\bD \in \RR^{p \times p}$ is a diagonal matrix with non-negative entries), and a function $f: \RR_{\ge 0} \to \RR_{\ge 0}$, $f(\bB)$ denotes the $p \times p$ positive semidefinite matrix $\bC f(\bD) \bC^{-1}$. 
    The notation $f(\bD)$ refers to a $p \times p$ diagonal matrix obtained by applying the function $f$ to each diagonal entry of $\bD$.
    
    For symmetric matrices $\bA$ and $\bB$, $\bA \preceq \bB$ denotes the Loewner ordering.
    For sequences of matrices $\bC_n$ and $\bD_n$, $\bC_n \asympequi \bD_n$ denotes a certain notion of asymptotic equivalence. For more details, see \Cref{def:deterministic-equivalent,def:deterministic-equivalent-D}.
    We use $O_p$ and $o_p$ to denote probabilistic big-O and little-o notation, respectively.
    We denote convergence in probability by ``$\pto$'', almost sure convergence by ``$\asto$'', and convergence in distribution by ``$\dto$''.

\subsection*{Organization}

Below we outline the structure of the rest of the supplement.

\begin{itemize}[leftmargin=7mm]

    \item 
    In \Cref{sec:setup}, we collect basic background and results from simple random sampling to analyze bagged predictors.
    
    \item In \Cref{sec:app-general-predictor}, 
    we present proofs of results related to general subagged predictors from \Cref{sec:general-predictors}.
    
    \item In \Cref{sec:appendix-with-replacement-ridge,sec:appendix-with-replacement-ridgeless}, we present proof of \Cref{thm:ver-with-replacement} related to subagging from \Cref{sec:bagging-with-replacement} for ridge and ridgeless predictors, respectively.
    The proofs for the two cases are separated due to length.
    However, the proof architecture for the two is similar.
    
    \item In \Cref{sec:appendix-without-replacement}, we present proof of \Cref{thm:ver-without-replacement} related to \splagging from \Cref{sec:bagging-without-replacement} for ridge and ridgeless predictors.
    Because some of this proof builds on that of \Cref{thm:ver-with-replacement}, we can combine the two cases of the ridge and ridgeless predictors, unlike the split cases for
    \Cref{thm:ver-with-replacement}.
    
    \item In \Cref{sec:appendix-risk-properties}, we present proofs of results related to the bias-variance component monotonicity properties in \Cref{prop:monotonicity-M-with-replacement,prop:monotonicity-M-without-replacement} for subagging and \splagging, respectively.
    In this section, we also provide proof of results related to cross-validation and profile monotonicity and those related to oracle properties of optimized bagging from \Cref{sec:monotonizing-risk-profiles}.
    
    \item In \Cref{sec:appendix-isotopic}, we present proofs of specialized results related to subagging and \splagging under isotopic features from \Cref{sec:isotropic_features}.
    
    \item In \Cref{sec:calculus_asymptotic_equivalents}, we formalize several calculus rules for a certain notion of conditional asymptotic equivalence of sequences of matrices that are used in the proofs of constituent lemmas in
    \Cref{sec:appendix-with-replacement-ridge,sec:appendix-with-replacement-ridgeless,sec:appendix-without-replacement}.
    
    \item In \Cref{sec:appendix-concerntration}, we collect various technical helper lemmas related to concentrations and convergences, along with their proofs that are used in proofs in \Crefrange{sec:app-general-predictor}{sec:appendix-without-replacement}.
    
    \item In \Cref{sec:appendix-additional-numerical-result}, we present additional numerical illustrations for \Cref{thm:ver-with-replacement,thm:ver-without-replacement,thm:monotonicity-phi}, and for specialized isotropic results from \Cref{sec:isotropic_features}.
\end{itemize}

\section{Background on simple random sampling}\label{sec:setup}
In~\cite{patil2022mitigating}, the authors have proposed a generic algorithm to improve the risk monotonicity behavior of general predictors using strategies analogous to bagging and boosting.
Specifically in the context of bagging, they have considered bagged predictors obtained by averaging ingredient predictors trained on $M$ subsets of the original data.
In that paper, the authors have characterized the risk of the bagged predictor for $M = 1$ and crudely bounded the risk of $M$-bagged predictor with that of $M = 1$.
However, the numerical simulations presented therein indicate that the $M$-bagged predictor for $M > 1$ can perform significantly better than that for $M = 1$, especially around the interpolation threshold.
In this paper, our primary goal is to develop tools that can help characterize the risk of the $M$-bagged predictor for any $M \ge 1$, and also to analyze the corresponding risk monotonization procedure.
The $M$-bagged predictor considered in the aforementioned work is obtained by simple random sampling with replacement.
In this paper, we also extend our analysis to other versions of bagged predictors to be described shortly.
The discussion in this section and that follows (\Cref{sec:general-predictors}) pertains to developing different versions of general bagged predictors and their risk characterization.

Because the $M$-bagged predictor is defined through simple random sampling with replacement, the well-known results from survey sampling \citep{chaudhuri_2014} are insightful for understanding its risk behavior.
They also provide other versions of the bagged predictors that one can consider.
For this reason, we find it helpful to collect and summarize some results from survey sampling about simple random sampling with and without replacement from an appropriate finite population.
We briefly mention \emph{simple random sampling with replacement} (SRSWR) and \emph{simple random sampling without replacement} (SRSWOR) on an abstract finite population below.

\paragraph{Sampling with replacement.\hspace{-2mm}}
Suppose we have $N$ numbers $\mathcal{N} = \{a_{1}, a_{2}, \ldots, a_{N}\}$, a finite population. Let the set of indices be $\mathcal{I} := \{1, \ldots, N\}$, a finite population. An SRSWR of size $M$ from $\mathcal{I}$ is an i.i.d.\ draw from $\cI$ with the uniform distribution. An unbiased estimator of the average of elements in the finite population $\mathcal{N}$ is given by
\[
\widehat{\mu}^{\WR}_{M,\mathcal{I}} = \frac{1}{M}\sum_{\ell = 1}^M a_{I_{\ell}},
\]
where $\{I_1, I_2, \ldots, I_M\}$ is an SRSWR sample of size $M$ from $\mathcal{I}$. It is very important to stress here that $a_{1}, \ldots, a_{N}$ are all fixed numbers, and only $I_1, \ldots, I_M$ are random. Standard results from survey sampling~\citep[Section 2.5]{chaudhuri_2014} imply that
\begin{equation}\label{eq:SRSWR-formulae}
\mathbb{E}[\widehat{\mu}^{\WR}_{M,\mathcal{I}}] = \frac{1}{N}\sum_{\ell=1}^N a_{\ell} =: \mu,\quad\mbox{and}\quad \mbox{Var}(\widehat{\mu}^{\WR}_{M,\mathcal{I}}) = \frac{1}{M}\left(\frac{1}{N}\sum_{\ell = 1}^N (a_{\ell} - \mu)^2\right).
\end{equation}

\paragraph{Sampling without replacement.\hspace{-2mm}}
An SRSWOR of size $M$ from $\cI$ is a sample drawn without replacement from $\cI$, i.e., $I_1$ is drawn from $\cI$ with each element of $\cI$ being equally likely, $I_2$ is drawn from $\cI\setminus \{I_1\}$ with each element being equally likely, and so on. Define
\[
\widehat{\mu}^{\WOR}_{M,\mathcal{I}} = \frac{1}{M}\sum_{\ell = 1}^M a_{I_\ell},
\]
where $I_1, \ldots, I_M$ are drawn sequentially without replacement from $\cI$. Once again the only randomness in $\widehat{\mu}^{\WOR}_{M,\mathcal{I}}$ stems from the randomness of $I_1, \ldots, I_M$. The results from~\citet[Section 2.5]{chaudhuri_2014} imply that
\begin{equation}\label{eq:SRSWOR-formulae}
\mathbb{E}[\widehat{\mu}^{\WOR}_{M,\mathcal{I}}] = \frac{1}{N}\sum_{\ell=1}^N a_{\ell} =: \mu,\quad\mbox{and}\quad \mbox{Var}(\widehat{\mu}^{\WOR}_{M,\mathcal{I}}) = \frac{N - M}{N-1}\frac{1}{M}\left(\frac{1}{N}\sum_{\ell = 1}^N (a_{\ell} - \mu)^2\right).
\end{equation}

Comparing formulas~\eqref{eq:SRSWR-formulae} and~\eqref{eq:SRSWOR-formulae}, one can note that both the averages are unbiased estimators of the mean of elements in the finite population, the variance of SRSWOR estimator is smaller than SRSWR (whenever $M > 1$)\footnote{Note that $(N-M)/(N-1) = 1 - (M-1)/(N-1) < 1$, if $M > 1$.}, and that the variance of SRSWOR estimator becomes zero if $M = N$. The last fact can be understood by noting that if we draw $N$ elements without replacement from a set of $N$ elements, we end up with the whole set, and no randomness is left.
Note $\widehat{\mu}_{M, \cI}^{\WOR}$
is not well-defined if $M > N$.
This particular restriction of $M > N$
becomes notationally cumbersome 
in the context of bagging and risk monotonization
to be discussed subsequently.
To avoid this,
we define $\widehat{\mu}^\WOR_{M, \cI}
= \widehat{\mu}^\WOR_{N, \cI}$ for $M > N$.
This is natural in the sense that
for $M \ge N$,
when sampling without replacement,
there is no randomness left in the estimator
$\widehat{\mu}^\WOR_{N, \cI}$.
By definition, %
the variance of the estimator $\widehat{\mu}^\WOR_{N, \cI}$ is 0.

\section{Proofs in \Cref{sec:general-predictors} (general bagged predictors)}\label{sec:app-general-predictor}

\subsection[Proof of Proposition \ref{prop:limiting-risk-for-arbitrary-M}
]
{Asymptotic data conditional risk, squared loss}

\begin{proof}[Proof of \Cref{prop:limiting-risk-for-arbitrary-M}]
    The key idea in the proof is to use 
    the conditional risk decomposition from \Cref{prop:bagged-predictors-conditional-mse}.
    Below we present the proof for sampling from $\cI_k$. The proof for sampling from $\cI_k^{\pi}$ is analogous.
    
    \paragraph{SRSWR.\hspace{-2mm}}
    We will do the case of SRSWR from $\cI_k$ first.
    From \Cref{prop:bagged-predictors-conditional-mse}, we have
    \begin{align}
         R(\tf_M;\cD_n) &= \EE_{(\bx,y)}[\mathbb{E}[(\tf_M - y)^2 \mid \mathcal{D}_n, (\bx,y)]] \notag\\
        &= \EE_{(\bx,y)}\left[\mathscr{B}_{\mathcal{I}_k}(\bx, y)\,\mid\, \cD_n\right] + \frac{1}{M}\EE_{(\bx,y)}\left[\mathscr{V}_{\mathcal{I}_k}(\bx, y) \,\mid\, \cD_n\right] \notag\\
        &=  R(\tf_\infty;\cD_n) + \frac{1}{M}C_n, \label{eq:sqauredrisk_decomposition_WR}
    \end{align}
    where $C_n=\EE_{(\bx,y)}\left[\frac{1}{|\mathcal{I}_k|}\sum_{I\in\mathcal{I}_k} \left(\hf(\bx; \mathcal{D}_{I}) - \widetilde{f}_{\infty,\mathcal{I}_k}(\bx)\right)^2 \,\mid\, \cD_n\right]$.

    Since for $M = 1$ and $M = 2$,
    we have
   \begin{align*}
        R(\tf_1;\cD_n)
        &= R(\tf_\infty;\cD_n)
        + C_n, \\
        R(\tf_2;\cD_n)
        &= R(\tf_\infty;\cD_n)
        + \tfrac{C_n}{2}.
   \end{align*}
   We can thus write $R(\tf_\infty;\cD_n)$ and $C_n$
   in terms of $R(\widetilde{f}_{1,\mathcal{I}_k}^{\textup{\texttt{WR}}};\cD_n)$ and $R(\widetilde{f}_{2,\mathcal{I}_k}^{\textup{\texttt{WR}}};\cD_n)$ as
   \begin{align*}
       R(\tf_\infty;\cD_n)
       &= 2 R(\tf_{2};\cD_n) - R(\tf_{1};\cD_n), \\
       C_n 
       &= 2 R(\tf_{1};\cD_n) - 2 R(\tf_{2};\cD_n).
   \end{align*}
   Substituting in \eqref{eq:sqauredrisk_decomposition_WR},
   we obtain
   \begin{align*}
        R(\tf_{M};\cD_n)
        &= 
        2 R(\tf_{2};\cD_n) - R(\tf_{1};\cD_n)
        + \tfrac{1}{M}
        \left( 2 R(\tf_{1};\cD_n) - 2 R(\tf_{2};\cD_n) \right) \\
        &=
        - \left( 1 -  \tfrac{2}{M} \right)
        R(\tf_{1};\cD_n)
        + \left( 2 - \tfrac{2}{M} \right)
        R(\tf_{2};\cD_n).
   \end{align*}
   Thus, subtracting the desired target in \eqref{eq:guarantees-for-WR-WOR} for with replacement from both sides,
   we get
   \begin{align*}
        R(\tf_{M};\cD_n)
        - \left[(2a_2-a_1)+ \frac{2(a_1-a_2)}{M}\right] =& -\left( 1 - \tfrac{2}{M} \right) \left(R(\tf_{1};\cD_n) - a_1\right)
        + \left( 2 - \tfrac{2}{M} \right) \left(R(\tf_{2};\cD_n) - a_2\right).
    \end{align*}
    Taking absolute values on both sides
    and using triangle inequality yields
    \begin{align*}
        \left|
            R(\tf_{M};\cD_n)
            -
            \left[(2a_2-a_1)+ \frac{2(a_1-a_2)}{M}\right]
        \right| 
        \le&
        \left| 1 - \tfrac{2}{M} \right|
        \left| R(\tf_{1};\cD_n) - a_1 \right|
        + \left( 2 - \tfrac{2}{M} \right)
        \left| R(\tf_{2};\cD_n) - a_2 \right|.
   \end{align*}
   Taking supremum over $M$,
   we have
    \begin{align*}
        \sup_{M\in\NN}
        \left|
            R(\tf_{M};\cD_n)
            -\left[(2a_2-a_1)+ \frac{2(a_1-a_2)}{M}\right]
        \right|
        \le&
        \left| R(\tf_{1};\cD_n) - a_1 \right|
        + 2
        \left| R(\tf_{2};\cD_n) - a_2 \right|.
   \end{align*}
    Finally, since we have
    \begin{align*}
        R(\tf_{1};\cD_n)&\asto a_1,\qquad R(\tf_{2};\cD_n)\asto a_2,
    \end{align*}
    the desired claim in \eqref{eq:guarantees-for-WR-WOR} for with replacement follows.

    \paragraph{SRSWOR.\hspace{-2mm}}
    For SRSWOR from $\cI_k$, similarly we have
    \begin{align}
         R(\tf_M;\cD_n) &= \EE_{(\bx,y)}[\mathbb{E}[(\tf_M - y)^2|\mathcal{D}_n, (\bx,y)]] \notag\\
        &= \EE_{(\bx,y)}\left[\mathscr{B}_{\mathcal{I}_k}(\bx, y)\,\mid\, \cD_n\right] + \frac{|\cI_k|-M}{|\cI_k|-1}\frac{1}{M}\EE_{(\bx,y)}\left[\mathscr{V}_{\mathcal{I}_k}(\bx, y) \,\mid\, \cD_n\right] \notag\\
        &=  R(\tf_\infty;\cD_n) + \frac{|\cI_k|-M}{|\cI_k|-1}\frac{1}{M}C_n\notag\\
        &= R(\tf_\infty;\cD_n) - \frac{C_n}{|\cI_k|-1} + \frac{1}{M}\cdot\frac{|\cI_k|C_n}{|\cI_k|-1}, \label{eq:sqauredrisk_decomposition_WOR}
    \end{align}
    where $C_n=\EE_{(\bx,y)}\left[\frac{1}{|\mathcal{I}_k|}\sum_{I\in\mathcal{I}_k} \left(\hf(\bx; \mathcal{D}_{I}) - \widetilde{f}_{\infty,\mathcal{I}_k}(\bx)\right)^2 \,\mid\, \cD_n\right]$.
    Since for $M = 1$ and $M = 2$,
   \begin{align*}
        R(\tf_1;\cD_n)
        &= R(\tf_\infty;\cD_n)
         - \frac{C_n}{|\cI_k|-1} + \frac{|\cI_k|C_n}{|\cI_k|-1}, \\
        R(\tf_2;\cD_n)
        &= R(\tf_\infty;\cD_n)
        -\frac{C_n}{|\cI_k|-1}+ \frac{1}{2}\cdot\frac{|\cI_k|C_n}{|\cI_k|-1}.
   \end{align*}
   We can thus write $R(\tf_\infty;\cD_n) - C_n/(|\cI_k|-1)$ and $|\cI_k|C_n/(|\cI_k|-1)$
   in terms of $R(\widetilde{f}_{1,\mathcal{I}_k}^{\textup{\texttt{WR}}};\cD_n)$ and $R(\widetilde{f}_{2,\mathcal{I}_k}^{\textup{\texttt{WR}}};\cD_n)$ as
   \begin{align*}
       R(\tf_\infty;\cD_n)
        -\frac{C_n}{|\cI_k|-1}
       &= 2 R(\tf_{2};\cD_n) -  R(\tf_{1};\cD_n), \\
       \frac{|\cI_k|C_n}{|\cI_k|-1} 
       &= 2 ( R(\tf_{1};\cD_n) -  R(\tf_{2};\cD_n) ).
   \end{align*}
   Substituting in \eqref{eq:sqauredrisk_decomposition_WOR},
   we obtain
   \begin{align*}
        R(\tf_{M};\cD_n) &=
        2 R(\tf_{2};\cD_n) -  R(\tf_{1};\cD_n)  + \frac{1}{M}\cdot 2 ( R(\tf_{1};\cD_n) -  R(\tf_{2};\cD_n) )\\
        &=- \left( 1- \frac{2}{M} \right)
        R(\tf_{1};\cD_n)
        + 2\left( 1  - \frac{1}{M}  \right)
        R(\tf_{2};\cD_n).
   \end{align*}
   Thus, subtracting the desired target in \eqref{eq:guarantees-for-WR-WOR} for with replacement from both sides,
   we get
   \begin{align*}
        R(\tf_{M};\cD_n)
        - \left[(2a_2-a_1)+ \frac{2(a_1-a_2)}{M}\right] =& -\left( 1 - \tfrac{2}{M} \right) \left(R(\tf_{1};\cD_n) - a_1\right)
        + \left( 2 - \tfrac{2}{M} \right) \left(R(\tf_{2};\cD_n) - a_2\right).
    \end{align*}
    Taking absolute values on both sides
    and using triangle inequality yields
    \begin{align*}
        \left|
            R(\tf_{M};\cD_n)
            -
            \left[(2a_2-a_1)+ \frac{2(a_1-a_2)}{M}\right]
        \right| 
        \le&
        \left| 1 - \tfrac{2}{M} \right|
        \left| R(\tf_{1};\cD_n) - a_1 \right|
        + \left( 2 - \tfrac{2}{M} \right)
        \left| R(\tf_{2};\cD_n) - a_2 \right|.
   \end{align*}
   Taking supremum over $M$,
   we have
    \begin{align*}
        \sup_{M\in\NN}
        \left|
            R(\tf_{M};\cD_n)
            -\left[(2a_2-a_1)+ \frac{2(a_1-a_2)}{M}\right]
        \right|
        \le& \left| R(\tf_{1};\cD_n) - a_1 \right|
        + 2
        \left| R(\tf_{2};\cD_n) - a_2 \right|.
   \end{align*}
    Finally, since we have
    \begin{align*}
        R(\tf_{1};\cD_n)&\asto a_1,\qquad R(\tf_{2};\cD_n)\asto a_2,
    \end{align*}
    the desired claim in \eqref{eq:guarantees-for-WR-WOR} for the case of sampling without replacement follows.
\end{proof}

\subsection[Proof of Proposition \ref{prop:limiting-risk-for-arbitrary-M-cond}
]
{Asymptotic subsample conditional risk, squared loss}

Before we present the proof for \Cref{prop:limiting-risk-for-arbitrary-M-cond}, we first show the upper bound of the squared subsample conditional risk for general $M$.

\begin{lemma}[Bounding the squared subsample conditional risk]\label{lem:squared_risk_decom}
    The subsample conditional prediction
    risk defined in \eqref{eq:unconditional-risk}
    for the bagged predictor $\hf_{M, \cI_k}$ can be bounded as:
    \begin{equation}\label{eq:prop:limiting-risk-for-arbitrary-M-cond-1}
        \begin{split}
            &\left|R(\tf_{M,\cI_k}; \cD_n, \{I_{\ell}\}_{\ell = 1}^M) - \left\{(2b_2 - b_1) + \frac{2(b_1 - b_2)}{M}\right\}\right|\\
            &\le \left|\frac{1}{M}\sum_{\ell = 1}^M R(\tf_{1,\cI_k}; \cD_n, \{I_{\ell}\}) - b_1\right| + 2\left|\frac{1}{M(M-1)}\sum_{i,j\in[M],i\neq j} R(\tf_{2,\cI_k}; \cD_n, \{I_i, I_j\}) - b_2\right|.
        \end{split}
    \end{equation}
\end{lemma}
\begin{proof}[Proof of \Cref{lem:squared_risk_decom}]
    We start by expanding the squared risk as:
    \begin{align*}
        &R(\tf_{M,\cI_k} ; \, \cD_n,\{I_{\ell}\}_{\ell = 1}^{M}) \\
        &= \int\left(y
        -\frac{1}{M}\sum\limits_{\ell=1}^M\hf(\bx;\cD_{I_\ell})\right)^2 \,\rd P(\bx,y) \\
        &=
        \int
        \left(
        \frac{1}{M}
        \sum_{\ell = 1}^{M}
        \big(y - \hf(\bx; \cD_{I_\ell})\big)
        \right)^2
        \, \mathrm{d}P(\bx, y) \\
        &=
        \frac{1}{M^2}
        \sum_{\ell=1}^{M}
        \int \big(y - \hf(\bx; \cD_{I_\ell})\big)^2
        \, \mathrm{d}P(\bx, y)  + 
        \frac{1}{M^2}
        \sum_{i = 1}^{M}
       \sum\limits_{\substack{j=1 \\ j \neq i}}^{M}
        \int
        \big(y - \hf(\bx; \cD_{I_i})\big) \big(y - \hf(\bx; \cD_{I_j})\big)
        \, \mathrm{d}P(\bx, y) \\
        &=
        \frac{1}{M^2}
        \sum_{\ell=1}^{M}
        R(\tf_{1, \cI_k}; \cD_n, I_\ell) + 
        \frac{1}{M^2}
        \sum_{i = 1}^{M}
       \sum\limits_{\substack{j=1 \\ j \neq i}}^{M}
        \int
        (y - \hf(\bx; \cD_{I_i})) (y - \hf(\bx; \cD_{I_j}))
        \, \mathrm{d}P(\bx, y) \\
        &\stackrel{(i)}{=}
        \frac{1}{M^2}
        \sum_{\ell=1}^{M}
        R(\tf_{1, \cI_k}; \cD_n, I_\ell) \\
        &\qquad
        + 
        \frac{1}{M^2}
        \sum_{i = 1}^{M}
       \sum\limits_{\substack{j=1 \\ j \neq i}}^{M}
        \int
        \frac{1}{2}
        \left\{
        4
        \Big(y - \frac{1}{2} \big(\hf(\bx; \cD_{I_i}) + \hf(\bx; \cD_{I_j})\big)  \Big)^2
        - \big(y - \hf(\bx; \cD_{I_i})\big)^2
        - \big(y - \hf(\bx; \cD_{I_j})\big)^2
        \right\}
        \, \mathrm{d}P(\bx, y) \\
        &=
        \frac{1}{M^2}
        \sum_{\ell=1}^{M}
        R(\tf_{1, \cI_k}; \cD_n, I_\ell) \\
        &\qquad
        + 
        \frac{1}{M^2}
        \sum_{i = 1}^{M}
       \sum\limits_{\substack{j=1 \\ j \neq i}}^{M}
       \frac{1}{2}
       \left\{
       4 R(\hf_{2, \cI_k}; \cD_n; I_{i}, I_{j})
       - R(\tf_{1, \cI_k}; \cD_{n}; I_i)
       - R(\tf_{1, \cI_k}; \cD_{n}; I_j)
       \right\} \\
        &=
        \frac{1}{M^2}
        \sum_{\ell=1}^{M}
        R(\tf_{1, \cI_k}; \cD_n, I_\ell)
        - \frac{1}{2M^2}
        \sum_{i=1}^{M} 
        \sum\limits_{\substack{j=1 \\ j \neq i}}^{M}
        R(\tf_{1,\cI_k}; I_i)
        - \frac{1}{2M^2}
        \sum_{i=1}^{M} 
        \sum\limits_{\substack{j=1 \\ j \neq i}}^{M}
        R(\tf_{1,\cI_k}; I_j)
        + 
        \frac{1}{M^2}
        \sum_{i=1}^{M}
        \sum\limits_{\substack{j = 1 \\j \neq i}}^{M}
       2 R(\hf_{2, \cI_k}; \cD_n; I_{i}, I_{j}) \\
       &=
       \frac{1}{M^2} \sum_{\ell = 1}^{M} R(\tf_{1, \cI_k}; \cD_n; I_{\ell})
       - \frac{1}{2M^2} 
       \cdot 2  \cdot (M - 1) \sum_{\ell=1}^{M}
       R(\tf_{1, \cI_k}; I_{\ell})
       + \frac{2}{M^2}
       \sum\limits_{\substack{i, j \in [M] \\ i \neq j}}
       R(\hf_{2, \cI_k}; \cD_n; I_i, I_j) \\
       &=
       \left(
       \frac{1}{M^2} 
       - \frac{(M - 1)}{M^2} 
       \right)
       \sum_{\ell = 1}^{M} R(\tf_{1, \cI_k}; \cD_n; I_{\ell})
       + \frac{2}{M^2}
       \sum\limits_{\substack{i, j \in [M] \\ i \neq j}}
       R(\hf_{2, \cI_k}; \cD_n; I_i, I_j) \\
        &= -\left(\frac{1}{M}-\frac{2}{M^2}\right)\sum_{\ell=1}^M R(\tf_{1,\cI_k}; \, \cD_n,\{I_{\ell}\}) + \frac{2}{M^2}\sum\limits_{\substack{i,j\in[M]\\i\neq j}}  R(\tf_{2,\cI_k} ; \, \cD_n, \{I_{i},I_j\}).
    \end{align*}
    In the expansion above,
    for equality $(i)$,
    we used the fact that $ab = \{ 4(a/2 + b/2)^2 - a^2 - b^2 \} /2$.
    
    Now, subtracting the desired limit on both sides yields
    \begin{align*}
        &\left|R(\tf_{M,\cI_k} ; \, \cD_n,\{I_{\ell}\}_{\ell = 1}^{M})- \left\{(2b_2-b_1)+ \frac{2(b_1-b_2)}{M}\right\}\right|  \nonumber \\
        &= \left|-\left(\frac{1}{M}-\frac{2}{M^2}\right)\sum_{\ell=1}^M (R(\tf_{1,\cI_k}; \, \cD_n, \{I_{\ell}\}) - b_1) + \frac{2}{M^2}\sum\limits_{\substack{i,j\in[M]\\i\neq j}} (R(\tf_{2,\cI_k} ; \, \cD_n, \{I_{i},I_j\}) - b_2) \right|  \nonumber \\
        &\leq \left|1 - \frac{2}{M}\right|\cdot  \left|\frac{1}{M}\sum_{\ell=1}^M R(\tf_{1,\cI_k}; \, \cD_n, \{I_{\ell}\}) - b_1\right| +  \frac{2(M-1)}{M} \left|\frac{1}{M(M-1)}\sum\limits_{\substack{i,j\in[M] \nonumber \\i\neq j}}R(\tf_{2,\cI_k} ; \, \cD_n, \{I_{i},I_j\}) - b_2\right| \nonumber \\
        &\leq   \left|\frac{1}{M}\sum_{\ell=1}^M R(\tf_{1,\cI_k}; \, \cD_n, \{I_{\ell}\}) - b_1\right| +  2\left|\frac{1}{M(M-1)}\sum\limits_{\substack{i,j\in[M]\\i\neq j}}R(\tf_{2,\cI_k} ; \, \cD_n, \{I_{i},I_j\}) - b_2\right|.
    \end{align*}
    This completes the proof of the upper bound.

\end{proof}

Next, we present the proof of \Cref{prop:limiting-risk-for-arbitrary-M-cond}.

\begin{proof}[Proof of \Cref{prop:limiting-risk-for-arbitrary-M-cond}]

\Cref{lem:risk_general_predictor_M12} implies the asymptotics for the data conditional risk.
Now, consider the asymptotics for the subsample conditional risk of the bagged predictors. 
From \eqref{eq:prop:limiting-risk-for-arbitrary-M-cond-1}
of \Cref{lem:squared_risk_decom}, 
it holds that
\begin{equation}
    \begin{split}
        &\left|R(\tf_{M,\cI_k}; \cD_n, \{I_{\ell}\}_{\ell = 1}^M) - \left\{(2b_2 - b_1) + \frac{2(b_1 - b_2)}{M}\right\}\right|\\
        &\le \left|\frac{1}{M}\sum_{\ell = 1}^M R(\tf_{1,\cI_k}; \cD_n, \{I_{\ell}\}) - b_1\right| + 2\left|\frac{1}{M(M-1)}\sum_{i,j\in[M],i\neq j} R(\tf_{2,\cI_k}; \cD_n, \{I_i, I_j\}) - b_2\right|.
    \end{split}
\end{equation}
This implies that
\begin{align*}
&\sup_{M\in\mathbb{N}}\left|R(\tf_{M,\cI_k}; \cD_n, \{I_{\ell}\}_{\ell = 1}^M) - \left\{(2b_2 - b_1) + \frac{2(b_1 - b_2)}{M}\right\}\right| \\
&\le \sup_{I\in\cI_k}|R(\tf_{1,\cI_k}; \cD_n, \{I\}) - b_1| + 2\sup_{M\ge2}\left|\frac{1}{M(M-1)}\sum_{i,j\in[M],i\neq j} R(\tf_{2,\cI_k}; \cD_n, \{I_i, I_j\}) - b_2\right|.
\end{align*}
The first term on the right hand side converges almost surely to zero by \Cref{lem:conv_cond_expectation}~\ref{lem:conv_cond_expectation-max}. To prove that the second term converges to zero, we start by noting that
\[
U_M = \frac{1}{M(M-1)}\sum_{i,j\in[M],i\neq j} \left\{R(\tf_{2,\cI_k}; \cD_n, \{I_i, I_j\}) - b_2\right\},
\]
is a $U$-statistics based on either an SRSWR or an SRSWOR sample $I_1, \ldots, I_M$ conditional on $\cD_n$. Theorem 2 in Section 3.4.2 of~\cite{lee2019u} implies that $\{U_M\}_{M\ge2}$ is a reverse martingale conditional on $\cD_n$ with respect to some filtration, when we have an SRSWR sample (which is the same as an i.i.d.\ sample). Lemma 2.1 of~\cite{sen1970hajek} proves the same result when we have an SRSWOR sample. This, combined with Theorem 3 (maximal inequality for reverse martingales) in Section 3.4.1 of~\cite{lee2019u} (for $r = 1$\footnote{Theorem 3 of Section 3.4.1 is only stated with $r > 1$, but from the proof, it is clear that $r = 1$ is a valid choice. }) yields
\begin{align*}
\mathbb{P}\left(\sup_{M\ge2}|U_M| \ge \delta \mathrel{\big|} \cD_n\right) &\le \frac{1}{\delta}\mathbb{E}\left[|U_2| \mathrel{\big|} \cD_n\right]\\ 
&= \frac{1}{\delta}\mathbb{E}\left[|R(\tf_{2,\cI_k}; \cD_n, \{I_1, I_2\}) - b_2| \mathrel{\big|} \cD_n\right].
\end{align*}
The right-hand side we know converges to zero almost surely. To see this, we first write as before the right-hand side as $\mathbb{E}[|R(\tf_{2,\cI_k}; \cD_n, \{I_1, I_2\})-b_2| \mid \cD_n = \cD_n(\omega)] = \mathbb{E}[|R(\tf_{2,\cI_k}; \cD_n(\omega), \{I_1, I_2\})-b_2|]$. We know that for all $\omega\in\mathcal{A}$, $R(\tf_{2,\cI_k}; \cD_n(\omega), \{I_1, I_2\}) \asto b_2$ as $n\to\infty$ (from the given assumption). Also, we know~\eqref{eq:every-omega-inequality} and that the right hand side of~\eqref{eq:every-omega-inequality} converges in $L_1$ to its probability limit. Hence, Vitali's theorem \citep[Theorem 4.5.4]{bogachev2007measure} implies that $\mathbb{E}[|R(\tf_{2,\cI_k}; \cD_n, \{I_1, I_2\}) - b_2| \mid \cD_n = \cD_n(\omega)]$ converges to zero for all $\omega\in\mathcal{A}$ as $n\to\infty$. Therefore, as $n\to\infty$, for all $\omega\in\mathcal{A}$,
\[
\mathbb{P}\left(\sup_{M\ge2}|U_M| \ge \delta \mathrel{\big|} \cD_n = \cD_n(\omega)\right) \to 0.
\]
Because probabilities are bounded by one, dominated convergence theorem implies that
\[
\mathbb{P}\left(\sup_{M\ge2}|U_M| \ge \delta\right) \to 0,\quad\mbox{as}\quad n\to\infty.
\]
Therefore,
\[
\sup_{M\in\mathbb{N}}\left|R(\tf_{M,\cI_k}; \cD_n, \{I_{\ell}\}_{\ell = 1}^M) - \left\{(2b_2 - b_1) + \frac{2(b_1 - b_2)}{M}\right\}\right| \pto 0.
\]
\end{proof}

\subsection[Proof of \Cref{prop:convex-sconvex-smooth} 
]
{Conditional risk bounds for convex, strongly-convex, and smooth losses}

\begin{proof}[Proof of \Cref{prop:convex-sconvex-smooth}]
We split the proof into two parts, 
depending on the assumption imposed on the loss function $L$.

    \paragraph{Part (1).\hspace{-2mm}}
    For any loss function $L:\mathbb{R}\times\mathbb{R}\to\mathbb{R}$ convex in the second argument,
    one can trivially obtain
    \begin{equation}
    \begin{split}
        R(\tf_{M, \cI_{k}}; \cD_n) &= \mathbb{E}[L(y, \tf_{M,\cI_k}(\bx)) \mid \cD_n]\\ &= \mathbb{E}[\mathbb{E}[L(y, \tf_{M,\cI_k}(\bx)) \mid \{I_{\ell}\}_{\ell = 1}^M] \mid \cD_n]\\ 
        &\ge \mathbb{E}[L(y, \mathbb{E}[\tf_{M,\cI_k}(\bx) \mid \{I_{\ell}\}_{\ell = 1}^M]) \mid \cD_n].
    \end{split}
    \end{equation}
    Here the last inequality follows from Jensen's inequality. Because $\mathbb{E}[\tf_{M,\cI_k}(\bx) \mid \{I_{\ell}\}_{\ell = 1}^M] = \tf_{\infty,\cI_k}(\bx)$, we get for any $M \ge 1$,
    \[
    R(\tf_{M,\cI_k}; \cD_n) \ge R(\tf_{\infty, \cI_k}; \cD_n).
    \]
    On the other hand, we have by Jensen's inequality
    \[
    R(\tf_{M,\cI_k}; \cD_n) = \mathbb{E}\left[L\left(y, \frac{1}{M}\sum_{\ell = 1}^M \tf(\bx; \cD_{I_{\ell}})\right) \mathrel{\Big|} \cD_n \right] 
    \le \mathbb{E}\left[\frac{1}{M}\sum_{\ell = 1}^M L(y, \tf(\bx; \cD_{I_{\ell}})) \mathrel{\Big|} \cD_n\right] = R(\tf_{1,\cI_k}; \cD_n).
    \]
    Summarizing, we get that for any $M \ge 1$,
    \[
    R(\tf_{1,\cI_k}; \cD_n) \geq R(\tf_{M,\cI_k}; \cD_n) \geq R(\tf_{\infty, \cI_k}; \cD_n).
    \]
    One can further obtain the monotonicity property by noting that for any $M \ge 1$, 
    \[
    \tf_{M+1,\mathcal{I}_k}(\bx, \{\cD_{I_{\ell}}\}_{\ell = 1}^{M+1}) = \frac{1}{M+1}\sum_{\ell = 1}^{M+1} \tf(\bx; \cD_{I_{\ell}}) = \frac{1}{(M+1)!}\sum_{\pi'}\left(\frac{1}{M}\sum_{\ell = 1}^M \tf(\bx; \cD_{I_{\pi'(\ell)}})\right),
    \]
    where $\pi'$ represents a permutation of $\{1, 2, \ldots, M+1\}.$ Therefore, for any loss function $L:\mathbb{R}\times \mathbb{R}\to\mathbb{R}$ that is convex in the second argument, we get
    \[
    L(y, \tf_{M+1,\cI_k}(\bx; \{\cD_{I_{\ell}}\}_{\ell = 1}^{M+1})) \le \frac{1}{(M+1)!}\sum_{\pi'} L\left(y, \tf(\bx; \{\cD_{I_{\pi'(\ell)}}\}_{\ell = 1}^M)\right).
    \]
    Because any (non-random) subset of a simple random sample with/without replacement is itself a simple random sample with/without replacement, taking conditional expectation on both sides conditional on $\cD_n$ yields
    \[
    R(\tf_{M+1,\cI_k}; \cD_n) \le R(\tf_{M, \cI_k}; \cD_n).
    \]
    This, in particular, implies that $R(\tf_{\infty,\cI_k}; \cD_n) \le R(\tf_{M,\cI_k}; \cD_n) \le R(\tf_{1, \cI_k}; \cD_n)$ for any $M\ge1$. 
    This finishes the proof of the first part of the statement.
    
    \paragraph{Part (2).\hspace{-2mm}}
    If
    we assume that the loss function is strongly convex and differentiable in the second argument, then we can improve the lower bound of Part 1 in terms of $\tf_{\infty}$. Formally, if $L:\mathbb{R}\times\mathbb{R}\to\mathbb{R}$ is $\underline{m}$-strongly convex, i.e., $L(a, b) - \underline{m} / 2 b^2$ is convex in $b$ (for every $a$), then
    \[
    L(y, \tf_{M,\cI_k}(\bx)) \ge L(y, \tf_{\infty, \cI_k}(\bx)) + \frac{\partial L(y, \tf_{\infty, \cI_k}(\bx))}{\partial b}(\tf_{M,\cI_k}(\bx) - \tf_{\infty, \cI_k}(\bx)) + \frac{\underline{m}}{2}(\tf_{M,\cI_k}(\bx) - \tf_{\infty, \cI_k}(\bx))^2.
    \]
    Applying~\Cref{prop:bagged-predictors-conditional-mse} and taking the expectation $(\bx,y)$ conditional on $\cD_n$, we obtain
    \begin{equation}
    \label{eq:sconvex-risk-lb}
    R(\tf_{M,\cI_k}; \cD_n) \ge R(\tf_{\infty, \cI_k}; \cD_n) + 
    \frac{\underline{m}}{2}
    \frac{1}{M}\int \frac{1}{|\cI_k|}\sum_{I\in\cI_k} (\widehat{f}(\bx; \cD_I) - \tf_{\infty, \cI_k}(\bx))^2\,\rd P(\bx, y).
    \end{equation}
    
    On the other hand,
    if we assume that the loss function $L : \RR \times \RR \to \RR$
    is $\overline{m}$ smooth in the second argument, then
    \[
        L(a, b)
        \le
        L(a, b')
        + \frac{\partial L(a, b')}{\partial b}
         (b - b')
        + \frac{\overline{m}}{2} (b - b')^2.
    \]
    It follows that
    \begin{equation}
    \label{eq:smooth-risk-ub}
        R(\tf_{M, \cI_k}; \cD_n)
        \le R(\tf_{\infty, \cI_k}; \cD_n)
        + \frac{\overline{m}}{2}
        \frac{K_{|\cI_k|,M}}{M}
        \int 
        \sum_{I \in \cI_k}
        (\hf(\bx; \cD_I) - \tf_{\infty, \cI_k}(\bx))^2 \, \rd P(\bx, y).
    \end{equation}
    Combining 
    the lower bound from \eqref{eq:sconvex-risk-lb}
    and the upper bound from \eqref{eq:smooth-risk-ub}
    finishes the proof of the second part of the statement.
\end{proof}

\subsection[Proof of Lemma \ref{lem:risk_general_predictor_M12}
]
{From subsample conditional to data conditional risk, $M = 1, 2$}
\begin{proof}[Proof of \Cref{lem:risk_general_predictor_M12}]
    Let us first prove the result when sampling with/without replacement from $\cI_k$. The proof for $\cI_k^{\pi}$ would be analogous.
    Note that $R(\tf_1; \cD_n) = \mathbb{E}[R(\tf_{1,\cI_k}; \cD_n, \{I_1\}) \mid \cD_n]$ where the expectation is taken over a random draw $I_1$ from $\cI_k$.
    We are given that $R(\tf_{1,\cI_k}; \cD_n, \{I\}) - b_1 \asto 0$
    for every $I\in\cI_k$.
    Although not explicitly highlighted, for clarity it is worth reminding that $I$ is a sequence implicitly indexed by $n$.
    Under this condition, let us note that 
    \begin{align*}
    \left|\mathbb{E}[R(\tf_{1,\cI_k}; \cD_n, \{I_1\}) \mid \cD_n] - b_1\right| &= \left|\frac{1}{|\cI_k|}\sum_{I\in\cI_k} R(\tf_{1,\cI_k}; \cD_n, \{I\}) - b_1\right|\\ 
    &\le \frac{1}{|\cI_k|}\sum_{I\in\cI_k} |R(\tf_{1,\cI_k}; \cD_n, \{I\}) - b_1|\\
    &\le \max_{I\in\cI_k}|R(\tf_{1,\cI_k}; \cD_n, \{I\}) - b_1|\\
    &\asto 0,
    \end{align*}
    by \Cref{lem:conv_cond_expectation}~\ref{lem:conv_cond_expectation-max}, since again the conditional risk for $M = 1$ converges for any sequences of indices.
    To be explicit, the underlying triangular array invoked in \Cref{lem:conv_cond_expectation}~(1) is as described follows: 
    For each $n$, recall $\cI_k = \{I^{(1)},\ldots,I^{(N_n)}\}$ where $N_n=|\cI_k|= \binom{n}{k}$.
    Note that the random quantity
    $R_{n, \ell} := R(\tf_{1,\cI_k}; \cD_n, I^{(\ell)})$ is indexed by $n$ and $\ell \in [N_n]$.
    For $I$ drawn from $\cI_k$, let $p_n$ be the index such that $I=I^{(p_n)}$.
    The convergence then follows from applying \Cref{lem:conv_cond_expectation}~(1) to the triangular array $R_{n,p_n}$.
    Hence, we proved that
    \begin{equation}\label{eq:one-bag-almost-sure-convergence-expected-value}
    R(\tf_{1,\cI_k}; \cD_n) \asto b_1,\quad\mbox{as}\quad n\to\infty.
    \end{equation}
    Now, observe that
    \begin{equation}\label{eq:Inequality-relating-M=2-and-M=1}
    R(\tf_{2,\cI_k}; \cD_n, \{I_{\ell}\}_{\ell = 1}^2) \le \frac{1}{2}R(\tf_{1,\cI_k}; \cD_n, \{I_1\}) + \frac{1}{2}R(\tf_{1,\cI_k}; \cD_n, \{I_2\}).
    \end{equation}
    We will now apply Pratt's lemma (see, e.g., Theorem 5.5 of \cite{gut_2005} or Chapter 5 Exercise 30 of \cite{resnick_2019}) to prove almost sure convergence of $\mathbb{E}[R(\tf_{2,\cI_k}; \cD_n, \{I_{\ell}\}_{\ell = 1}^2) \mid \cD_n]$. Usually, Pratt's lemma is applied unconditionally, and here we apply it conditional on $\cD_n$. For an easier understanding of the proof, let us write $\cD_n(\omega)$ in place of $\cD_n$ in order to make it clear that we are conditioning on $\cD_n$. Recall that $\cD_n$ is independent of the subsamples $\{I_{\ell}\}_{\ell = 1}^M$ for any $M\ge1$. In this notation, inequality~\eqref{eq:Inequality-relating-M=2-and-M=1} becomes
    \begin{equation}\label{eq:every-omega-inequality}
    0 \le R(\tf_{2,\cI_k}; \cD_n(\omega), \{I_{\ell}\}_{\ell = 1}^2) \le \frac{1}{2}R(\tf_{1,\cI_k}; \cD_n(\omega), \{I_1\}) + \frac{1}{2} R(\tf_{1,\cI_k}; \cD_n(\omega), \{I_2\}).
    \end{equation}
    Because $R(\tf_{1,\cI_k}; \cD_n, \{I\})\asto b_1$ for every $I\in\cI_k$, there exists a set $\mathcal{A}\subseteq\Omega$ such that $\mathbb{P}(\mathcal{A}) = 1$ and for all $\omega\in\mathcal{A}$, $R(\tf_{1,\cI_k}; \cD_n(\omega), \{I\})\asto b_1$ for every $I\in\cI_k$. Applying Pratt's lemma for every $\omega\in \mathcal{A}$, as $n\to\infty$ and using the fact~\eqref{eq:one-bag-almost-sure-convergence-expected-value} as well as the assumption $R(\tf_{2,\cI_k}; \cD_n(\omega), \{I_{\ell}\}_{\ell = 1}^2) \asto b_2$, we get that
    \[
    \mathbb{E}[R(\tf_{2,\cI_k}; \cD_n(\omega), \{I_{\ell}\}_{\ell = 1}^2)] ~\to~ b_2,\quad\mbox{for all}\quad \omega\in\mathcal{A}.
    \]
    Note that $R(\tf_{2,\cI_2}; \cD_n(\omega)) = \mathbb{E}[R(\tf_{2,\cI_k}; \cD_n, \{I_{\ell}\}_{\ell = 1}^2) \mid \cD_n = \cD_n(\omega)]$. Therefore, we conclude 
    \begin{equation}\label{eq:two-bag-almost-sure-convergence-expected-value}
    R(\tf_{2,\cI_2}; \cD_n) \asto b_2,\quad\mbox{as}\quad n\to\infty.
    \end{equation}
    Therefore, \eqref{eq:data_cond_risk_M12} applies to yield asymptotics for the data conditional risk uniformly over $M\in\mathbb{N}$.
\end{proof}

\subsection[Proof of Theorem \ref{thm:risk_general_predictor}
]
{From subsample conditional to data conditional risk, general $M$}
\begin{proof}[Proof of \Cref{thm:risk_general_predictor}]
    The proof follows by combining 
    \Cref{prop:limiting-risk-for-arbitrary-M,prop:limiting-risk-for-arbitrary-M-cond},
    and \Cref{lem:risk_general_predictor_M12}.
\end{proof}

\section
[Proof of \Cref{thm:ver-with-replacement} (subagging with replacement, ridge predictor)]
{Proof of \Cref{thm:ver-with-replacement} (subagging with replacement, ridge predictor)}
\label{sec:appendix-with-replacement-ridge}

    For $\tfWR{M}{\cI_k}$ defined in \Cref{thm:ver-with-replacement}, we present the proof for ridge and ridgeless predictors in \Cref{thm:ver-ridge,thm:ver-ridgeless}.
    For $\tfWOR{M}{\cI_k}$ defined in \Cref{thm:ver-with-replacement}, the conclusion still holds since the limits of the proportions of intersection between two SRSWR and SRSWOR draws from $\cI_k$ are the same from \Cref{lem:i0_mean}.
    For proving the asymptotic conditional risks, we will treat $\cI_k$ as fixed and use $\tfWR{\lambda}{M}$ to denote the ingredient predictor associated with regularization parameter $\lambda$.
    
    \subsection{Proof assembly}
    Before we present the proof, recall the nonnegative constants defined in \eqref{eq:fixed-point-ridge} and \eqref{eq:v0-def}: $v(-\lambda;\theta)\geq 0$ is the unique solution to the fixed-point equation
    \begin{align}
        v(-\lambda;\theta)^{-1} &= \lambda + \theta\int r(1 + v(-\lambda;\theta)r)^{-1}\,\rd H(r) \label{eq:v_ridge},
    \end{align}
    and the nonnegative constants $\tv(-\lambda;\vartheta,\theta)$, and $\tc(-\lambda;\theta)$ are defined via the following equations
    \begin{align}
        \tv(-\lambda,\vartheta,\theta) 
        &= \frac{\vartheta\int r^2 (1+ v(-\lambda; \theta)r)^{-2}\,\rd H(r)}{v(-\lambda; \theta)^{-2}-\vartheta \int r^2 (1+ v(-\lambda; \theta)r)^{-2}\,\rd H(r)},\ 
        \tc(-\lambda;\theta)=\int   r (1+v(-\lambda; \theta))r)^{-2} \,\rd G(r). \label{eq:tv_tc_ridge}
    \end{align}
    It helps to first slightly rewrite the statement of \Cref{thm:ver-with-replacement} for $\lambda > 0$ as follows.
    Though it suffices to analyze the case $M=2$ according to \Cref{thm:risk_general_predictor}, below we will do the risk decomposition for general $M$.
    
    \begin{theorem}[Risk characterization of subagged ridge predictor]\label{thm:ver-ridge}
        Let $\tfWR{\lambda}{M}$ be the ingredient predictor as defined in \eqref{eq:ingredient-predictor} for $\lambda > 0$.
        Suppose that Assumptions \ref{asm:rmt-feat}-\ref{asm:spectrum-spectrumsignproj-conv} hold, then for $M=\{1,2,3,\ldots\}$, as $k,n,p\rightarrow\infty$, $p/n\rightarrow\phi\in[0,\infty)$ and $p/k\rightarrow\phi_s\in[\phi,\infty]$, there exists a deterministic function $\RlamM{M}{\phi}$ such that for $I_1, \ldots, I_M \overset{\texttt{\textup{SRSWR}}}{\sim} \mathcal{I}_k$, 
        \begin{align*}
            \sup_{M\in\NN}| R(\tfWR{\lambda}{M}; \cD_n,\{{I_{\ell}}\}_{\ell=1}^M) - \RlamM{M}{\phi} | \pto 0,
        \end{align*}
        and
        $$\sup_{M\in\NN}|R(\tfWR{\lambda}{M};\cD_n) - \RlamM{M}{\phi}| \asto 0.$$
        Furthermore, $\RlamM{M}{\phi}$ decomposes as $$ \RlamM{M}{\phi}:= \sigma^2 + \BlamM{M}{\phi}  + \VlamM{M}{\phi},$$ 
        where $\BlamM{M}{\phi} = M^{-1}  B_{\lambda}(\phi,\phi_s) + (1-M^{-1}) B_{\lambda}(\phi,\phi_s)$, 
        and $\VlamM{M}{\phi} = M^{-1} V_{\lambda}(\phi_s,\phi_s) + (1 - M^{-1}) V_{\lambda}(\phi,\phi_s) $ with
        \begin{align*}
            B_{\lambda}(\vartheta, \theta)
            = \rho^2  (1+\tv(-\lambda; \vartheta, \theta)) \tc(-\lambda; \theta),\qquad
            V_{\lambda}(\vartheta, \theta)
            =\sigma^2\tv (-\lambda; \vartheta, \theta),\qquad \theta \in (0, \infty], \; \vartheta \le \theta,
        \end{align*}
        where $\tv(-\lambda; \vartheta, \theta) $ and $\tc(-\lambda; \theta)$ are as defined in \eqref{eq:tv_tc_ridge}.
    \end{theorem}    
    \begin{proof}[Proof of \Cref{thm:ver-ridge}]
        In what follows, we will prove the results for $n,k,p$ being a sequence of integers $\{n_m\}_{m=1}^{\infty}$, $\{k_m\}_{m=1}^{\infty}$, $\{p_m\}_{m=1}^{\infty}$.
        For simplicity, we drop the subscript when it is clear from the context.
        
        For any $m\in[M]$, let $I_m$ be a sample from $\cI_k$, and $\bL_m\in\RR^{n\times n}$ be a diagonal matrix with $(\bL_{m})_{ll}=1$ if $l\in I_m$ and 0 otherwise.
        An illustration of these notations for $M=2$ is shown in Figure \ref{fig:dataset}.
        The proof will reduce to analyze the individual terms concerning one dataset $\cD_{I_m}$, or the cross terms concerning $\cD_{I_m}$ and $\cD_{I_l}$ for $m\neq l$.
        \begin{figure}[!ht]
            \centering
            \begin{tikzpicture}
        
                \node[rectangle,draw,
                    fill = lightgray, 
                    minimum width = 1cm, 
                    minimum height = 1cm] (rect10) at (0,0) {};
                    
                \node[below of =rect10,
                    node distance=1.0cm,
                    rectangle, draw,
                    fill = lightgray,
                    minimum width = 1cm, 
                    minimum height = 1.2cm] (rect11) at (0,0) {};
                
                \node[right of =rect10,
                    node distance=2cm,
                    rectangle, draw,
                    fill = lightgray,
                    minimum width = 1cm, 
                    minimum height = 1cm] (rect20) {};
                    
                \node[below of =rect20,
                    node distance=2.2cm,
                    rectangle, draw,
                    fill = lightgray,
                    minimum width = 1cm, 
                    minimum height = 1.2cm] (rect22) {};

                \node[above of =rect10,
                        node distance=1cm
                        ] (text1) {$p$};
                        
                \node[below of =rect10,
                        node distance=3.2cm
                        ] (textdata1) {$\cD_{I_1}$};
                        
                \node[above of =rect20,
                        node distance=1cm
                        ] (text2) {$p$};
                        
                \node[below of =rect20,
                        node distance=3.2cm
                        ] (textdata2) {$\cD_{I_2}$};
                        
                \node[left of =rect10,
                        node distance=1.5cm
                        ] (textsample1) {$i_0$};
                        
                \node[left of =rect11,
                        node distance=1.5cm
                        ] (textsample2) {$k-i_0$};
                        
                \node[left of =rect22,
                        node distance=3.5cm
                        ] (textsample2) {$k-i_0$};
             
                \draw[] (7.4,0.5)--(4,0.5)--(4,-2.9)--node[below]{$\bL_1$}(7.4,-2.9)--cycle;
                
                \draw[very thick](4,0.5)--(6.2,-1.7);
                
                \draw[] (11.8,0.5)--(8.4,0.5)--(8.4,-2.9)--node[below]{$\bL_2$}(11.8,-2.9)--cycle;
                
                \draw[very thick](8.4,0.5)--(9.4,-0.5);
                
                \draw[very thick](10.6,-1.7)--(11.8,-2.9);
                
            \end{tikzpicture}
               
            \caption{Illustration of subsampled datasets $\cD_{I_1}$ and $\cD_{I_2}$ from $\cD_n$.
            The design matrix of each of them can be represented as $\bL_j\bX$ ($j=1,2$), where $\bX\in\RR^{n\times p}$ is the full design matrix.}
            \label{fig:dataset}
        \end{figure}
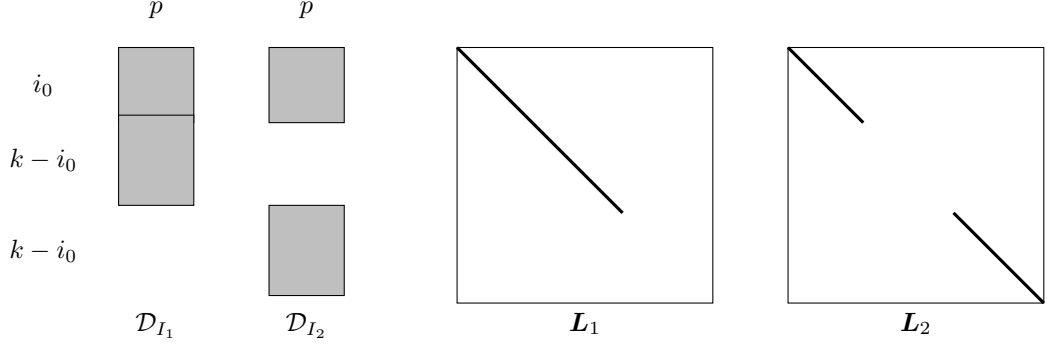
        
        The ingredient estimator takes the form:
        \begin{align*}
            \tbeta_{\lambda,M}(\{\cD_{I_\ell}\}_{\ell=1}^M) &=\frac{1}{M}\sum_{m=1}^M\betaridge(\cD_{I_m})\\
            &= \frac{1}{M}\sum_{m=1}^M (\bX^{\top}\bL_m\bX/k+\lambda\bI_p)^{-1}(\bX^{\top}\bL_m\by/k)\\
            &= \frac{1}{M}\sum_{m=1}^M \left[
            \left(\frac{\bX^{\top}\bL_m\bX}{k}+\lambda\bI_p\right)^{-1}\frac{\bX^{\top}\bL_m}{k} \bbeta_0  + \left(\frac{\bX^{\top}\bL_m\bX}{k}+\lambda\bI_p\right)^{-1}\frac{\bX^{\top}\bL_m}{k} \bepsilon\right].
        \end{align*}
        Denote $\tbetaridge{M}$ by $\tilde{\bbeta}_{\lambda,M}$ for simplicity. Let $\bM_m= (\bX^{\top}\bL_m\bX/k+\lambda\bI_p)^{-1}$ for $m\in[M]$, we have
        \begin{align*}
            \tilde{\bbeta}_{\lambda,M}&=\frac{1}{M}\sum_{m=1}^M (\bI_p-\lambda \bM_m) \bbeta_0  + \frac{1}{M}\sum_{m=1}^M \bM_m(\bX^{\top}\bL_m/k)\bepsilon,
        \end{align*}
        which yields
        \begin{align*}
            \bbeta_0-\tilde{\bbeta}_{\lambda,M}&=\frac{1}{M}\sum_{m=1}^M \lambda \bM_m \bbeta_0  - \frac{1}{M}\sum_{m=1}^M \bM_m(\bX^{\top}\bL_m/k)\bepsilon.
        \end{align*}
        Thus, the conditional risk is given by
        \begin{align*}
            R(\tilde{f}_{M,\lambda};\cD_n,\{I_\ell\}_{\ell=1}^M) &= \EE_{(\bx_0,y_0)}[(y_0-\bx_0^{\top}\tilde{\bbeta}_{\lambda,M})^2] \\
            &= \sigma^2 + (\bbeta_0-\tbeta_{\lambda,M})^{\top}\bSigma(\bbeta_0-\tbeta_{\lambda,M}) \\
            &= \sigma^2 + T_C + T_B + T_V, 
        \end{align*}
        where the constant term $T_C$, bias term $T_B$, and the variance term $T_V$ are given by
        \begin{align}
            T_C&= -\tfrac{2\lambda}{M^2}\cdot \bepsilon^{\top}\left(\sum_{m=1}^M\bM_m\tfrac{\bX^{\top}\bL_m}{k}\right)^{\top} \bSigma \left(\sum_{m=1}^M\bM_m\right) \bbeta_0, \label{eq:ridge-C0}\\
            T_B &= \tfrac{\lambda^2}{M^2}\cdot \bbeta_0^{\top}\left(\sum_{m=1}^M\bM_m\right)\bSigma\left(\sum_{m=1}^M\bM_m\right)\bbeta_0, \label{eq:ridge-B0}\\
            T_V&= \tfrac{1}{M^2}\cdot \bepsilon^{\top}\left(\sum_{m=1}^M\bM_m\tfrac{\bX^{\top}\bL_m}{k}\right)^{\top} \bSigma \left(\sum_{m=1}^M\bM_m\tfrac{\bX^{\top}\bL_m}{k}\right) \bepsilon. \label{eq:ridge-V0}
        \end{align}
        
        Next we analyze the three terms separately for $M\in\{1,2\}$.
        From Lemmas \ref{lem:ridge-conv-C0} and \ref{lem:ridge-conv-V0}, we have that $T_C\asto0$, and
        \begin{align*}
            T_V &=\frac{1}{M^2} \sum_{m=1}^M \bepsilon^{\top}\bM_m\tfrac{\bX^{\top}\bL_m}{k} \bSigma \bM_m\tfrac{\bX^{\top}\bL_m}{k} \bepsilon + \frac{1}{M^2}\sum_{m=1}^M\sum_{l=1}^M\bepsilon^{\top}\bM_m\tfrac{\bX^{\top}\bL_m}{k} \bSigma \bM_l\tfrac{\bX^{\top}\bL_l}{k} \bepsilon\\
            &\asto \frac{1}{M^2}\sum_{m=1}^M \tfrac{\sigma^2}{k}\tr( \bM_m\hSigma_m\bM_m \bSigma) + \frac{1}{M^2}\sum_{m\neq l} \tfrac{\sigma^2}{k^2}\tr( \bM_l\bX^{\top}\bL_l\bL_m\bX \bM_m \bSigma) := T_V'.
        \end{align*}
        Thus, it remains to obtain the deterministic equivalent for the bias term $T_B$ and the trace term $T_V'$.
        From Lemma \ref{lem:ridge-B0} and Lemma \ref{lem:ridge-V0}, we have that for all $I_1\in\cI_k$ when $M=1$ and for all $I_m,I_l\overset{\texttt{\textup{SRSWR}}}{\sim}\cI_k$ when $M=2$, it holds that
        \begin{align*}
            T_B &= \frac{\lambda^2}{M^2}\sum_{m=1}^M \bbeta_0^{\top}\bM_m \bSigma \bM_m\bbeta_0 + \frac{\lambda^2}{M^2}\sum_{m=1}^M\sum_{l=1}^M \bbeta_0^{\top}\bM_m \bSigma \bM_l\bbeta_0\\
            &\asto \frac{\rho^2}{M}(1+\tv(-\lambda;\phi_s,\phi_s))\tc(-\lambda;\phi_s) + \frac{\rho^2(M-1)}{M}(1+\tv(-\lambda;\phi,\phi_s))\tc(-\lambda;\phi_s)\\
            T_V'&\asto\tfrac{\sigma^2}{M}\tv(-\lambda;\phi_s,\phi_s)+ \tfrac{\sigma^2(M-1)}{M}\tv(-\lambda;\phi,\phi_s),
        \end{align*}
        as $n,k,p\rightarrow\infty$, $p/n\rightarrow\phi\in(0,\infty)$, and $p/k\rightarrow\phi_s\in[\phi,\infty)$, where the nonnegative constants $\tv(-\lambda;\phi,\phi_s)$ and $\tc(-\lambda;\phi_s)$ are as defined in \eqref{eq:tv_tc_ridge}.
        Therefore, we have shown that for all $I\in\cI_k$,
        $$R(\tilde{f}_{\lambda,1};\cD_n,\{I\}) \asto \RlamM{1}{\phi},$$
        and for all $I_1,I_2\overset{\texttt{\textup{SRSWR}}}{\sim} \cI_k$,
        $$R(\tilde{f}_{\lambda,2};\cD_n,\{I_\ell\}_{\ell=1}^2) \asto \RlamM{2}{\phi},$$
        where
        \begin{align*}
            \RlamM{M}{\phi} &= \sigma^2 +\tfrac{1}{M}(B_{\lambda}(\phi_s,\phi_s)+V_{\lambda}(\phi_s,\phi_s)) + \tfrac{M-1}{M}(B_{\lambda}(\phi,\phi_s)+V_{\lambda}(\phi,\phi_s)),
        \end{align*}
        and the components are:
        \begin{align*}
            B_{\lambda}(\phi,\phi_s) &= \rho^2 (1+\tv(-\lambda;\phi,\phi_s))\tc(-\lambda;\phi_s),&&
            V_{\lambda}(\phi,\phi_s) = \sigma^2\tv(-\lambda;\phi,\phi_s).
        \end{align*}
        The proof for the boundary case when $\phi_s=\infty$ follows from \Cref{prop:Rdet-ridge-infinity}.
        Then, we have that the function $\RlamM{M}{\phi}$ is continuous on $[\phi,\infty]$.
        
        Finally, the risk expression for general $M$ and the uniformity claim over $M\in\NN$ follow from \Cref{thm:risk_general_predictor}. 
    \end{proof}

    \subsection{Component concentrations}
        In this subsection, we will show that the cross-term $T_C$ converges to zero and the variance term $T_V$ converge to its corresponding trace expectation.
    
        \subsubsection{Convergence of the cross term}
        \begin{lemma}[Convergence of the cross term]\label{lem:ridge-conv-C0}
            Under Assumptions \ref{asm:rmt-feat}-\ref{asm:spectrum-spectrumsignproj-conv},
            for $T_C$ as defined in \eqref{eq:ridge-C0},
            we have $T_C\asto 0$ as $k,p\rightarrow\infty$ and $p/k\rightarrow\phi_s$.
        \end{lemma}
        \begin{proof}[Proof of Lemma \ref{lem:ridge-conv-C0}]
            Note that
        \begin{align*}
            T_C &= - \tfrac{2\lambda}{M^2} \cdot \tfrac{1}{k} \left\langle \left(\sum_{m=1}^M\bM_m\bX^{\top}\bL_m\right)^{\top} \bSigma \left(\sum_{m=1}^M\bM_m\right) \bbeta_0
            , \bepsilon\right\rangle.
        \end{align*}
        We next bound the squared norm
        \begin{align*}
            &\tfrac{1}{k}\norm{\frac{1}{M}\left(\sum_{m=1}^M\bM_m \bX^{\top}\bL_m\right)^{\top} \bSigma \left(\sum_{m=1}^M\bM_m\right) \bbeta_0}_2^2\\
            & \leq  \sum_{j=1}^M\sum\limits_{l=1}^M\tfrac{1}{M^2k}\norm{( \bM_j\bX^{\top}\bL_j)^{\top} \bSigma \bM_l \bbeta_0}_2^2\\
            & \leq 
            \frac{\norm{\bbeta_0}_2^2}{M^2}\cdot \sum_{j=1}^M\sum\limits_{l=1}^M\tfrac{1}{k}
            \norm{\bM_l \bSigma  \bM_j  \bX^{\top} \bL_j\bX\bM_j \bSigma \bM_l}_{\oper} \\
            & \leq \frac{\norm{\bbeta_0}_2^2}{M^2}\cdot \sum_{j=1}^M\sum\limits_{l=1}^M
            \norm{\bM_l}_{\oper}^2 \norm{\bSigma}_{\oper}^2\norm{ \bM_j (\bX^{\top} \bL_j\bX/k)\bM_j }_{\oper} \\
            & = \frac{\norm{\bbeta_0}_2^2}{M^2}\cdot \sum_{j=1}^M\sum\limits_{l=1}^M
            \norm{\bM_l}_{\oper}^2 \norm{\bSigma}_{\oper}^2\norm{\bM_j}_{\oper}\norm{ \bI_p - \lambda\bM_j }_{\oper} \\
            & \leq \tfrac{\norm{\bbeta_0}_2^2r_{\max}^2}{\lambda^3}, 
        \end{align*}
        where the last inequality is due to Assumption \ref{asm:covariance-bounded-eigvals} and the fact that $\|\bM_j\|_{\oper}\leq1/\lambda$.
        By Assumption \ref{asm:signal-bounded-norm}, the above quantity is uniformly bounded in $p$.
        Applying 
        \Cref{lem:concen-linform},
        we thus have that $T_C\asto 0$.
        \end{proof}

        \subsubsection{Convergence of the variance term}
        \begin{lemma}[Convergence of the variance term]\label{lem:ridge-conv-V0}
            Under Assumptions \ref{asm:rmt-feat}-\ref{asm:spectrum-spectrumsignproj-conv}, let $M\in\NN$. For all $m\in[M]$ and $I_m\in \cI_k$, let $\hSigma_m=\bX^{\top}\bL_m\bX/k$, $\bL_m\in\RR^{n\times n}$ be a diagonal matrix with $(\bL_{m})_{ll}=1$ if $l\in I_m$ and 0 otherwise, and $\bM_m= (\bX^{\top}\bL_m\bX/k+\lambda\bI_p)^{-1}$.
            Then, for all
            $m,l\in[M]$ and $m\neq l$, it holds that
            \begin{align*}
                \tfrac{1}{k^2} \bepsilon^{\top}\bL_m\bX\bM_m \bSigma \bM_m\bX^{\top}\bL_m\bepsilon-\tfrac{\sigma^2}{k}\tr(  \bM_m\hSigma_m\bM_m\bSigma) \asto 0,\\
                \tfrac{1}{k^2} \bepsilon^{\top}\bL_m\bX\bM_m \bSigma \bM_l\bX^{\top}\bL_l\bepsilon-\tfrac{\sigma^2}{k^2}\tr( \bM_l\bX^{\top}\bL_l\bL_m\bX \bM_m \bSigma)\asto 0,
            \end{align*}
            as $n,k,p\rightarrow\infty$, $p/n\rightarrow\phi\in(0,\infty)$, and $p/k\rightarrow\phi_s\in[\phi,\infty)$.
        \end{lemma}
        \begin{proof}[Proof of Lemma \ref{lem:ridge-conv-V0}]
            Note that the first term is the same as the variance term for ridge predictor trained on $k$ i.i.d.\ samples $(\bL_{m}\bX,\bL_{m}\by)$.
            Notice that $\bL_m\bepsilon$ is independent of $\bL_m\bX\bM_m \bSigma \bM_m\bX^{\top}\bL_m$, and
            \begin{align*}
                \tfrac{1}{k}\norm{\bL_m\bX\bM_m \bSigma \bM_m\bX^{\top}\bL_m}_{\oper} &\leq \norm{\hSigma_m}_{\oper}^{\tfrac{1}{2}}\norm{\bM_m}_{\oper}\norm{\bSigma }_{\oper}\norm{ \bM_m}_{\oper}\norm{\hSigma_m}_{\oper}^{\tfrac{1}{2}} \\
                & = \norm{\hSigma_m}_{\oper}\norm{\bM_m}_{\oper}^2\norm{\bSigma }_{\oper}\\
                &\leq \tfrac{r_{\max}}{\lambda^2}\norm{\hSigma_m}_{\oper}.
            \end{align*}
            Now, we have $\limsup\norm{\hSigma_m}_{\oper}\leq \limsup\max_{1\leq i\leq p} s_i^2\leq r_{\max}(1+\sqrt{\phi_s})^2$ almost surely as $k,p\rightarrow\infty$ and $p/k\rightarrow\phi_s\in(0,\infty)$ from \citet{bai2010spectral}.
            From 
            \Cref{lem:concen-quadform},
            it follows that
            \begin{align*} \bepsilon^{\top}\tfrac{\bL_m^{\top}\bX}{k}\bM_m \bSigma \bM_m\tfrac{\bX^{\top}\bL_m}{k}\bepsilon - \tfrac{\sigma^2}{k^2}\tr(\bL_m\bX\bM_m \bSigma \bM_m\bX^{\top}\bL_m) \asto 0.
            \end{align*}
            Since $\tr(\bL_m\bX\bM_m \bSigma \bM_m\bX^{\top}\bL_m)/k^2=\tr(  \bM_m\hSigma_m\bM_m\bSigma)/k=\tr(  \bM_m^2\hSigma_m\bSigma)/k$, we further have $\forall\ m\in[M]$,
            \begin{align} \bepsilon^{\top}\tfrac{\bL_m^{\top}\bX}{k}\bM_m \bSigma \bM_m\tfrac{\bX^{\top}\bL_m}{k}\bepsilon - \tfrac{\sigma^2}{k}\tr(  \bM_m\hSigma_m\bM_m\bSigma) \asto 0. \label{eq:conv-V0-2}
            \end{align}
            
            The second term involves the cross-term $\bM_m\bSigma\bM_l$.
            Note that
            \begin{align*}
                \tfrac{1}{n}\norm{\bL_m\bX\bM_m \bSigma \bM_l\bX^{\top}\bL_l}_{\oper } &\leq \tfrac{k}{n}\norm{\hSigma_m}_{\oper}^{\tfrac{1}{2}}\norm{\hSigma_l}_{\oper}^{\tfrac{1}{2}}\norm{\bM_m}_{\oper}\norm{\bM_l}_{\oper}\norm{\bSigma }_{\oper} \leq \tfrac{r_{\max}}{\lambda^2}\tfrac{k}{n}\norm{\hSigma_m}_{\oper}^{\tfrac{1}{2}}\norm{\hSigma_l}_{\oper}^{\tfrac{1}{2}}.
            \end{align*}
            Because $\norm{\hSigma_m}_{\oper}$ for $m \in [M]$ are uniformly bounded almost surely, again by 
            \Cref{lem:concen-quadform},
            it follows that
            \begin{align*}
                \tfrac{1}{n}\bepsilon^{\top}\bL_m\bX\bM_m \bSigma \bM_l\bX^{\top}\bL_l\bepsilon - \tfrac{\sigma^2}{n}\tr(\bL_m\bX\bM_m \bSigma \bM_l\bX^{\top}\bL_l)\asto0.
            \end{align*}
            Since $k/n\rightarrow\phi_s/\phi$, we have
            \begin{align}
                \tfrac{1}{k^2}\bepsilon^{\top}\bL_1\bX\bM_1 \bSigma \bM_2\bX^{\top}\bL_2\bepsilon - \tfrac{\sigma^2}{k^2}\tr( \bM_l\bX^{\top}\bL_l\bL_m\bX \bM_m \bSigma)\asto0. \label{eq:conv-V0-3}
            \end{align}
        \end{proof}

    \subsection{Component deterministic approximations}
    
        \subsubsection{Deterministic approximation of the bias functional}
        \begin{lemma}[Deterministic approximation of the bias functional]\label{lem:ridge-B0}
            Under Assumptions \ref{asm:rmt-feat}-\ref{asm:spectrum-spectrumsignproj-conv}, for all $m\in[M]$ and $I_m\in \cI_k$, let $\hSigma_m=\bX^{\top}\bL_m\bX/k$, $\bL_m\in\RR^{n\times n}$ be a diagonal matrix with $(\bL_{m})_{ll}=1$ if $l\in I_m$ and 0 otherwise, and $\bM_m= (\bX^{\top}\bL_m\bX/k+\lambda\bI_p)^{-1}$.
            Then, it holds that
            \begin{enumerate}[leftmargin=7mm]
                \item for all $m\in[M]$ and $I_m\in\cI_k$,
                \begin{align*}
                    \lambda^2\bbeta_0^{\top}\bM_m \bSigma \bM_m\bbeta_0 &\asto  \rho^2(1+\tv(-\lambda;\phi_s,\phi_s)) \tc(-\lambda;\phi_s),
                \end{align*}
                
                \item for all $m,l\in[M]$, $m\neq l$ and $I_m,I_l\overset{\texttt{\textup{SRSWR}}}{\sim}\cI_k$,
                \begin{align*}
                    \lambda^2\bbeta_0^{\top}\bM_m \bSigma \bM_l\bbeta_0 &\asto \rho^2(1+\tv(-\lambda;\phi,\phi_s)) \tc(-\lambda;\phi_s),
                \end{align*}
            \end{enumerate}
            as $n,k,p\rightarrow\infty$, $p/n\rightarrow\phi\in(0,\infty)$, and $p/k\rightarrow\phi_s\in[\phi,\infty)$, where $\phi_0=\phi_s^2/\phi$, 
            $T_B$ is as defined in \eqref{eq:ridge-B0}, and the nonnegative constants $\tv(-\lambda  ;\phi,\phi_s)$ and $\tc(-\lambda;\phi_s)$ are as defined in \eqref{eq:tv_tc_ridge}.
        \end{lemma}
        \begin{proof}[Proof of Lemma \ref{lem:ridge-B0}]
            From Lemma \ref{lem:deter-approx-generalized-ridge}~\ref{eq:detequi-ridge-genbias} we have that for $m\in[M]$,
            \begin{align}
                \lambda^2
                \bM_m \bSigma \bM_m
                \asympequi 
                (\tv_b(-\lambda; \phi_s)+1) \cdot (v(-\lambda; \phi_s) \bSigma + \bI_p)^{-1}
                 \bSigma 
                (v(-\lambda; \phi_s) \bSigma + \bI_p)^{-1}.\label{eq:lem-det-approx-B0-0}
            \end{align}
            By the definition of deterministic equivalent, we have
            \begin{align}
                \lambda^2\bbeta_0^{\top} \bM_m \bSigma \bM_m\bbeta_0 &\asto \lim\limits_{p\rightarrow\infty}(1+\tv_b(-\lambda; \phi_s))\sum\limits_{i=1}^p \tfrac{r_i}{(1 + r_i v(-\lambda; \phi_s))^2} (\bbeta_0^{\top}w_i)^2 \notag\\
                &= \lim\limits_{p\rightarrow\infty}\|\bbeta_0\|_2^2(1+\tv_b(-\lambda; \phi_s))\int \tfrac{r}{(1 + v(-\lambda; \phi_s)r)^2} \,\rd G_{p}(r)\notag\\
                &= \rho^2(1+\tv_b(-\lambda; \phi_s)) \int \tfrac{r}{(1 + v(-\lambda; \phi_s)r)^2} \,\rd G(r),\label{eq:lem-det-approx-B0-1}
            \end{align}
            where the last equality holds since $G_p$ and $G$ have compact supports and Assumptions \ref{asm:signal-bounded-norm} and \ref{asm:spectrum-spectrumsignproj-conv}.
            
        For the cross term, it suffices to derive the deterministic equivalent of $\bbeta_0^{\top}\bM_1\bSigma\bM_2\bbeta_0/2$. 
        We begin with analyze the deterministic equivalent of $\bM_1\bSigma\bM_2$.
        Let $i_0=\tr(\bL_1\bL_2)$ be the number of shared samples between $\cD_{I_1}$ and $\cD_{I_2}$, we use the decomposition
        \begin{align*}
            \bM_j^{-1}&=\tfrac{i_0}{k}(\hat{\bSigma}_0+\lambda\bI_p) +  \tfrac{k-i_0}{k}(\hat{\bSigma}_j^{\text{ind}} + \lambda\bI_p),\qquad j=1,2,
        \end{align*} where $\hat{\bSigma}_0=\bX^{\top}\bL_1\bL_2\bX/i_0$ and $\hat{\bSigma}_{j}^{\text{ind}}=\bX^{\top}(\bL_j-\bL_1\bL_2)\bX/(k-i_0)$ are the common and individual covariance estimators of the two datasets.
        Let $\bN_0=(\hat{\bSigma}_0+\lambda\bI_p)^{-1}$ and $\bN_j=(\hat{\bSigma}_j^{\text{ind}}+\lambda\bI_p)^{-1}$ for $j=1,2$. Then
        \begin{align}
            \bM_j&= \left(\tfrac{i_0}{k}\bN_0^{-1}+\tfrac{k-i_0}{k}\bN_j^{-1}\right)^{-1}, \qquad j=1,2, \label{eq:ridge-bias-term1}
        \end{align}
        where the equalities hold because $\bN_0$ is invertible when $\lambda>0$.
        Conditioning on $i_0$, we will show a sequence of deterministic equivalents
        \begin{align*}
            \lambda^2\bM_1\bSigma\bM_2 \overset{(a)}{\asympequi} \lambda\bM^{\det}_{\bN_0,i_0}\bSigma\bM_2\overset{(b)}{\asympequi}\bM^{\det}_{\bN_0,i_0}\bSigma\bM^{\det}_{\bN_0,i_0} \overset{(c)}{\asympequi} (\lambda^2\bM_1\bSigma\bM_2)^{\det}_{i_0},
        \end{align*}
        where in each step, we consider randomness from $\bN_1$, $\bN_2$, and $\bN_0$, respectively, since they are independent of each other conditioning on $i_0$. The subscript of the deterministic equivalent indicates the dependence on the corresponding random variables, and we will specify each deterministic equivalent in the following proof.
        
        When $i_0=k$, we have $\bM_1=\bM_2$ and the above asymptotic equation reduces to \eqref{eq:lem-det-approx-B0-0}.
        We next prove the case when $i_0<k$.

        \paragraph{Part (a).\hspace{-2mm}} Since $\bN_1$ is independent of $\bN_0$ conditioning on $i_0$, from \Cref{def:cond-deterministic-equivalent} and \Cref{lem:deter-approx-ridge-extend}~\ref{eq:lem:deter-approx-ridge-extend} we have
        \begin{align*} 
            \lambda\bM_1 \asympequi\bM^{\det}_{\bN_0,i_0}&:= \tfrac{k}{k-i_0}\left(v_1\bSigma +\bI_p+ \bC_1\right)^{-1} \,\Big|\, i_0,
        \end{align*}
        where $v_1=v(-\lambda; \gamma_1,\bSigma_{\bC_1})$, $\bSigma_{\bC_1}=(\bI_p+\bC_1)^{-\tfrac{1}{2}}\bSigma (\bI_p+\bC_1)^{-\tfrac{1}{2}}$, $\bC_1=i_0(\lambda(k-i_0))^{-1}\bN_0^{-1}$, and $\gamma_1=p/(k-i_0)$. Here the subscripts of $v_1$ and $\bC_1$ are related to the aspect ratio $\gamma_1$. 
        Because of the sub-multiplicativity of operator norm, we have
        \begin{align*}
            \norm{\bSigma\bM_2}_{\oper} \leq \norm{\bSigma}_{\oper}\norm{\bM_2}_{\oper} \leq \tfrac{r_{\max} }{\lambda}.
        \end{align*}
        By \Cref{prop:cond-calculus-detequi}~\ref{lem:cond-calculus-detequi-item-product}, we have $\lambda^2\bM_1\bSigma\bM_2 \asympequi \lambda\bM^{\det}_{\bN_0,i_0}\bSigma\bM_2\mid i_0$.
        
        \paragraph{Part (b).\hspace{-2mm}} Analogously, we have
        \begin{align*}
            &\lambda\bM^{\det}_{\bN_0,i_0}\bSigma\bM_2 \asympequi\bM^{\det}_{\bN_0,i_0}\bSigma\bM^{\det}_{\bN_0,i_0}\\
            \asympequi&\left(\tfrac{k}{k-i_0}\right)^2\left( v_1\bSigma +\bI_p+\bC_1\right)^{-1}\bSigma \left( v_1\bSigma +\bI_p+ \bC_1\right)^{-1}\mid i_0,
        \end{align*}
        as $\norm{\bM^{\det}_{\bN_0,i_0}}_{\oper}\leq 1$.
        
        \paragraph{Part (c).\hspace{-2mm}} As we have symmetrized the expression, we have
        \begin{align*}
            \lambda^2\bM_1\bSigma\bM_2 \asympequi \bM^{\det}_{\bN_0,i_0}\bSigma\bM^{\det}_{\bN_0,i_0}= \tfrac{k^2}{i_0^2}\lambda^2 (\bN_0^{-1}+\lambda\bC_0)^{-1}\bSigma(\bN_0^{-1}+\lambda\bC_0)^{-1} \mid i_0,
        \end{align*}
        where $\bC_0=(k-i_0)/i_0\cdot(v_1\bSigma+\bI_p)$.
        Define $\bSigma_{\bC_0}=(\bI+\bC_0)^{-\tfrac{1}{2}}\bSigma(\bI+\bC_0)^{-\tfrac{1}{2}}$.
        Conditioning on $i_0$, by \Cref{lem:deter-approx-ridge-extend}~\ref{eq:lem:deter-approx-ridge-extend}, 
        we have
        \begin{align*}
            \tr[\bSigma_{\bC_1}(v_1\bSigma_{\bC_1}+\bI_p)^{-1}] &= \tr[\bSigma(v_1\bSigma+\bI_p+\bC_1)^{-1}]\\
            &= \tfrac{\lambda(k-i_0)}{i_0}\tr\left[\bSigma\left(\bN_0^{-1}+ \tfrac{\lambda(k-i_0)}{i_0}( v_1\bSigma+\bI_p)\right)^{-1}\right]\\
            &\stackrel{a.s.}{=} \tfrac{k-i_0}{i_0}\tr\left[\bSigma\left(v_0\bSigma+\bI_p+ \tfrac{k-i_0}{i_0}( v_1\bSigma+\bI_p)\right)^{-1}\right]\\
            &= \tr\left[\bSigma\left(
            \left(\tfrac{i_0}{k-i_0}v_0 + v_1 \right)\bSigma+ \tfrac{k}{k-i_0}\bI_p\right)^{-1}\right],
        \end{align*}
        where $v_0=v(-\lambda;\gamma_0,\bSigma_{\bC_0})$and $\gamma_0=p/i_0$.
        Note that the fixed-point solution $v_0$ depends on $v_1$. The fixed-point equations reduce to
        \begin{align*}
            \tfrac{1}{v_0} &=\lambda + \gamma_0 \tr[\bSigma_{\bC_0}(v_0\bSigma_{\bC_0}+\bI_p)^{-1}]/p= \lambda + \tfrac{p}{k}\tr \left[\bSigma\left(\left(\tfrac{i_0}{k}v_0+\tfrac{k-i_0}{k}v_1\right)\bSigma + \bI_p\right)^{-1}\right]/p \\
            \tfrac{1}{v_1} &= \lambda + \gamma_1 \tr[\bSigma_{\bC_1}(v_1\bSigma_{\bC_1}+\bI_p)^{-1}]/p= \lambda + \tfrac{p}{k} \tr\left[\bSigma\left(
            \left(\tfrac{i_0}{k}v_0 +\tfrac{k-i_0}{k} v_1\right)\bSigma+ \bI_p\right)^{-1}\right]/p
        \end{align*}
        almost surely.
        Note that the solution $(v_0, v_1)$ to the above equations is a pair of positive numbers and does not depend on samples.
        If $(v_0, v_1)$ is a solution to the above system, then $(v_1, v_0)$ is also a solution. Thus, any solution to the above equations must be unique.
        On the other hand, since $v_0=v_1=v(-\lambda;p/k)$ satisfies the above equations, it is the unique solution.
        By \Cref{lem:substitue}, we can replace $v(-\lambda;\gamma_1,\bSigma_{\bC_1})$ by the solution $v_0=v_1=v(-\lambda; p/k)$ of the above system, which does not depend on samples. Thus,
        \begin{align}
            \lambda^2\bM_1\bSigma\bM_2\asympequi  \tfrac{k^2}{i_0^2}\lambda^2 (\bN_0^{-1}+\lambda\bC^*)^{-1}\bSigma(\bN_0^{-1}+\lambda\bC^*)^{-1} \mid i_0,\label{eq:ridge-B0-cross-term}
        \end{align}
        where $\bC^*= (k-i_0)/i_0\cdot(v(-\lambda; p/k)\bSigma+\bI_p)$.
        By Lemma \ref{lem:deter-approx-ridge-extend}~\ref{eq:lem:deter-approx-ridge-extend-gbias}, we have
        \begin{align}
             \bM^{\det}_{\bN_0,i_0}\bSigma\bM^{\det}_{\bN_0,i_0} \asympequi (\lambda^2\bM_1\bSigma\bM_2)^{\det}_{i_0}:=\tfrac{k^2}{i_0^2}(\tv_b(-\lambda; \gamma_0,\bC^*)+1) (v(-\lambda; \gamma_0,\bC^*) \bSigma + \bI_p+ \bC^*)^{-2}\bSigma\mid i_0,\label{eq:lem-ridge-B0-term-1}
        \end{align} where $\gamma_0=p/i_0$, and $v(-\lambda; \gamma_0,\bC^*)$ and $\tv_b(-\lambda; \gamma_0,\bC^*)$ are defined through the following equations:
        \begin{align*}
            \tfrac{1}{v(-\lambda; \gamma_0,\bC^*)} &= \lambda + \gamma_0 \tr[\bSigma(v(-\lambda; \gamma_0,\bC^*)\bSigma+\bI_p+\bC^*)^{-1}]/p\\
             \tfrac{1}{\tv_b(-\lambda; \gamma_0,\bC^*)} &=\tfrac{\gamma_0 \tr[\bSigma^2(v(-\lambda; \gamma_0,\bC^*)\bSigma+\bI_p+\bC^*)^{-2}]/p}{v(-\lambda; \gamma_0,\bC^*)^{-2}-\gamma_0 \tr[\bSigma^2(v(-\lambda; \gamma_0,\bC^*)\bSigma+\bI_p+\bC^*)^{-2}]/p}.
        \end{align*}
        From Parts (a) to (c), we have shown that $\lambda^2\bM_1\bSigma\bM_2\asympequi (\lambda^2\bM_1\bSigma\bM_2)^{\det}_{i_0}\mid i_0$ for $i_0<k$.
        Note that the above equivalence also holds for $i_0=k$. That is, this holds for all $i_0\in\{0,1,\cdots,k\}$.
        By Proposition \ref{prop:cond-calculus-detequi}~\ref{lem:cond-calculus-detequi-item-uncond}, we can obtain the unconditioned asymptotic equivalence $\bM^{\det}_{\bN_0,i_0}\bSigma\bM^{\det}_{\bN_0,i_0} \asympequi (\lambda^2\bM_1\bSigma\bM_2)^{\det}_{i_0}$.

        Note that from \Cref{lem:ridge-fixed-point-v-properties} $\tilde{v}_b(-\lambda;\gamma)$ and $v(-\lambda;\gamma)$ are continuous on $\gamma$, and from Lemma \ref{lem:i0_mean}, $i_0/k\asto \phi/\phi_s$, where $\phi_s\in(0,\infty)$ is the limiting ratio such that $p/k\rightarrow\phi_s$ as $k,p\rightarrow\infty$.
        We have
        \begin{align*}
            (\lambda^2\bM_1\bSigma\bM_2)^{\det}_{i_0}
            \asympequi& \tfrac{\phi_s^2}{\phi^2}(\tv_b(-\lambda; \phi_0,\bSigma_{\bC'})+1) (v(-\lambda; \phi_0,\bSigma_{\bC'}) \bSigma + \bI_p+ \bC')^{-2}\bSigma,
        \end{align*}
        where $\bC'=(\phi_s-\phi)/\phi\cdot( v(-\lambda; \phi_s)\bSigma +\bI_p)$ and $\phi_0=\phi_s^2/\phi$.
        Note that
        \begin{align*}
            \tfrac{1}{v(-\lambda; \phi_0,\bSigma_{\bC'})} &= \lambda+ \phi_0 \int \tfrac{r}{1 + r v(-\lambda; \phi_0,\bSigma_{\bC'})}\,\rd H(r;\bSigma_{\bC'})\\
            &= \lambda+ \phi_s \lim\limits_{p\rightarrow\infty}\tr\left[\bSigma\left(\tfrac{\phi}{\phi_s}(v(-\lambda; \phi_0,\bSigma_{\bC'})\bSigma+\bI_p) + \left(1-\tfrac{\phi}{\phi_s}\right)(v(-\lambda;\phi_s)\bSigma+\bI_p )\right)^{-1}\right]/p\\
            \tfrac{1}{v(-\lambda; \phi_s)} &= \lambda+ \phi_s \lim\limits_{p\rightarrow\infty}\tr\left[\bSigma(v(-\lambda;\phi_s)\bSigma+\bI_p )^{-1}\right]/p.
        \end{align*}
        We have
        \begin{align}
            v(-\lambda; \phi_0,\bSigma_{\bC'})=v(-\lambda;\phi_s) \label{eq:vc}
        \end{align}is a solution to the first fixed-point equation.
        From \Cref{lem:properties-sol}~\ref{lem:properties-sol-item-f-ridge}, this solution is also unique.
        Then, we have
        \begin{align*}
            1+\tv_b(-\lambda; \phi_0,\bSigma_{\bC'}) &= \lim\limits_{p\rightarrow\infty} \tfrac{v(-\lambda; \phi_0,\bC')^{-2}}{v(-\lambda; \phi_0,\bC')^{-2}-\phi_0 \tr[\bSigma^2(v(-\lambda; \phi_0,\bC')\bSigma+\bI_p+\bC')^{-2}]/p}\\
            &= \lim\limits_{p\rightarrow\infty} \tfrac{v(-\lambda; \phi_s)^{-2}}{v(-\lambda; \phi_s)^{-2}-\phi \tr[\bSigma^2(v(-\lambda; \phi_s)\bSigma+\bI_p)^{-2}]/p}\\
            &=\ddfrac{v(-\lambda; \phi_s)^{-2}}{v(-\lambda; \phi_s)^{-2}-\phi \int\tfrac{r^2}{(1+ v(-\lambda; \phi_s)r)^2}\,\rd H(r)}:= 1+\tv(-\lambda;\phi,\phi_s). 
        \end{align*}
        From \Cref{lem:properties-sol}~\ref{lem:properties-sol-item-f'-ridge}, we have that $1+\tv(-\lambda;\phi,\phi_s)>0$.
        To conclude, we have shown that
        \begin{align}
            \lambda^2\bM_1\bSigma\bM_2 \asympequi (1+\tv(-\lambda;\phi,\phi_s)) \left(v(-\lambda; \phi_s) \bSigma + \bI_p\right)^{-2}\bSigma. \label{eq:lem-ridge-B0-term-det}
        \end{align}
        By the definition of deterministic equivalent, we have
        \begin{align}
            \lambda^2\bbeta_0^{\top} \bM_1 \bSigma \bM_2\bbeta_0
            &\asto \lim\limits_{p\rightarrow\infty}\sum\limits_{i=1}^p \tfrac{(1+\tv(-\lambda;\phi,\phi_s)) r_i}{\left(1+v(-\lambda; \phi_s)r_i\right)^2} (\bbeta_0^{\top}w_i)^2 \notag\\
            &= \lim\limits_{p\rightarrow\infty}\|\bbeta_0\|_2^2\int \tfrac{(1+\tv(-\lambda;\phi,\phi_s)) r}{\left(1+v(-\lambda; \phi_s)r\right)^2} \,\rd G_{p}(r)\notag\\
            &= \rho^2 \int \tfrac{(1+\tv(-\lambda;\phi,\phi_s)) r}{\left(1+v(-\lambda; \phi_s)r\right)^2} \,\rd G(r), \label{eq:lem-det-approx-B0-2}
        \end{align}
        where in the last line we used the fact that $G_p$ and $G$ have compact supports and Assumptions \ref{asm:signal-bounded-norm} and \ref{asm:spectrum-spectrumsignproj-conv}. The conclusion follows by combining \eqref{eq:lem-det-approx-B0-1} and \eqref{eq:lem-det-approx-B0-2}.
        \end{proof}

        \subsubsection{Deterministic approximation of the variance functional}
        \begin{lemma}[Deterministic approximation of the variance functional]\label{lem:ridge-V0}
            Under Assumptions \ref{asm:rmt-feat}-\ref{asm:spectrum-spectrumsignproj-conv}, for all $m\in[M]$ and $I_m\in \cI_k$, let $\hSigma_m=\bX^{\top}\bL_m\bX/k$, $\bL_m\in\RR^{n\times n}$ be a diagonal matrix with $(\bL_{m})_{ll}=1$ if $l\in I_m$ and 0 otherwise, and $\bM_m= (\bX^{\top}\bL_m\bX/k+\lambda\bI_p)^{-1}$.
            Then, it holds that
            \begin{enumerate}[leftmargin=7mm]
                \item for all $m\in[M]$ and $I_m\in\cI_k$,
                \begin{align*}
                    \tfrac{1}{k}\tr(  \bM_m\hSigma_m\bM_m\bSigma) \asto  \tv(-\lambda;\phi_s,\phi_s),
                \end{align*}
                
                \item for all $m,l\in[M]$, $m\neq l$ and $I_m,I_l\overset{\texttt{\textup{SRSWR}}}{\sim}\cI_k$,
                \begin{align*}
                    \tfrac{1}{k^2}\tr( \bM_l\bX^{\top}\bL_l\bL_m\bX \bM_m \bSigma) \asto \tv(-\lambda;\phi,\phi_s) , %
                \end{align*}
            \end{enumerate}
            as $n,k,p\rightarrow\infty$, $p/n\rightarrow\phi\in(0,\infty)$, and $p/k\rightarrow\phi_s\in[\phi,\infty)$, where the nonnegative constant $\tv(\lambda  ;\phi,\phi_s)$ is as defined in \eqref{eq:tv_tc_ridge}.
        \end{lemma}
        \begin{proof}[Proof of Lemma \ref{lem:ridge-V0}]
            From Lemma \ref{lem:deter-approx-generalized-ridge}~\ref{eq:detequi-ridge-genvar}, we have that for $j\in[M]$,
            \begin{align}
                \bM_j\hSigma_j\bM_j\bSigma \asympequi \tv_v(-\lambda; \phi_s) (v(-\lambda; \phi_s) \bSigma + \bI_p)^{-2} \bSigma^2. \label{eq:lem-V0-term-1}
            \end{align}
            By the trace rule \Cref{lem:calculus-detequi}~\ref{lem:calculus-detequi-item-trace} , we have
            \begin{align}
                \tfrac{1}{k}\tr(  \bM_j\hSigma_j\bM_j\bSigma)&\asto \lim\limits_{p\rightarrow\infty}\tfrac{p}{k}\cdot\tfrac{1}{p}\tr(\tv_v(-\lambda; \phi_s) (v(-\lambda; \phi_s) \bSigma + \bI_p)^{-2} \bSigma^2)\notag\\
                &= \phi_s\tv_v(-\lambda; \phi_s) \lim\limits_{p\rightarrow\infty}\tfrac{1}{p}\sum\limits_{i=1}^p\tfrac{r_i^2}{(v(-\lambda; \phi_s) r_i + 1)^2}\notag\\
                &=\phi_s\tv_v(-\lambda; \phi_s)\lim\limits_{p\rightarrow\infty}\int \tfrac{r^2}{(v(-\lambda; \phi_s) r + 1)^2}\,\rd H_p(r)\notag\\
                &=\phi_s\tv_v(-\lambda; \phi_s)\int \tfrac{r^2}{(v(-\lambda; \phi_s) r + 1)^2}\,\rd H(r),\qquad j=1,2, \label{eq:lem:ridge-V0-1}
            \end{align}
            where in the last line we used the fact that $H_p$ and $H$ have compact supports and Assumption \ref{asm:spectrum-spectrumsignproj-conv}.
            
            For the cross term, it suffices to derive the deterministic equivalent of $\bM_1 \hSigma_0 \bM_2 \bSigma$ where $\hSigma_0=\bX^{\top}\bL_1\bL_2\bX/i_0$ and $i_0=\tr(\bL_1\bL_2)$. We again show a sequence of deterministic equivalents as in the proof for \Cref{lem:ridge-B0}:
            \begin{align*}
                \bM_1 \hSigma_0 \bM_2 \bSigma \overset{(a)}{\asympequi} \bM^{\det}_{\bN_0,i_0} \hSigma_0 \bM_2 \bSigma \overset{(b)}{\asympequi}\bM^{\det}_{\bN_0,i_0} \hSigma_0 \bM^{\det}_{\bN_0,i_0} \bSigma \overset{(c)}{\asympequi} (\bM_1 \hSigma_0 \bM_2 \bSigma)^{\det}_{i_0}\mid i_0.
            \end{align*}
            When $i_0=k$, this reduces to \eqref{eq:lem-V0-term-1}.
            We next show the case when $i_0<k$.
            
            \paragraph{Part (a).\hspace{-2mm}} We use \Cref{lem:deter-approx-ridge-extend}~\ref{eq:lem:deter-approx-ridge-extend} to obtain
            \begin{align}
                \bM_1 & \asympequi \bM^{\det}_{\bN_0,i_0}:= \tfrac{ k}{\lambda(k-i_0)}\left(v(-\lambda;\gamma_1,\bSigma_{\bC_1})\bSigma +  \bI_p+ \bC_1\right)^{-1} \mid i_0
            \end{align}
            where $\bSigma_{\bC_1}=(\bI_p+\bC_1)^{-\tfrac{1}{2}}\bSigma (\bI_p+\bC_1)^{-\tfrac{1}{2}}$, $\bC_1=i_0(\lambda(k-i_0))^{-1}\bN_0^{-1}$, and $\gamma_1=p/(k-i_0)$.
            Let $\gamma_0=p/i_0$. Note that conditioning on $i_0$, $\limsup\norm{\hSigma_0}_{\oper}\leq r_{\max}(1+\sqrt{\phi_0})^2$ almost surely as $i_0,p\rightarrow\infty$ and $\gamma_0\rightarrow\phi_0\in(0,\infty)$ from \citet{bai2010spectral}.
            Then $\hSigma_0\bM_2\bSigma$ has bounded operator norm and we have $\bM_1 \hSigma_0 \bM_2 \bSigma\asympequi \bM^{\det}_{\bN_0,i_0} \hSigma_0 \bM_2\bSigma\mid i_0$ by \Cref{prop:cond-calculus-detequi}~\ref{lem:cond-calculus-detequi-item-product}.

            \paragraph{Part (b).\hspace{-2mm}} Similarly, we have $\bM_2\asympequi \bM^{\det}_{\bN_0,i_0}\mid i_0$ and $\bM_1 \hSigma_0 \bM_2 \bSigma\asympequi \bM^{\det}_{\bN_0,i_0} \hSigma_0 \bM^{\det}_{\bN_0,i_0}\bSigma\mid i_0$.
            
            \paragraph{Part (c).\hspace{-2mm}}
            Note that
            \begin{align*}
                \bM^{\det}_{\bN_0,i_0} \hSigma_0 \bM^{\det}_{\bN_0,i_0} \bSigma &= \tfrac{k^2}{\lambda^2(k-i_0)^2}\left(  v(-\lambda;\gamma_1,\bSigma_{\bC_1})\bSigma +  \bI_p+ \bC_1\right)^{-1} \hSigma_0 \left( v(-\lambda;\gamma_1,\bSigma_{\bC_1})\bSigma +  \bI_p+ \bC_1\right)^{-1}\bSigma\\
                &= \tfrac{k^2}{i_0^2}\left( \bN_0^{-1}+\lambda\bC_0\right)^{-1} \hSigma_0 \left( \bN_0^{-1} + \lambda\bC_0\right)^{-1}\bSigma,
            \end{align*}
            where $\bC_0= (k-i_0)/i_0\cdot(v(-\lambda; \gamma_1,\bSigma_{\bC_1})\bSigma +\bI_p)$.
            Define $\bSigma_{\bC_0}=(\bI+\bC_0)^{-\tfrac{1}{2}}\bSigma(\bI+\bC_0)^{-\tfrac{1}{2}}$.
            Conditioning on $i_0$, by \Cref{lem:deter-approx-ridge-extend}~\ref{eq:lem:deter-approx-ridge-extend} we have
            \begin{align*}
                \tr[\bSigma_{\bC_1}(v_1\bSigma_{\bC_1}+\bI_p)^{-1}] &= \tr[\bSigma(v_1\bSigma+\bI_p+\bC_1)^{-1}]\\
                &= \tfrac{\lambda(k-i_0)}{i_0}\tr\left[\bSigma\left(\bN_0^{-1}+ \tfrac{\lambda(k-i_0)}{i_0}( v_1\bSigma+\bI_p)\right)^{-1}\right]\\
                &\stackrel{a.s.}{=} \tfrac{k-i_0}{i_0}\tr\left[\bSigma\left(v_0\bSigma+\bI_p+ \tfrac{k-i_0}{i_0}( v_1\bSigma+\bI_p)\right)^{-1}\right]\\
                &= \tr\left[\bSigma\left(
                \left(\tfrac{i_0}{k-i_0}v_0 + v_1 \right)\bSigma+ \tfrac{k}{k-i_0}\bI_p\right)^{-1}\right]
            \end{align*}
            where $v_0=v(-\lambda;\gamma_0,\bSigma_{\bC_0})$ and $\gamma_0=p/i_0$.
            Note that the fixed-point solution $v_0$ depends on $v_1$. The fixed-point equations reduce to
            \begin{align*}
                \tfrac{1}{v_0} &=\lambda + \gamma_0 \tr[\bSigma_{\bC_0}(v_0\bSigma_{\bC_0}+\bI_p)^{-1}]/p= \lambda + \tfrac{p}{k}\tr \left[\bSigma\left(\left(\tfrac{i_0}{k}v_0+\tfrac{k-i_0}{k}v_1\right)\bSigma + \bI_p\right)^{-1}\right]/p\\
                \tfrac{1}{v_1} &= \lambda + \gamma_1 \tr[\bSigma_{\bC_1}(v_1\bSigma_{\bC_1}+\bI_p)^{-1}]/p= \lambda + \tfrac{p}{k} \tr\left[\bSigma\left(
                \left(\tfrac{i_0}{k}v_0 +\tfrac{k-i_0}{k} v_1\right)\bSigma+ \bI_p\right)^{-1}\right]/p
            \end{align*}
            almost surely.
            By the same argument as in the proof for \Cref{lem:ridge-B0}, we have that the solution $v_0=v_1=v(-\lambda; p/k)$ of the above system does not depend on samples and equals to $v(-\lambda;\gamma_1,\bSigma_{\bC_1})$ or $v(-\lambda;\gamma_0,\bSigma_{\bC_0})$ almost surely. Thus, by \Cref{lem:substitue},
            \begin{align*}
                 \bM^{\det}_{\bN_0,i_0} \hSigma_0 \bM^{\det}_{\bN_0,i_0} \bSigma\asympequi  \tfrac{k^2}{i_0^2} (\bN_0^{-1}+\lambda\bC^*)^{-1}\hSigma_0(\bN_0^{-1}+\lambda\bC^*)^{-1} \mid i_0,
            \end{align*}
            where $\bC^*= (k-i_0)/i_0\cdot(v(-\lambda; p/k)\bSigma+\bI_p)$.
            From \Cref{lem:deter-approx-ridge-extend}~\ref{eq:lem:deter-approx-ridge-extend-gvar}, we have
            \begin{align*}
                \bM^{\det}_{\bN_0,i_0} \hSigma_0 \bM^{\det}_{\bN_0,i_0} \bSigma \asympequi (\bM_1 \hSigma_0 \bM_2 \bSigma)^{\det}_{i_0}&:= \tfrac{k^2}{i_0^2}\tv_v(-\lambda; \gamma_0,\bSigma_{\bC^*})(  v(-\lambda; \gamma_0,\bSigma_{\bC^*}) \bSigma + \bI_p+\bC^* )^{-2}\bSigma^2\mid i_0,
            \end{align*}
            where $\gamma_0=p/i_0$.
            
            From Parts (a) to (c), we have shown that $\bM_1 \hSigma_0 \bM_2 \bSigma\asympequi(\bM_1 \hSigma_0 \bM_2 \bSigma)^{\det}_{i_0} \mid i_0$ for $i_0<k$.
            Note that this also holds for $i_0=k$.
            Then by \Cref{prop:cond-calculus-detequi}, $\bM_1 \hSigma_0 \bM_2 \bSigma\asympequi(\bM_1 \hSigma_0 \bM_2 \bSigma)^{\det}_{i_0}$.
            
            Note that from \Cref{lem:ridge-fixed-point-v-properties}, $\tilde{v}_b(-\lambda;\gamma)$ and $v(-\lambda;\gamma)$ are continuous on $\gamma$, and from Lemma \ref{lem:i0_mean}, $i_0/k\asto \phi/\phi_s$ where $\phi_s\in(0,\infty)$ is the limiting ratio such that $p/k\rightarrow\phi_s$ as $k,p\rightarrow\infty$.
            We have
            \begin{align*}
                \bM_1 \hSigma_0 \bM_2 \bSigma&\asympequi \tfrac{\phi_s^2}{\phi^2}\tv_v(-\lambda; \phi_0,\bSigma_{\bC'})( v(-\lambda; \phi_0,\bSigma_{\bC'}) \bSigma + \bI_p+\bC')^{-2}\bSigma^2,
            \end{align*}
            where $\phi_0=\phi_s^2/\phi$, $\bSigma_{\bC'}=(\bI_p+\bC')^{-\tfrac{1}{2}}\bSigma(\bI_p+\bC')^{-\tfrac{1}{2}}$, and $\bC'=(\phi_s-\phi)/\phi\cdot(v(-\lambda; \phi_s)\bSigma +\bI_p)$.
            From \eqref{eq:vc}, we have that $v(-\lambda; \phi_0;\bSigma_{\bC'})=v(-\lambda; \phi_s)$, and
            \begin{align*}
                \phi\tv_v(-\lambda;\phi_0,\bSigma_{\bC'}) &=  \lim\limits_{p\rightarrow\infty} \tfrac{\phi}{v(-\lambda; \phi_0,\bC')^{-2}-\phi_0 \tr[\bSigma^2(v(-\lambda; \phi_0,\bC')\bSigma+\bI_p+\bC')^{-2}]/p}\\
            &= \lim\limits_{p\rightarrow\infty} \tfrac{\phi}{v(-\lambda; \phi_s)^{-2}-\phi \tr[\bSigma^2(v(-\lambda; \phi_s)\bSigma+\bI_p)^{-2}]/p}\\
            &=\ddfrac{\phi}{v(-\lambda; \phi_s)^{-2}-\phi \int\tfrac{r^2}{(1+ v(-\lambda; \phi_s)r)^2}\,\rd H(r)}:= v_v(-\lambda;\phi,\phi_s). 
            \end{align*}
            From \Cref{lem:properties-sol}~\ref{lem:properties-sol-item-f'-ridge}, we have that $v_v(-\lambda;\phi,\phi_s)>0$.
            Then we have
            \begin{align}
                \bM_1 \hSigma_0 \bM_2 \bSigma&\asympequi \phi^{-1}v_v(-\lambda;\phi,\phi_s)( v(-\lambda; \phi_s) \bSigma + \bI_p)^{-2}\bSigma^2,\label{eq:lem-ridge-V0-term-1}
            \end{align}
        and thus, we have
           \begin{align}
                \tfrac{i_0}{k^2}\tr( \bM_1 \hSigma_0 \bM_2 \bSigma))& \asto \lim\limits_{p\rightarrow\infty}\tfrac{i_0p}{k^2}\tfrac{1}{\phi}\cdot \tfrac{1}{p}\tr(v_v(-\lambda; \phi,\phi_s)(v(-\lambda; \phi_s) \bSigma + \bI_p)^{-2}\bSigma^2)\notag\\
                &= \lim\limits_{p\rightarrow\infty}\tfrac{1}{p}\sum_{i=1}^p \tfrac{v_v(-\lambda; \phi,\phi_s) r_i^2 }{(1+v(-\lambda; \phi_s) r_i)^2}\notag\\
                &= \lim\limits_{p\rightarrow\infty}\int \tfrac{v_v(-\lambda; \phi,\phi_s) r^2 }{(1+v(-\lambda; \phi_s)r)^2}\,\rd H_p(r)\notag\\
                &= \int \tfrac{ v_v(-\lambda; \phi,\phi_s) r^2 }{(1+v(-\lambda; \phi_s) r)^2}\,\rd H(r):= \tv(-\lambda; \phi,\phi_s) , \label{eq:lem:ridge-V0-2}
            \end{align}
            where in the last line we used the fact that $H_p$ and $H$ have compact supports and Assumption \ref{asm:spectrum-spectrumsignproj-conv}.
        \end{proof}

    \subsection{Boundary case: diverging subsample aspect ratio}
    \begin{proposition}[Risk approximation when $\phi_s\rightarrow+\infty$]\label{prop:Rdet-ridge-infinity}
    Under Assumptions \ref{asm:rmt-feat}-\ref{asm:spectrum-spectrumsignproj-conv}, it holds for all $M\in\NN$ $$R(\tfWR{\lambda}{M};\cD_n,\{I_\ell\}_{\ell=1}^M) \asto \RlamMe{M}{\phi}{\infty},$$
    as $k,n,p\rightarrow\infty$, $p/n\rightarrow\phi\in(0,\infty)$ and $p/k\rightarrow\infty$, where
    \begin{align}
        \RlamMe{M}{\phi}{\infty}:=\lim\limits_{\phi_s\rightarrow+\infty}\RlamM{M}{\phi} = \sigma^2 + \rho^2\int r\,\rd G(r) \label{eq:Rdet-ridge-infinity}
    \end{align}
    and $\RlamM{M}{\phi}$ is defined in \Cref{thm:ver-ridge}.
    \end{proposition}
    \begin{proof}[Proof of \Cref{prop:Rdet-ridge-infinity}]
        Note that 
        \begin{align*}
            R(\tfWR{\lambda}{M};\cD_n,\{I_\ell\}_{\ell=1}^M) &= \EE_{(\bx_0,y_0)}[(y_0-\bx_0^{\top}\tbetaridge{M})^2] \\
            &= \EE_{(\bx_0,y_0)}[(\bepsilon_0
            +\bx_0^{\top}(\bbeta_0-\tbetaridge{M}))^2] \\
            &= \sigma^2 + \EE_{(\bx_0,y_0)}[(\bbeta_0-\tbetaridge{M})^{\top}\bx_0\bx_0^{\top}(\bbeta_0-\tbetaridge{M}) ] \\
            &= \sigma^2 + (\bbeta_0-\tbetaridge{M})^{\top}\bSigma(\bbeta_0-\tbetaridge{M}).
        \end{align*}
        Then, by the Cauchy-Schwarz inequality, we have
        \begin{align*}
            R(\tfWR{\lambda}{M};\cD_n,\{I_\ell\}_{\ell=1}^M) - (\bbeta_0^{\top}\bSigma\bbeta_0 + \sigma^2) &= \|\bSigma^{\tfrac{1}{2}}\tbetaridge{M}\|_2^2 - 2 \tbetaridge{M}^{\top}\bSigma\bbeta_0\\
            &\leq \tfrac{1}{r_{\min}} \|\tbetaridge{M}\|_2^2 + 2 \|\tbetaridge{M}\|_2 \|\bSigma\|_2 \\
            &\leq \tfrac{1}{r_{\min}} \|\tbetaridge{M}\|_2^2 + 2 r_{\max}\rho \|\tbetaridge{M}\|_2,
        \end{align*}
        almost surely as $k,n,p\rightarrow$ and $p/k\rightarrow\infty$.
        Thus, we have the following holds almost surely:
        \begin{align*}
            \|\tbetaridge{M}(\{\cD_{I_\ell}\}_{\ell=1}^M)\|_2 &\leq \tfrac{1}{M}\sum_{m=1}^M \|(\bX^{\top}\bL_m\bX/k+\lambda\bI_p)^{-1}(\bX^{\top}\bL_m\by/k)\|_2 \\
            &\leq \tfrac{1}{M}\sum_{m=1}^M \|(\bX^{\top}\bL_m\bX/k+\lambda\bI_p)^{-1}\bX^{\top}\bL_m/\sqrt{k}\|\cdot\|\bL_m\by/\sqrt{k}\|_2\\
            &\leq C\sqrt{\rho^2+\sigma^2}\cdot \tfrac{1}{M}\sum_{m=1}^M\|(\bX^{\top}\bL_m\bX/k+\lambda\bI_p)^{-1}\bX^{\top}\bL_m/\sqrt{k}\|,
        \end{align*}
        where the last inequality holds eventually almost surely since Assumptions \ref{asm:rmt-feat}-\ref{asm:signal-bounded-norm} imply that the entries of $\by$ have bounded $4$-th moment, and thus from the strong law of large numbers, $\| \bL_m \by / \sqrt{k} \|_2$ is eventually almost surely
        bounded above by $C\sqrt{\EE[y_1^2]} = C\sqrt{\rho^2 + \sigma^2}$ for some constant $C$.
        Observe that operator norm of the matrix $(\bX^\top \bL_m\bX / k+ \lambda\bI_p)^{-1} \bX\bL_m / \sqrt{k}$ is upper bounded $\max_i s_i/(s_i^2+\lambda)\leq 1/s_{\min}$ where $s_i$'s are the singular values of $\bX$ and $s_{\min}$ is the smallest nonzero singular value.
        As $k, p \to \infty$ such that $p / k \to \infty$, $s_{\min} \to \infty$ almost surely (e.g., from results in \cite{bloemendal2016principal}) and therefore, $\| \tbetaridge{M} \|_2 \to 0$ almost surely.
        Thus, we have shown that
        $$R(\tfWR{\lambda}{M};\cD_n,\{I_\ell\}_{\ell=1}^M) \asto \sigma^2+\bbeta_0^{\top}\bSigma\bbeta_0 .$$
        or equivalently 
        $$R(\tfWR{\lambda}{M};\cD_n,\{I_\ell\}_{\ell=1}^M) \asto \sigma^2+ \rho^2\int r\,\rd G(r).$$
        
        From \Cref{lem:ridge-fixed-point-v-properties}, we have $$\lim_{\phi_s\rightarrow+\infty}v(-\lambda;\phi_s) = \lim_{\phi_s\rightarrow+\infty}\tv_b(-\lambda;\phi_s)=\lim_{\phi_s\rightarrow+\infty}\tv_v(-\lambda;\phi_s).$$ Thus, $$\lim_{\phi_s\rightarrow+\infty}V_{\lambda}(\phi,\phi_s)=\lim_{\phi_s\rightarrow+\infty}V_{\lambda}(\phi,\phi_s) = 0$$
        and $$\lim_{\phi_s\rightarrow+\infty}B_{\lambda}(\phi,\phi_s)=\lim_{\phi_s\rightarrow+\infty}B_{\lambda}(\phi,\phi_s) = \rho^2\int r\,\rd G(r).$$
        Therefore, we have $\RlamMe{M}{\phi}{\infty}:=\lim_{\phi_s\rightarrow+\infty}\RlamM{M}{\phi} = \sigma^2 + \rho^2\int r\,\rd G(r)$.
        Thus, $\RlamMe{M}{\phi}{\infty}$ is well defined and $\RlamM{M}{\phi}$ is right continuous at $\phi_s=+\infty$.
    \end{proof}
    
\section
[Proof of \Cref{thm:ver-with-replacement} (subagging with replacement, ridgeless predictor)]
{Proof of \Cref{thm:ver-with-replacement} (subagging with replacement, ridgeless predictor)}
\label{sec:appendix-with-replacement-ridgeless}

   As done in \Cref{sec:appendix-with-replacement-ridge},
   for proving the asymptotic conditional risks, we will treat $\cI_k$ or $\cI_k^{\sigma}$ as fixed.
   We will use $\tfWR{0}{M}$ to denote the ingredient predictor associated with regularization parameter $\lambda = 0$. 
    
    \subsection{Proof assembly}
    
    We first explicitly write out the statement of \Cref{thm:ver-with-replacement}
    for the ridgeless case of $\lambda = 0$.
    As in \Cref{sec:appendix-with-replacement-ridge}, we obtain the risk decomposition for general $M$ though it suffices to analyze the case $M=2$ according to \Cref{thm:risk_general_predictor}.
    
    For ridgeless predictors ($\lambda=0$) and $\theta>1$, 
    the scalar $v(0;\theta)$ is the unique fixed-point solution 
    to the following equation:
    \begin{align}
        v(0;\theta)^{-1} = \theta\int r(1 + v(0;\theta)r)^{-1}\,\rd H(r). \label{eq:v_ridgeless}
    \end{align}
    and the nonnegative constants $\tv(0;\vartheta,\theta)$ and $\tc(0;\theta)$ are defined via the following equations:
    \begin{align}
        \tv(0;\vartheta,\theta) = \frac{\vartheta \int r^2(1+ v(0; \theta)r)^{-2}\,\rd H(r)}{v(0; \theta)^{-2}-\vartheta \int r^2 (1+ v(0; \theta)r)^{-2}\,\rd H(r)},\qquad 
        \tc(0;\theta)& =\int  r (1 + v(0; \theta)r)^{-2} \,\rd G(r). \label{eq:tv_tc_ridgeless}
    \end{align}
    When $\theta\leq 1$, the quantities defined in \eqref{eq:v_ridgeless} and \eqref{eq:tv_tc_ridgeless} are interpreted as $\lim_{\lambda\rightarrow0^+}v(-\lambda; \theta) = \infty$, $\lim_{\lambda\rightarrow0^+}\tc(-\lambda; \theta) = 0$ and $\lim_{\lambda\rightarrow0^+}\tv(-\lambda; \vartheta, \theta) = \vartheta (1-\vartheta)^{-1}$.
    
    \begin{theorem}[Risk characterization of subagged ridgeless predictor]\label{thm:ver-ridgeless}
        Let $\tfWR{0}{M}$ be the ingredient predictor as defined in \eqref{eq:ingredient-predictor}.
        Suppose 
        Assumptions \ref{asm:rmt-feat}-\ref{asm:spectrum-spectrumsignproj-conv} hold for the dataset $\cD_n$.
        Then,
        as
        $k,n,p\rightarrow\infty$
        such that
        $p/n\rightarrow\phi\in(0,\infty)$ and $p/k\rightarrow\phi_s\in[\phi,\infty]$
        and $\phi_s \neq 1$,
        there exists a deterministic function $\RzeroM{M}{\phi}$, $M \in \NN$,
        such that for $I_1, \ldots, I_M \overset{\texttt{\textup{SRSWR}}}{\sim} \mathcal{I}_k$, 
        \begin{align*}
            \sup_{M\in\NN}| R(\tfWR{0}{M}; \{\cD_{I_{\ell}}\}_{\ell=1}^M) - \RzeroM{M}{\phi} | \pto 0,
        \end{align*}
        and
        $$ \sup_{M\in\NN}
        | R(\tf_{0, M}; \cD_n) - \RzeroM{M}{\phi} | \asto 0.$$
        Furthermore, the function $\RzeroM{M}{\phi}$
        decomposes as
        $\RzeroM{M}{\phi} = \sigma^2 + \BzeroM{M}{\phi} + \VzeroM{M}{\phi}$,
        where 
        the terms are given by
        $
            \BzeroM{M}{\phi}
            = M^{-1} B_0(\phi_s, \phi_s)
            + (1 -  M^{-1}) B_{0}(\phi,\phi_s),
        $
        and 
        $
            \VzeroM{M}{\phi}
            = M^{-1} V_{0}(\phi_s, \phi_s)
            + (1 - M^{-1})
            V_{0}(\phi,\phi_s),
        $
        and the functions $B_{0}(\cdot, \cdot)$
        and $V_{0}(\cdot, \cdot)$ are defined as
        \begin{align*}
            B_0(\vartheta, \theta)
            &= 
            \begin{dcases}
            0 
            & \theta \in (0, 1), \, \vartheta \le \theta \\
            \rho^2 
            (1 + \tv(0;\vartheta, \theta)\tc(0;\theta)& \theta \in (1, \infty], \, \vartheta \le \theta
            \end{dcases}, \qquad
            V_0(\vartheta,\theta)
            =
            \begin{dcases}
                \sigma^2
                \tfrac{\vartheta}{1  - \vartheta}
                & \theta \in (0, 1), \, \vartheta \le \theta \\
                \sigma^2 \tv(0;\vartheta,\theta)
                & \theta \in (1, \infty], \, \vartheta \le \theta
            \end{dcases},
        \end{align*}
        where the nonnegative constants $\tv(0;\vartheta,\theta)$ and $\tc(0;\theta)$ are as defined in \eqref{eq:tv_tc_ridgeless}.
    \end{theorem}    
    \begin{proof}[Proof of Theorem \ref{thm:ver-ridgeless}]
        We use the same notations as in the proof for Theorem \ref{thm:ver-ridge} and let $\hSigma_m=\bX^{\top}\bL_m\bX/k$ for all $m\in[M]$.
        Note that
        \begin{align*}
            \bbeta_0-\tbeta(\{\cD_{I_\ell}\}_{\ell=1}^M)  &=
            \tfrac{1}{M}\sum\limits_{m=1}^M(\bI_p-\hSigma_m^{+}\hSigma_m)\bbeta_0 -            \tfrac{1}{M}\sum\limits_{m=1}^M\hSigma_m^{+}\tfrac{\bX^{\top}\bL_m\bepsilon}{k} .
        \end{align*}
        We have 
        \begin{align*}
         R(\tfWR{0}{M};\cD_n,\{I_\ell\}_{\ell=1}^M) &= \sigma^2 + (\bbeta_0-\tbeta(\{\cD_{I_\ell}\}_{\ell=1}^M))^{\top}\bSigma(\bbeta_0-\tbeta(\{\cD_{I_\ell}\}_{\ell=1}^M)) \\
         &= \sigma^2 + T_B + T_V + T_C,
        \end{align*}
        where
        \begin{align}
            T_C &= - \tfrac{2}{M^2}\bepsilon ^{\top}\left(\sum_{m=1}^M\hat{\bSigma}_m^{+}\tfrac{\bX^{\top}\bL_m}{k}\right)^{\top}\bSigma\left(\sum_{m=1}^M(\bI_p - \hat{\bSigma}_m^{+}\hat{\bSigma}_m)\right)   \bbeta_0,\label{eq:ridgeless-C0}\\
            T_B &= \tfrac{1}{M^2}\bbeta_0^{\top} \left(\sum_{m=1}^M(\bI_p - \hat{\bSigma}_m^{+}\hat{\bSigma}_m)\right) \bSigma \left(\sum_{m=1}^M(\bI_p - \hat{\bSigma}_m^{+}\hat{\bSigma}_m)\right) \bbeta_0,\label{eq:ridgeless-B0}\\
            T_V&=\tfrac{1}{M^2} \bepsilon^{\top} \left(\sum_{m=1}^M\hat{\bSigma}_m^{+}\tfrac{\bX^{\top}\bL_m}{k}\right)^{\top}
            \bSigma \left(\sum_{m=1}^M\hat{\bSigma}_m^{+}\tfrac{\bX^{\top}\bL_m}{k}\right)
            \bepsilon. \label{eq:ridgeless-V0}
        \end{align}
        
        Next we analyze the three term separately for $M\in\{1,2\}$.
        From \Cref{lem:ridgeless-conv-C0}, we have that $T_C\asto0$.
        Further, from \Cref{lem:ridgeless-B0}, \Cref{lem:ridgeless-V0-phis_lt1}, and
        \Cref{lem:ridgeless-V0-phis-gt1}, for all $I_1\in\cI_k$ when $M=1$ and for all $I_m,I_l\overset{\texttt{\textup{SRSWR}}}{\sim}\cI_k$ when $M=2$, it holds that $$R(\tilde{f}_{M,\lambda};\{\cD_{I_{\ell}}\}_{\ell=1}^M) \asto \RzeroM{M}{\phi}$$ as $n,k,p\rightarrow\infty$, $p/n\rightarrow\phi\in(0,\infty)$, and $p/k\rightarrow\phi_s\in[\phi,\infty)\setminus\{1\}$, where
        \begin{align*}
            \RzeroM{M}{\phi} &= \sigma^2 + \tfrac{1}{M}(B_{0}(\phi_s,\phi_s)+V_{0}(\phi_s,\phi_s)) + \tfrac{M-1}{M}(B_{0}(\phi,\phi_s) +V_{0}(\phi,\phi_s)).
        \end{align*}
        Here, the components are:
        \begin{align*}
            B_{0}(\phi,\phi_s) &= \begin{cases}
            0,&\phi_s\in(0,1)\\
            \rho^2(1+\tv(0;\phi,\phi_s))\tc(0;\phi_s),&\phi_s\in(1,\infty)
            \end{cases},
            &&
            V_{0}(\phi,\phi_s) = \begin{cases}
            \sigma^2 \tfrac{\phi}{1-\phi},&\phi_s\in(0,1)\\
            \sigma^2\tv(0;\phi,\phi_s), &\phi_s\in(1,\infty)
            \end{cases},
        \end{align*}
        and the nonnegative constants $\tv(0;\phi,\phi_s)$ and $\tc(0;\phi_s)$ 
        are as defined in \eqref{eq:tv_tc_ridgeless}.
        The proof for the boundary case when $\phi_s=\infty$ follows from \Cref{prop:Rdet-ridgeless-infinity}.
        Then, we have that the function $\RzeroM{M}{\phi}$ is continuous on $[\phi,\infty]\setminus\{1\}$ and lower-semi continuous on $[\phi,\infty]$.
        
        Finally, the risk expression for general $M$ and the uniformity claim over $M\in\NN$ follow from \Cref{thm:risk_general_predictor}. 
    \end{proof}

    \subsection{Component concentrations}
        In this subsection, we will show that the cross-term $C_0$ converges to zero 
        and the variance term $T_V$ converge to its corresponding trace expectation.
    
        \subsubsection{Convergence of the cross term}
        \begin{lemma}[Convergence of the cross term]\label{lem:ridgeless-conv-C0}
            Under Assumptions \ref{asm:rmt-feat}-\ref{asm:spectrum-spectrumsignproj-conv}, 
            for $T_C$ as defined in \eqref{eq:ridgeless-C0},
            we have $T_C\asto 0$ as $k,p\rightarrow\infty$ and $p/k\rightarrow\phi_s\in(0,1)\cup(1,\infty)$, 
        \end{lemma}
        \begin{proof}[Proof of Lemma \ref{lem:ridgeless-conv-C0}]
            Note that
        \begin{align*}
            T_C &= - \tfrac{2}{M^2} \cdot \tfrac{1}{k} \left\langle \left(\sum_{m=1}^M\hat{\bSigma}_m^{+}\tfrac{\bX^{\top}\bL_m}{k}\right)^{\top}\bSigma\left(\sum_{m=1}^M(\bI_p - \hat{\bSigma}_m^{+}\hat{\bSigma}_m)\right)   \bbeta_0, \bepsilon \right\rangle.
        \end{align*}
        We next bound the norm
        \begin{align*}
            &\tfrac{1}{k}\norm{\frac{1}{M}\left(\sum_{m=1}^M\hat{\bSigma}_m^{+}\tfrac{\bX^{\top}\bL_m}{k}\right)^{\top}\bSigma\left(\sum_{m=1}^M(\bI_p - \hat{\bSigma}_m^{+}\hat{\bSigma}_m)\right)   \bbeta_0}_2^2\\
            &\leq  \frac{1}{M^2}\sum_{m=1}^M\sum\limits_{l=1}^M\tfrac{1}{k}\norm{( \hat{\bSigma}_m^{+}\bX^{\top}\bL_m)^{\top} \bSigma (\bI_p - \hat{\bSigma}_l^{+}\hat{\bSigma}_l) \bbeta_0}_2^2\\
            &\leq 
            \frac{\norm{\bbeta_0}_2^2}{M^2}\cdot \sum_{m=1}^M\sum\limits_{l=1}^M
            \norm{(\bI_p - \hat{\bSigma}_l^{+}\hat{\bSigma}_l) \bSigma  \bSigma_m^{+}\bSigma_m \hSigma_m^{+}  \bSigma(\bI_p - \hat{\bSigma}_l^{+}\hat{\bSigma}_l)}_{\oper} \\
            &\leq \frac{\norm{\bbeta_0}_2^2}{M^2}\cdot \sum_{j=1}^M\sum\limits_{l=1}^M
            \norm{\bI_p - \hat{\bSigma}_l^{+}\hat{\bSigma}_l}_{\oper}^2 \norm{\bSigma}_{\oper}^2\norm{\hSigma_j^{+}\bSigma_j \hSigma_j^{+}}_{\oper} \\
            &= \frac{\norm{\bbeta_0}_2^2}{M^2}\cdot \sum_{j=1}^M\sum\limits_{l=1}^M
            \norm{\bI_p - \hat{\bSigma}_l^{+}\hat{\bSigma}_l}_{\oper}^2 \norm{\bSigma}_{\oper}^2\norm{ \hSigma_j^{+}}_{\oper}\\
            &\leq \norm{\bbeta_0}_2^2r_{\max}^2 \cdot \frac{1}{M}\sum_{j=1}^M\norm{ \hSigma_m^{+}}_{\oper},
        \end{align*}
        where the last inequality is due to Assumption \ref{asm:covariance-bounded-eigvals} and the fact that $ \norm{\bI_p - \hat{\bSigma}_l^{+}\hat{\bSigma}_l}_{\oper}\leq 1$.
        By Assumption \ref{asm:signal-bounded-norm}, $\|\bbeta_0\|^2_2$ is uniformly bounded in $p$.
        From \citet{bai2010spectral}, $\liminf\min_{1\leq i\leq p}s_i^2\geq r_{\min}(1-\sqrt{\phi_s})^2$ almost surely for $\phi_s\in(0,1)\cup(1,\infty)$. Thus, $\limsup\norm{\hSigma_m^{+}}_{\oper}\leq C$ for some constant $C<\infty$ almost surely.
        Applying 
        \Cref{lem:concen-linform},
        we thus have that $T_C\asto 0$.
        \end{proof}

        \subsubsection{Convergence of the variance term}
        
        \begin{lemma}[Convergence of the variance term]\label{lem:ridgeless-conv-V0}
            Under Assumptions \ref{asm:rmt-feat}-\ref{asm:spectrum-spectrumsignproj-conv}, for all $m\in[M]$ and $I_m\in \cI_k$, let $\hSigma_m=\bX^{\top}\bL_m\bX/k$ and $\bL_m\in\RR^{n\times n}$ be a diagonal matrix with $(\bL_{m})_{ll}=1$ if $l\in I_m$ and 0 otherwise.
            Then, it holds that
            \begin{enumerate}[leftmargin=7mm]
                \item for all $m\in[M]$ and $I_m\in\cI_k$,
                \begin{align*}
                    \tfrac{1}{k^2} \bepsilon^{\top}\bL_m\bX\hSigma_m^{+} \bSigma \hSigma_m^{+}\bX^{\top}\bL_m\bepsilon -\tfrac{\sigma^2}{k}\tr(  \hSigma_j^{+}\bSigma) \asto 0,
                \end{align*}
                
                \item for all $m,l\in[M]$, $m\neq l$ and $I_m,I_l\overset{\texttt{\textup{SRSWR}}}{\sim}\cI_k$,
                \begin{align*}
                    \tfrac{1}{k^2} \bepsilon^{\top}\bL_m\bX\hSigma_m^{+} \bSigma \hSigma_l^{+}\bX^{\top}\bL_l\bepsilon-\tfrac{\sigma^2}{k^2}\tr( \hSigma_l^{+} \bX^{\top}\bL_l\bL_m \hSigma_m^{+} \bSigma) \asto 0,
                \end{align*}
            \end{enumerate}
            as $n,k,p\rightarrow\infty$, $p/n\rightarrow\phi\in(0,\infty)$, and $p/k\rightarrow\phi_s\in[\phi,\infty)\setminus\{1\}$.
        \end{lemma}
        \begin{proof}[Proof of Lemma \ref{lem:ridge-conv-V0}]
            Note that the term is the same as the variance terms for ridge predictor trained on $k$ i.i.d.\ samples $(\bL_{m}\bX,\bL_{m}\by)$.
            Notice that $\bL_m\bepsilon$ is independent of $\bL_m\bX\hSigma_m^{+} \bSigma \hSigma_m^{+}\bX^{\top}\bL_m$, and
            \begin{align*}
                \tfrac{1}{k}\norm{\bL_m\bX\hSigma_m^{+} \bSigma \hSigma_m^{+}\bX^{\top}\bL_m}_{\oper} &\leq \norm{\hSigma_m^{+}}_{\oper}^2\norm{\hSigma_m}_{\oper} \norm{\bSigma }_{\oper}\leq r_{\max}\norm{\hSigma_m^{+}}_{\oper}^2\norm{\hSigma_m}_{\oper}.
            \end{align*}
            Observe that $\liminf\norm{\hSigma_m}_{\oper}\geq \liminf\min_{1\leq i\leq p} s_i^2\geq r_{\max}(1-\sqrt{\phi_s})^2$ and $\limsup\norm{\hSigma_m}_{\oper}\leq \limsup\max_{1\leq i\leq p} s_i^2\leq r_{\max}(1+\sqrt{\phi_s})^2$ almost surely as $k,p\rightarrow\infty$ and $p/k\rightarrow\phi_s\in(0,1)\cup(1,\infty)$ from \citet{bai2010spectral}.
            We have $\limsup\norm{\hSigma_m^{+}}_{\oper}\leq C$ and $\limsup\norm{\hSigma_m}_{\oper}\leq C$ for some constant $C<\infty$ almost surely.
            From 
            \Cref{lem:concen-quadform},
            it follows that
            \begin{align*} 
            \tfrac{1}{k^2} \bepsilon^{\top}\bL_m\bX\hSigma_m^{+} \bSigma \hSigma_m^{+}\bX^{\top}\bL_m\bepsilon -\tfrac{\sigma^2}{k^2}\tr(  \bL_m\bX\hSigma_m^{+} \bSigma \hSigma_m^{+}\bX^{\top}\bL_m)\asto 0.
            \end{align*}
            Since $\tr(\bL_m\bX\hSigma_m^{+} \bSigma \hSigma_m^{+}\bX^{\top}\bL_m)/k^2=\tr(  \hSigma_m^{+}\hSigma_m\hSigma_m^{+} \bSigma )/k=\tr(  \hSigma_m^{+}\bSigma)/k$, we further have
            \begin{align} \tfrac{1}{k^2} \bepsilon^{\top}\bL_m\bX\hSigma_m^{+} \bSigma \hSigma_m^{+}\bX^{\top}\bL_m\bepsilon  - \tfrac{\sigma^2}{k}\tr(  \hSigma_m^{+}\bSigma) \asto 0. \label{eq:ridgeless-conv-V0-2}
            \end{align}
            
            The second term involves the cross term $\bM_m\bSigma\bM_l$.
            Note that
            \begin{align*}
                \tfrac{1}{n}\norm{\bL_m\bX\hSigma_m^{+} \bSigma \hSigma_l^{+}\bX^{\top}\bL_l}_{\oper } &\leq \tfrac{k}{n}r_{\max}\norm{\hSigma_m}_{\oper}^{\tfrac{1}{2}}\norm{\hSigma_m}_{\oper}^{\tfrac{1}{2}}\norm{\hSigma_l^{+}}_{\oper}\norm{\hSigma_l^{+}}_{\oper},
            \end{align*}
            because $\norm{\hSigma_m^{+}}_{\oper}$ and $\norm{\hSigma_m}_{\oper}$
            for $m \in [M]$
            are uniformly bounded almost surely.
            By 
            \Cref{lem:concen-quadform},
            it follows that
            \begin{align*}
                \tfrac{1}{n}\bepsilon^{\top}\bL_m\bX\hSigma_m^{+} \bSigma \hSigma_l^{+}\bX^{\top}\bL_l\bepsilon - \tfrac{\sigma^2}{n}\tr(\bL_m\bX\hSigma_m^{+} \bSigma \hSigma_l^{+}\bX^{\top}\bL_l)\asto0.
            \end{align*}
            Since $k/n\rightarrow\phi_s/\phi$, we have
            \begin{align*}
                \tfrac{1}{k^2}\bepsilon^{\top}\bL_m\bX\hSigma_m^{+} \bSigma \hSigma_l^{+}\bX^{\top}\bL_l\bepsilon - \tfrac{\sigma^2}{k^2}\tr(\bL_m\bX\hSigma_m^{+} \bSigma \hSigma_l^{+}\bX^{\top}\bL_l)\asto0.
            \end{align*}
        \end{proof}

    \subsection{Component deterministic approximations}
    
        \subsubsection{Deterministic approximation of the bias functional}
        \begin{lemma}[Deterministic approximation of the bias functional]\label{lem:ridgeless-B0}
            Under Assumptions \ref{asm:rmt-feat}-\ref{asm:spectrum-spectrumsignproj-conv}, for all $m\in[M]$ and $I_m\in \cI_k$, let $\hSigma_m=\bX^{\top}\bL_m\bX/k$ and $\bL_m\in\RR^{n\times n}$ be a diagonal matrix with $(\bL_{m})_{ll}=1$ if $l\in I_m$ and 0 otherwise.
            Then, it holds that
            \begin{enumerate}[leftmargin=7mm]
                \item for all $m\in[M]$ and $I_m\in\cI_k$,
                \begin{align*}
                    \bbeta_0^{\top} (\bI_p - \hat{\bSigma}_m^{+}\hat{\bSigma}_m) \bSigma(\bI_p - \hat{\bSigma}_m^{+}\hat{\bSigma}_m)\bbeta_0 \asto \begin{cases}
                    0 &\phi_s\in(0,1)\\
                    \rho^2(1+\tv(0;\phi_s,\phi_s))\tc(0;\phi_s) &\phi_s\in(1,\infty),
                    \end{cases}
                \end{align*}
                
                \item for all $m,l\in[M]$, $m\neq l$ and $I_m,I_l\overset{\texttt{\textup{SRSWR}}}{\sim}\cI_k$,
                \begin{align*}
                    \bbeta_0^{\top} (\bI_p - \hat{\bSigma}_m^{+}\hat{\bSigma}_m) \bSigma(\bI_p - \hat{\bSigma}_l^{+}\hat{\bSigma}_l)\bbeta_0 \asto \begin{cases}
                    0 &\phi_s\in(0,1)\\
                    \rho^2(1+\tv(0;\phi,\phi_s))\tc(0;\phi_s) &\phi_s\in(1,\infty),
                    \end{cases}
                \end{align*}
            \end{enumerate}
            as $n,k,p\rightarrow\infty$, $p/n\rightarrow\phi\in(0,\infty)$, and $p/k\rightarrow\phi_s\in[\phi,\infty)\setminus\{1\}$, where the nonnegative constants $\tv(0; \phi,\phi_s)$ and $\tc(0;\phi_s)$ are as defined in \eqref{eq:tv_tc_ridgeless}.
        \end{lemma}
        \begin{proof}[Proof of Lemma \ref{lem:ridgeless-B0}]
            For the first term, we have that for $m\in[M]$,
            \begin{align}
                \bbeta_0^{\top} (\bI_p - \hat{\bSigma}_m^{+}\hat{\bSigma}_m) \bSigma(\bI_p - \hat{\bSigma}_m^{+}\hat{\bSigma}_m)\bbeta_0 \asto \begin{cases}   
                0 &\text{if }\phi_s\in(0,1) \\
                \rho^2(1+\tv_b(0; \phi_s)) \int \tfrac{r}{(1+ v(0; \phi_s)  r)^2} \,\rd G(r)&\text{if }\phi_s\in(1,\infty). \end{cases}\label{eq:lem-ridgeless-det-approx-B0-1}
            \end{align}
            
            Next we analyze the second term, by considering the following two cases separately for $(m,l)=(1,2)$.
            
            \noindent 
            \underline{(1) $\phi_s\in(0,1)$}.
            Since the singular values of $\hSigma_j$'s are almost surely lower bounded away from 0, we have $\hSigma_j^{+}\hSigma_j=\bI_p$ almost surely.
            Then $\bbeta_0^{\top}(\bI_p-\hat{\bSigma}_1^{+}\hat{\bSigma}_1) \bSigma(\bI_p - \hat{\bSigma}_2^{+}\hat{\bSigma}_2)\bbeta_0\asto 0$ when $k,p\rightarrow\infty$ and $p/k\rightarrow\phi_s\in(0,1)$.

            \noindent 
            \underline{(2) $\phi_s\in(1,\infty)$}.
            We begin with analyzing the deterministic equivalent of $(\bI_p-\hat{\bSigma}_1^{+}\hat{\bSigma}_1) \bSigma(\bI_p - \hat{\bSigma}_2^{+}\hat{\bSigma}_2)$.
            Recall that $i_0$ is the number of shared samples between $\cD_{I_1}$ and $\cD_{I_2}$, and $\hat{\bSigma}_0=\bX^{\top}\bL_1\bL_2\bX^{\top}/i_0$ and $\hat{\bSigma}_{j}^{\text{ind}}=\bX^{\top}(\bL_j-\bL_1\bL_2)\bX^{\top}/(k-i_0)$ are the common and individual covariance estimators of the two datasets.
            Also note that from \eqref{eq:lem-ridge-B0-term-det}, we have $\lambda^2\bM_1\bSigma\bM_2 \asympequi (\lambda^2\bM_1\bSigma\bM_2)^{\det}$, where
            \begin{align}  (\lambda^2\bM_1\bSigma\bM_2)^{\det}=(1+\tv(-\lambda;\phi_s,\phi) )\left(v(-\lambda; \phi_s)\bSigma + \bI_p\right)^{-2}\bSigma>0,\label{eq:lem-ridgeless-B0-term-1}
            \end{align}
            and $\tv(-\lambda;\phi_s,\phi)$ is as defined in \eqref{eq:tv_tc_ridgeless}.
            Let $\lambda\in\Lambda=[0,\lambda_{\max}]$ where $\lambda_{\max}<\infty$. For any matrix $\bT\in\RR^{p\times p}$ with trace norm uniformly bounded by $M$,
            \begin{align*}
                |\tr[\lambda^2\bM_1\bSigma\bM_2\bT]| & \leq \lambda^2\norm{\bM_1}_{\oper}\norm{\bM_2}_{\oper}\norm{\bSigma}_{\oper} |\tr[(\bT^{\top}\bT)^{\tfrac{1}{2}}]| \leq Mr_{\max}\norm{\bSigma}_{\oper}
            \end{align*}
            where the second inequality holds because $\norm{\bM_1}_{\oper}\leq \lambda^{-1}$ and $\norm{\bSigma}_{\oper}\leq r_{\max}$.
            Since $\phi_0\geq\phi_s>1$, it follows from \citet[Lemma S.6.14]{patil2022mitigating} that, there exists $M'>0$ such that  the magnitudes of $v(-\lambda; \phi_s)$ and $v_b(\lambda,\phi_s,\phi)-1$, and their derivatives with respect to $\lambda$ are continuous and bounded by $M'$ for all $\lambda\in\Lambda$.
            Thus,
            we get
            \begin{align*}
                \left|\tr\left[(\lambda^2\bM_1\bSigma\bM_2)^{\det}\bT\right]\right|
                &\leq (1+M') \norm{v(-\lambda; \phi_s)\bSigma + \bI_p}_{\oper}^{-2}\norm{\bSigma}_{\oper} |\tr[(\bT^{\top}\bT)^{\tfrac{1}{2}}]|\\
                &\leq  (1+M')Mr_{\max}.
            \end{align*}
            Similarly, in the same interval the derivatives of $\tr\left[\lambda^2\bM_1\bSigma\bM_2\bT\right]$ and $\tr\left[(\lambda^2\bM_1\bSigma\bM_2)^{\det}\bT\right]$ with respect to $\lambda$ also have bounded magnitudes for $\lambda\in\Lambda$.
            Therefore, the family of functions 
            $$\tr[\lambda^2\bM_1\bSigma\bM_2\bT]-\tr\left[(\lambda^2\bM_1\bSigma\bM_2)^{\det}\bT\right]$$
            forms an equicontinuous family in $\lambda$ over $\lambda \in \Lambda$.
            Thus, the convergence in Part 1 of \Cref{lem:deter-approx-generalized-ridge} is uniform in $\lambda$. 
            We can now use the Moore-Osgood theorem and the continuity property from \Cref{lem:fixed-point-v-lambda-properties} to interchange the limits to obtain
            \begin{align*}
                &\lim\limits_{p\rightarrow\infty}\tr\left[(\bI_p-\hat{\bSigma}_1^{+}\hat{\bSigma}_1) \bSigma(\bI_p - \hat{\bSigma}_2^{+}\hat{\bSigma}_2)\bT\right] - \tr\left[((\bI_p-\hat{\bSigma}_1^{+}\hat{\bSigma}_1) \bSigma(\bI_p - \hat{\bSigma}_2^{+}\hat{\bSigma}_2))^{\det}\bT\right]\\
                &= \lim\limits_{p\rightarrow\infty}\lim\limits_{\lambda\rightarrow0^+}\tr\left[\lambda^2\bM_1\bSigma\bM_2\bT\right]-\tr\left[ (\lambda^2\bM_1\bSigma\bM_2)^{\det}\bT\right]\\ &=\lim\limits_{\lambda\rightarrow0^+}\lim\limits_{p\rightarrow\infty}\tr\left[\lambda^2\bM_1\bSigma\bM_2\bT\right]-\tr\left[ (\lambda^2\bM_1\bSigma\bM_2)^{\det}\bT\right]\\
                &=0,
            \end{align*}
            where $$((\bI_p-\hat{\bSigma}_1^{+}\hat{\bSigma}_1) \bSigma(\bI_p - \hat{\bSigma}_2^{+}\hat{\bSigma}_2))^{\det}= (1+\tv(0;\phi,\phi_s) ) (v(0; \phi_s) \bSigma + \bI_p)^{-2}\bSigma.$$
            As $p\rightarrow\infty$, replacing the empirical distribution $G_p(r)$ by limiting distribution $G(r)$ yields the desired results.
        \end{proof}
        
        \subsubsection{Deterministic approximation of the variance functional}

        \begin{lemma}[Deterministic approximation of the variance functional when $\phi_s<1$]\label{lem:ridgeless-V0-phis_lt1}
            Under Assumptions \ref{asm:rmt-feat}-\ref{asm:spectrum-spectrumsignproj-conv}, for all $m\in[M]$ and $I_m\in \cI_k$, let $\hSigma_m=\bX^{\top}\bL_m\bX/k$ and $\bL_m\in\RR^{n\times n}$ be a diagonal matrix with $(\bL_{m})_{ll}=1$ if $l\in I_m$ and 0 otherwise.
            Then, it holds that
            \begin{enumerate}[leftmargin=7mm]
                \item for all $m\in[M]$ and $I_m\in\cI_k$,
                \begin{align*}
                    \tfrac{1}{k}\tr(  \hSigma_m^{+}\bSigma) \asto \tfrac{\phi_s}{1-\phi_s},
                \end{align*}
                
                \item for all $m,l\in[M]$, $m\neq l$ and $I_m,I_l\overset{\texttt{\textup{SRSWR}}}{\sim}\cI_k$,
                \begin{align*}
                    \tfrac{1}{k}\tr( \hSigma_l^{+} \bX^{\top}\bL_l\bL_m \hSigma_m^{+} \bSigma)  \asto \tfrac{\phi}{1-\phi} %
                \end{align*}
            \end{enumerate} 
            as $n,k,p\rightarrow\infty$, $p/n\rightarrow\phi\in(0,\infty)$, and $p/k\rightarrow\phi_s\in[\phi,\infty)\cap (0,1)$.
        \end{lemma}
        \begin{proof}[Proof of \Cref{lem:ridgeless-V0-phis_lt1}]
        For the first term, from \citet[Proposition S.3.2]{patil2022mitigating} we have that for $m\in[M]$,
            \begin{align}
                \tfrac{1}{k}\tr(  \hSigma_m^{+}\bSigma) \asto \begin{dcases}   
                \tfrac{\phi_s}{1-\phi_s} &\text{if }\phi_s\in(0,1) \\
                \phi_sv_v(0; \phi,\phi_s) \int \tfrac{r^2}{(1+ v(0; \phi_s)  r)^2} \,\rd H(r)&\text{if }\phi_s\in(1,\infty) \end{dcases}.\label{eq:lem-ridgeless-det-approx-V0-1}
            \end{align}

            Next we analyze the second term for $\phi_s\in(0,1)$. It suffices to analyze the case when $(m,l)=(1,2)$.
            From \citet{bai2010spectral}, we have
            \begin{align*}
                r_{\min}(1-\sqrt{\phi_s})^2\leq \liminf\norm{\hSigma_j}_{\oper}&\leq \limsup\norm{\hSigma_j}_{\oper}\leq r_{\max}(1+\sqrt{\phi_s})^2 ,\quad j=1,2.
            \end{align*}
            Then $\hSigma_j$'s are invertible almost surely.
            From \Cref{lem:det-equiv-subsample}, we have that for $j=1,2$,
            \[
                \hSigma_j^{-1}=\left(\tfrac{i_0}{k} \hSigma_0 + \tfrac{k-i_0}{k} \hSigma_1^{\mathrm{ind}}\right)^{-1}
                \asympequi \left(\tfrac{i_0}{k} \hSigma_0 +  (1 - \phi_s)\tfrac{k-i_0}{k} \bSigma\right)^{-1},
            \]
            where $\hSigma_0=\bX^{\top}\bL_1\bL_2\bX/i_0$ and $\hSigma_j^{\text{ind}}=\bX^{\top}\bL_j\bX/(k-i_0)$ for $j=1,2$, defined analogously as in the proof for \Cref{thm:ver-ridge}.
            Thus, conditional on $\hSigma_0$ and $i_0$, we have
            \begin{align*}
                \hSigma_1^{-1}
                \hSigma_0
                \hSigma_2^{-1}\bSigma
                &\asympequi
               \left(\tfrac{i_0}{k} \hSigma_0 +  (1 - \phi_s)\tfrac{k-i_0}{k} \bSigma\right)^{-1}
                \hSigma_0
                \left(\tfrac{i_0}{k} \hSigma_0 +  (1 - \phi_s)\tfrac{k-i_0}{k} \bSigma\right)^{-1}
                \\
                &= \tfrac{i_0^2}{k^2} \left(\hSigma_0 + (1-\phi_s)\tfrac{k-i_0}{i_0}\bSigma\right)^{-1} \hSigma_0 \left(\hSigma_0 + (1-\phi_s)\tfrac{k-i_0}{i_0}\bSigma\right)^{-1}\bSigma
            \end{align*}
            by applying the conditional product rule from \Cref{prop:cond-calculus-detequi}.
            When $i_0<k$, let $\hSigma'=c \bSigma^{-\tfrac{1}{2}}\hSigma_0\bSigma^{-\tfrac{1}{2}}$ and $c=(1-\phi_s)(k-i_0)/i_0$, we further have
            \begin{align*}
               \hSigma_1^{-1}
                \hSigma_0
                \hSigma_2^{-1}\bSigma
               &\asympequi\tfrac{i_0^2}{k^2c^{2}}\bSigma^{-\tfrac{1}{2}} (\hSigma' + \bI_p)^{-1} \hSigma' (\hSigma' + \bI_p)^{-1}\bSigma^{-\tfrac{1}{2}}\bSigma\\
                &\asympequi \tfrac{i_0^2}{k^2} \tv_v(-1;\gamma_0,c^{-1}\bI_p)(v(-1;\gamma_0,c^{-1}\bI_p)+c)^{-2}  \bI_p,
            \end{align*}
            where $\gamma_0=p/i_0$, the second equality is from \Cref{lem:deter-approx-generalized-ridge}~\ref{eq:detequi-ridge-genvar} and the fixed point solutions are defined by
            \begin{align*}
                \tfrac{1}{v(-1;\gamma_0,c^{-1}\bI_p)} &= 1 +  \tfrac{\gamma_0}{c+v(-1;\gamma_0,c^{-1}\bI_p)}\\
                \tfrac{1}{\tv_v(-1;\gamma_0,c^{-1}\bI_p)} &= \tfrac{1}{v(-1;\gamma_0,c^{-1}\bI_p)^2} - \tfrac{\gamma_0}{(c+v(-1;\gamma_0,c^{-1}\bI_p))^2}.
            \end{align*}
            When $i_0=k$, the above equivalent is also valid, which reduces to the case for $\hSigma^{+}_j\hSigma_j\hSigma^{+}_j$ as in \eqref{eq:lem-ridgeless-det-approx-V0-1}.
            Note that from \Cref{lem:ridge-fixed-point-v-properties}, $\tilde{v}_v(-\lambda;\gamma)$ and $v(-\lambda;\gamma)$ are continuous on $\gamma$, and from Lemma \ref{lem:i0_mean}, $i_0/k\asto \phi/\phi_s$ where $\phi_s\in(0,\infty)$ is the limiting ratio such that $p/k\rightarrow\phi_s$ as $k,p\rightarrow\infty$.
            We have
            \begin{align*}
               \hSigma_1^{-1}
                \hSigma_0
                \hSigma_2^{-1}\bSigma&\asympequi \tfrac{\phi_s^2}{\phi_0^2} \tv_v(-1;\phi_0,c_0^{-1}\bI_p)(v(-1;\phi_0,c_0^{-1}\bI_p)+c_0)^{-2}  \bI_p,
            \end{align*}
            where $c_0=\lim_{p\rightarrow\infty}c=(1-\phi_s)(\phi_s-\phi)/\phi$ and the fixed solutions reduce to
            \begin{align*}
                v(-1;\gamma_0,c_0^{-1}\bI_p) &= 1-\phi_s,\qquad \tv_v(-1;\gamma_0,c_0^{-1}\bI_p) = \tfrac{(1-\phi_s)^2}{1-\phi}.
            \end{align*}
            Then, we have
            \begin{align}
                \tfrac{i_0}{k^2}\tr[ \hSigma_1^{+} \hSigma_0 \hSigma_2^{+} \bSigma] & \asto \lim\limits_{p\rightarrow\infty}\tfrac{i_0p}{k^2}\cdot \tfrac{1}{p}\tr\left[\tfrac{\phi_s^2}{\phi^2}\tfrac{(1-\phi_s)^2}{1-\phi} \left(  1-\phi_s + \tfrac{(1-\phi_s)(\phi_s-\phi)}{\phi} \right)^{-2}\bI_p\right]=\tfrac{\phi}{1-\phi}. \label{eq:lem:ridgeless-V0-2}
            \end{align}
            Combining \eqref{eq:lem-ridgeless-det-approx-V0-1} and \eqref{eq:lem:ridgeless-V0-2}, the conclusion follows.
        \end{proof}
        
        \begin{lemma}[Deterministic approximation of the variance functional when $\phi_s>1$]\label{lem:ridgeless-V0-phis-gt1}
            Under Assumptions \ref{asm:rmt-feat}-\ref{asm:spectrum-spectrumsignproj-conv}, for all $m\in[M]$ and $I_m\in \cI_k$, let $\hSigma_m=\bX^{\top}\bL_m\bX/k$ and $\bL_m\in\RR^{n\times n}$ be a diagonal matrix with $(\bL_{m})_{ll}=1$ if $l\in I_m$ and 0 otherwise.
            Then, it holds that
            \begin{enumerate}[leftmargin=7mm]
                \item for all $m\in[M]$ and $I_m\in\cI_k$,
                \begin{align*}
                    \tfrac{1}{k}\tr(  \hSigma_j^{+}\bSigma) \asto \tfrac{1}{2}\tv(0;\phi_s,\phi_s) ,
                \end{align*}
                
                \item for all $m,l\in[M]$, $m\neq l$ and $I_m,I_l\overset{\texttt{\textup{SRSWR}}}{\sim}\cI_k$,
                \begin{align*}
                    \tfrac{1}{k^2}\bepsilon^{\top}\bL_1\bX\hSigma_m^{+} \bSigma \hSigma_l^{+}\bX^{\top}\bL_2\bepsilon \asto \tfrac{1}{2}\tv(0;\phi,\phi_s), %
                \end{align*}
            \end{enumerate} 
            as $n,k,p\rightarrow\infty$, $p/n\rightarrow\phi\in(0,\infty)$, and $p/k\rightarrow\phi_s\in[\phi,\infty)\cap(1,\infty)$, where the nonnegative constants $v(0;\phi_s)$ and $\tv(0;\phi,\phi_s)$ are as defined in \eqref{eq:v_ridgeless} and \eqref{eq:tv_tc_ridgeless}.
        \end{lemma}
        \begin{proof}[Proof of Lemma \ref{lem:ridgeless-V0-phis-gt1}]
        
        From \eqref{eq:lem-ridgeless-det-approx-V0-1} we have 
        \begin{align}
            \tfrac{1}{k}\tr(  \hSigma_m^{+}\bSigma) \asto \tv(0;\phi,\phi_s). \label{eq:lem-ridgeless-V0-phis-gt1-term-2}
        \end{align}
        For the second term, it suffices to consider the case when $(m,l)=(1,2)$. Let $P_0=\bepsilon^{\top}\bL_1\bX\hSigma_1^{+} \bSigma \hSigma_2^{+}\bX^{\top}\bL_2\bepsilon/k^2$ and $P_{\lambda}=\bepsilon^{\top}\bL_1\bX\bM_1 \bSigma$ $ \bM_2\bX^{\top}\bL_2\bepsilon /k^2$ where $\bM_j=(\hSigma_j+\lambda\bI_p)^{-1}$.
        Note that $\lim_{\lambda\rightarrow0^+}P_{\lambda}=P_0$.
        Note that $\lim_{\lambda\rightarrow0^+}P_{\lambda}=P_0$.
        From \Cref{lem:ridge-conv-V0} and \Cref{lem:ridge-V0}, we have that for any fixed $\lambda>0$,
        \begin{align*}
            P_{\lambda} \asto Q_{\lambda}:=\tv(-\lambda;\phi,\phi_s),
        \end{align*}
        as $n,k,p\rightarrow\infty$, $p/n\rightarrow\phi\in(0,\infty)$, and $p/k\rightarrow\phi_s\in[\phi,\infty)\setminus\{1\}$, where $\tv(\lambda,\phi_s,\phi)$ is as defined in \eqref{eq:tv_tc_ridge}.
        Because of the continuity of $\tv_v(-\lambda;\phi)$ and $v(-\lambda;\phi)$ in $\lambda$ from \Cref{lem:fixed-point-v-lambda-properties}, we have that 
        \begin{align*}
            \lim\limits_{\lambda\rightarrow0^+}Q_{\lambda} = Q_0:=\tv(0;\phi,\phi_s).
        \end{align*}
        As $n,p\rightarrow\infty$, we have that almost surely
        \begin{align*}
            |P_{\lambda}| &= \phi|\tr(  \bM_2 \hSigma_0 \bM_1 \bSigma)/p| \leq \phi \|\bM_1\hSigma_0\bM_2\|_{\oper}\|\bSigma\|_{\oper} \leq \tfrac{\phi_s^2r_{\max}}{\phi},
        \end{align*}
        where the last inequality is because $\norm{\hSigma_0}_{\oper}\leq r_{\max}$, and
        \begin{align}
             \norm{\bM_1 \hSigma_0 \bM_2}_{\oper}\leq \tfrac{k^2}{i_0^2} \cdot \max\limits_{i}\tfrac{l_i}{\left(l_i+ \tfrac{k-i_0}{i_0}\lambda\right)^2} \leq \tfrac{k^2}{i_0^2}, \label{eq:lem-ridgeless-V0-term-1}
         \end{align}
        where $l_i$'s are the eigenvalues of $\hSigma_0$. Similarly, we have $|P_0|$ is almost surely bounded.
        Thus, $|P_{\lambda}|$ is almost surely bounded over $\lambda\in\Lambda[0,\lambda_{\max}]$ for some constant $\lambda_{\max}>0$.
        Next we consider the derivative
        \begin{align*}
            \tfrac{\partial P_{\lambda}}{\partial\lambda} &= \bepsilon^{\top}\bL_1\bX\tfrac{\partial \bM_1}{\partial\lambda} \bSigma \bM_2\bX^{\top}\bL_2\bepsilon /k^2 + \bepsilon^{\top}\bL_1\bX\bM_1 \bSigma \tfrac{\partial \bM_2}{\partial\lambda}\bX^{\top}\bL_2\bepsilon /k^2\\
            &= -\bepsilon^{\top}\bL_1\bX\bM_1^2\bSigma \bM_2\bX^{\top}\bL_2\bepsilon /k^2 - \bepsilon^{\top}\bL_1\bX\bM_1 \bSigma \bM_2^2\bX^{\top}\bL_2\bepsilon /k^2
        \end{align*}
        Note that for $\lambda\in\Lambda$,
        we can bound
        \begin{align*}
             \norm{\bM_1^2 \hSigma_0 \bM_2}_{\oper}\leq \tfrac{k^2}{i_0^2} \cdot \max\limits_{i}\tfrac{l_i}{\left(l_i+ \tfrac{k-i_0}{i_0}\lambda\right)^3} \leq \tfrac{k^2}{i_0^2},
         \end{align*}
        where $l_i$'s are the eigenvalues of $\hSigma_0$. Similarly, we have that $\norm{\bM_1 \hSigma_0 \bM_2^2}_{\oper}$ is almost surely bounded for $\lambda\in\Lambda.$
        By similar argument as in \Cref{lem:ridgeless-conv-V0}, the following holds almost surely as $n,p\rightarrow\infty$,
        \begin{align*}
            \left|\tfrac{\partial P_{\lambda}}{\partial\lambda}\right| & = \phi |\tr(\bM_1^2\hSigma_0 \bM_2 \bSigma) + \tr(\bM_1\hSigma_0 \bM_2^2 \bSigma)| \leq \tfrac{\phi_s^2r_{\max}}{\phi}.
        \end{align*}
        That is, $|\partial P_{\lambda}/\partial \lambda|$ is almost surely bounded over $\lambda\in\Lambda[0,\lambda_{\max}]$.
        
        Since $\phi_0\geq\phi_s>1$, it follows from \citet[Lemma S.6.14]{patil2022mitigating} that, there exists $M'>0$ such that  the magnitudes of $v(-\lambda; \phi_s)$ and $ v_v(\lambda,\phi_s,\phi)/\phi$, and their derivatives with respect to $\lambda$ are continuous and bounded by $M'$ for all $\lambda\in\Lambda$.
        Thus, $|Q_{\lambda}|\leq \phi_0 M' r_{\max}^2$ over $\lambda\in\Lambda$.
        Similarly, we have $|\partial Q_{\lambda}/\partial\lambda|_{\lambda=0^+}|$ are uniformly bounded over $\lambda\in\Lambda$.
        We can now use the Moore-Osgood theorem and the continuity property from \Cref{lem:fixed-point-v-lambda-properties} to interchange the limits to obtain
        \begin{align*}
                &\lim\limits_{p\rightarrow\infty} P_0 - Q_0
                = \lim\limits_{p\rightarrow\infty}\lim\limits_{\lambda\rightarrow0^+}P_{\lambda}-Q_{\lambda} =\lim\limits_{\lambda\rightarrow0^+}\lim\limits_{p\rightarrow\infty}P_{\lambda}-Q_{\lambda} =0,
            \end{align*}
        and the conclusion follows.
        \end{proof}

\subsection{Boundary case: diverging subsample aspect ratio}
    \begin{proposition}[Risk approximation when $\phi_s\rightarrow\infty$]\label{prop:Rdet-ridgeless-infinity}
    Under Assumptions \ref{asm:rmt-feat}-\ref{asm:spectrum-spectrumsignproj-conv}, it holds for all $M\in\NN$ $$R(\tfWR{0}{M};\cD_n,\{I_\ell\}_{\ell=1}^M) \asto \RzeroMe{M}{\phi}{\infty},$$
    as $k,n,p\rightarrow\infty$, $p/n\rightarrow\phi\in(0,\infty)$ and $p/k\rightarrow\infty$, where
    \begin{align}
        \RzeroMe{M}{\phi}{\infty}:=\lim\limits_{\phi_s\rightarrow\infty}\RzeroM{M}{\phi} = \sigma^2 + \rho^2\int r\,\rd G(r), \label{eq:Rdet-ridgeless-infinity}
    \end{align}
    and $\RzeroM{M}{\phi}$ is as defined in \Cref{thm:ver-ridgeless}.
    \end{proposition}
    \begin{proof}[Proof of \Cref{prop:Rdet-ridgeless-infinity}]
        Note that 
        \begin{align*}
            R(\tfWR{0}{M};\cD_n,\{I_\ell\}_{\ell=1}^M) &= \EE_{(\bx_0,y_0)}[(y_0-\bx_0^{\top}\tbetamntwo{M})^2] \\
            &= \EE_{(\bx_0,y_0)}[(\bepsilon_0
            +\bx_0^{\top}(\bbeta_0-\tbetamntwo{M}))^2] \\
            &= \sigma^2 + \EE_{(\bx_0,y_0)}[(\bbeta_0-\tbetamntwo{M})^{\top}\bx_0\bx_0^{\top}(\bbeta_0-\tbetamntwo{M}) ] \\
            &= \sigma^2 + (\bbeta_0-\tbetamntwo{M})^{\top}\bSigma(\bbeta_0-\tbetamntwo{M}).
        \end{align*}
        Then, by the Cauchy-Schwarz inequality, we have
        \begin{align*}
            R(\tfWR{0}{M};\cD_n,\{I_\ell\}_{\ell=1}^M) - (\bbeta_0^{\top}\bSigma\bbeta_0 + \sigma^2) &= \|\bSigma^{\tfrac{1}{2}}\tbetamntwo{M}\|_2^2 - 2 \tbetamntwo{M}^{\top}\bSigma\bbeta_0\\
            &\leq \tfrac{1}{r_{\min}} \|\tbetamntwo{M}\|_2^2 + 2 \|\tbetamntwo{M}\|_2 \|\bSigma\|_2 \\
            &\leq \tfrac{1}{r_{\min}} \|\tbetamntwo{M}\|_2^2 + 2 r_{\max}\rho \|\tbetamntwo{M}\|_2 
        \end{align*}
        almost surely as $k,n,p\rightarrow$ and $p/k\rightarrow\infty$.
        Thus, we have the following holds almost surely:
        \begin{align*}
            \|\tbeta^{0}_M(\cD_n)\|_2 &\leq \tfrac{1}{M}\sum_{m=1}^M \|(\bX^{\top}\bL_m\bX/k)^{+}(\bX^{\top}\bL_m\by/k)\|_2 \\
            &\leq \tfrac{1}{M}\sum_{m=1}^M \|(\bX^{\top}\bL_m\bX/k)^{+}\bX^{\top}\bL_m/\sqrt{k}\|\cdot\|\bL_m\by/\sqrt{k}\|_2\\
            &\leq C\sqrt{\rho^2+\sigma^2}\cdot \tfrac{1}{M}\sum_{m=1}^M\|(\bX^{\top}\bL_m\bX/k)^{+}\bX^{\top}\bL_m/\sqrt{k}\|
        \end{align*}
        where the last inequality holds eventually almost surely since Assumptions \ref{asm:rmt-feat}-\ref{asm:signal-bounded-norm} imply that the entries of $\by$ have bounded 4-th moment, and thus from the strong law of large numbers, $\| \bL_m\by / \sqrt{k} \|_2$ is eventually almost surely
        bounded above by $C\sqrt{\EE[y_1^2]} = C\sqrt{\rho^2 + \sigma^2}$ for some constant $C$.
        Observe that operator norm of the matrix $(\bX^\top \bL_m\bX / k)^{+} \bX\bL_m / \sqrt{k}$ is upper bounded $1/s_{\min}$, where $s_{\min}$ is the smallest nonzero singular value of $\bX$.
        As $k, p \to \infty$ such that $p / k \to \infty$, $s_{\min} \to \infty$ almost surely (e.g., from results in \cite{bloemendal2016principal}), and therefore, $\| \tbetamntwo{M} \|_2 \to 0$ almost surely.
        Thus, we have shown that
        $$R(\tfWR{0}{M};\cD_n,\{I_\ell\}_{\ell=1}^M) \asto \sigma^2+\bbeta_0^{\top}\bSigma\bbeta_0,$$
        or equivalently, 
        $$R(\tfWR{0}{M};\cD_n,\{I_\ell\}_{\ell=1}^M) \asto \sigma^2+ \rho^2\int r\,\rd G(r).$$
        
        From \Cref{lem:fixed-point-v-properties} we have $$\lim_{\phi_s\rightarrow\infty}v(0;\phi_s) = \lim_{\phi_s\rightarrow\infty}\tv_b(0;\phi_s)=\lim_{\phi_s\rightarrow\infty}\tv_v(0;\phi_s).$$ Thus, $$\lim_{\phi_s\rightarrow\infty}V_{0}(\phi_s,\phi_s)=\lim_{\phi_s\rightarrow\infty}V_{0}(\phi,\phi_s) = 0,$$
        and $$\lim_{\phi_s\rightarrow\infty}B_{0}(\phi_s,\phi_s)=\lim_{\phi_s\rightarrow\infty}B_{0}(\phi,\phi_s) = \rho^2\int r\,\rd G(r).$$
        Therefore, we have $\sR^{\textup{\texttt{WR}}}_{0,M}(\phi,\infty):=\lim_{\phi_s\rightarrow\infty}\RzeroM{M}{\phi} = \sigma^2 + \rho^2\int r\,\rd G(r)$.
        Thus, $\sR^{\textup{\texttt{WR}}}_{0,M}(\phi,\infty)$ is well defined and $\sR^{\textup{\texttt{WR}}}_{0,M}(\phi,\phi_s)$ is right continuous at $\phi_s=\infty$.
    \end{proof}
    
\section
[Proof of \Cref{thm:ver-without-replacement} (splagging without replacement, ridge, ridgeless predictors)]
{Proof of \Cref{thm:ver-without-replacement} (splagging without replacement, ridge and ridgeless predictors)}
\label{sec:appendix-without-replacement}

    \begin{proof}[Proof of \Cref{thm:ver-without-replacement}]
        For $M\in\{1,2,\ldots,\lfloor\liminf n/k \rfloor\}$, following the proof in \Cref{thm:ver-ridge}, the conditional risk is given by
        \begin{align*}
            R(\tfWOR{\lambda}{M};\{\cD_{I_{\ell}}\}_{\ell=1}^M) &= \sigma^2 + T_C + T_B + T_V,
        \end{align*}
        where $T_C$, $T_B$, and $T_V$ are defined as
        \begin{align}
            T_C&= -\tfrac{\lambda}{M}\cdot \bepsilon^{\top}\left(\sum_{m=1}^M\bM_m\tfrac{\bX^{\top}\bL_m}{k}\right)^{\top} \bSigma \left(\sum_{m=1}^M\bM_m\right) \bbeta_0, \label{eq:ridge-TC-without-replace} \\
            T_B &= \tfrac{\lambda^2}{M^2}\cdot \bbeta_0^{\top}\left(\sum\limits_{i=1}^M\bM_{I_i}\right)\bSigma\left(\sum\limits_{i=1}^M\bM_{I_i}\right)\bbeta_0 \label{eq:ridge-TB-without-replace}\\
            &= \tfrac{\lambda^2}{M}\sum\limits_{i=1}^M \bbeta_0^{\top}\bM_{I_i}\bSigma\bM_{I_i}\bbeta_0 + \tfrac{\lambda^2(M-1)}{M}\sum\limits_{i,j=1}^M \bbeta_0^{\top}\bM_{I_i}\bSigma\bM_{I_j}\bbeta_0, \notag\\
            T_V&= \tfrac{1}{M^2}\cdot \bepsilon^{\top}\left(\sum\limits_{i=1}^M\bM_{I_i}\tfrac{\bX^{\top}\bL_i}{k}\right)
            ^{\top} \bSigma \left(\sum\limits_{i=1}^M\bM_{I_i}\tfrac{\bX^{\top}\bL_i}{k}\right) \bepsilon \notag\\
            &= \tfrac{1}{M}\sum\limits_{i=1}^M\left(\bM_{I_i}\tfrac{\bX^{\top}\bL_i}{k}\right)^{\top}\bSigma\left(\bM_{I_i}\tfrac{\bX^{\top}\bL_i}{k}\right)+ \tfrac{M-1}{M}\sum\limits_{i,j=1}^M\left(\bM_{I_i}\tfrac{\bX^{\top}\bL_i}{k}\right)^{\top}\bSigma\left(\bM_{I_j}\tfrac{\bX^{\top}\bL_j}{k}\right),\notag\\
        \end{align}
        where $\bM_{I_\ell}=(\bX^{\top}\bL_{\ell}\bX/k + \lambda\bI_p)^{-1}$ and $\bL_{\ell}$ is a diagonal matrix with diagonal entry being $1$ if the $\ell$th sample $X_{\ell}$ is in the sub-sampled dataset $\cD_{I_\ell}$ and 0 otherwise.
        Note that for \splagging, $I_i\cap I_j=\varnothing$ for all $i\neq j$.
        
        We analyze each term separately for $M\in\{1,2\}$. From \Cref{lem:ridge-conv-C0}, we have that $T_C\asto0$.
        From \Cref{lem:ridge-conv-V0}, we have that
        \begin{align}
            T_V-\tfrac{1}{M}\sum\limits_{j=1}^M\tfrac{\sigma^2}{k}\tr(  \bM_{I_j}\hSigma_j\bM_{I_j}\bSigma) \asto 0,\label{eq:splaaging-1}
        \end{align}
        since the datasets have no overlaps and the cross term vanishes because $\bL_l\bL_m=\zero_{n\times n}$ for $l\neq m$.
        Then, from \eqref{eq:lem-det-approx-B0-1} and \eqref{eq:lem:ridge-V0-1}, we have that for $\ell\in[M]$,
        \begin{align}
            \lambda^2\bbeta_0^{\top}\bM_{I_{\ell}}\bSigma\bM_{I_{\ell}}\bbeta_0 &\asto \rho^2 \tv(-\lambda;\phi_s,\phi_s)\tc(-\lambda;\phi_s), \label{eq:splaaging-2}\\
            \tfrac{\sigma^2}{k}\tr(  \bM_{I_{\ell}}\hSigma_{\ell}\bM_{I_{\ell}}\bSigma) &\asto\tfrac{\sigma^2}{2}\tv(-\lambda;\phi_s,\phi_s), \label{eq:splaaging-4}
        \end{align}
        as $n,k,p\rightarrow\infty$, $p/n\rightarrow\phi\in(0,\infty)$, and $p/k\rightarrow\phi_s=2\phi$, where the positive constants $\tv(\lambda  ;\phi_s,\phi)$, and $\tc(-\lambda;\phi_s)$ are as defined in \eqref{eq:tv_tc_ridge}.
        For the cross term ($i\neq j$), setting $i_0=0$ in \eqref{eq:ridge-B0-cross-term} yields that
        \begin{align*}
            \bM_{I_i}\bSigma\bM_{I_j} & \asympequi  (v(-\lambda; \phi_s)\bSigma+\bI_p)^{-1}\bSigma(v(-\lambda; \phi_s)\bSigma+\bI_p)^{-1}.
        \end{align*}
        Thus,
        \begin{align}
            \lambda^2\bbeta_0^{\top}\bM_{I_i}\bSigma\bM_{I_j}\bbeta_0 &\asto \rho^2\int \tfrac{r}{(1 + v(-\lambda; \phi_s)r)^2} \,\rd G(r) = \rho^2\tc(0;\phi_s). \label{eq:splaaging-3}
        \end{align}
        Combining \eqref{eq:ridge-TC-without-replace}-\eqref{eq:splaaging-3}, we have shown that $R(\tfWOR{\lambda}{M};\{\cD_{I_\ell}\}_{\ell=1}^M) \asto \bRlamM{M}{\phi}$, where
        \begin{align*}
            \bRlamM{M}{\phi} &= \sigma^2 +
            \bBlamM{M}{\phi} +  \bVlamM{M}{\phi},
        \end{align*}
        and the components are:
        \begin{align*}
            \bBlamM{M}{\phi} = \frac{1}{M}B_{\lambda}(\phi_s,\phi_s) + \left(1-\frac{1}{M}\right)C_{\lambda}(\phi_s),\qquad \bVlamM{M}{\phi} \frac{1}{M}V_{\lambda}(\phi_s,\phi_s),
        \end{align*}
        with $B_{\lambda}(\phi,\phi_s) = \rho^2 (1+\tv(-\lambda;\phi,\phi_s)) \tc(-\lambda;\phi_s)$, $ C_{\lambda}(\phi_s)= \rho^2 \tc(-\lambda;\phi_s)$, and $V_{\lambda}(\phi,\phi_s) = \sigma^2 \tv(-\lambda;\phi,\phi_s)$.

        From \Cref{prop:Rdet-ridge-infinity} and \Cref{prop:Rdet-ridgeless-infinity}, we have that for all $\lambda\in[0,\infty)$ and $M\in\{1,2\}$, $$\lim_{\phi_s\rightarrow+\infty}R(\tfWOR{\lambda}{M};\{\cD_{k}^{(m)}\}_{m=1}^M)= \sigma^2 +\rho^2\int r\,\rd G(r),$$ $\lim_{\phi_s\rightarrow+\infty}B_{\lambda}(\phi_s,\phi_s)=\rho^2\int r\,\rd G(r)$ and $\lim_{\phi_s\rightarrow+\infty}v(-\lambda;\phi_s)=\lim_{\phi_s\rightarrow+\infty}V_{\lambda}(\phi_s,\phi_s)=0$.
        Then $$\lim_{\phi_s\rightarrow+\infty}\tc(-\lambda;\phi_s)= \lim_{\phi_s\rightarrow+\infty}\int r(1 +v(-\lambda;\phi_s) r )^{-2}\,\rd G(r) = \int r\,\rd G(r).$$
        Thus, the approximation holds when $\phi_s=\infty$: $\lim_{\phi_s\rightarrow+\infty}R(\tfWOR{\lambda}{M};\{\cD_{I_{\ell}}\}_{\ell=1}^M)=\lim_{\phi_s\rightarrow+\infty}\bRlamM{M}{\phi}$.
        
        Finally, the risk expression for general $M$ and the uniform statement for all $M\leq \lfloor n/k\rfloor$ follow from \Cref{thm:risk_general_predictor}.
    \end{proof}
    
\section{Proofs related to bagged risk properties}\label{sec:appendix-risk-properties}

\subsection[Proof of \Cref{prop:monotonicity-M-with-replacement}
]
{Bias-variance monotonicities in the number of bags, subagging with replacement}

\begin{proof}[Proof of \Cref{prop:monotonicity-M-with-replacement}]
    Recall that from the proof for \Cref{thm:ver-ridge}, we have
    \begin{align*}
        \BlamM{1}{\phi} &= \rho^2(1+\tv(-\lambda, \phi_s,\phi_s)) \tc(-\lambda; \phi_s)&&
        \VlamM{1}{\phi} = \sigma^2\tv(-\lambda; \phi_s,\phi_s)\\
        \BlamM{\infty}{\phi} &= \rho^2 (1+\tv(-\lambda, \phi,\phi_s)) \tc(-\lambda; \phi_s) &&
        \BlamM{\infty}{\phi} = \sigma^2\tv(-\lambda;\phi,\phi_s)
    \end{align*}
    where the nonnegative constants $\tv(-\lambda, \phi,\phi_s)$ and $\tc(-\lambda; \phi_s)$ are defined in \eqref{eq:v_ridge}.
    Since $H$ has positive support, $\tv(-\lambda;\phi,\phi_s)$ is strictly increasing in $\phi$, and thus, $\BlamM{\infty}{\phi}=\BlamM{1}{\phi}$ when $\phi_s= \phi$, and $\BlamM{1}{\phi}> \BlamM{\infty}{\phi}$ when $\phi_s> \phi$.
    Similarly, $\VlamM{\infty}{\phi}=\VlamM{1}{\phi}$ when $\phi_s= \phi$ and $\VlamM{1}{\phi}<\VlamM{\infty}{\phi}$ when $\phi_s> \phi$.
    Recall that the definitions of $\BlamM{M}{\phi} = 1/ M \cdot B_{\lambda}(\phi_s,\phi_s) + (1-1/M) B_{\lambda}(\phi,\phi_s)$ and $\VlamM{M}{\phi} = 1 / M \cdot V_{\lambda}(\phi_s,\phi_s) + (1 - 1 / M) V_{\lambda}(\phi,\phi_s)$ are a convex combination of $B_{\lambda}(\phi,\phi_s)$ and $B_{\lambda}(\phi_s,\phi_s)$, and $V_{\lambda}(\phi,\phi_s)$ and $V_{\lambda}(\phi_s,\phi_s)$, respectively.
    The proof for ridgeless predictor follows by setting $\lambda=0$ except $B_{0}(\phi,\phi_s)=B_{0}(\phi,\phi_s)=0$ for $\phi_s<1$.
\end{proof}

\subsection[Proof of Proposition \ref{prop:monotonicity-M-without-replacement}
]
{Bias-variance monotonicities in the number of bags, splagging without replacement}
    \begin{proof}[Proof of \Cref{prop:monotonicity-M-without-replacement}]
        For the variance term, $\bVlamM{M}{\phi}= M^{-1}V_{\lambda}(\phi_s,\phi_s)$ as a linear function of $M^{-1}$ is strictly decreasing in $M$ if $\phi_s<\infty$ and is zero if $\phi_s=\infty$ or $\sigma^2=0$.
    
        For the bias term, when $\phi_s>1$, since $\tc(-\lambda;\phi_s)>0$, we have that $B_{\lambda}(\phi_s,\phi_s)\geq C_{\lambda}(\phi_s)$ with equality holds if and only if $\tv(-\lambda;\phi,\phi_s)=0$ or $\tc(-\lambda;\phi_s)=0$, if and only if $\phi_s=\infty$.
        Then we have
        \begin{align*}
            \bBlamM{M}{\phi} &= \frac{1}{M}B_{\lambda}(\phi_s,\phi_s)+\left(1-\frac{1}{M}\right) C_{\lambda}(\phi_s) \\
            &= \frac{1}{M}(B_{\lambda}(\phi_s,\phi_s)-C_{\lambda}(\phi_s))+ C_{\lambda}(\phi_s)\\
            &\geq \frac{1}{M+1}(B_{\lambda}(\phi_s,\phi_s)-C_{\lambda}(\phi_s))+ C_{\lambda}(\phi_s)\\
            &= \frac{1}{M+1}B_{\lambda}(\phi_s,\phi_s)+\left(1-\frac{1}{M+1}\right) C_{\lambda}(\phi_s) \\
            &= \bBlamM{M+1}{\phi}.
        \end{align*}
        with equality holds if $\phi_s=\infty$ or $\rho^2=0$.
        When $\phi_s<1$, $B_{\lambda}(\phi_s,\phi_s)\geq C_{\lambda}(\phi_s)$ with equality holds if and only if $\tc(-\lambda;\phi_s)=0$, if and only if $\lambda=0$.
        The monotonicity of $\bVlamM{M}{\phi}$ in $M$ follows analogously.
        
        As $M\leq \phi_s/\phi$, we further have $ \bVlamM{M}{\phi}\geq \bVlamM{\phi_s/\phi}{\phi}$ and $\bVlamM{M}{\phi}\geq \bVlamM{\phi_s/\phi}{\phi}$ for all $M=1,\ldots,\lfloor \liminf n/k\rfloor$.
    \end{proof}

\subsection[Proof of Theorem \ref{thm:cv_general}
]
{Risk monotonization of general bagged predictors by cross-validation}
\begin{proof}[Proof of \Cref{thm:cv_general}]
    We present the proof for bagging with replacement, and the proof for bagging without replacement follows by restricting the support of $\phi_s\mapsto \mathscr{R}_M(\phi,\phi_s)$ to $[M\phi,\infty]$.
    From \Cref{thm:risk_general_predictor}, we have that for any $M\in\NN$ and $\{I_{\ell}\}_{\ell=1}^M$ simple random samples from $\cI_k$ or $\cI_k^{\pi}$, it holds that
    $$R(\tf_{M};\cD_n,\{I_{\ell}\}_{\ell=1}^M) \pto \mathscr{R}_M(\phi,\phi_s)$$
    as $k,n,p\rightarrow$, $p/n\rightarrow\phi\in(0,\infty)$, and $p/k\rightarrow \phi_s\in[\phi,\infty)$, where
    \begin{align*}
        \mathscr{R}_M(\phi,\phi_s) := (2\mathscr{R}(\phi,\phi_s) - \mathscr{R}(\phi_s,\phi_s)) + \frac{2}{M}(\mathscr{R}(\phi_s,\phi_s) - \mathscr{R}(\phi,\phi_s)).
    \end{align*}
    From \citet[Lemma 3.8 and Theorem 3.4]{patil2022mitigating}, we have that
    \begin{align*}
        \left(R(\hf_{M,\cI_{\hat{k}}}^\cv; \cD_n)-\mathscr{R}_M(\phi,\phi_s)\right)_{+}
        \pto 0.
    \end{align*}
    In \cite{patil2022mitigating},
    we have assumed that the risk is bounded away from 0
    in order to conclude that
    the relative error converges to 0.
    But in \Cref{thm:cv_general},
    we conclude only the positive part
    of the absolute error converges to 0,
    for which we do not require the risk to be bounded away from 0.

    Since $\mathscr{R}_M(\phi,\phi_s)$ is increasing in $\phi$ for any fixed $\phi_s$.
    For $0<\phi_1\leq \phi_2<\infty$,
    \begin{align*}
        \min_{\phi_s \ge \phi_1}
        \mathscr{R}_M(\phi_1,\phi_s)
        &\le
        \min_{\phi_s \ge \phi_2}
        \mathscr{R}_M(\phi_1,\phi_s) \le
        \min_{\phi_s \ge \phi_2}
        \mathscr{R}_M(\phi_2,\phi_s)
    \end{align*}
    where the first inequality follows because $\{ \phi_s: \phi_s \ge \phi_1 \} \supseteq \{ \phi_s : \phi_s \ge \phi_2 \}$,
    and the second inequality follows because $\mathscr{R}_M(\phi,\phi_s)$ is increasing in $\phi$ for a fixed $\phi_s$.
    Thus, $\min_{\phi_s \ge \phi}\mathscr{R}_M(\phi,\phi_s)$ is a monotonically increasing function in $\phi$.
\end{proof}

\subsection[Proof of \Cref{thm:monotonicity-phi}
]
{Risk monotonization of ridge bagged predictors by cross-validation}

\begin{proof}[Proof of \Cref{thm:monotonicity-phi}]
    It suffices to verify the two conditions (i) and (ii) in \Cref{thm:cv_general}.
    From \Cref{thm:ver-with-replacement} and \Cref{thm:ver-without-replacement}, condition (i) holds naturally with $\mathscr{R}_{M}(\phi,\phi_s)$ being the limiting risk $\RlamM{M}{\phi}$ (or $\bRlamM{M}{\phi}$) for fixed $\lambda\geq 0$.
    For condition (ii), note that when $\lambda>0$, $\mathscr{R}_{M}(\phi,\phi_s)$ is continuous over $[\phi,\infty]$. 
    When $\lambda=0$, $\mathscr{R}_{M}(\phi,\phi_s)$ is continuous over $[\phi,\infty]\setminus \{1\}$ and can takes value infinity when $\phi_s$ tends to 1 from both sides.
    Thus, $\mathscr{R}_{M}(\phi,\phi_s)$ is lower semi-continuous over $[\phi,\infty]$ and continuous on the set $\argmin_{\psi:\psi\geq\phi}\mathscr{R}_{M}(\phi,\psi)\subseteq[\phi,\infty]\setminus \{1\}$.

    Following \Cref{rem:uniform-consistency-estimated-risk},
    the uniform risk closeness condition for $k\in\cK_n$ holds.
    Then by \Cref{thm:cv_general}, we have that 
    \begin{align*}
        \left(R(\hf_{M}^\cv;\cD_n,\{I_{\hat{k},\ell}\}_{\ell=1}^M)-\min_{\phi_s\geq\phi}\RlamM{M}{\phi}\right)_+ \pto 0.
    \end{align*}
    
    Recall that for any fixed $\theta$, the function
    \begin{align*}
        \tv(-\lambda;\vartheta,\theta) = \tfrac{\vartheta \int{r^2}{(1+ v(-\lambda; \theta)r)^{-2}}\,\rd H(r)}{v(-\lambda; \phi_s)^{-2}-\vartheta \int{r^2}{(1+ v(-\lambda; \theta)r)^{-2}}\,\rd H(r)} \geq 0
    \end{align*}
    is increasing in $\vartheta$.
    Then $\RlamM{M}{\phi}$ as a function of $\tv(-\lambda;\vartheta,\theta)$ through \eqref{eq:risk-det-with-replacement} and \eqref{eq:Blam_V_lam} is also increasing in $\phi$ for any fixed $\phi_s$.
    For $0<\phi_1\leq \phi_2<\infty$,
    \begin{align*}
        \min_{\phi_s \ge \phi_1}
        \RlamM{M}{\phi_1}
        &\le
        \min_{\phi_s \ge \phi_2}
        \RlamM{M}{\phi_1} \le
        \min_{\phi_s \ge \phi_2}
        \RlamM{M}{\phi_2}
    \end{align*}
    where the first inequality follows because $\{ \phi_s: \phi_s \ge \phi_1 \} \supseteq \{ \phi_s : \phi_s \ge \phi_2 \}$,
    and the second inequality follows because $\RlamM{M}{\phi}$ is increasing in $\phi$ for a fixed $\phi_s$.
    Thus, $\min_{\phi_s \ge \phi}\RlamM{M}{\phi}$ is a monotonically increasing function in $\phi$.
\end{proof}
    
\subsection
[Proof of Proposition \ref{prop:improve-with-without-replace}
]
{Optimal subagging versus optimal splagging}

    \begin{proof}[Proof of \Cref{prop:improve-with-without-replace}]
        Recall that
        \begin{align*}
            \RlamM{M}{\phi} &=
            \tfrac{1}{M}(B_{\lambda}(\phi_s,\phi_s) + V_{\lambda}(\phi_s,\phi_s)) +\left(1-\tfrac{1}{M}\right)(B_{\lambda}(\phi,\phi_s) + V_{\lambda}(\phi,\phi_s)),\qquad M\in\NN\\
            \bRlamM{M}{\phi} &=\tfrac{1}{M}(B_{\lambda}(\phi_s,\phi_s) + V_{\lambda}(\phi_s,\phi_s))
            + \left(1 - \tfrac{1}{M} \right)C_{\lambda}( \phi_s),\qquad M=1,\ldots,\lfloor\tfrac{n}{k}\rfloor.
        \end{align*}
        From \Cref{prop:monotonicity-M-with-replacement}, we have that
        \begin{align}
            \RlamM{M}{\phi} &\geq \RlamM{\infty}{\phi} \notag\\
            &= B_{\lambda}(\phi,\phi_s) + V_{\lambda}(\phi,\phi_s)\notag\\
            &=\rho^2(1+\tv(-\lambda;\phi,\phi_s))\tc(-\lambda;\phi_s) + \sigma^2\tv(-\lambda;\phi,\phi_s) \notag\\
            &= \rho^2
            \tc(-\lambda;\phi_s) + \tv(-\lambda;\phi,\phi_s)) ( \rho^2 \tc(-\lambda;\phi_s) + \sigma^2).\label{eq:improve-with-without-replace-ridge-1}
        \end{align}
        where $\tc(-\lambda;\phi_s)=\int r/{(1 + v(-\lambda; \phi_s)r)^2} \,\rd G(r) $.
        From \Cref{prop:monotonicity-M-without-replacement}, we have that for $M\in\NN$,
        \begin{align}
            \bRlamM{M}{\phi} &\geq \bRlamM{\phi_s/\phi}{\phi} = \tfrac{\phi}{\phi_s}
            (B_{\lambda}(\phi_s,\phi_s) + V_{\lambda}(\phi_s,\phi_s))
            + \left(1 - \tfrac{\phi}{\phi_s} \right)C_{\lambda}( \phi_s). \label{eq:improve-with-without-replace-ridge-2}
        \end{align}
        On the other hand,
        \begin{align}
            \bRlamM{\phi_s/\phi}{\phi} &= \frac{\phi}{\phi_s}\rho^2(1+\tv(-\lambda;\phi_s,\phi_s))\tc(-\lambda;\phi_s) + \frac{\phi}{\phi_s}\sigma^2\tv(-\lambda;\phi_s,\phi_s))+\left(1-\frac{\phi}{\phi_s}\right)\rho^2\tc(-\lambda;\phi_s)\notag\\
            &= \rho^2\tc(-\lambda;\phi_s) + \tfrac{\phi}{\phi_s}\tv(-\lambda;\phi_s,\phi_s)) ( \rho^2 \tc(-\lambda;\phi_s) + \sigma^2).\label{eq:improve-with-without-replace-ridge-3}
        \end{align}
        Since $v(-\lambda;\phi_s)$ is strictly decreasing in $\phi_s$ from \Cref{lem:ridge-fixed-point-v-properties} and $G$ has nonnegative support from Assumption \ref{asm:spectrum-spectrumsignproj-conv}, we have that $\tc(-\lambda;\phi_s)$ is nonnegative and increasing in $\phi_s$.
        Also note that
        \begin{align}
            \tfrac{\phi}{\phi_s}\tv(-\lambda;\phi_s,\phi_s)) &= \ddfrac{\phi\int\tfrac{r^2}{(1+ v(-\lambda; \phi_s)r)^2}\,\rd H(r)}{v(-\lambda; \phi_s)^{-2}-\phi_s \int\tfrac{r^2}{(1+ v(-\lambda; \phi_s)r)^2}\,\rd H(r)}\notag \\
            &\geq \ddfrac{\phi\int\tfrac{r^2}{(1+ v(-\lambda; \phi_s)r)^2}\,\rd H(r)}{v(-\lambda; \phi_s)^{-2}-\phi \int\tfrac{r^2}{(1+ v(-\lambda; \phi_s)r)^2}\,\rd H(r)}\notag\\
            &= \tv(-\lambda;\phi,\phi_s). \label{eq:improve-with-without-replace-ridge-4}
        \end{align}
        Suppose that $\phi^*\in\argmin_{\inf_{\phi_s\in[\phi,\infty]}}\bRlamM{M}{\phi}$, we have
        \begin{align*}
             \inf_{M\in\NN,\phi_s\in[\phi,\infty]}\RlamM{M}{\phi} &= \inf_{\phi_s\in[\phi,\infty]}\RlamM{\infty}{\phi} \\
             &\leq \RlamMe{\infty}{\phi}{\phi^*}\\
             &=\rho^2
            \tc(-\lambda;\phi_s^*) + \tv(-\lambda;\phi_s^*,\phi)) ( \rho^2 \tc(-\lambda;\phi_s^*) + \sigma^2) \\
            &\leq \rho^2\tc(-\lambda;\phi_s) + \tfrac{\phi}{\phi_s}\tv(-\lambda;\phi_s,\phi_s)) ( \rho^2 \tc(-\lambda;\phi_s) + \sigma^2)\\ 
             &= \bRlamMe{M}{\phi_s}{\phi_s^*}  \\
             &= \inf_{\phi_s\in[\phi,\infty]}\bRlamM{\phi_s/\phi}{\phi} 
             \\
             &\leq \inf_{M\in\NN,\phi_s\in[\phi,\infty]}\bRlamM{M}{\phi} 
        \end{align*}
        where in the second inequality we use \eqref{eq:improve-with-without-replace-ridge-4} and the last inequality is from \eqref{eq:improve-with-without-replace-ridge-2}.
    \end{proof}

\subsection[Proof of \Cref{prop:opt-risk-ridgeless}
]
{Optimal bag size for ridgeless predictors}

    \begin{proof}[Proof of \Cref{prop:opt-risk-ridgeless}]
        The proof of \Cref{prop:opt-risk-ridgeless} follows by combining results from \Cref{lem:opt-risk-subagged-ridgeless} and \Cref{lem:opt-risk-splagged-ridgeless} for subagged and \splagged ridgeless predictors, respectively.
    \end{proof}

    \begin{lemma}[Optimal risk for subagged ridgeless predictor]\label{lem:opt-risk-subagged-ridgeless}
        Suppose the conditions in \Cref{thm:ver-with-replacement} hold, and $\sigma^2,\rho^2 \ge 0$ are the noise variance and signal strength from Assumptions \ref{asm:lin-mod} and \ref{asm:signal-bounded-norm}. Let $\SNR=\rho^2/\sigma^2$. For any $\phi\in(0,\infty)$, the properties of the optimal asymptotic risk $\mathscr{R}_{0,\infty}^{\sub}(\phi, \phi_s^{\sub}(\phi))$ in terms of $\SNR$ and $\phi$ are characterized as follows:
        \begin{enumerate}[label={(\arabic*)},leftmargin=7mm]
            \item $\SNR=0 \; (\rho^2 = 0, \sigma^2 \neq 0)$:  For all $\phi \ge 0$, the global minimum $\mathscr{R}_{0,\infty}^{\sub}(\phi, \phi_s^{\sub}(\phi)) = \sigma^2$ is obtained with $\phi_s^\sub(\phi)=\infty$.
            
            \item $\SNR>0$: For all $\phi \ge 0$, the global minimum of $\phi_s \mapsto \RzeroM{\infty}{\phi}$ is obtained at $\phi_s^\sub(\phi)\in(1,\infty)$.
            
            \item $\SNR=\infty \; (\rho^2 \neq 0, \sigma^2 = 0)$: If $\phi\in(0,1]$, the global minimum $\mathscr{R}_{0,\infty}^{\sub}(\phi, \phi_s^{\sub}(\phi)) = 0$ is obtained with any $\phi_s^\sub(\phi)\in[\phi,1]$.
            If $\phi\in(1,\infty)$, then the global minimum $\mathscr{R}_{0,\infty}^{\sub}(\phi, \phi_s^{\sub}(\phi))$ is obtained at $\phi_s^\sub(\phi)\in[\phi,\infty)$.
        \end{enumerate}
    \end{lemma}
    \begin{proof}[Proof of \Cref{lem:opt-risk-subagged-ridgeless}]
        From \Cref{thm:ver-with-replacement}, the limiting risk for bagged ridgeless with $M=\infty$ is given by
        \begin{align*}
            \RzeroM{\infty}{\phi} &= \rho^2(1+\tv(0;\phi,\phi_s))\tc(0;\phi_s) + \sigma^2 (1+\tv(0;\phi,\phi_s)).
        \end{align*}
        Defined in \eqref{eq:v_ridgeless}-\eqref{eq:tv_tc_ridgeless},  $\tv(0;\phi,\phi_s)\geq 0$ and $\tc(0;\phi_s)\geq 0$ are continuous functions of $v(0;\phi_s)$, which is strictly decreasing over $\phi_s\in(1,\infty)$ 
        and satisfies $\lim_{\phi_s\rightarrow\infty}v(0;\phi_s)=0$ from \Cref{lem:fixed-point-v-properties}. Then we have $\tv(0;\phi,\phi_s)$ is decreasing in $\phi_s$ over $(1,\infty)$, $\tc(0;\phi_s)$ is increasing in $\phi_s$ over $(1,\infty)$, and
        \begin{align*}
            \lim_{\phi_s\rightarrow\infty}\tv(0;\phi,\phi_s)=0,\qquad \lim_{\phi_s\rightarrow\infty}\tc(0;\phi_s)=\int r \,\rd G(r).
        \end{align*}
        Also, $\tv(0;\phi,\phi_s)=\phi/(1-\phi)$ and $\tc(0;\phi_s)=0$ remain constant for $\phi_s\in(0,1]$ from \eqref{remark:ridgeless}.
        Then to determine the global minimum, it suffices to consider the case when $\phi_s\in[1,\infty)$.
        Next, we consider various cases depending on the value of $\SNR$.
        
        \begin{itemize}[leftmargin=*]
            \item 
            We first consider the case $\SNR>0$.
            We consider further sub-cases depending the value of the pair $(\phi, \phi_s)$.
            
            \begin{enumerate}[leftmargin=*]
                
                \item
                When $\phi\in(0,1)$ and $\phi_s\in(1,\infty]$, 
                \begin{align*}
                    \frac{\partial \RzeroM{\infty}{\phi}}{\partial \phi_s} &= \frac{\partial \RzeroM{\infty}{\phi}}{\partial v(0;\phi_s)} \frac{\partial v(0;\phi_s)}{\partial \phi_s} \\
                    &=  \rho^2 \ddfrac{\phi \int \frac{v(0;\phi_s)r^2}{(1+ v(0;\phi_s) r)^3}\,\rd H(r)}{\left(1 - \phi \int \left(\frac{v(0;\phi_s)r}{(1+ v(0;\phi_s) r)}\right)^2\,\rd H(r)\right)^2} \int \frac{r}{(1 + v(0;\phi_s) r)^2} \,\rd G(r) \cdot \frac{\partial v(0;\phi_s)}{\partial \phi_s} \\
                    &\qquad - 2\rho^2 \ddfrac{\frac{1}{v(0;\phi_s)^2}}{\frac{1}{v(0;\phi_s)^2} - \phi \int \frac{r^2}{(1+ v(0;\phi_s) r)^2}\,\rd H(r)} \int \frac{r^2}{(1 + v(0;\phi_s) r)^3} \,\rd G(r) \cdot \frac{\partial v(0;\phi_s)}{\partial \phi_s} \\
                    &\qquad \qquad + \sigma^2 \ddfrac{\phi \int \frac{v(0;\phi_s)r^2}{(1+ v(0;\phi_s) r)^3}\,\rd H(r)}{\left(1 - \phi \int \left(\frac{v(0;\phi_s)r}{(1+ v(0;\phi_s) r)}\right)^2\,\rd H(r)\right)^2}\cdot \frac{\partial v(0;\phi_s)}{\partial \phi_s} .
                \end{align*}
                Note that from \Cref{lem:ridge-fixed-point-v-properties} $v(0;\phi_s)$ is differentiable in $\phi_s\in(0,\infty]$ with
                $$\frac{\partial v(0;\phi_s)}{\partial \phi_s} = - \ddfrac{\int\frac{r}{1+v(0;\phi_s)r} \,\rd H(r)}{\frac{1}{v(0;\phi_s)^2} - \phi_s \int \frac{r^2}{(1+ v(0;\phi_s) r)^2}\,\rd H(r)}$$ 
                being negative over $\phi_s\in(1,\infty)$ and continuous in $\phi_s\in(1,\infty]$, and $$\lim_{\phi_s\rightarrow 1^+}\frac{\partial v(0;\phi_s)}{\partial \phi_s} = -\infty,\qquad \lim_{\phi_s\rightarrow \infty}\frac{\partial v(0;\phi_s)}{\partial \phi_s} = -\lim_{\phi_s\rightarrow\infty}\tv_v(0;\phi_s)\int\frac{r}{1+v(0;\phi_s)r} \,\rd H(r) = 0$$ 
                by \Cref{lem:fixed-point-v-properties} with $\tv_v$ defined therein.
                We have that $\partial \RzeroM{\infty}{\phi}/\partial \phi_s$ is continuous over $\phi_s\in(1,\infty]$.
                Since $\lim_{\phi_s\rightarrow\infty}v(0;\phi_s)=0$ from \Cref{lem:fixed-point-v-properties}, we have that
                \begin{align}
                    \lim_{\phi_s\rightarrow\infty} \ddfrac{\phi \int \frac{v(0;\phi_s)r^2}{(1+ v(0;\phi_s) r)^3}\,\rd H(r)}{\left(1 - \phi \int \left(\frac{v(0;\phi_s)r}{(1+ v(0;\phi_s) r)}\right)^2\,\rd H(r)\right)^2}%
                    &= 0 \label{eq:thm:opt-risk-ridgeless-1}\\
                    \lim_{\phi_s\rightarrow\infty}\ddfrac{\frac{1}{v(0;\phi_s)^2}}{\frac{1}{v(0;\phi_s)^2} - \phi \int \frac{r^2}{(1+ v(0;\phi_s) r)^2}\,\rd H(r)} \int \frac{r^2}{(1 + v(0;\phi_s) r)^3} \,\rd G(r) &=  \frac{1}{1-\phi}\int r^2\,\rd G(r)>0.\label{eq:thm:opt-risk-ridgeless-2}
                \end{align}
                Since $\partial v(0;\phi_s)/\partial \phi_s$ is negative over $(1,\infty)$ and $\lim_{\phi_s\rightarrow\infty0}\partial v(0;\phi_s)/\partial \phi_s=0$, we have
                \begin{align}
                    \frac{\partial \RzeroM{\infty}{\phi}}{\partial \phi_s}\big|_{\phi_s=\infty} = -2\rho^2\int r^2\,\rd G(r) \cdot \lim_{\phi_s\rightarrow\infty}\frac{\partial v(0;\phi_s)}{\partial \phi_s}  = 0. \label{eq:thm:opt-risk-ridgeless-3}
                \end{align}
                Combining \eqref{eq:thm:opt-risk-ridgeless-1}-\eqref{eq:thm:opt-risk-ridgeless-3}, we have that when $\phi_s$ is large, $\partial \RzeroM{\infty}{\phi}/\partial \phi_s$ approaching zero from above as $\phi_s$ tends to $\infty$.
                On the other hand, since for $k=1,2$,
                \begin{align*}
                    &\lim_{\phi_s\rightarrow1^+}\int \frac{r^{k}}{(1 + v(0;\phi_s) r)^{k+1}} \,\rd G(r) \cdot \frac{\partial v(0;\phi_s)}{\partial \phi_s} \\
                    =& \lim_{\phi_s\rightarrow1^+}\int \frac{v(0;\phi_s)r^{k}}{(1 + v(0;\phi_s) r)^{k+1}} \,\rd G(r) \cdot \lim_{\phi_s\rightarrow1^+} \frac{1}{v(0;\phi_s)}\frac{\partial v(0;\phi_s)}{\partial \phi_s}=0,
                \end{align*}
                we have
                \begin{align}
                    \frac{\partial \RzeroM{\infty}{\phi}}{\partial \phi_s}\big|_{\phi_s=1^+} = \sigma^2 \frac{\phi}{1-\phi} \cdot \lim_{\phi_s\rightarrow1^+}\frac{\partial v(0;\phi_s)}{\partial \phi_s}  < 0. \notag
                \end{align}
                Thus, there exists $\phi^*\in(1,\infty)$ such that
                \begin{align*}
                    \RzeroMe{\infty}{\phi}{\phi^*} < \RzeroMe{\infty}{\phi}{1} = \RzeroMe{\infty}{\phi}{\phi}.
                \end{align*}
                
                \item
                When $\phi=1$, $\RzeroMe{\infty}{1}{1}=\infty$ while $\RzeroMe{\infty}{\phi}{\phi_s}<\infty$ for all $\phi_s\in(1,\infty]$.
                Since $\RzeroMe{\infty}{\phi}{\phi_s}$ is continuous and finite in $(1,\infty]$, by continuity and \eqref{eq:thm:opt-risk-ridgeless-1}-\eqref{eq:thm:opt-risk-ridgeless-3} we have $\phi^*\in(1,\infty)$.

                \item
                When $\phi\in(1,\infty)$, the optimal $\phi^*\geq\phi>1$ must be obtained in $[\phi,\infty)$ because of \eqref{eq:thm:opt-risk-ridgeless-1}-\eqref{eq:thm:opt-risk-ridgeless-3}.
            \end{enumerate}
            
            \item
            Next consider the case when $\SNR=0$, i.e., $\rho^2=0$ and $\sigma^2\neq 0$,
            since $\RzeroM{\infty}{\phi}=\sigma^2+\sigma^2\tv(0;\phi,\phi_s)>0$ is increasing in $v(0;\phi_s)$ and $v(0;\phi_s)\geq 0$ is decreasing in $\phi_s$, we have that $\RzeroM{\infty}{\phi}$ is decreasing in $\phi_s$. Thus, the global minimum $\RzeroM{\infty}{\infty}=\sigma^2$ is obtained at $\phi_s^*=\infty$.
            
            \item
            Finally,
            consider the case
            when $\SNR=\infty$, i.e. $\rho^2\neq 0$ and $\sigma^2=0$, $\RzeroM{\infty}{\phi}=\rho^2(1+\tv(0;\phi,\phi_s))\tc(0;\phi_s)$.
            As the bias term is zero when $\phi_s\in(0,1]$ and positive when $\phi_s\in(1,\infty]$, we have that $\RzeroM{\infty}{\phi}\geq \RzeroMe{\infty}{\phi}{\phi^*}=0$ for all $\phi_s^*\in[\phi,1]$ when $\phi\in(0,1]$.
            If $\phi\in(1,\infty)$, since the risk is continuous over $[\phi,\infty]$, the global minimum exists.
            Since the derivative $\partial \RzeroM{\infty}{\phi}/\partial \phi_s$ is continuous over $\phi_s\in(1,\infty]$ and \eqref{eq:thm:opt-risk-ridgeless-1}-\eqref{eq:thm:opt-risk-ridgeless-3}, the minimizer satisfies $\phi^*\in[\phi,\infty)$.
        \end{itemize}
    \end{proof}

\begin{lemma}[Optimal \splagged ridgeless]\label{lem:opt-risk-splagged-ridgeless}
    Suppose the conditions in \Cref{thm:ver-without-replacement} hold, and $\sigma^2,\rho^2 \ge 0$ are the noise variance and signal strength from Assumptions \ref{asm:lin-mod} and \ref{asm:signal-bounded-norm}. Let $\SNR=\rho^2/\sigma^2$. For any $\phi\in(0,\infty)$, the properties of the optimal asymptotic risk $\mathscr{R}_{0,\infty}^{\spl}(\phi, \phi_s^{\spl}(\phi))$ in terms of $\SNR$ and $\phi$ are characterized as follows:
    \begin{enumerate}[label={(\arabic*)},leftmargin=7mm]
        \item $\SNR=0 \; (\rho^2 = 0, \sigma^2 \neq 0)$:  For all $\phi \ge 0$, the global minimum $\mathscr{R}_{0,\infty}^{\spl}(\phi, \phi_s^{\spl}(\phi)) = \sigma^2$ is obtained with $\phi_s^\spl(\phi)=\infty$.
        
        \item $\SNR>0$: For $\phi \ge 1$, there exists global minimum of $\phi_s \mapsto \bRzeroM{\infty}{\phi}$ in $(1,\infty)$. For $\phi\in(0,1)$, the global minimum is in $\{\phi\}\cup(1,\infty)$.
        
        \item $\SNR=\infty \; (\rho^2 \neq 0, \sigma^2 = 0)$: If $\phi\in(0,1]$, the global minimum $\mathscr{R}_{0,\infty}^{\spl}(\phi, \phi_s^{\spl}(\phi)) = 0$ is obtained with any $\phi_s^\spl(\phi)\in[\phi,1]$.
        If $\phi\in(1,\infty)$, then the global minimum $\mathscr{R}_{0,\infty}^{\spl}(\phi, \phi_s^{\spl}(\phi))$ is obtained at $\phi_s^\spl(\phi)\in[\phi,\infty)$.
    \end{enumerate}
\end{lemma}
    
    \begin{proof}[Proof of \Cref{lem:opt-risk-splagged-ridgeless}]
        From \Cref{thm:ver-without-replacement}, the limiting risk for bagged ridgeless with $M=\phi_s/\phi$ is given by
        \begin{align*}
            \bRzeroM{\phi_s/\phi}{\phi} &= \sigma^2+\frac{\phi}{\phi_s}\left[\rho^2(1+\tv(0;\phi_s,\phi_s))\tc(0;\phi_s)  + \sigma^2 \tv(0;\phi_s,\phi_s)\right] + \left(  1-\frac{\phi}{\phi_s}\right)\rho^2\tc(0;\phi_s)\\
            &=\sigma^2+\rho^2\tc(0;\phi_s) + \phi\tfrac{\tv(0;\phi_s,\phi_s)}{\phi_s} ( \rho^2 \tc(0;\phi_s) + \sigma^2).
        \end{align*}
        Defined in \eqref{eq:v_ridgeless}-\eqref{eq:tv_tc_ridgeless},  $\tv(0;\phi_s,\phi_s)\geq 0$ and $\tc(0;\phi_s)\geq 0$ are continuous functions of $v(0;\phi_s)$, which is strictly decreasing over $\phi_s\in(1,\infty)$ 
        and satisfies $\lim_{\phi_s\rightarrow\infty}v(0;\phi_s)=0$ from \Cref{lem:fixed-point-v-properties}.
        Then $\tc(0;\phi_s)$ is increasing in $\phi_s$ over $(1,\infty)$ and $\lim_{\phi_s\rightarrow\infty}\tc(0;\phi_s)=\int r \,\rd G(r)$.

        \begin{itemize}[leftmargin=*]
            \item 
            We first consider the case $\SNR>0$.
            We consider further sub-cases depending the value of the pair $(\phi, \phi_s)$.

        \begin{enumerate}[leftmargin=*]
            \item When $\phi\in(0,1)$ and $\phi_s\in(1,\infty]$,

        Define functions $h_1$ and $h_2$ as follows:
        \begin{align}
            h_1(\phi_s)=\SNR \cdot \tc(0;\phi_s),\qquad h_2(\phi_s) = \frac{\tv(0;\phi_s,\phi_s)}{\phi_s} &= \tv_v(0;\phi_s) \int \left(\frac{r}{1+v(0;\phi_s)r} \right)^2\,\rd H(r),
        \end{align}
        where $\tv_v$ is defined in \Cref{lem:fixed-point-v-properties}. 
        Then $\bRzeroM{\phi_s/\phi}{\phi}=\sigma^2+\sigma^2(h_1(\phi_s) + \phi h_2(\phi_s) (1+h_1(\phi_s)))$, with $h_1$ increasing in $\phi_s$ and
        $$\lim_{\phi_s\rightarrow1^+}h_1(\phi_s)= 0,\qquad \lim_{\phi_s\rightarrow\infty}h_1(\phi_s)= \SNR \int r\,\rd G(r), \qquad \lim_{\phi_s\rightarrow1^+}h_2(\phi_s) = +\infty,\qquad\lim_{\phi_s\rightarrow\infty}h_2(\phi_s) = 0.$$
        Next we study the property of $h_2$.
        Simple calculation yields that
        \begin{align*}
            &\frac{\partial h_2(\phi_s)}{\partial \phi_s} \notag\\
            =&\tv_v(0;\phi_s)^2 \left[\frac{2}{v(0;\phi_s)^3}\int\frac{r^2}{(1+v(0;\phi_s)r)^3}\,\rd H(r)\cdot \frac{\partial v(0;\phi_s)}{\partial\phi_s} + \left(\int\frac{r^2}{(1+v(0;\phi_s)r)^2}\,\rd H(r)\right)^2\right]\notag\\
            =& \tv_v(0;\phi_s)^2 \left[ \left(\int\frac{r^2}{(1+v(0;\phi_s)r)^2}\,\rd H(r)\right)^2 - \frac{2\tv_v(0;\phi_s)}{v(0;\phi_s)^3}\int\frac{r^2}{(1+v(0;\phi_s)r)^3}\,\rd H(r)\int\frac{r}{1+v(0;\phi_s)r}\,\rd H(r) \right].
        \end{align*}
        From \Cref{lem:fixed-point-v-properties}~\ref{lem:fixed-point-v-properties-item-vb-properties}, we have that $\lim_{\phi_s\rightarrow\infty}\tv_v(0;\phi_s)/v(0;\phi_s)^2 \lim_{\phi_s\rightarrow\infty}[1+\tv_b(0;\phi_s)] = 1$ where $\tv_b(0;\phi_s)$ is defined in \Cref{lem:fixed-point-v-properties}.
        Analogously, $\lim_{\phi_s\rightarrow1^+}\tv_v(0;\phi_s)/v(0;\phi_s)^2= +\infty$.
        Then as in the proof of \Cref{prop:opt-risk-ridgeless}, one can verify that 
        \begin{align*}
            \frac{\partial \bRzeroM{\phi_s/\phi}{\phi}}{\partial\phi_s} &= -\sigma^2\tv_v(0;\phi_s) \left[\SNR(1 + \phi h_2(\phi_s)) \int \frac{r^2}{(1+v(0;\phi_s)r)^3}\,\rd G(r) \cdot \int\frac{r}{1+v(0;\phi_s)r}\,\rd H(r)\right. \\
            &\quad + \phi  (1+h_1(\phi_s)) \frac{2\tv_v(0;\phi_s)^2}{v(0;\phi_s)^3}\int\frac{r^2}{(1+v(0;\phi_s))^3}\,\rd H(r) \cdot \int\frac{r}{1+v(0;\phi_s)r}\,\rd H(r) \\
            &\quad \left.- \tv_v(0;\phi_s) \phi  (1+h_1(\phi_s)) \left(\int\frac{r^2}{(1+v(0;\phi_s)r)^2}\,\rd H(r)\right)^2\right]
        \end{align*}
        satisfies $\lim_{\phi_s\rightarrow1^+}\partial \bRzeroM{\phi_s/\phi}{\phi}/\partial \phi_s=-\infty$ and $\lim_{\phi_s\rightarrow\infty}\partial \bRzeroM{\phi_s/\phi}{\phi}/\partial \phi_s=0$ by utilizing properties in \Cref{lem:fixed-point-v-properties}.
        Furthermore, as 
        \begin{align}
            \lim_{\phi_s\rightarrow\infty } \tv_v(0;\phi_s) ^{-1} \frac{\partial \bRzeroM{\phi_s/\phi}{\phi}}{\partial\phi_s} = -\rho^2 \int r^2\,\rd G(r) \cdot \int r\,\rd H(r) <0, \label{eq:prop:opt-risk-splagged-ridgeless-1}
        \end{align}
        we have that when $\phi_s$ is large, $\partial \bRzeroM{\phi_s/\phi}{\phi}/\partial \phi_s$ approaching zero from above as $\phi_s$ tends to $\infty$.
        Thus, the minimum of $\bRzeroM{\phi_s/\phi}{\phi}$ over $[1,\infty]$ is obtained in the open interval $(1,\infty)$.
        
        \item When $\phi<1$ and $\phi_s\in[\phi,1)$, since the term $\tc(0;\phi_s)$ is zero, $\bRzeroM{\phi_s/\phi}{\phi}=\sigma^2+\sigma^2\phi(1-\phi_s)^{-1}$ is increasing in $\phi_s$.
        So the minimum over $[\phi,1]$ is obtained at $\phi_s=\phi$.
        
        \item
        When $\phi=1$, $\bRzeroMe{\phi_s/\phi}{1}{1}=\infty$ while $\bRzeroMe{\phi_s/\phi}{\phi}{\phi_s}<\infty$ for all $\phi_s\in(1,\infty]$.
        Since $\bRzeroMe{\phi_s/\phi}{\phi}{\phi_s}$ is continuous and finite in $(1,\infty]$, by continuity and \eqref{eq:prop:opt-risk-splagged-ridgeless-1} we have $\phi_s^*\in(1,\infty)$.

        \item
        When $\phi\in(1,\infty)$, the optimal $\phi_s^*\geq\phi>1$ must be obtained in $[\phi,\infty)$ because of \eqref{eq:prop:opt-risk-splagged-ridgeless-1}.
        \end{enumerate}
        
        \item Next consider the case when $\SNR=0$, i.e., $\rho^2=0$ and $\sigma^2\neq 0$.
        Then $h_1\equiv 0$ and $\bRzeroM{\infty}{\phi_s/\phi}=\sigma^2+\sigma^2\phi\tv(0;\phi_s,\phi_s)/\phi_s$.
        When $\phi_s\in(0,1),$ $\tv(0;\phi_s,\phi_s)/\phi_s=(1-\phi_s)^{-1}$ is increasing in $\phi_s$; when $\phi_s>1$, $\tv(0;\phi_s,\phi_s)/\phi_s\geq 0=\lim_{\phi_s\rightarrow\infty }\tv(0;\phi_s,\phi_s)/\phi_s = 0.$
        Therefore, the global minimum $\RzeroM{\infty}{\infty}=\sigma^2$ is obtained at $\phi_s^*=\infty$.
        
        \item
        Finally,
        consider the case
        when $\SNR=\infty$, i.e. $\rho^2\neq 0$ and $\sigma^2=0$, $\bRzeroM{\infty}{\phi_s/\phi}=\rho^2\tc(0;\phi_s)+\rho^2 \phi\phi_s^{-1}\tv(0;\phi,\phi_s)\tc(0;\phi_s)$.
        As the term $\tc(0;\phi_s)$ is zero when $\phi_s\in(0,1]$ and positive when $\phi_s\in(1,\infty]$, we have that $\RzeroM{\infty}{\phi}\geq \RzeroMe{\infty}{\phi}{\phi_s^*}=0$ for all $\phi_s^*\in[\phi,1]$ when $\phi\in(0,1]$.
        If $\phi\in(1,\infty)$, since the risk is continuous over $[\phi,\infty]$, the global minimum exists.
        Since the derivative $\partial \RzeroM{\infty}{\phi}/\partial \phi_s$ is continuous over $\phi_s\in(1,\infty]$ and \eqref{eq:prop:opt-risk-splagged-ridgeless-1}, the minimizer satisfies $\phi_s^*\in[\phi,\infty)$.
        \end{itemize}
    \end{proof}
    
\section{Proofs in \Cref{sec:isotropic_features} (isotropic features)}\label{sec:appendix-isotopic}

    \subsection[Proof of \Cref{cor:isotropic-ridgeless}
    ]
    {Bagged risk for ridgeless regression}
    \begin{proof}[Proof of \Cref{cor:isotropic-ridgeless}]
        Since $\bSigma=\bI_p$, we have that $\rd G=\rd H=\delta_1$.
        Then, $v(0;\phi_s)$, $\tv(0;\phi,\phi_s)$ and $\tc(0;\phi_s)$ defined in \eqref{eq:v_ridgeless} and \eqref{eq:tv_tc_ridgeless} for $\phi_s>1$ reduce to
        \begin{align*}
            v(0;\phi_s) &= \tfrac{1}{\phi_s-1},\qquad
            \tv(0; \phi,\phi_s) = %
            \tfrac{\phi}{\phi_s^2-\phi},\qquad \tc(0;\phi_s) = \frac{(\phi_s-1)^2}{\phi_s^2}.
            \end{align*}
            Thus, we have
            \begin{align*}
            B_{0}(\phi,\phi_s) &= \begin{dcases}
            0 ,&\phi_s\in(0,1)\\
            \rho^2 \tfrac{\phi_s-1}{\phi_s}, &\phi_s\in(1,\infty)
            \end{dcases},
            \qquad
            V_{0}(\phi,\phi_s) = \begin{dcases}
            \sigma^2\tfrac{\phi_s}{1-\phi_s} ,&\phi_s\in(0,1)\\
            \sigma^2\tfrac{1}{\phi_s-1}, &\phi_s\in(1,\infty)
            \end{dcases},
        \end{align*}
        and 
        \begin{align*}
            C_{0}(\phi_s) = \begin{dcases}
            0 ,&\phi_s\in(0,1)\\
             \rho^2\frac{(\phi_s-1)^2}{\phi_s^2}, &\phi_s\in(1,\infty)
            \end{dcases}.
        \end{align*}
    \end{proof}

    From \Cref{cor:isotropic-ridgeless}, we are able to derive the asymptotic bias and variance for $M=1$ and $M=\infty$ for ridgeless regression with replacement:
    \begin{align*}
        \BzeroM{1}{\phi} &= \begin{dcases}
        0 ,&\phi_s\in(0,1)\\
        \rho^2 \tfrac{\phi_s-1}{\phi_s}, &\phi_s\in(1,\infty)
        \end{dcases}&&
        \VzeroM{1}{\phi} = \begin{dcases}
        \sigma^2\tfrac{\phi_s}{1-\phi_s} ,&\phi_s\in(0,1)\\
        \sigma^2\tfrac{1}{\phi_s-1}, &\phi_s\in(1,\infty)
        \end{dcases}\\
        \BzeroM{\infty}{\phi} &= \begin{dcases}
        0,&\phi_s\in(0,1)\\
        \rho^2\tfrac{(\phi_s-1)^2}{\phi_s^2-\phi},&\phi_s\in(1,\infty)
        \end{dcases}&&
        \VzeroM{\infty}{\phi} = \begin{dcases}
        \sigma^2 \tfrac{\phi}{1-\phi},&\phi_s\in(0,1)\\
        \sigma^2\tfrac{\phi}{\phi_s^2-\phi}, &\phi_s\in(1,\infty)
        \end{dcases}.
    \end{align*}
    Then the asymptotic bias and variance for general $M$ would be convex combinations of the above quantities.
    
    On the other hand, the asymptotic bias and variance for \splagging without replacement are given by
    \begin{align*}
       \bBlamM{M}{\phi}&=M^{-1}B_{\lambda}(\phi_s, \phi_s)+(1-M^{-1})C_{\lambda}(\phi_s) ,\qquad
       \bVlamM{M}{\phi} =M^{-1}V_{\lambda}(\phi_s,\phi_s).
    \end{align*}

\subsection[Proof of \Cref{thm:ridgeless-isotipic-optrisk}
]
{Optimal subagged ridgeless regression with replacement}
    \begin{proof}[Proof of \Cref{thm:ridgeless-isotipic-optrisk}]
        For $\phi\in(0,1)$ and $\phi_s\in(1,\infty]$, we have that
        $$\RzeroM{\infty}{\phi} = \sigma^2 + \rho^2\tfrac{(\phi_s-1)^2}{\phi_s^2-\phi}+ \sigma^2\tfrac{\phi}{\phi_s^2-\phi}.$$
        Taking the derivative of the right hand side with respect to $\phi_s$
        \begin{align*}
            \tfrac{\partial \RzeroM{\infty}{\phi}}{\partial \phi_s} &= 2\sigma^2 \tfrac{\SNR (\phi_s-1)(\phi_s-\phi) - \phi\phi_s}{(\phi_s^2-\phi)^2}
        \end{align*}
        and setting it to zero yields that
        \begin{align}
            \phi_s &= A \pm \sqrt{A^2-\phi}. \label{eq:ridgeless-isotipic-optrisk-sols}
        \end{align}
        where $A=(\phi+1 + \phi/\SNR)/2$.
        Since $A - \sqrt{A^2-\phi} < \sqrt{\phi}\leq 1$, we have $\phi_s^* = A + \sqrt{A^2-\phi}$ is a minimizer and
        \begin{align}
            \RzeroMe{\infty}{\phi}{\phi^*} &= \sigma^2+ \sigma^2\tfrac{\phi+A-\sqrt{A^2-\phi}+\SNR(1-\phi)/\phi(A-\phi-\sqrt{A^2-\phi})}{2 \sqrt{A^2-\phi}}\notag\\
            & = \frac{\sigma^2}{2}\left[ 1 + \frac{\phi-1}{\phi} \SNR + 
             \frac{2\SNR}{\phi}\sqrt{A^2-\phi}\right] 
              \notag\\
              &=  \frac{\sigma^2}{2}\left[ 1 + \frac{\phi-1}{\phi} \SNR + 
             \sqrt{\left(1 - \frac{\phi-1}{\phi}\SNR\right)^2 + 4 \SNR}\right],
         \end{align}    
         which gives the simplified formula.
         Note that
         \begin{align}
            \RzeroMe{\infty}{\phi}{\phi^*} &= \sigma^2+\sigma^2\left(\tfrac{\phi}{2 \sqrt{A^2-\phi}}+ \tfrac{A-\sqrt{A^2-\phi}}{2 \sqrt{A^2-\phi}} + \tfrac{1-\phi}{\phi}\SNR \tfrac{A-\phi-\sqrt{A^2-\phi}}{2 \sqrt{A^2-\phi}}\right) \notag \\
            &= \sigma^2+\sigma^2 h(\SNR) - \sigma^2\delta(\SNR) \label{eq:ridgeless-isotipic-optrisk-decomposition}
        \end{align}
        where for all $r\geq 0$, the functions $h$ and $\delta$ are defined as $h(r)=h_1(r)+h_2(r)+h_3(r)$ and $\delta(r)=(1-\phi)r h_1(r)/\phi$, with $A(r)=(\phi+1 + \phi/r)/2$ and
        \begin{align*}
            h_1(r) &= \tfrac{\phi}{2 \sqrt{A(r)^2-\phi}}\\
            h_2(r) &= \tfrac{A(r)-\sqrt{A(r)^2-\phi}}{2 \sqrt{A(r)^2-\phi}} = \tfrac{1}{2 \sqrt{1-\phi/A(r)^2}} -\tfrac{1}{2}\\
            h_3(r) &= \tfrac{1-\phi}{\phi} r  h_2(r).
        \end{align*}
        Since $h_1$, $h_2$, and $h_3$ are nonngative over $(0,\infty)$, $h$ and $\delta$ are also nonnegative.
        Also noted that %
        \begin{align*}
            \delta(0) &= \tfrac{1-\phi}{\phi} \lim\limits_{r\rightarrow0^+}r h_1(r)  = 0,\qquad \delta(\infty) = \tfrac{1-\phi}{\phi} \lim\limits_{r\rightarrow+\infty}r h_1(r)  = +\infty,
        \end{align*}
        we obtain the upper bound for $\RzeroMe{\infty}{\phi}{\phi^*}$ as follows:
        \begin{align}
            \RzeroMe{\infty}{\phi}{\phi^*} &\leq \sigma^2 + \sigma^2 h(\SNR),
        \end{align}
        with equality obtained if and only if $\SNR=0$.
        
        Next we analyze the function $h(r)$.
        Note that $A(r)>0$ is decreasing in $r$, we have that the functions $h_1$ and $h_2$ are nonnegative and monotone increasing in $\SNR$. Hence $h_3$ as the product of nonnegative and monotone increasing functions, is also nonnegative and monotone increasing in $\SNR$.
        Thus, $h$ is monotone increasing in $\SNR$ and 
        \begin{align*}
            h(\SNR) &\leq \lim\limits_{r\rightarrow\infty} h(r) \\
            &= \lim\limits_{r\rightarrow\infty} h_1(r) + \lim\limits_{r\rightarrow\infty} h_2(r) + \lim\limits_{r\rightarrow\infty} rh_2(r)\\
            &= \tfrac{\phi}{1-\phi} + \tfrac{\phi}{1-\phi} + \tfrac{1}{\phi}\lim\limits_{r\rightarrow\infty}\tfrac{A(r)-\sqrt{A(r)^2-\phi}}{\tfrac{1}{r}}\\
            &= \tfrac{\phi}{1-\phi} + \tfrac{\phi}{1-\phi} + \tfrac{1}{\phi}\lim\limits_{r\rightarrow\infty} \ddfrac{
            -\tfrac{\phi}{2r^2}+ \ddfrac{\tfrac{A(r)\phi}{r^2}}
            {2\sqrt{A(r)^2-\phi}}}{-\tfrac{1}{r^2}} \\
            &= \tfrac{\phi}{1-\phi}
        \end{align*}
        where the third equality is due to the L'Hospital's rule.
        Note that the risk for $\phi_s\in[\phi,1)$ is given by $\sigma^2+\sigma^2\phi/(1-\phi)$, we have that $\phi_s^*$ obtained the global minimum of $\RzeroM{\infty}{\phi}$ over $\phi_s\in[\phi,\infty]$.
        
        For $\phi\in[1,\infty)$ and $\phi_s\in[\phi,\infty)$, from \eqref{eq:ridgeless-isotipic-optrisk-sols} and $A-\sqrt{A^2-\phi}\leq \sqrt{\phi}\leq\phi$, we have again $\phi_s^*=A+\sqrt{A^2-\phi}$ is a minimizer.
        
        When $\SNR=0$, since the bias term is zero and variance term is increasing over $\phi_s<1$ and increasing over $\phi_s>1$, we have that when $\phi_s>1$ (whenever $\phi\leq\phi_s$),
        \begin{align*}
            \RzeroMe{\infty}{\phi}{\phi_s} &=\sigma^2 + V_{0}(\phi,\phi_s)\geq \sigma^2 + V_{0}(\phi,\infty) = \sigma^2.
        \end{align*}
        When $\phi<1$, we have $\RzeroMe{\infty}{\phi}{\phi_s}\geq \RzeroMe{\infty}{\phi}{\phi} = \sigma^2 /(1-\phi)>\sigma^2$.
        Therefore, $\RzeroMe{\infty}{\phi}{\phi_s}\geq \RzeroMe{\infty}{\phi}{\infty}=\sigma^2$ for all $\phi\in(1,\infty]$.
        
        When $\SNR=\infty$, the variance term $V_{0}(\phi,\phi_s)=0$ for all $\phi_s\in[\phi,\infty]$. 
        If $\phi\in(0,1]$, then $B_{0}(\phi,\phi_s)=0$ for all $\phi_s\in[\phi,1]$.
        If $\phi\in(1,\infty)$, then $B_{0}(\phi,\phi_s)$ is increasing over $\phi_s\in[\phi,\infty]$.
        Hence, the conclusions follow.
    \end{proof}    
    
\subsection[Proof of \Cref{thm:comparison_optimal_ridge}
]
{Comparison between subagged and optimal ridge regression}
    
    \begin{proof}[Proof of \Cref{thm:comparison_optimal_ridge}]
        As $n,p\rightarrow\infty$ and $p/n\rightarrow\phi$, the optimal regularization parameter is given by  $\lambda^* = \phi \sigma^2/\rho^2$ under the isotopic model \citep{dobriban_wager_2018}.
        The limiting risk of the optimal ridge regression is given by
        \begin{align*}
            \sR^{\textup{\texttt{WR}}}_{\lambda^*,1}(\phi,\phi)&=  \frac{\sigma^2}{2}\left[1+\frac{\phi-1}{\phi} \SNR+\sqrt{\left(1-\frac{\phi-1}{\phi} \SNR\right)^2+4 \SNR}\right]
        \end{align*}
        which is the same the formula given in  \Cref{thm:ridgeless-isotipic-optrisk}.
        Thus, the conclusion follows.
    \end{proof}

\subsection*{Fixed-point equation details for ridge regression}
\label{subsec:isotropic-ridge}
    For isotopic features $\bSigma=\bI_p$, $\rd G=\rd H=\delta_1$.
    When $n,p\rightarrow$ and $p/n\rightarrow\phi\in(0,\infty)$, \eqref{eq:def-v-ridge}-\eqref{eq:def-tv-v-ridge} reduce to
    \begin{align*}
        v(-\lambda; \phi)^{-1} &= \lambda + \phi(1+v(-\lambda; \phi))^{-1} \\
        \tv_b(-\lambda; \phi)&= \ddfrac{\phi (1+v(-\lambda; \phi) )^{-2}}{v(-\lambda; \phi)^{-2}- \phi(1+v(-\lambda; \phi) )^{-2}}\\
        \tv_v(-\lambda; \phi)^{-1} &=v(-\lambda; \phi)^{-2}- \phi (1+v(-\lambda; \phi) )^{-2}.
    \end{align*}
    Solving the first equation for $v(-\lambda; \phi)\geq 0$ gives
    \begin{align}
        v(-\lambda; \phi)&= \tfrac{1}{2\lambda}(-(\phi+\lambda-1) + \sqrt{(\phi+\lambda-1)^2+4\lambda}).\label{eq:v-identity}
    \end{align}
    Then the asymptotic bias and variance defined in \Cref{thm:ver-ridge} can be evaluated accordingly.

\section{Auxiliary asymptotic equivalency results}
\label{sec:calculus_asymptotic_equivalents}

\subsection{Preliminaries}

We use the notion of asymptotic equivalence of sequences of random matrices in various proofs.
In this section, we provide a basic review of the related definitions
and corresponding calculus rules.

\begin{definition}[Asymptotic equivalence: 
deterministic version]
    \label{def:deterministic-equivalent-D}
    Consider sequences $\{ \bA_p \}_{p \ge 1}$ and $\{ \bB_p \}_{p \ge 1}$
    of (random or deterministic) matrices of growing dimensions.
    We say that $\bA_p$ and $\bB_p$ are equivalent and write
    $\bA_p \asympequi_D \bB_p$ if
    $\lim_{p \to \infty} | \tr[\bC_p (\bA_p - \bB_p)] | = 0$ almost surely
    for any sequence of matrices $\bC_p$ with bounded trace norm
    such that $\limsup_{p\rightarrow\infty} \| \bC_p \|_{\mathrm{tr}} < \infty$.
\end{definition}
    
We emphasize that recent work \citep{dobriban_sheng_2021,patil2022mitigating} used the deterministic version of the asymptotic equivalence, implicitly assuming that $\bC_p$ in the definition is deterministic.
However, in this paper we need to investigate the asymptotic equivalence relationship conditional on some other sequences.
In that direction, we first extend Definition \ref{def:deterministic-equivalent-D} to allow for random $\bC_p$, as in Definition \ref{def:deterministic-equivalent}.

\begin{definition}[Asymptotic equivalence: random version]
    \label{def:deterministic-equivalent}
    Consider sequences $\{ \bA_p \}_{p \ge 1}$ and $\{ \bB_p \}_{p \ge 1}$
    of (random or deterministic) matrices of growing dimensions.
    We say that $\bA_p$ and $\bB_p$ are equivalent and write
    $\bA_p \asympequi_R \bB_p$ if
    $\lim_{p \to \infty} | \tr[\bC_p (\bA_p - \bB_p)] | = 0$ almost surely
    for any sequence of random matrices $\bC_p$ independent of $\bA_p$ and $\bB_p$, with bounded trace norm
    such that $\limsup_{p\rightarrow\infty} \| \bC_p \|_{\mathrm{tr}} < \infty$ almost surely.
\end{definition}
    
Even though Definition \ref{def:deterministic-equivalent-D} seems to be more restrictive than Definition \ref{def:deterministic-equivalent}, they are indeed equivalent as shown in Proposition \ref{prop:equi_D_R}.
The latter definition allows for more general definition for ``conditional'' asymptotic equivalents.

\begin{proposition}[Equivalence of $\asympequi_D$ and $\asympequi_R$]\label{prop:equi_D_R}
    The asymptotic equivalent relations $\asympequi_D$ in Definition \ref{def:deterministic-equivalent-D} and $\asympequi_R$ in Definition \ref{def:deterministic-equivalent} are equivalent. \end{proposition} \begin{proof}[Proof of Proposition \ref{prop:equi_D_R}] Let $\{\bA_p\}$ and $\{\bB\}_p$ be two sequences of random matrices. Suppose that $\bA_p\asympequi_D \bB_p$. We next show that $\bA_p\asympequi_R \bB_p$ holds.
    For any sequence of random matrices $\bC_p$ that is independent of $\bA_p$ and $\bB_p$ for all $p\in\NN$, and has bounded trace norm such that $\limsup \norm{\bC_p}_{\tr} < \infty$ as $p\rightarrow\infty$ almost surely.
    Let $A$ denote the event that $\lim_{p\rightarrow\infty}|\tr[\bC_p(\bA_p-\bB_p)]| = 0$. Then
    \begin{align*}
        \PP\left(A\right) &=
        \EE[\ind_A] \overset{(a)}{=} \EE[\EE[\ind_A\mid \{\bC_p\}_{p\geq 1}]] \overset{(b)}{=} \EE[1] = 1.
    \end{align*}
    Above, equality (a) follows from the law of total expectation. Inequality (b) holds almost surely because $\bA_p\asympequi_D \bB_p$ and $\bC_p$ is independent of $\bA_p$ and $\bB_p$.
    This can be seen as follows.
    Note that $\ind_A(\{\bC_p\}, (\{\bA_p\}, \{\bB_p\}))$ 
    is a function of
    random variables
    $\{ \bC_p \}$ and $(\{\bA_p\}, \{\bB_p\})$.
    Let $\EE[\ind_A(\{\bc_p\}, (\{\bA_p\}, \{\bB_p\}))] = h(\{\bc_p\})$ where the expectation is taken over
    the randomness in $(\{\bA_p\}, \{\bB_p\})$.
    Since $\{\bC_p\}$ and $(\{ \bA_p \}, \{ \bB_p \})$
    are independent
    and $\EE[ | \ind_A |] \le 1 < \infty$,
    we have that
    (see, e.g., \citet[Chapter 2, Section 7, Equation (16)]{shiryaev2016probability},
    or \citet[Example 5.1.5]{durrett2010probability})
    \[
        \EE[\ind_A \mid \{ \bC_p \}]
        = h(\{ \bC_p \}),
    \]
    and from \Cref{def:deterministic-equivalent-D},
    we have $h(\{ \bC_p \}) = 1$ almost surely.
    Thus, we can conclude that $\bA_p\asympequi_R\bB_p$.
    
    On the other hand, by definition, $\bA_p\asympequi_R \bB_p$ directly implies $\bA_p\asympequi_D \bB_p$, which completes the proof.
\end{proof}

    The properties for the two types of deterministic equivalents are summarized in \Cref{lem:calculus-detequi}.
    Though most of the calculus rules are the direct consequences from \citet{dobriban_wager_2018,dobriban_sheng_2021}, the product rule involving random matrices $\bC_p$ does not immediately follow
    from previous work.

    \begin{lemma}
        [Calculus of deterministic equivalents]
        \label{lem:calculus-detequi}
        Let $\bA_p$, $\bB_p$, $\bC_p$ and $\bD_p$ be sequences of random matrices.
        The calculus of deterministic equivalents ($\asympequi_D$ and $\asympequi_R$) satisfies the following properties:
        \begin{enumerate}[label={(\arabic*)},leftmargin=7mm]
            \item 
            \label{lem:calculus-detequi-item-equivalence}
            Equivalence:
            The relation $\asympequi$ is an equivalence relation.
            \item 
            \label{lem:calculus-detequi-item-sum}
            Sum:
            If $\bA_p \asympequi \bB_p$ and $\bC_p \asympequi \bD_p$, then $\bA_p + \bC_p \asympequi \bB_p + \bD_p$.
            \item 
            \label{lem:calculus-detequi-item-product}
            Product:
            If $\bA_p$ has uniformly bounded operator norms such that $\limsup_{p\rightarrow\infty}\| \bA_p \|_{\oper} < \infty$, $\bA_p$ is independent of $\bB_p$ and $\bC_p$ for $p\ge 1$,
            and $\bB_p \asympequi \bC_p$, then $\bA_p \bB_p \asympequi \bA_p \bC_p$.
            \item 
            \label{lem:calculus-detequi-item-trace}
            Trace:
            If $\bA_p \asympequi \bB_p$, then $\tr[\bA_p] / p - \tr[\bB_p] / p \to 0$ almost surely.
            \item 
            \label{lem:calculus-detequi-item-differentiation}
            Differentiation:
            Suppose $f(z, \bA_p) \asympequi g(z, \bB_p)$ where the entries of $f$ and $g$
            are analytic functions in $z \in S$ and $S$ is an open connected subset of $\CC$.
            Suppose that for any sequence $\bC_p$ of deterministic matrices with bounded trace norm
            we have $| \tr[\bC_p (f(z, \bA_p) - g(z, \bB_p))] | \le M$ for every $p$ and $z \in S$.
            Then we have $f'(z, \bA_p) \asympequi g'(z, \bB_p)$ for every $z \in S$,
            where the derivatives are taken entrywise with respect to $z$.
        \end{enumerate}
    \end{lemma}
    \begin{proof}
         The conclusions for $\asympequi_D$ directly follow from \citet{dobriban_wager_2018,dobriban_sheng_2021}.
         Then, the proof of property \ref{lem:calculus-detequi-item-equivalence}, \ref{lem:calculus-detequi-item-sum}, \ref{lem:calculus-detequi-item-trace}, and \ref{lem:calculus-detequi-item-differentiation} for $\asympequi_R$ follows from  Proposition \ref{prop:equi_D_R}.
         It remains to show that the product rule holds for $\asympequi_R$.
         Since $\bB_p \asympequi_R \bC_p$, we have $\bB_p \asympequi_D \bC_p$.
         Then for any sequence of random matrices $\{\bD_p\}_{p\geq 1}$ that have bounded trace norm and are independent of $\bB_p$ and $\bC_p$, we have
         \begin{align*}
             \PP\left(\lim_{p \to \infty} | \tr[\bD_p (\bB_p - \bC_p)] | = 0\right) = 1.
         \end{align*}
         Because $| \tr[\bD_p (\bA_p\bB_p -\bA_p \bC_p)] |\leq \|\bA_p\|_{\oper}| \tr[\bD_p (\bB_p - \bC_p)] |$ and $\limsup_{p\rightarrow\infty}\|\bA_p\|_{\oper}<\infty$, we have that $\lim_{p \to \infty} | \tr[\bD_p (\bB_p - \bC_p)] | = 0$ implies $ \lim_{p \to \infty} | \tr[\bD_p (\bA_p\bB_p - \bA_p\bC_p)] | = 0$ conditioning on $\{\bA_p\}_{p\geq 1}$.
         Thus,
         \begin{align*}
             \PP\left(\lim_{p \to \infty} | \tr[\bD_p (\bA_p\bB_p - \bA_p\bC_p)] | = 0\,\mid\, \{\bA_p\}_{p\geq 1}\right) = 1.
         \end{align*}
         and by law of total expectation
         \begin{align*}
             \PP\left(\lim_{p \to \infty} | \tr[\bD_p \bA_p(\bB_p - \bC_p)] | = 0\right) = 1,
         \end{align*}
         which holds for any sequence of random matrices $\{\bD_p\bA_p\}_{p\geq 1}$ that have bounded trace norm and are independent of $\bB_p$ and $\bC_p$.
         By definition, we have $\bA_p \bB_p \asympequi \bA_p \bC_p$.
    \end{proof}
    
    Since the asymptotic equivalent relation $\asympequi_D$ is equivalent to $\asympequi_R$, we will just ignore the subscript and use the notation ``$\asympequi$'' for simplicity.
    The subscript will be specified when needed.

\subsection{Conditioning and calculus}    

In this section,
we extend the notion of asymptotic equivalence
of two sequences of random matrices
from \Cref{def:deterministic-equivalent-D,def:deterministic-equivalent}
to incorporate conditioning on another sequence of random matrices.

    \begin{definition}[Conditional asymptotic equivalence]
        \label{def:cond-deterministic-equivalent}
        Consider sequences $\{ \bA_p \}_{p \ge 1}$, $\{ \bB_p \}_{p \ge 1}$ and $\{ \bD_p \}_{p \ge 1}$
        of (random or deterministic) matrices of growing dimensions.
        We say that $\bA_p$ and $\bB_p$ are equivalent given $\bD_p$ and write
        $\bA_p \asympequi \bB_p\mid \bD_p$ if
        $\lim_{p \to \infty} | \tr[\bC_p (\bA_p - \bB_p)] | = 0$ almost surely conditional on $\{\bD_p\}_{p\ge 1}$, i.e.,
        \begin{align*}
            \PP\left(\lim\limits_{p\rightarrow\infty}|\tr[\bC_p(\bA_p-\bB_p)]| =0\,\mid\, \{\bD_p\}_{p\ge 1}\right) =1,
        \end{align*}
        for any sequence of random matrices $\bC_p$ independent of $\bA_p$ and $\bB_p$ conditional on $\bD_p$, with bounded trace norm
        such that $\limsup \| \bC_p \|_{\mathrm{tr}} < \infty$
        as $p \to \infty$.
    \end{definition}

    Below we formalize additional calculus rules
    that hold for conditional asymptotic equivalence \Cref{def:cond-deterministic-equivalent}.

    \begin{proposition}[Calculus of conditional asymptotic equivalents] \label{prop:cond-calculus-detequi}
        Let $\bA_p$, $\bB_p$, $\bC_p$, and $\bE_p$ be sequences of random matrices.
        \begin{enumerate}[label={(\arabic*)},leftmargin=7mm]
            \item \label{lem:cond-calculus-detequi-item-uncond} Unconditioning: If $ \bA_p\asympequi \bB_p\mid \bE_p$, then $ \bA_p\asympequi \bB_p$.
            \item \label{lem:cond-calculus-detequi-item-product}Product: If $\bA_p$ has bounded operator norms such that $\limsup_{p\rightarrow\infty}\| \bA_p \|_{\oper} < \infty$, $\bA_p$ is conditional independent of $\bB_p$ and $\bC_p$ given $\bE_p$ for $p\ge 1$, and $\bB_p \asympequi \bC_p\mid \bE_p$, then $\bA_p \bB_p \asympequi \bA_p \bC_p\mid \bE_p$.
        \end{enumerate}
    \end{proposition}
    \begin{proof}[Proof of Proposition \ref{prop:cond-calculus-detequi}]
    Proofs for the two parts appear below.
        \paragraph{Part (1).\hspace{-2mm}}
        For any sequence of deterministic matrices $\bC_p$ with bounded trace norm, we have 
        \begin{align*}
            \PP\left(\lim\limits_{p\rightarrow\infty}|\tr[\bC_p(\bA_p-\bB_p)]| =0 \,\mid\, \{\bD_p\}_{p\ge 1}\right) =1
        \end{align*}
        because $\bA_p \asympequi \bB_p\mid\bE_p$.
        By the law of total expectation, we have
        \begin{align*}
            \PP\left(\lim\limits_{p\rightarrow\infty}|\tr[\bC_p(\bA_p-\bB_p)]| =0\right) =1.
        \end{align*}
        Thus, $\bA_p \asympequi_D \bB_p$. By Proposition \ref{prop:equi_D_R}, we further have $\bA_p \asympequi_R \bB_p$.
        
        \paragraph{Part (2).\hspace{-2mm}}
        For any sequence of random matrices $\bD_p$, let $E_1$ and $E_2$ denote the event $\lim_{p\rightarrow\infty} |\tr[\bD_p(\bB_p-\bC_p)]|=0$ and $\lim_{p\rightarrow\infty} |\tr[\bD_p(\bA_p\bB_p-\bA_p\bC_p)]|=0$, respectively.
        Because $\bB_p \asympequi \bC_p\mid \bE_p$, by definition we have
        \begin{align*}
            \PP\left(E_1\,\mid\, \{\bE_p\}_{p\geq 1}\right) =1
        \end{align*}
        Because $|\tr[\bD_p(\bA_p\bB_p-\bA_p\bC_p)]|\leq \|\bA_p\|_{\oper}|\tr[\bD_p(\bB_p-\bC_p)]| $ and $\limsup_{p\rightarrow\infty}\|\bA_p\|_{\oper}<\infty$, we have $E_1$ implies $E_2$ conditioning on $\{\bE_p\}_{p\geq 1}$. Thus we have 
        \begin{align*}
            \PP\left(E_2\,\mid\, \{\bE_p\}_{p\geq 1}\right) =1
        \end{align*}
        holds for any $\{\bD_p\}_{p\geq 1}$.
        This implies that $\bA_p \bB_p \asympequi \bA_p \bC_p\,\mid\, \bE_p$.
    \end{proof}
    Other rules in \Cref{lem:calculus-detequi} also hold for conditional asymptotic equivalents.
    A direct implication of this is that the deterministic equivalents for resolvents we will derive in \Cref{append:det-equi-resol} based on these rules can be naturally generalized to allow for conditional asymptotic equivalents given a common sequence of random matrices that are independent of the source sequence.
   
\subsection{Standard ridge resolvents and extensions}
\label{append:det-equi-resol}

In this section, we collect various asymptotic equivalents that are used in the proofs of \Cref{lem:ridge-B0,lem:ridge-V0}, and \Cref{lem:ridgeless-B0,lem:ridgeless-V0-phis_lt1,lem:ridgeless-V0-phis-gt1}, which serve to prove \Cref{thm:ver-with-replacement}.
These equivalents are also subsequently used in the proof of \Cref{thm:ver-without-replacement}.

\subsubsection{Standard ridge resolvents}

The following lemma supplies an asymptotic equivalent for the standard ridge resolvent.
The lemma is adapted from Theorem 1 of \cite{rubio_mestre_2011} and Theorem 3 of \cite{dobriban_sheng_2021}.

\begin{lemma}[Asymptotic equivalent for standard ridge resolvent]
    \label{lem:basic-ridge-resolvent-equivalent}
    Suppose $\bx_i \in \RR^{p}$ for $i \in [n]$ are i.i.d.\ random vectors that each multiplicatively decompose such that $\bx_i = \bz_{i} \bSigma^{1/2}$, where $\bz_i$ is a random vector consisting of i.i.d.\ entries $z_{ij}$ for $j \in [p]$, satisfying $\EE[z_{ij}] = 0$, $\EE[z_{ij}^2] = 1$, and $\EE[|z_{ij}|^{8+\alpha}] \le M_\alpha$ for some constants $\alpha > 0$ and $M_\alpha < \infty$, and $\bSigma \in \RR^{p \times p}$ is a positive semidefinite matrix satisfying $0 \preceq \bSigma \preceq r_{\max} \bI_p$ for some constant $r_{\max} < \infty$ that is independent of $p$.
    Let $\bX \in \RR^{n \times p}$ be the concatenated matrix with $\bx_i^\top$ for $i \in [n]$ as rows, and let $\hSigma \in \RR^{p \times p}$ denote the random matrix $\bX^\top \bX / n$.
    Let $\gamma := p / n$.
    Then, for $z \in \CC^{>0}$, as $n, p \to \infty$ such that $0 < \liminf \gamma \le \limsup \gamma < \infty$, the following asymptotic equivalence holds:
    \begin{equation}
        (\hSigma - z \bI_p)^{-1}
        \asympequi
        (c(e(z; \gamma)) \bSigma - z \bI_p)^{-1}.
    \end{equation}
    Here the scalar $c(e(z; \gamma))$ is defined in terms of another scalar $e(z; \gamma)$ by the equation:
    \begin{equation}
        \label{eq:basic-ridge-equivalence-c-e-relation}
        c(e(z; \gamma))
        = \frac{1}{ 1 + \gamma e(z; \gamma)},
    \end{equation}
    where $e(z; \gamma)$ is the unique solution in $\CC^{>0}$ to the fixed-point equation:
    \begin{equation}
        \label{eq:basic-ridge-equivalence-e-fixed-point}
        e(z; \gamma)
        = 
        \tr[ \bSigma (c(e(z; \gamma)) \bSigma  - z I_p)^{-1} ] / p.
    \end{equation}
\end{lemma}

It is worth noting that both scalars $c(e(z; \gamma))$ and $e(z; \gamma)$ implicitly depend on $\bSigma$. 
For the sake of concise notation, this dependence is not always explicitly indicated. 
Nonetheless, such dependence will be clearly stated when considering certain forthcoming extensions. 
Please see the remark following \Cref{lem:deter-approx-generalized-ridge} for more discussion.

Observe that $e(z; \gamma)$ can be removed from the statement of \Cref{lem:basic-ridge-resolvent-equivalent} by combining \eqref{eq:basic-ridge-equivalence-c-e-relation} and \eqref{eq:basic-ridge-equivalence-e-fixed-point}.
For $z \in \CC^{>0}$, we then have the following asymptotic equivalence:
\[
    (\hSigma  - z \bI_p)^{-1}
    \asympequi (c(z; \gamma) \bSigma - z \bI_p)^{-1}.
\]
Here $c(z)$ is the unique solution in $\CC_{<0}$ to the following fixed-point equation:
\[
    \frac{1}{c(z; \gamma)}
    = 1 + \gamma \tr[\bSigma (c(z; \gamma) \bSigma - z \bI_p)^{-1}] / p.
\]

The next corollary is a simple consequence of \Cref{lem:basic-ridge-resolvent-equivalent}.
It supplies an asymptotic equivalent for the regularization scaled ridge resolvent.
We assume the regularization parameter $\lambda$ is real from here onwards.

\begin{corollary}
    [Asymptotic equivalent for the scaled ridge resolvent]
    \label{cor:asympequi-scaled-ridge-resolvent}
    Assume the setting of \Cref{lem:basic-ridge-resolvent-equivalent}.
    For $\lambda > 0$,
    the following asymptotic equivalence holds:
    \[
        \lambda (\hSigma + \lambda \bI_p)^{-1}
        \asympequi
        (v(-\lambda; \gamma) \bSigma + \bI_p)^{-1}.
    \]
    Here
    $v(-\lambda; \gamma) > 0$
    is the unique solution to the fixed-point equation
    \begin{align}
        \label{eq:basic-ridge-equivalence-v-fixed-point}
        \frac{1}{v(-\lambda; \gamma)}
        =
        \lambda
        + \gamma \int \frac{r}{1+v(-\lambda; \gamma)r} \, \rd H_n(r),
    \end{align}
    where $H_n$ is the empirical distribution 
    of the eigenvalues of $\bSigma$ that supported on $\RR_{\ge 0}$.
\end{corollary}

A side remark:
the parameter $v(-\lambda; \gamma)$ that appears in \Cref{cor:asympequi-scaled-ridge-resolvent} is the companion Stieltjes transform of the spectral distribution of the sample covariance matrix $\hSigma$.
It is also the Stieltjes transform of the spectral distribution of the gram matrix $\bX \bX^\top / n$.

The subsequent lemma utilizes \Cref{cor:asympequi-scaled-ridge-resolvent}, in conjunction with the calculus of deterministic equivalents as outlined in \Cref{lem:calculus-detequi}).
It provides deterministic equivalents for resolvents required to derive limiting bias and variance of standard ridge regression.
The lemma is adapted from Lemma S.6.10 of \cite{patil2022mitigating}.
These equivalents are standard and well-established in the prior literature. For example, the variance resolvent in \Cref{lem:deter-approx-generalized-ridge} can be obtained from results in \citet{dobriban_sheng_2021}, while the bias resolvent in \Cref{lem:deter-approx-generalized-ridge} can be obtained from results in \citet{hastie2022surprises}.

\begin{lemma}[Asymptotic equivalents for ridge resolvents associated with generalized bias and variance]
    \label{lem:deter-approx-generalized-ridge}
    Suppose $\bx_i \in \RR^{p}$ for $i \in [n]$ are i.i.d.\ random vectors with each decomposing multiplicatively into $\bx_i = \bz_{i} \bSigma^{1/2}$, where $\bz_i \in \RR^{p}$ is a random vector that contains i.i.d.\ random variables $z_{ij}$, $j \in [p]$, each with $\EE[z_{ij}] = 0$, $\EE[z_{ij}^2] = 1$, and $\EE[|z_{ij}|^{8+\alpha}] \le M_\alpha$ for some constants $\alpha > 0$ and $M_\alpha < \infty$,
    and $\bSigma \in \RR^{p \times p}$ is a positive semidefinite matrix with $r_{\min} \bI_p \preceq \bSigma \preceq r_{\max} \bI_p$
    for some constants $r_{\min} > 0$ and $r_{\max} < \infty$ that is independent of $p$.
    Let $\bX \in \RR^{n \times p}$ be the concatenated random matrix with $\bx_i$ for $i \in [n]$ as its rows, and define $\hSigma:=\bX^\top \bX / n \in \RR^{p \times p}$.
    Let $\gamma:= p / n$.
    Then, for $\lambda > 0$,
    as $n, p \to \infty$
    with $0 < \liminf \gamma \le \limsup \gamma < \infty$,
    the following asymptotic equivalents hold:
    \begin{enumerate}[label={(\arabic*)},leftmargin=7mm]
        \item\label{eq:detequi-ridge-genbias}  Bias of ridge regression:
        \begin{equation}
            \lambda^2
            (\hSigma + \lambda \bI_p)^{-1} \bSigma (\hSigma + \lambda \bI_p)^{-1}
            \asympequi 
            (v(-\lambda; \gamma) \bSigma + \bI_p)^{-1}
            (\tv_b(-\lambda; \gamma) \bSigma + \bSigma)
            (v(-\lambda; \gamma) \bSigma + \bI_p)^{-1}.
        \end{equation}

        \item\label{eq:detequi-ridge-genvar} Variance of ridge regression:
        \begin{equation}
            (\hSigma + \lambda \bI_p)^{-2} \hSigma \bSigma
            \asympequi
            \tv_v(-\lambda; \gamma) (v(-\lambda; \gamma) \bSigma + \bI_p)^{-2} \bSigma\bSigma.
        \end{equation}
    \end{enumerate}
    Here $v(-\lambda; \gamma,\bSigma) > 0$
        is the unique solution to the fixed-point equation:
        \begin{equation}
            \label{eq:def-v-ridge}
            \frac{1}{v(-\lambda; \gamma,\bSigma)}
            = \lambda+\int \frac{\gamma r}{1+v(-\lambda; \gamma,\bSigma) r} \, \rd H_n(r;\bSigma),
        \end{equation}
        and $\tv_b(-\lambda; \gamma,\bSigma)$ and $\tv_v(-\lambda; \gamma,\bSigma)$
        are defined through $v(-\lambda; \gamma,\bSigma)$ by the following equations:
        \begin{align}
            \tv_b(-\lambda; \gamma,\bSigma)
            &=
            \ddfrac{\int \gamma r^2(1+v(-\lambda; \gamma,\bSigma) r)^{-2} \, \rd H_n(r;\bSigma)}{v(-\lambda; \gamma,\bSigma)^{-2}- \int \gamma r^2(1+v(-\lambda; \gamma,\bSigma) r)^{-2} \, \rd H_n(r;\bSigma)}
            , \label{eq:def-v-b-ridge}\\
            \tv_v(-\lambda; \gamma,\bSigma)^{-1}
           & = 
                v(-\lambda; \gamma,\bSigma)^{-2}
                - \int \gamma r^2(1+v(-\lambda; \gamma,\bSigma) r)^{-2} \, \rd H_n(r;\bSigma),
                \label{eq:def-tv-v-ridge}
        \end{align}
    where $H_n(\cdot;\bSigma)$ is the empirical distribution 
    of the eigenvalues of $\bSigma$ that is supported on $[r_{\min}, r_{\max}]$.
\end{lemma}

Two remarks on \Cref{lem:deter-approx-generalized-ridge} follow: 

\begin{enumerate}[(a)]
    \item 
    The dependency of various scalar parameters appearing in \Cref{lem:deter-approx-generalized-ridge}
    on the matrix $\bSigma$ is explicitly highlighted the statement.
    This is because when we extend the current results later in \Cref{lem:deter-approx-ridge-extend}, 
    these parameters depend
    on the distributions of eigenvalues of matrices other than $\bSigma$.
    In places where it is clear from context,
    we will write $H_n(r)$, $v(-\lambda; \gamma)$, $\tv_b(-\lambda; \gamma)$, and $\tv_v(-\lambda; \gamma)$ to denote $H_n(r;\bSigma)$, $v(-\lambda; \gamma,\bSigma)$, $\tv_b(-\lambda; \gamma,\bSigma)$, and $\tv_v(-\lambda; \gamma,\bSigma)$, respectively, for notational simplicity.
    \item
    \Cref{lem:basic-ridge-resolvent-equivalent,lem:deter-approx-generalized-ridge} assume existence of moments of order $8 + \alpha$ for some $\alpha > 0$
    on the entries of $\bz_i$, $1 \le i \le k_m$, mentioned in assumption \ref{asm:lin-mod}.
    As done in the proof of Theorem 6 of \cite{hastie2022surprises} in Appendix A.4 of that paper, this can be relaxed to only requiring the existence of moments of order $4 + \alpha$ by a truncation argument. 
    We omit the details.
\end{enumerate}

\subsubsection{Extended ridge resolvents}
\label{sec:asympequi-extended-ridge-resolvents}

The subsequent lemma provides an extension to the asymptotic equivalents presented in \Cref{lem:deter-approx-generalized-ridge}. 
It supplies asymptotic equivalents for Tikhonov resolvents, in which the regularization matrix $\lambda \bI_p$ is substituted by $\lambda(\bI_p+ \bC)$.
Here $\bC\in\RR^{p\times p}$ represents an arbitrary positive semidefinite random matrix.
While the derivation of extended resolvents in \Cref{lem:deter-approx-ridge-extend} naturally follows from \Cref{lem:deter-approx-generalized-ridge}, we have specifically isolated these extensions. This abstraction facilitates their repeated application in our conditioning arguments, especially in the proofs of \Cref{thm:ver-with-replacement,thm:ver-without-replacement}.

\begin{lemma}[Asymptotic equivalents for Tikhonov resolvents]\label{lem:deter-approx-ridge-extend}
    Suppose the conditions in Lemma \ref{lem:deter-approx-generalized-ridge} holds.
    Let $\bC\in\RR^{p\times p}$ be any symmetric and positive semidefinite random matrix with uniformly bounded operator norm in $p$ that is independent of $\bX$ for all $n,p\in\NN$. Let $\bN=(\hSigma + \lambda \bI_p)^{-1}$.
    Then the following asymptotic equivalences hold:
    \begin{enumerate}[label={(\arabic*)},leftmargin=7mm]
        \item \label{eq:lem:deter-approx-ridge-extend}Tikhonov resolvent:
        \begin{align}
            \lambda(\bN^{-1} + \lambda\bC )^{-1} &\asympequi \tilde{\bSigma}_{\bC}^{-1}. 
        \end{align}
        
        \item \label{eq:lem:deter-approx-ridge-extend-gbias}Bias of Tikhonov regression:
        \begin{align}
            \lambda^2 (\bN^{-1} + \lambda\bC )^{-1}\bSigma (\bN^{-1} + \lambda\bC )^{-1} &\asympequi    \tilde{\bSigma}_{\bC}^{-1} (\tv_b(-\lambda; \gamma,\bSigma_{\bC})\bSigma+\bSigma)\tilde{\bSigma}_{\bC}^{-1}.
        \end{align}
        
        \item \label{eq:lem:deter-approx-ridge-extend-gvar}Variance of Tikhonov regression:
        \begin{align}
            (\bN^{-1} + \lambda\bC )^{-1}\hSigma(\bN^{-1} + \lambda\bC )^{-1}\bSigma  &\asympequi \tv_v(-\lambda; \gamma,\bSigma_{\bC})\tilde{\bSigma}_{\bC}^{-1}\bSigma\tilde{\bSigma}_{\bC}^{-1}\bSigma,
        \end{align}
        
    \end{enumerate}
    where $\bSigma_{\bC}= (\bI_p+ \bC)^{-\tfrac{1}{2}}\bSigma  (\bI_p+ \bC)^{-\tfrac{1}{2}}$, $\tilde{\bSigma}_{\bC}=v(-\lambda; \gamma,\bSigma_{\bC}) \bSigma + \bI_p+\bC $.
    Here $v(-\lambda; \gamma,\bSigma_{\bC})$, $\tilde{v}_b(-\lambda; \gamma,\bSigma_{\bC}),$ and $\tilde{v}_v(-\lambda; \gamma,\bSigma_{\bC})$ defined by \eqref{eq:def-v-ridge}-\eqref{eq:def-tv-v-ridge} simplify to the following:
    \begin{align}
        \tfrac{1}{v(-\lambda; \gamma,\bSigma_{\bC})}
             &= \lambda
            + \gamma \tr[ (v(-\lambda; \gamma,\bSigma_{\bC}) \bSigma + \bI_p +\bC)^{-1}\bSigma] / p, \label{eq:def-v-c-ridge} \\
            \tfrac{1}{\tilde{v}_v(-\lambda; \gamma,\bSigma_{\bC})} &= \tfrac{1}{v(-\lambda; \gamma,\bSigma_{\bC})^2} - \gamma \tr[ (v(-\lambda; \gamma,\bSigma_{\bC}) \bSigma + \bI_p +\bC)^{-2}\bSigma^2] / p, \label{eq:def-tv-v-c-ridge}\\
            \tilde{v}_b(-\lambda; \gamma,\bSigma_{\bC}) &=  \gamma \tr[ (v(-\lambda; \gamma,\bSigma_{\bC}) \bSigma + \bI_p +\bC)^{-2}\bSigma^2] / p \cdot \tilde{v}_v(-\lambda; \gamma,\bSigma_{\bC}). \label{eq:def-tv-b-c-ridge}
    \end{align}
    If $\gamma\rightarrow\phi\in(0,\infty)$, then $\gamma$ in \ref{eq:lem:deter-approx-ridge-extend}-\ref{eq:lem:deter-approx-ridge-extend-gvar} can be replaced by $\phi$, with the empirical distribution $H_n$ of eigenvalues of $\bSigma$ being replaced by the limiting distribution $H$.
    \end{lemma}
    \begin{proof}[Proof of Lemma \ref{lem:deter-approx-ridge-extend}]
    Proofs for the different parts are separated below. 
    
        \paragraph{Part (1).\hspace{-2mm}} Note that
        \begin{align}
            \lambda(\bN^{-1} + \lambda\bC )^{-1} &= \lambda (\hSigma + \lambda(\bI_p+ \bC) )^{-1}
            = (\bI_p+ \bC)^{-\tfrac{1}{2}}
            \lambda(\hSigma_{\bC}+\lambda\bI_p)^{-1}
            (\bI_p+ \bC)^{-\tfrac{1}{2}}, \label{eq:lem:deter-approx-ridge-extend-1}
        \end{align}
        where $\hSigma_{\bC}=\bSigma_{\bC}^{\tfrac{1}{2}}(\bZ^{\top}\bZ/n)\bSigma_{\bC}^{\tfrac{1}{2}}$, and $\bSigma_{\bC} = (\bI_p+ \bC)^{-\tfrac{1}{2}}\bSigma  (\bI_p+ \bC)^{-\tfrac{1}{2}}$.
        Using \Cref{lem:basic-ridge-resolvent-equivalent}, we have
        \begin{align}
            \label{eq:lem:deter-approx-ridge-extend-2}
             \lambda (\hSigma_{\bC} + \lambda \bI_p)^{-1}
            \asympequi
            (v(-\lambda; \gamma,\bSigma_{\bC}) \bSigma_{\bC} + \bI_p)^{-1},
        \end{align}
        where $v(-\lambda; \gamma,\bSigma_{\bC})$ is the unique solution to the fixed point equation \eqref{eq:basic-ridge-equivalence-v-fixed-point} such that
        $$\tfrac{1}{v(-\lambda; \gamma,\bSigma_{\bC})}
            =\lambda
            + \gamma \tr[\bSigma_{\bC} (v(-\lambda; \gamma,\bSigma_{\bC}) \bSigma_{\bC} + \bI_p)^{-1}] / p = \lambda
            + \gamma \tr[\bSigma (v(-\lambda; \gamma,\bSigma_{\bC}) \bSigma + \bI_p +\bC)^{-1}] / p .$$
        Note that $\norm{(\bI_p+ \bC)^{-1}}_{\oper}\leq 1$. 
        We can apply the product rule from \Cref{lem:calculus-detequi}~\ref{lem:calculus-detequi-item-product} and get
        \begin{align*}
            \lambda(\bN^{-1} + \lambda\bC )^{-1} & \asympequi  (\bI_p+ \bC)^{-\tfrac{1}{2}}
            (v(-\lambda; \gamma,\bSigma_{\bC}) \bSigma_{\bC} + I_p)^{-1}
            (\bI_p+ \bC)^{-\tfrac{1}{2}} = (v(-\lambda; \gamma,\bSigma_{\bC})\bSigma +\bI_p+ \bC)^{-1},
        \end{align*}
        by combining  \eqref{eq:lem:deter-approx-ridge-extend-1}-\eqref{eq:lem:deter-approx-ridge-extend-2}.
        
        \paragraph{Part (2).\hspace{-2mm}} 
        From \Cref{lem:deter-approx-generalized-ridge}~\ref{eq:detequi-ridge-genbias}, we have
        \begin{align*}
            &\lambda^2 (\bN^{-1} + \lambda\bC )^{-1}\bSigma (\bN^{-1} + \lambda\bC )^{-1} 
            \\
            &= \lambda^2 (\bI_p+ \bC)^{-\tfrac{1}{2}}
            \cdot [(\hSigma_{\bC}+\lambda\bI_p)^{-1}
            (\bI_p+ \bC)^{-\tfrac{1}{2}}\cdot \bSigma \cdot  (\bI_p+ \bC)^{-\tfrac{1}{2}}
            (\hSigma_{\bC}+\lambda\bI_p)^{-1}
            ]\cdot(\bI_p+ \bC)^{-\tfrac{1}{2}}\\
            &\asympequi (\bI_p+ \bC)^{-\tfrac{1}{2}}\cdot [(v(-\lambda; \gamma,\bSigma_{\bC}) \bSigma_{\bC} + \bI_p)^{-1}\\
            &\qquad 
            \cdot (\tv_b(-\lambda; \gamma,\bSigma_{\bC}) \bSigma_{\bC} + (\bI_p+ \bC)^{-\tfrac{1}{2}}\bSigma (\bI_p+ \bC)^{-\tfrac{1}{2}})
            \cdot (v(-\lambda; \gamma,\bSigma_{\bC}) \bSigma_{\bC} + \bI_p)^{-1}]\cdot (\bI_p+ \bC)^{-\tfrac{1}{2}}\\
            &= (v(-\lambda; \gamma,\bSigma_{\bC}) \bSigma + \bI_p + \bC)^{-1} (\tv_b(-\lambda; \gamma,\bSigma_{\bC})\bSigma+\bSigma) (v(-\lambda; \gamma,\bSigma_{\bC}) \bSigma + \bI_p + \bC)^{-1}.
        \end{align*}
        
        \paragraph{Part (3).\hspace{-2mm}}
        Similar to Part (2),
        from \Cref{lem:deter-approx-generalized-ridge} ~\ref{eq:detequi-ridge-genvar}, we have
        \begin{align*}
            &(\bN^{-1} + \lambda\bC )^{-1}\hSigma(\bN^{-1} + \lambda\bC )^{-1}\bSigma \\
            &=  (\bI_p+ \bC)^{-\tfrac{1}{2}}\cdot 
             (\hSigma_{\bC}+\lambda\bI_p)^{-1} \hSigma_{\bC} (\hSigma_{\bC}+\lambda\bI_p)^{-1}\cdot(\bI_p+ \bC)^{-\tfrac{1}{2}}\bSigma \\
            &\asympequi
             (\bI_p+ \bC)^{-\tfrac{1}{2}} \cdot \tv_v(-\lambda; \gamma,\bSigma_{\bC}) (v(-\lambda; \gamma,\bSigma_{\bC}) \bSigma_{\bC} + \bI_p)^{-1} \bSigma_{\bC} (v(-\lambda; \gamma,\bSigma_{\bC}) \bSigma_{\bC} + \bI_p)^{-1} \cdot(\bI_p+ \bC)^{-\tfrac{1}{2}}\bSigma\\
            &= \tv_v(-\lambda; \gamma,\bSigma_{\bC})(v(-\lambda; \gamma,\bSigma_{\bC}) \bSigma + \bI_p+\bC )^{-1}\bSigma(v(-\lambda; \gamma,\bSigma_{\bC}) \bSigma + \bI_p+\bC )^{-1}\bSigma.
        \end{align*}
        Note that the distribution of $\bSigma_{\bC}$'s eigenvalue has positive support.
        By the continuity of $v(-\lambda;\cdot,\bSigma_{\bC})$, $\tilde{v}_b(-\lambda;\cdot,\bSigma_{\bC})$, and $\tilde{v}_v(-\lambda;\cdot,\bSigma_{\bC})$ from Lemma \ref{lem:ridge-fixed-point-v-properties}~\ref{lem:ridge-fixed-point-v-properties-item-v-properties}, \ref{lem:ridge-fixed-point-v-properties-item-vb-properties} and \ref{lem:ridge-fixed-point-v-properties-item-vv-properties}, $\gamma$ can by replaced by its limit $\phi$ as $n,p \to \infty$. 
    \end{proof}
    
The following lemma concerns the deterministic equivalents of the precision matrix as the weighted average of two sample covariance matrices of subsamples, when the full sample covariance matrix is invertible almost surely.
It is useful when we aim to condition on one of the subsampled covariance matrix, which is used in the proof of \Cref{lem:ridgeless-V0-phis_lt1}.

\begin{lemma}[Deterministic equivalent of subsamples in the underparameterized regime]\label{lem:det-equiv-subsample}
    Suppose the conditions in Lemma \ref{lem:deter-approx-generalized-ridge} holds.
    Let $\hSigma_0$ be the sample covariance matrix computed using $i$ observations of $\bX$,  
    and $\hSigma_1$ be the sample covariance matrix computed using the remaining $n-i$ samples.
    Let $\pi_0 = i / n$ and $\pi_1 = (n - i) / n$.
    Suppose that $p/n\rightarrow\phi\in(0,1)$ as $n,p\rightarrow\infty$.
    Then, we have
    \[
        (\pi_0 \hSigma_0 + \pi_1 \hSigma_1)^{-1}
        \asympequi (\pi_0\hSigma_0 + (1-\phi)\pi_1 \bSigma)^{-1}.
    \]
\end{lemma}     
\begin{proof}[Proof of \Cref{lem:det-equiv-subsample}]
    We first note that when $\phi\in(0,1)$, the eigenvalues of $\hSigma = \pi_0\hSigma_0 + \pi_1 \hSigma_1$ are bounded away from zero almost surely \citep{bai2010spectral} and hence the inverse is well defined almost surely as $n,p\rightarrow\infty$.

    The idea for the proof is to consider the perturbed resolvent
    $(\pi_0\hSigma_0 + \mu \bI_p + \pi_1 \hSigma_1)^{-1}$ for $\mu>0$.
    Note that since the matrix $(\pi_0\hSigma_0 + \pi_1 \hSigma_1)$ is almost surely invertible.
    Then, 
    \[
    \lim_{\mu \to 0^+} (\pi_0\hSigma_0 + \mu \bI_p + \pi_1 \hSigma_1)^{-1} 
    = (\pi_0\hSigma_0 + \pi_1 \hSigma_1)^{-1}.
    \]
    
    Conditioned on $(\pi_0\hSigma_0 + \mu \bI_p)$, we have
    \begin{align*}
        (\pi_0\hSigma_0 + \mu\bI_p + \pi_1 \hSigma_1)^{-1}
        &= a (\bA + \hSigma_1)^{-1} \\
        &= a \bA^{-\tfrac{1}{2}} (\bI_p + \bA^{-\tfrac{1}{2}} \hSigma_1 \bA^{-\tfrac{1}{2}}) \bA^{-\tfrac{1}{2}} \\
        &= a \bA^{-\tfrac{1}{2}} (\bI_p + \hSigma_{1, \bA})^{-1} \bA^{-\tfrac{1}{2}} \\
        &\asympequi a \bA^{-\tfrac{1}{2}} (\bI_p + c \bSigma_{\bA}) \bA^{-\tfrac{1}{2}} \\
        &= a (\bA + c \bSigma)^{-1} \\
        &= (\pi_0\hSigma_0 + \mu\bI_p + c \pi_1 \bSigma)^{-1},
    \end{align*}
    where the intermediate constants are $a = \pi_1^{-1}$, $\bA = a \pi_0\hSigma_0 + a \mu\bI_p$,
    $\hSigma_{1, \bA} = \bA^{-\tfrac{1}{2}} \hSigma \bA^{-\tfrac{1}{2}}$, $\bSigma_{\bA} = \bA^{-\tfrac{1}{2}} \bSigma \bA^{-\tfrac{1}{2}}$,
    and $c$ satisfy the fixed-point equation
    \begin{align*}
        \tfrac{1}{c}
        &= 1 + \tfrac{p}{n - i} \tr[\bSigma_{\bA} (c \bSigma_{\bA} + \bI_p)^{-1}] / p \\
        &= 1 + \tfrac{p}{k} \tfrac{k}{n - i} \tr[\bA^{-1/2} \bSigma \bA^{-1/2} (c \bA^{-1/2} \bSigma \bA^{-1/2} + \bI_p)^{-1}] / p \\
        &= 1 + \phi a \tr[\bSigma (c \bSigma + \bA)^{-1}] / p \\
        &= 1 + \phi \tr[\bSigma (c \pi_1 \bSigma + \pi_0\hSigma_0 + \mu\bI_p)^{-1}] / p \\
        &= 1 + \phi  \tr[\bSigma (\pi_0\hSigma_0 + \mu\bI_p + \pi_1 \hSigma_1)^{-1}] / p,
    \end{align*}
    where in final equality, we used the trace property of the asymptotic equivalence
    \[
    (\pi_0\hSigma_0 + \mu\bI_p + c \pi_1 \bSigma)^{-1} \asympequi (\pi_0\hSigma_0 + \mu\bI_p + \pi_1 \hSigma_1)^{-1}.
    \]
    Now note that
    \[
        (\pi_0\hSigma_0 + \mu\bI_p + \pi_1 \hSigma_1)^{-1}
        = (\hSigma + \mu\bI_p)^{-1}
        \asympequi (c' \bSigma + \mu\bI_p)^{-1}
    \]
    where $c'$ solves the fixed-point equation
    \[
        \tfrac{1}{c'} = 1 + \phi \tr[\bSigma (c' \bSigma + \mu\bI_p)^{-1}] / p.
    \]
    Thus, the fixed-point in $c$ can be written as
    \[
        \tfrac{1}{c}
        = 1 + \phi \tr[\bSigma (c' \bSigma + \mu\bI_p)^{-1}] / p.
    \]
    We note that $c = c'$ satisfy the fixed-point equation for $c$
    (from the fixed-point equation for $c'$).
    Since $c$ is a unique solution, this must be the solution.
    Letting $\mu \to 0^{+}$, we observe that $c' = 1 - \phi$ is the solution
    for the fixed-point equation in $c'$.
    Thus, we also have $c = 1 - \phi$.
\end{proof}

\subsection{Analytic properties of associated fixed-point equations}

In this section, we compile results regarding analytical properties of the fixed-point solution $v(-\lambda;\phi)$ as defined in \eqref{eq:basic-ridge-equivalence-v-fixed-point}.

The subsequent lemma affirms the existence and uniqueness of the solution $v(-\lambda;\phi)$. It establishes a connection between the properties of the derivatives described in \Cref{lem:properties-sol} and the properties of $\tv_v(-\lambda;\phi)$ as defined in \eqref{eq:def-tv-v-ridge}.
Note the latter equals $-f'(x)$, where the function $f$ is as defined in \eqref{eq:lem:properties-sol:fx}

\begin{lemma}[Properties of the solution of the fixed-point equation]\label{lem:properties-sol}
    Let $\lambda,\phi,a > 0$ and $b < \infty$ be real numbers.
    Let $P$ be a probability measure supported
    on $[a, b]$.
    Define the function as follows:
    \begin{align}
        f(x) = \tfrac{1}{x} - \phi \int\tfrac{r}{1 + r x}\rd P(r) - \lambda .\label{eq:lem:properties-sol:fx}
    \end{align}
    Then the following properties hold:
    \begin{enumerate}[label={(\arabic*)},leftmargin=7mm]
        \item \label{lem:properties-sol-item-f-ridgeless}For $\lambda=0$ and $\phi\in(1,\infty)$, there is a unique $x_0\in(0,\infty)$ such that $f(x_0)=0$. The function $f$ is positive and strictly decreasing
        over $(0, x_0)$ and negative
        over $(x_0,\infty)$, with $\lim_{x \to 0^{+}} f(x) = \infty$ and $\lim_{x \to \infty} f(x) = 0$.
        
        \item \label{lem:properties-sol-item-f-ridge}For $\lambda>0$ and $\phi\in(0,\infty)$, there is a unique $x_0^{\lambda}\in(0,\infty)$ such that $f(x_0^{\lambda})=0$. The function $f$ is positive and strictly decreasing
        over $(0, x_0^{\lambda})$ and negative
        over $(x_0^{\lambda},\infty)$, with $\lim_{x \to 0^{+}} f(x) = \infty$ and $\lim_{x \to \infty} f(x) = -\lambda$.

        \item \label{lem:properties-sol-item-f'-ridgeless}For $\lambda=0$ and $\phi\in(1,\infty)$, $f$ is differentiable on $(0,\infty)$ and its derivative $f'$ is strictly increasing over $(0, x_0)$,
        with $\lim_{x \to 0^{+}} f'(x) = - \infty$
        and $f'(x_0) < 0$.
        
        \item \label{lem:properties-sol-item-f'-ridge}For $\lambda>0$ and $\phi\in(0,\infty)$, $f$ is differentiable on $(0,\infty)$ and its derivative $f'$ is strictly increasing over $(0, \infty)$,
        with $\lim_{x \to 0^{+}} f'(x) = - \infty$
        and $f'(x_0^{\lambda}) < 0$.
        
    \end{enumerate}
\end{lemma}
\begin{proof}[Proof of \Cref{lem:properties-sol}]
    We consider different parts separately below.
    
    \paragraph{Part (1).\hspace{-2mm}}
    Observe that
        \[
            f(x)
            =
            \tfrac{1}{x}
            - \phi \int \tfrac{r}{xr + 1}
            \, \mathrm{d}P(r)
            = g_1(x)h_1(x),
        \]
        where $$g_1(x)=\tfrac{1}{x},\qquad h_1(x)= 1 - \phi \int \tfrac{xr}{xr + 1}
            \, \mathrm{d}P(r).$$
        Note that $g_1$ is positive and strictly decreasing over $(0, \infty)$
        with $\lim_{x \to 0^{+}} g_1(x) = \infty$ and $\lim_{x \to \infty} g_1(x) = 0$,
        while $h_1$ is strictly decreasing
        over $(0, \infty)$ with $h_1(0) = 1$ and $\lim_{x \to \infty} h_1(x) = 1 - \phi < 0$.
        Thus, there is a unique $0 < x_0 < \infty$ such that $h_1(x_0) = 0$,
        and consequently $f(x_0) = 0$.
        Because $h_1$ is positive over
        $(0, x_0)$, and negative over $(x_0,\infty)$, $f$ is positive strictly decreasing over $(0, x_0)$ and negative over $(x_0, \infty)$, with $\lim_{x \to 0^{+}} f(x) = \infty$ and $\lim_{x \to \infty} f(x) = 0$.
       
       \paragraph{Part (2).\hspace{-2mm}} 
        Note that $f(x)=g_1(x)h_1(x)-\lambda$. Since from (1) $\lim_{x\rightarrow0}g_1(x)h_1(x)=\infty$ and $\lim_{x\rightarrow0}g_1(x)h_1(x)=0$, we have that $\lim_{x \to 0^+} f(x) = +\infty$ and $\lim_{x \to \infty} f(x) = -\lambda<0$.
        
        For $\phi>1$, since $g_1(x)h_1(x)$ is positive and strictly decreasing over $(0,x_0)$ and negative over $(x_0,\infty)$, and $\lim_{x\rightarrow0^+}g_1(x)h_1(x)=\infty$, we have that there exists $x_0^{\lambda}\in(0,x_0)$ such that $f(x_0^{\lambda})=0$.
        The properties of $f$ over $(0,x_0^{\lambda})$ and $(x_0^{\lambda},\infty)$ follow analogously as in (1).
        
        For $\phi\in(0,1]$, since $g_1h_1$ is continuous, positive and strictly decreasing over $(0,\infty)$, by intermediate value theorem, there exists $x_0^{\lambda}\in(0,\infty)$ such that $f(x_0^{\lambda})=0$, $f$ is positive and strictly decreasing for $x<x_0^{\lambda}$ and negative for $x>x_0^{\lambda}$, with $\lim_{x \to 0^{+}} f(x) = \infty$ and $\lim_{x \to \infty} f(x) = -\lambda$.

       \paragraph{Part (3).\hspace{-2mm}} 
        Since $f$ is monotone and continuous, it is differentiable. The derivative $f'$ at $x$ is given by
    \[
        f'(x)
        = - \tfrac{1}{x^2}
        + \phi \int \tfrac{r^2}{(x r + 1)^2}
        \, \mathrm{d}P(r) =- g_2(x)h_2(x),
    \]
    where
    \begin{align*}
        g_2(x)=\tfrac{1}{x^2},\qquad h_2(x)=
        \left(
            1 - \phi \int \left(\tfrac{xr}{xr + 1}\right)^2
            \, \mathrm{d}P(r).
        \right)
    \end{align*}
    Note that the function $g_2$ is positive and strictly decreasing over $(0, \infty)$ with $\lim_{x \to 0^{+}} g_2(x) = \infty$ and $\lim_{x \to \infty} g_2(x) = 0$.
    On the other hand, the function $h_2$ is strictly decreasing over $(0, \infty)$ with $h_2(0) = 1$ and $h_2(x_0) > 0$.
    This follows because for $x \in (0, x_0]$,
    \begin{align*}
            \label{eq:bound-deriv-v-in-lambda-part-2}
            \phi \int \left( \tfrac{xr}{xr + 1} \right)^2 \, \mathrm{d}P(r)
            &\le \tfrac{x_0 b}{x_0 b + 1}
            \phi 
            \int \left( \tfrac{xr}{xr + 1} \right) \, \mathrm{d}P(r) <
             \int \tfrac{\phi xr}{xr + 1} \, \mathrm{d}P(r)
            \le
             \int \tfrac{\phi x_0 r}{x_0 r + 1} \, \mathrm{d}P(r)
            =
            1,
    \end{align*}
    where the first inequality in the chain above
    follows as the support of $P$ is $[a, b]$,
    and the last inequality follows since $f(x_0) = 0$
    and $x_0 > 0$,
    which implies that
    \[
        \tfrac{1}{x_0}
        = \phi \int \tfrac{r}{x_0 r + 1} \, \mathrm{d}P(r),
        \quad
        \text{ or equivalently that}
        \quad
        1
        = \phi \int \tfrac{x_0 r}{x_0 r + 1} \, \mathrm{d}P(r).
    \]
    Thus, $-f'$, a product of two positive strictly decreasing functions,
    is strictly decreasing, and in turn, $f'$ 
    is strictly increasing.
    Moreover, $\lim_{x \to 0^{+}} f'(x) = -\infty$ and $f'(x_0) < 0$.
    
    \paragraph{Part (4).\hspace{-2mm}}
    The conclusion follows by noting that $h_2(x_0^{\lambda})> h_2(x_0) >0$ from (3).
\end{proof}

\Cref{lem:ridge-fixed-point-v-properties} provides the continuity and limiting behavior of the function $\phi\mapsto v(-\lambda;\phi)$ for ridge regression ($\lambda>0$). 
\Cref{lem:fixed-point-v-properties} does the same for ridgeless regression ($\lambda=0$).

\begin{lemma}[Continuity properties in the aspect ratio for ridge regression]
    \label{lem:ridge-fixed-point-v-properties}
    Let $\lambda,a > 0$ and $b < \infty$ be real numbers.
    Let $P$ be a probability measure supported
    on $[a, b]$.
    Consider the function $v(-\lambda; \cdot) : \phi \mapsto v(-\lambda; \phi)$,
    over $(0, \infty)$,
    where $v(-\lambda; \phi) > 0$ is the unique solution to the following fixed-point equation:
    \begin{equation}
        \label{eq:ridge-fixed-point-gen-phi}
        \tfrac{1}{v(-\lambda; \phi)}
         = \lambda + \phi \int \tfrac{r}{1+rv(-\lambda; \phi)} \rd P(r)
    \end{equation}
    Then the following properties hold:
    \begin{enumerate}[label={(\arabic*)},leftmargin=7mm]
        \item 
        \label{lem:ridge-fixed-point-v-properties-item-v-bound}
        The range of the function $v(-\lambda; \cdot)$ is a subset of $(0,\lambda^{-1})$.
        
        \item 
        \label{lem:ridge-fixed-point-v-properties-item-v-properties}
        The function $v(-\lambda; \cdot)$ is continuous and strictly decreasing over $(0, \infty)$. Furthermore, $\lim_{\phi \to 0^{+}} v(-\lambda; \phi) = \lambda^{-1}$, and $\lim_{\phi \to \infty} v(-\lambda; \phi) = 0$.

        \item
        \label{lem:ridge-fixed-point-v-properties-item-vv-properties}
        The function 
        $\tv_v(-\lambda; \cdot) : \phi \mapsto \tv_v(-\lambda; \phi)$,
        where
        \[
           \tv_v(-\lambda; \phi)
           =
           \left(
                v(-\lambda; \phi)^{-2}
                - \int \phi  r^2(1 + r v(-\lambda; \phi))^{-2} 
                \, \mathrm{d}P(r)
           \right)^{-1},
        \]
        is positive and continuous over $(0, \infty)$.
        Furthermore,
        $\lim_{\phi \to 0^{+}} \tv_v(-\lambda; \phi) = \lambda^{-2}$,
        and $\lim_{\phi \to \infty} \tv_v(-\lambda; \phi) = 0$.
        
        \item
        \label{lem:ridge-fixed-point-v-properties-item-vb-properties}
        The function 
        $\tv_b(-\lambda; \cdot) : \phi \mapsto \tv_b(-\lambda; \phi)$,
        where
        \[
            \tv_b(-\lambda; \phi)
            = \tv_v(-\lambda; \phi)
            \int
            \phi r^2(1 + v(-\lambda; \phi) r)^{-2}
            \, \mathrm{d}P(r),
        \]
        is positive and continuous over $(0, \infty)$.
        Furthermore,
        $\lim_{\phi \to 0^{+}} \tv_b(-\lambda; \phi) =\lim_{\phi \to \infty} \tv_b(-\lambda; \phi) = 0$.
    \end{enumerate}
\end{lemma}
\begin{proof}[Proof of Lemma \ref{lem:ridge-fixed-point-v-properties}]
Proofs for the different parts appear below.

  \paragraph{Part (1).\hspace{-2mm}} Since $P$ has positive support, we have
  \begin{align*}
        \tfrac{1}{v(-\lambda; \phi)}
         &= \lambda + \phi \int \tfrac{r}{1+rv(-\lambda; \phi)} \rd P(r) > \lambda,\\
         \tfrac{1}{v(-\lambda; \phi)}
         &= \lambda + \phi \int \tfrac{r}{1+rv(-\lambda; \phi)} \rd P(r)< \lambda+\phi b
  \end{align*}
  which implies that $0<(\lambda+\phi b)^{-1}<v(-\lambda;\phi)<\lambda^{-1}$.

    \paragraph{Part (2).\hspace{-2mm}}
  Rearranging \eqref{eq:ridge-fixed-point-gen-phi} yields
  \begin{align*}
      \tfrac{1}{\phi} &= \tfrac{1}{1-\lambda v(-\lambda;\phi)} \left(1-\int \tfrac{1}{1+rv(-\lambda;\phi)}\rd P(r)\right).
  \end{align*}
  From \citet[][Lemma S.6.13]{patil2022mitigating}, the function $$h_1:t\mapsto 1-\int \tfrac{1}{1+rt}\rd P(r)$$ is strictly increasing and continuous over $(0,\infty)$, $\lim_{t\rightarrow0}h_1(t)=0$, and $\lim_{t\rightarrow\infty}h_1(t)=1$.
  It is also positive from (1).
  Since $h_2:t\mapsto 1/(1-\lambda t)$ is positive, strictly increasing and continuous over $t\in(0,\lambda^{-1})$, we have that the function
  \begin{align*}
      T:t\mapsto \tfrac{1}{1-\lambda t}\left(1-\int \tfrac{1}{1+rt}\rd P(r)\right)
  \end{align*}
  is strictly increasing and continuous over $(0,\lambda^{-1})$.
  By the continuous inverse theorem, we have $T^{-1}$ is strictly increasing and continuous.
  Note that $v(-\lambda;\phi)=T^{-1}(\phi^{-1})$.
  Since $\phi \mapsto \phi^{-1}$ is continuous and strictly decreasing in $\phi$, 
  we have $\phi \mapsto v(-\lambda;\phi)$ is continuous and strictly decreasing over $\phi\in(0,\infty)$.
  Moreover, $\lim_{\phi\rightarrow0^+}T^{-1}(\phi^{-1})=\lambda^{-1}$ and $\lim_{\phi\rightarrow\infty}T^{-1}(\phi^{-1})=0$.

    \paragraph{Part (3).\hspace{-2mm}}
  From (2), $\phi \mapsto v(-\lambda;\phi)^{-2}$ is continuous in $\phi$ and 
  $$T_2:\phi\mapsto\phi\int \tfrac{r^2}{(1 + r v(-\lambda; \phi))^2} \, \mathrm{d}P(r)$$
  is also continuous in $\phi$.
  Thus, the function $\tv_v(-\lambda;\cdot)^{-1}$ is continuous.
  Note that
  \begin{align*}
      \tfrac{v(-\lambda; \phi)^2}{\tv_v(-\lambda;\phi)} & =     1  - \phi \int\tfrac{r^2v(-\lambda; \phi)^2}{(1+rv(-\lambda; \phi))^2}\rd P(r) > 1  - \phi \int\tfrac{rv(-\lambda; \phi)}{1+rv(-\lambda; \phi)}\rd P(r) = 0,
  \end{align*}
  where the inequality holds because $rv(-\lambda; \phi)/(1+rv(-\lambda; \phi))$ is strictly positive and $P(r)$ has positive support.
  Then we have that $\phi \mapsto \tv_v(-\lambda;\phi)^{-1}>0$ and $\tv_v(-\lambda;\cdot)$ is continuous over $(0,\infty)$.
  Since $\lim_{\phi\rightarrow0^+}v(-\lambda; \phi)=\lambda^{-1}$, it follows that $\lim_{\phi\rightarrow0^+}\tv_v(-\lambda; \phi)=\lambda^{-2}.$
  Similarly, from $\lim_{\phi\rightarrow\infty}v(-\lambda; \phi)=0$, $\lim_{\phi\rightarrow\infty}\phi v(-\lambda; \phi)=1$ and the fact that
\[
\lim_{\phi \to \infty}
\int \tfrac{r^2}{(1 + r v(-\lambda; \phi))^2} \, \mathrm{d}P(r)
\in [a^2,b^2],
\]
  it follows that $$\lim_{\phi\rightarrow\infty}\tv_v(-\lambda; \phi)=\lim_{\phi\rightarrow\infty}v(-\lambda; \phi)^2\cdot \left(
    1
    - v(-\lambda; \phi)\cdot \phi v(-\lambda; \phi) \cdot  \int r^2(1 + r v(-\lambda; \phi))^{-2} \, \mathrm{d}P(r) \right)^{-1}= 0.$$
  
    \paragraph{Part (4).\hspace{-2mm}}
  The continuity of $\tv_b(-\lambda;\cdot)$ follows from the continuity of $v(-\lambda;\cdot)$ and $\tv_v(-\lambda;\cdot)$.
  Note that
  \begin{align*}
      \tfrac{1}{1+\tv_b(-\lambda;\phi)} &= 1 - v(-\lambda;\phi)\cdot\phi v(-\lambda;\phi)\cdot \int \tfrac{r^2}{(1 + r v(-\lambda; \phi))^2} \, \mathrm{d}P(r).
  \end{align*}
  From the proof in (3), we have 
  \begin{align*}
        \lim_{\phi\rightarrow0^+}\tfrac{1}{1+\tv_b(-\lambda;\phi)} &= 1 - \lim_{\phi\rightarrow0^+}v(-\lambda;\phi)\cdot\phi v(-\lambda;\phi)\cdot \int \tfrac{r^2}{(1 + r v(-\lambda; \phi))^2} \, \mathrm{d}P(r)=1\\
      \lim_{\phi\rightarrow\infty}\tfrac{1}{1+\tv_b(-\lambda;\phi)} &= 1 - \lim_{\phi\rightarrow\infty}v(-\lambda;\phi)\cdot\phi v(-\lambda;\phi)\cdot \int \tfrac{r^2}{(1 + r v(-\lambda; \phi))^2} \, \mathrm{d}P(r)=1
  \end{align*}
  and thus, $\lim_{\phi\rightarrow0^+}\tv_b(-\lambda;\phi)=\lim_{\phi\rightarrow\infty}\tv_b(-\lambda;\phi)=0$.
\end{proof}

\begin{lemma}[Continuity properties in the aspect ratio for ridgeless regression, adapted from \citet{patil2022mitigating}]
    \label{lem:fixed-point-v-properties}
    Let $a > 0$ and $b < \infty$ be real numbers.
    Let $P$ be a probability measure supported
    on $[a, b]$.
    Consider the function $v(0; \cdot) : \phi \mapsto v(0; \phi)$,
    over $(1, \infty)$,
    where $v(0; \phi) > 0$ is the unique solution to the followinn fixed-point equation:
    \begin{equation}
       \label{eq:fixed-point-gen-phi}
        \tfrac{1}{\phi}
        = \int \tfrac{v(0; \phi) r}{1 + v(0; \phi) r} \, \mathrm{d}P(r).
    \end{equation}
    Then
    the following properties hold:
    \begin{enumerate}[label={(\arabic*)},leftmargin=7mm]
        \item 
        \label{lem:fixed-point-v-properties-item-v-properties}
        The function $v(0; \cdot)$ is continuous 
        and strictly decreasing over $(1, \infty)$.
        Furthermore, $\lim_{\phi \to 1^{+}} v(0; \phi) = \infty$,
        and $\lim_{\phi \to \infty} v(0; \phi) = 0$.
        
        \item \label{lem:fixed-point-v-properties-item-phivinverse-properties}
        The function 
        $\phi \mapsto (\phi v(0; \phi))^{-1}$
        is strictly increasing over $(1, \infty)$.
        Furthermore,
        $\lim_{\phi \to 1^{+}} (\phi v(0; \phi))^{-1} = 0$
        and $\lim_{\phi \to \infty} (\phi v(0; \phi))^{-1} = 1$.
        
        \item
        
        \label{lem:fixed-point-v-properties-item-vv-properties}
        The function 
        $\tv_v(0; \cdot) : \phi \mapsto \tv_v(0; \phi)$,
        where
        \[
           \tv_v(0; \phi)
           =
           \left(v(0; \phi)^{-2}
                - \phi
                \int r^2(1 + r v(0; \phi))^{-2} 
                \, \mathrm{d}P(r)
           \right)^{-1},
        \]
        is positive and continuous over $(1, \infty)$.
        Furthermore,
        $\lim_{\phi \to 1^{+}} \tv_v(0; \phi) = \infty$,
        and $\lim_{\phi \to \infty} \tv_v(0; \phi) = 0$.
        \item
        \label{lem:fixed-point-v-properties-item-vb-properties}
        The function 
        $\tv_b(0; \cdot) : \phi \mapsto \tv_b(0; \phi)$,
        where
        \[
            \tv_b(0; \phi)
            = \tv_v(0; \phi
            )
            \int
            r^2(1 + v(0; \phi) r)^{-2}
            \, \mathrm{d}P(r),
        \]
        is positive and continuous over $(1, \infty)$.
        Furthermore,
        $\lim_{\phi \to 1^{+}} \tv_b(0; \phi) = \infty$,
        and $\lim_{\phi \to \infty} \tv_b(0; \phi) = 0$.
    \end{enumerate}
\end{lemma}

\Cref{lem:fixed-point-v-lambda-properties}, adapted from \cite{patil2022mitigating}, confirms the continuity and differentiability of the function $\lambda\mapsto v(-\lambda;\phi)$ on the closed interval $[0,\lambda_{\max}]$ for a certain constant $\lambda_{\max}$, provided $\phi\in(1,\infty)$. 
This ensures that $v(0;\phi)=\lim_{\lambda\rightarrow0^+}v(-\lambda;\phi)$ is well-defined for $\phi>1$, and also implies that related functions are bounded.

\begin{lemma}[Differentiability properties in the regularization parameter for $\phi\in(1,\infty)$, adapted from \citet{patil2022mitigating}]
    \label{lem:fixed-point-v-lambda-properties}
    Let $0<a \leq b<\infty$ be real numbers.
    Let $P$ be a probability measure supported on $[a, b]$.
    Let $\phi \in (1, \infty)$ be a real number.
    Let $\Lambda = [0, \lambda_{\max}]$
    for some constant $\lambda_{\max}\in(0,\infty)$.
    For $\lambda \in \Lambda$,
    let $v(-\lambda; \phi) > 0$
    denote the solution
    to the fixed-point equation
    \[
        \tfrac{1}{v(-\lambda; \phi)}
        =  \lambda
        + \phi \int \tfrac{r}{v(-\lambda; \phi) r + 1} \, \mathrm{d}P(r).
    \]
    Then, the function 
    $\lambda \mapsto v(-\lambda; \phi)$
    is 
    twice differentiable over $\Lambda$.
    Furthermore, 
    over $\Lambda$,
    $v(-\lambda; \phi)$,
    $\partial / \partial \lambda [v(-\lambda; \phi)]$,
    and
    $\partial^2 / \partial \lambda^2 [v(-\lambda; \phi)]$
    are bounded.
\end{lemma}

\begin{lemma}[Substitutability of the fixed-point solution]\label{lem:substitue}
    Let $v:\RR^{p\times p}\rightarrow\RR$ and $f(v(\bC),\bC):\RR^{p\times p}\rightarrow\RR^{p\times p}$ be a matrix function for matrix $\bC\in\RR^{p\times p}$ and $p\in\NN$,  that is continuous in the first augment with respect to operator norm. If $v(\bC)\stackrel{a.s.}{=} v(\bD)$ such that $\bC$ is independent of $\bD$, then $f(v(\bC),\bC)\asympequi f(v(\bD),\bC)\mid \bC$.

\end{lemma}
\begin{proof}
    For any matrix $\bT$ whose trace norm is bounded by $M$, conditioning on $\{\bC\}_{p\geq 1}$, we have
    \begin{align*}
        |\tr[(f(v(\bC), \bC)-f(v(\bD), \bC))\bT]| &\leq \norm{f(v(\bC), \bC)-f(v(\bD), \bC)}_{\oper} \tr(\bT)\\
        &\leq M\norm{f(v(\bC), \bC)-f(v(\bD), \bC)}_{\oper}.
    \end{align*}
    Since $v(\bC)\asto v(\bD)$ and $f$ is continuous in the first argument with respect to operator norm, we have $\lim_{p\rightarrow\infty}\norm{f(v(\bC), \bC)-f(v(\bD), \bC)}_{\oper}=0$.
    Thus,
    \begin{align*}
        \lim\limits_{p\rightarrow\infty}|\tr[(f(v(\bC), \bC)-f(v(\bD), \bC))\bT]| = 0,
    \end{align*}
    conditioning on  $\{\bC\}_{p\geq 1}$.
\end{proof}

The lemma below specializes the solution to the fixed-point equations under the isotopic model.

\begin{lemma}[Properties of the fixed-point solution with isotopic features]\label{lem:prop-fix-point-v-isotopic}
    Let $P$ be a probability measure supported on $\{a\}$ for $a>0$.
    For $\lambda>0$ and $\phi>0$, the fixed-point equation
    \[
        \tfrac{1}{v(-\lambda; \phi)}
        =  \lambda
        + \phi \int \tfrac{r}{v(-\lambda; \phi) r + 1} \, \mathrm{d}P(r) = \lambda+\tfrac{\phi a}{1+v(-\lambda;\phi)a}
    \]
    has a closed-form solution given by:
    \begin{align*}
        v(-\lambda; \phi) &= \tfrac{-(\lambda/a + \phi - 1)+\sqrt{(\lambda/a + \phi - 1)^2+4\lambda/a}}{2\lambda}.
    \end{align*}
    Define $\tv_b(-\lambda; \phi)$ and $\tv_v( -\lambda; \phi)$ via the follow equations:
    \begin{align*}
        \tv_b(-\lambda; \phi)
        &=
        \frac{\int \phi r^2(1+v(-\lambda; \phi) r)^{-2}\rd P(r)}{v(-\lambda; \phi)^{-2}- \int \phi r^2(1+v(-\lambda; \phi) r)^{-2}\rd P(r)},\\
        \tv_v(-\lambda; \phi)^{-1}
        &= v(-\lambda; \phi)^{-2}
            - \int \phi r^2(1+v(-\lambda; \phi) r)^{-2}\rd P(r).
    \end{align*}
    As $\lambda\rightarrow0^+,$
    we have the following different cases:
    \begin{align*}
        &(1)&&\phi\in(0,1):\qquad && v(0;\phi)=\infty,\quad  &&\tv_b(0;\phi)=\tfrac{\phi}{1-\phi},\quad  &&\tv_v(0;\phi)=\infty,\\
        &(2)&&\phi=1:\qquad && v(0;\phi)=\infty,\quad  &&\tv_b(0;\phi)=\infty,\quad  &&\tv_v(0;\phi)=\infty,\\
        &(3)&&\phi\in(1,\infty):\qquad && v(0;\phi)=\tfrac{1}{a(\phi-1)} ,\quad  &&\tv_b(0;\phi)=\tfrac{1}{\phi-1},\quad  &&\tv_v(0;\phi)=\tfrac{\phi}{a^2(\phi-1)^3},\\
        &(4)&&\phi=\infty:\qquad && v(0;\phi)=0,\quad  &&\tv_b(0;\phi)=0,\quad  &&\tv_v(0;\phi)=0,
    \end{align*}
\end{lemma}
\begin{proof}[Proof of \Cref{lem:prop-fix-point-v-isotopic}]
    For $\phi\in(0,1)$, we have $v(0; \phi) = \lim_{\lambda\rightarrow 0^+}v(-\lambda; \phi)=\infty$.
    For $\phi> 1$, 
    \begin{align*}
        v(0; \phi) = \lim\limits_{\lambda\rightarrow 0^+}v(-\lambda; \phi) = \tfrac{1}{2a}\lim\limits_{\lambda\rightarrow0^+}\left( - 1 + \tfrac{\lambda/a+\phi+1}{\sqrt{(\lambda/a+\phi-1)^2+4\lambda/a}}\right) =
        \tfrac{1}{a(\phi-1)},
    \end{align*}
    by applying L'Hospital's rule for indeterminate forms. When $\phi=1$, we have
    \begin{align*}
        v(0; 1) = \lim\limits_{\lambda\rightarrow 0^+}v(-\lambda; 1) =\lim\limits_{\lambda\rightarrow0^+} \tfrac{1}{2a}\left(-1+\sqrt{1+\tfrac{a}{\lambda}}\right) =\infty.
    \end{align*}
    Since $\tv_b(0;\phi)$ and $\tv_v(0;\phi)$ are continuous functions of $v(0;\phi) $, we have 
    \begin{align*}
        \tv_v(0;\phi) &= \begin{dcases}
            \infty, &\phi\in(0,1]\\
            \tfrac{\phi}{a^2(\phi-1)^3} , & \phi\in(1,\infty)
        \end{dcases}
    \end{align*}
    and $\tv_b(0;\phi)=1/(\phi-1)$ for $\phi\in(1,\infty)$.
    For $\phi\in(0,1]$, we apply the L'Hospital' rule to obtain $\tv_b(0;\phi)=\phi/(1-\phi)$.

\end{proof}

\section{Helper concentration results}\label{sec:appendix-concerntration}

\subsection{Size of the intersection of randomly sampled datasets}

In this section,
we collect various helper results
concerned with concentrations and convergences
that are used in the proofs of
\Cref{lem:risk_general_predictor_M12},
\Cref{lem:ridge-B0,lem:ridge-V0,lem:ridgeless-V0-phis_lt1}.

Below we recall the definition
of a hypergeometric random variable,
along with its mean and variance.
See, e.g., \citet{greene2017exponential}
for more related details.

    \begin{definition}[Hypergeometric random variable]
        A random variable $X$ follows the hypergeometric distribution $X\sim \operatorname {Hypergeometric} (n,K,N)$ if the probability mass function of $X$ is as follows:
         $$\PP(X=k)=\tfrac{
         \binom{K}{k}\binom{N-K}{n-k}
         }{\binom{N}{n}
         },\quad \text{where} \quad \max\{0,n+K-N\}\leq k\leq \min\{n,K\}.$$
        The expectation and variance of $X$ are given by:
        \begin{align*}
            \EE[X] &= \tfrac{nK}{N},
            \quad
            \text{and}
            \quad
            \Var(X) = \tfrac{nK(N-K)(N-n)}{N^2(N-1)}.
        \end{align*}
    \end{definition}

The following lemma provides tail bounds for the number of shared observations in two simple random samples adapted from 
\citep{hoeffding1963probability,serfling1974probability}. 
See also \citet{greene2017exponential}.

    \begin{lemma}[Concentration bounds for the number of shared observations]\label{lem:i0_tailbound}
        For $n\in\NN$, define $\mathcal{I}_k := \{\{i_1, i_2, \ldots, i_k\}:\, 1\le i_1 < i_2 < \ldots < i_k \le n\}$.
        Let $I_1,I_2\overset{\textup{\texttt{SRSWR}}}{\sim}\cI_k$, define the random variable $i_{0}^{\textup{\texttt{SRSWR}}} :=|I_1\cap I_2|$ to be the number of shared samples, and define $i_{0}^{\textup{\texttt{SRSWOR}}}$ accordingly.
        Then the following statements hold:
        \begin{enumerate}[leftmargin=7mm]
            \item[(1)] $i_0^{\textup{\texttt{SRSWR}}}$ follows a binomial distribution, $i_0^{\textup{\texttt{SRSWR}}}\sim \textup{\text{Binomial}}(k,k/n)$ with mean $\EE[i_0^{\textup{\texttt{SRSWR}}}]=k^2/n$. It holds that for all $t>0$,
            \begin{align*}
                \PP\left(i_0^{\textup{\texttt{SRSWR}}} - \EE[i_0^{\textup{\texttt{SRSWR}}}] \geq kt\right) \leq \exp\left(-2 k t^2\right).
            \end{align*}
            
            \item[(2)] $i_0^{\textup{\texttt{SRSWOR}}}$ follows a hypergeometric distribution, $i_0^{\textup{\texttt{SRSWOR}}} \sim \operatorname {Hypergeometric} (k,k,n)$ with mean $\EE[i_0^{\textup{\texttt{SRSWOR}}} ]=k^2/n$. 
            It holds that for all $t>0$,
            \begin{align}
                \PP\left(i_0^{\textup{\texttt{SRSWOR}}}  - \EE[i_0^{\textup{\texttt{SRSWOR}}} ] \geq kt\right) \leq \exp\left(-\tfrac{2nkt^2}{n-k+1}\right).
            \end{align}
        \end{enumerate}
    \end{lemma}

    The following lemma characterizes the limiting proportions of shared observations in two simple random samples under proportional asymptotics when both the subsample and full data sizes tend to infinity.
    \begin{lemma}[Asymptotic proportions of the shared observations]\label{lem:i0_mean}
        Consider the setting in Lemma \ref{lem:i0_tailbound}. Let $\{k_m\}_{m=1}^{\infty}$ and $\{n_m\}_{m=1}^{\infty}$ be two sequences of positive integers such that $n_m$ is strictly increasing in $m$, $n_m^{\nu}\leq k_m\leq n_m$ for some constant $\nu\in(0,1)$, and $k_m/n_m\rightarrow \omega_s\in[0,1]$.
        Then, $i_0^{\textup{\texttt{SRSWR}}}/k_m\asto \omega_s$, and $i_0^{\textup{\texttt{SRSWOR}}}/k_m\asto \omega_s$.
    \end{lemma}
    \begin{proof}
        Proofs for the two parts are split below.
        \paragraph{Part (1).\hspace{-2mm}}
        For all $\delta>0$,
        \begin{align*}
            \sum\limits_{m=1}^{\infty} \PP\left(\tfrac{1}{k_m}|i_0^{\textup{\texttt{SRSWR}}}-\EE[i_0^{\textup{\texttt{SRSWR}}}]|>\delta\right) &\leq 2\sum\limits_{m=1}^{\infty} \exp\left(- 2k_m\delta^2\right) .
        \end{align*}
        Because $k_m,n_m\rightarrow\infty$ and $k_m=\Omega(n_m^{\nu})$, there exists $m_0\in\NN$, such that for all $m> m_0$, $\exp(-2k_m\delta^2)\leq n_m^{-(1+\nu)}$.
        Thus,
        \begin{align*}
            \sum\limits_{m=1}^{\infty} \PP\left(\tfrac{1}{k_m}|i_0^{\textup{\texttt{SRSWR}}}-\EE[i_0^{\textup{\texttt{SRSWR}}}]|>\delta\right) &\leq 2 \sum\limits_{m=1}^{m_0} \exp\left(- 2k_m\delta^2\right) + 2 \sum\limits_{m=m_0}^{\infty}\tfrac{1}{n_m^{1+\nu}} <\infty.
        \end{align*}
        By the Borel–Cantelli lemma,
        we have $$\tfrac{i_0^{\textup{\texttt{SRSWR}}}}{k_m}-\tfrac{\EE[i_0^{\textup{\texttt{SRSWR}}}]}{k_m}\asto0 .$$
        As $\lim_{m\rightarrow\infty}\EE[i_0^{\textup{\texttt{SRSWR}}}]/{k_m}=\lim_{m\rightarrow\infty}{k_m}/{n_m}=\omega_s,$ we further have ${i_0^{\textup{\texttt{SRSWR}}}}/{k_m}\asto \omega_s$.
        
        \paragraph{Part (2).\hspace{-2mm}}
        Note that
        \begin{align*}
            \PP\left(i_0^{\textup{\texttt{SRSWOR}}}  - \EE[i_0^{\textup{\texttt{SRSWOR}}} ] \geq kt\right) \leq \exp\left(-\tfrac{2nkt^2}{n-k+1}\right)\leq \exp\left(-2kt^2\right).
        \end{align*}
        The conclusion then follows analogously, as in Part 1.
    \end{proof}
    
\subsection{Convergence of random linear and quadratic forms}

In this section, we collect helper lemmas
on the concentration of linear and quadratic forms of random vectors that are used in the proofs of \Cref{lem:ridge-conv-C0,lem:ridge-conv-V0,lem:ridgeless-conv-C0,lem:ridgeless-conv-V0}.

The following lemma provides the concentration of a linear form of a random vector with independent components.
It follows from a moment bound from Lemma 7.8 of \cite{erdos_yau_2017}, along with the Borel-Cantelli lemma and is adapted from Lemma S.8.5 of \cite{patil2022mitigating}.

\begin{lemma}
    [Concentration of linear form with independent components]
    \label{lem:concen-linform}
    Let $\bz_p \in \RR^{p}$ be a sequence of random vector with i.i.d.\ entries $z_{pi}$, $i = 1, \dots, p$ such that for each i, $\EE[z_{pi}] = 0$, $\EE[z_{pi}^2] = 1$, $\EE[|z_{pi}|^{4+\alpha}] \le M_\alpha$ for some $\alpha > 0$ and constant $M_\alpha < \infty$.
    Let $\ba_p \in \RR^{p}$ be a sequence of random vectors independent of $\bz_p$ such that $\limsup_{p} \| \ba_p \|^2 / p \le M_0$ almost surely for a constant $M_0 < \infty$.
    Then, $\ba_p^\top \bz_p / p \to 0$ almost surely as $p \to \infty$.
\end{lemma}

The following lemma provides the concentration of a quadratic form of a random vector with independent components.
It follows from a moment bound from Lemma B.26 of \cite{bai2010spectral}, along with the Borel-Cantelli lemma, and is adapted from Lemma S.8.6 of \cite{patil2022mitigating}.

\begin{lemma}
    [Concentration of quadratic form with independent components]
    \label{lem:concen-quadform}
    Let $\bz_p \in \RR^{p}$ be a sequence of random vector with i.i.d.\ entries $z_{pi}$, $i = 1, \dots, p$ such that for each i, $\EE[z_{pi}] = 0$, $\EE[z_{pi}^2] = 1$, $\EE[|z_{pi}|^{4+\alpha}] \le M_\alpha$ for some $\alpha > 0$ and constant $M_\alpha < \infty$.
    Let $\bD_p \in \RR^{p \times p}$ be a sequence of random matrix such that $\limsup \| \bD_p \|_{\oper} \le M_0$ almost surely as $p \to \infty$ for some constant $M_0 < \infty$.
    Then, $\bz_p^\top \bD_p \bz_p / p - \tr[\bD_p] / p \to 0$
    almost surely as $p \to \infty$.
\end{lemma}

\subsection{Convergence of Ce\`saro-type mean and max for triangular array}

In this section,
we collect a helper lemma on deducing almost sure convergence 
of a Ce\`saro-type mean from almost sure convergence
of the original sequence.
It is used in the proof of \Cref{prop:limiting-risk-for-arbitrary-M-cond} and \Cref{lem:risk_general_predictor_M12}.

    \begin{lemma}[Convergence of conditional expectation]\label{lem:conv_cond_expectation}
        For $n\in\NN$, suppose $\{R_{n,\ell}\}_{\ell=1}^{N_n}$ is a set of $N_n$ %
        random variables defined over the probability space $(\Omega,\cF,\PP)$, with $1<N_n<\infty$ almost surely.
        If there exists a constant $c$ such that $R_{n,p_n}\asto c$ for all 
        deterministic sequences $\{p_n\in[N_n]\}_{n=1}^{\infty}$, 
        then  the following statements hold:
        \begin{enumerate}[(1),leftmargin=7mm]
            \item\label{lem:conv_cond_expectation-max}
            $\max_{\ell\in[N_n]}\left|R_{n,\ell}(\omega)-c\right|\asto 0$,
            \item\label{lem:conv_cond_expectation-avg} $N_n^{-1}\sum_{\ell=1}^{N_n}R_{n,\ell}\asto c$.
        \end{enumerate}
    \end{lemma}
    \begin{proof}[Proof of \Cref{lem:conv_cond_expectation}]
        Proofs for the two parts are split below.
        \paragraph{Part (1).\hspace{-2mm}}
        We concatenate the sets $\{R_{n,\ell}\}_{\ell=1}^{N_n}$ for all $n\in\NN$ to form a new sequence
        $$W=(W_1,W_2,\cdots)=(R_{1,1},\cdots,R_{1,N_1},R_{2,1},\cdots,R_{2,N_2},\cdots).$$
        That is, $W_t=R_{n,\ell}$ for $t=\sum_{j=1}^{n}N_{j} + \ell$.
        See \Cref{fig:cesaro-mean-triangular-array} for an illustration.
        \begin{figure}[!ht]
            \centering
            \includegraphics[width=0.75\textwidth]{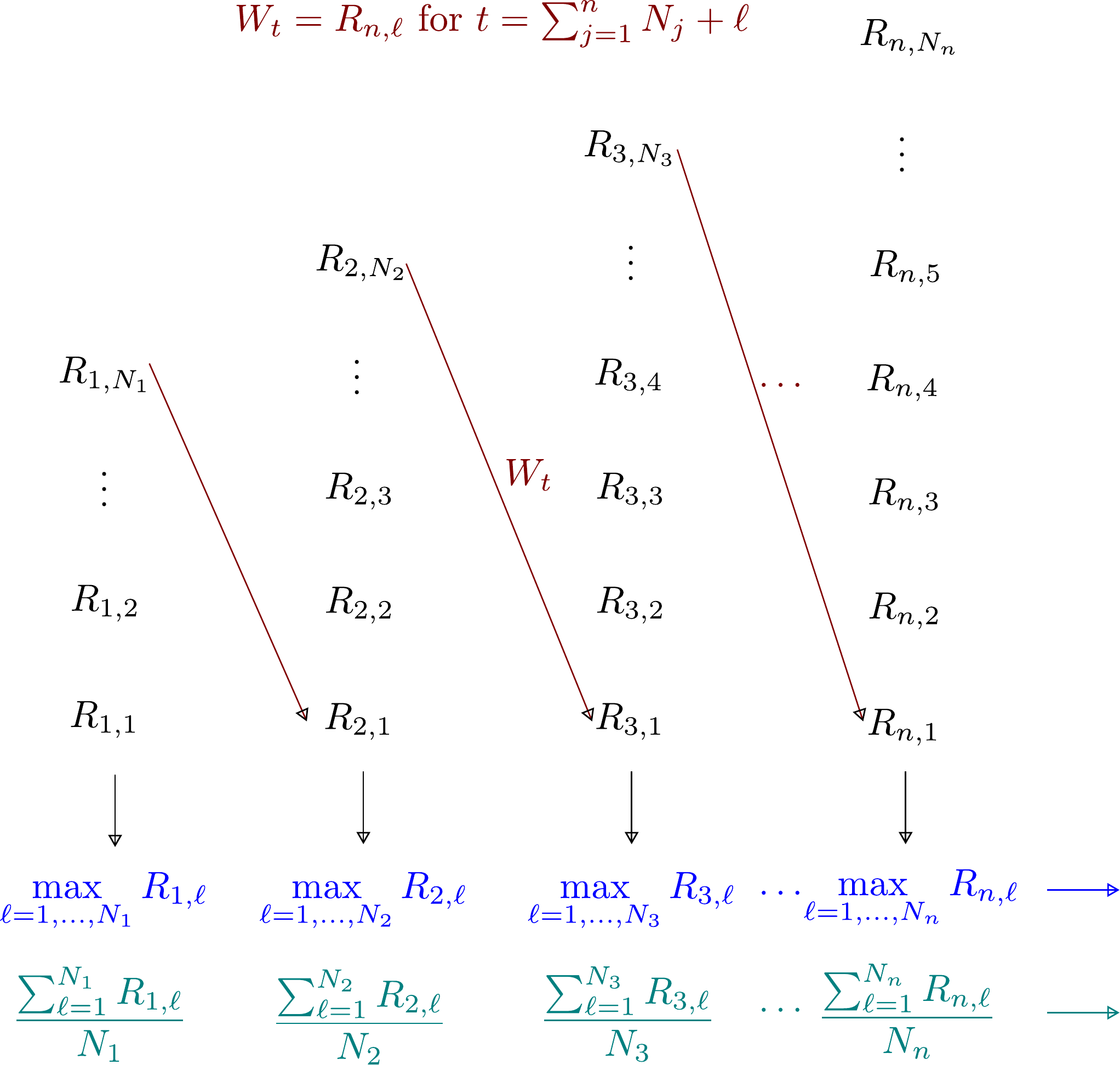}
            \caption{Illustration of the concatenated sequence $\{ W_t \}$ (in maroon) constructed from the triangle array $\{ R_{n, \ell} \}_{\ell=1}^{N_n}, n \in \NN$ (in black), used in the proof of \Cref{lem:conv_cond_expectation},
            along with the max sequence (in blue) and the average sequence (in teal).}
            \label{fig:cesaro-mean-triangular-array}
        \end{figure}
        Because $N_n\rightarrow \infty$ if and only if $n\rightarrow\infty$ if and only if $t\rightarrow\infty$,
        it holds that $W_{t}\asto c$ as $t\rightarrow\infty$.
        Then, by \citet[Chapter 2, Section 10, Theorem 1]{shiryaev2016probability}, we have that for all $\epsilon>0$, 
        \begin{align*}
            \lim_{s\rightarrow\infty}\PP\left(\bigcup_{t=s}^{\infty}\{\omega\in\Omega:|W_t(\omega)-c|> \epsilon\}\right) = 0.
        \end{align*}
        Now, for $s\in\NN$, let $m$ be the smallest natural number such that $\sum_{j=1}^mN_j\geq s$. Since
        \begin{align*}
            \bigcup_{t=s}^{\infty}\{\omega\in\Omega:|W_t(\omega)-c| > \epsilon\}&\supseteq \bigcup_{n=m}^{\infty}\bigcup_{\ell=1}^{N_n}\{\omega\in\Omega:|R_{n,\ell}(\omega)-c|> \epsilon\} \\
            &= \bigcup_{n=m}^{\infty}\left\{\omega\in\Omega:\max_{\ell\in[N_n]}\left|R_{n,\ell}(\omega)-c\right|> \epsilon\right\}.
        \end{align*}
        We further have
        \begin{align*}
            0 & \leq \lim_{m\rightarrow\infty } \PP\left(\bigcup_{n=m}^{\infty}\left\{\omega\in\Omega:\max_{\ell\in[N_n]}\left|R_{n,\ell}(\omega)-c\right|> \epsilon\right\}\right)\leq \lim_{s\rightarrow\infty }\PP\left(\bigcup_{t=s}^{\infty}\{\omega\in\Omega:|W_t(\omega)-c|> \epsilon\}\right) = 0,
        \end{align*}
        or in other words,
        $$\lim_{m\rightarrow\infty } \PP\left(\bigcup_{n=m}^{\infty}\left\{\omega\in\Omega:\max_{\ell\in[N_n]}\left|R_{n,\ell}(\omega)-c\right|> \epsilon\right\}\right) = 0.$$
        Thus, we have that $\max_{\ell\in[N_n]}\left|R_{n,\ell}(\omega)-c\right|\asto 0$ by \citet[Chapter 2, Section 10, Theorem 1]{shiryaev2016probability}.
        
        \paragraph{Part (2).\hspace{-2mm}}
        We will use the first part.
        Note that by triangle inequality, $$\left|N_n^{-1}\sum_{\ell=1}^{N_n}R_{n,\ell}- c\right|\leq N_n^{-1}\sum_{\ell=1}^{N_n}\left|R_{n,\ell}- c\right|\leq \max_{\ell\in[N_n]}\left|R_{n,\ell}(\omega)-c\right|.$$
        Invoking the first part,
        we have that
        $N_n^{-1}\sum_{\ell=1}^{N_n}R_{n,\ell}\asto c$.

    \end{proof}
    
\clearpage
\section{Additional numerical illustrations}\label{sec:appendix-additional-numerical-result}

\subsection{Additional illustrations for \Cref{thm:ver-with-replacement}}
\label{sec:empirical_varying_SNR}

\subsubsection{Prediction risk curves for subagged ridgeless and ridge predictors with varying $M$}
    \begin{figure}[!ht]
        \centering
        \includegraphics[width=0.90\textwidth]{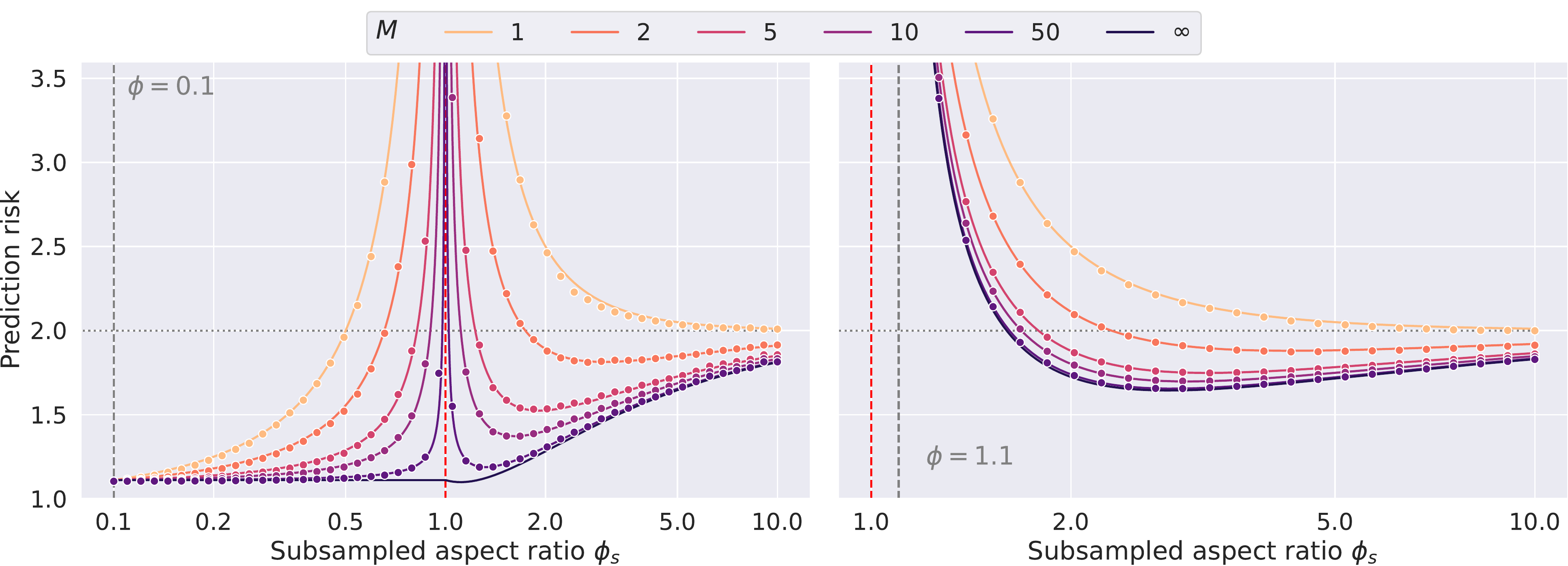}
        \caption{Asymptotic prediction risk curves in \eqref{eq:risk-det-with-replacement} for ridgeless predictors ($\lambda=0$), under model \eqref{eq:model} when $\rho^2=1$ and $\sigma^2=1$ for varying bag size $k=\lfloor p/\phi_s\rfloor$ and number of bags $M$.
         The null risk is marked as a dotted line. For each value of $M$, the points denote finite-sample risks averaged over 100 dataset repetitions, with $n=1000$ and $p=\lfloor n\phi\rfloor$.
         The left and the right panels correspond to the cases when $p<n$ ($\phi=0.1$) and $p>n$ ($\phi=1.1$), respectively.
        }
        \label{fig:ridgeless-with-replacement-varing-M-iso}
    \end{figure}

    \begin{figure}[!ht]
        \centering
        \includegraphics[width=0.90\textwidth]{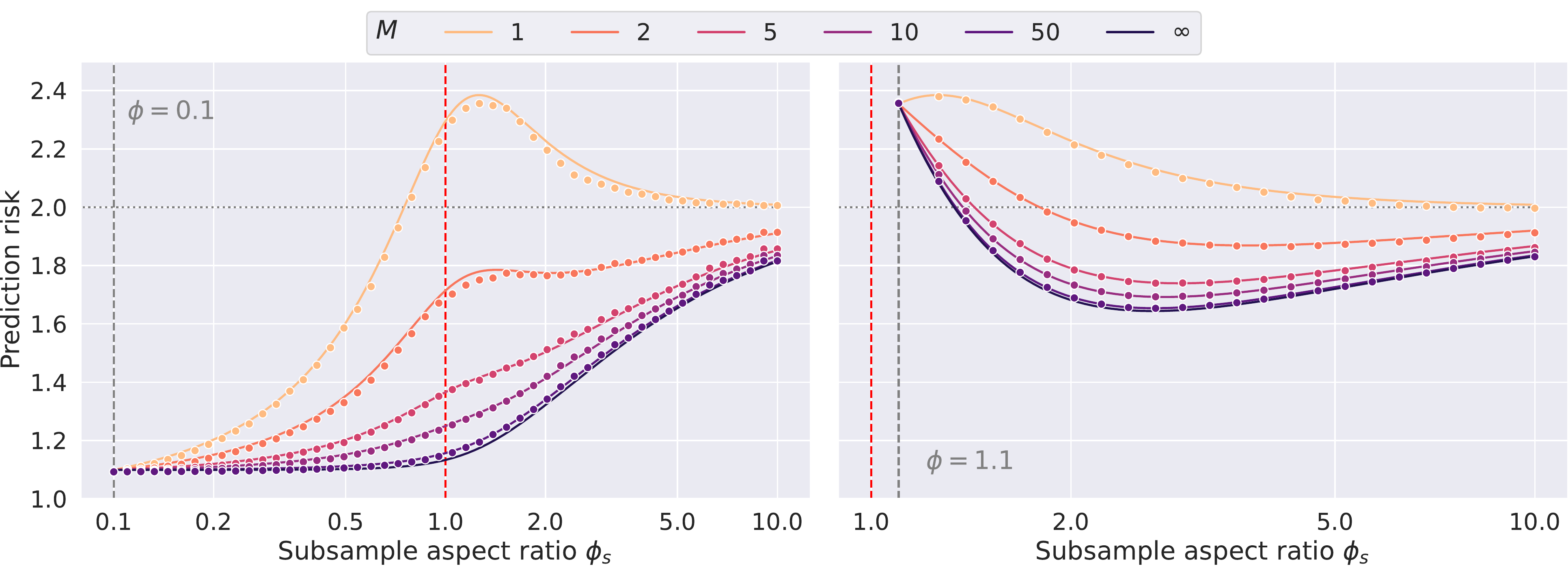}
        \caption{Asymptotic prediction risk curves in \eqref{eq:risk-det-with-replacement} for subagged ridge predictors ($\lambda=0.1$), under model \eqref{eq:model} when $\rho^2=1$ and $\sigma^2=1$ for varying bag size $k=\lfloor p/\phi_s\rfloor$ and number of bags $M$.
         The null risk is marked as a dotted line. For each value of $M$, the points denote finite-sample risks averaged over 100 dataset repetitions, with $n=1000$ and $p=\lfloor n\phi\rfloor$.
         The left and the right panels correspond to the cases when $p<n$ ($\phi=0.1$) and $p>n$ ($\phi=1.1$), respectively.}
        \label{fig:ridge-with-replacement-varing-M-iso}
    \end{figure}

\clearpage
\subsubsection{Bias-variance curves for subagged ridgeless and ridge predictors with varying $M$}
    
    \begin{figure}[!ht]
        \centering
        \includegraphics[width=0.90\textwidth]{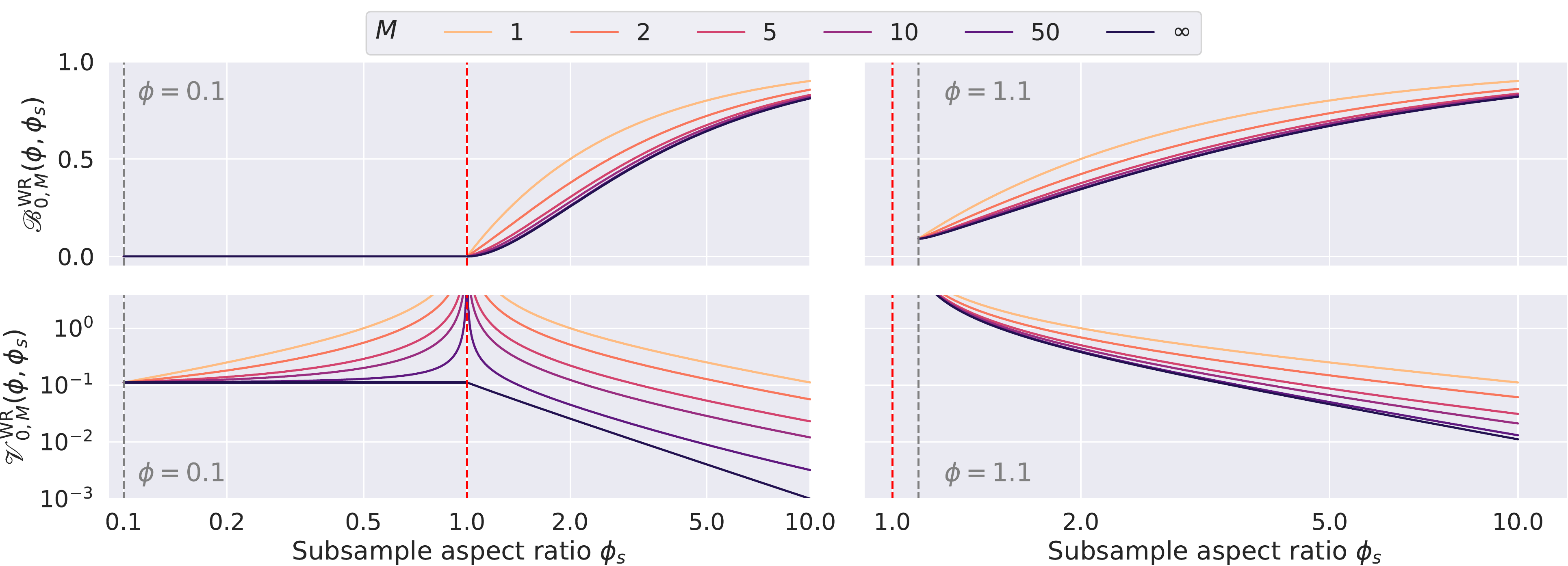}
        \caption{Asymptotic bias and variance curves in \eqref{eq:Blam_V_lam} for subagged ridgeless predictors ($\lambda=0$), under model \eqref{eq:model} when $\rho^2=1$ and $\sigma^2=0.25$ for varying bag size $k=\lfloor p/\phi_s\rfloor$ and number of bags $M$.
        The left and the right panels correspond to the cases when $p<n$ ($\phi=0.1$) and $p>n$ ($\phi=1.1$), respectively.
        The values of $\VzeroM{M}{\phi}$ are shown on a $\log$-10 scale.
        }
        \label{fig:ridgeless-bias-var-with-replacement-iso}
    \end{figure}
    
    \begin{figure}[!ht]
        \centering
        \includegraphics[width=0.90\textwidth]{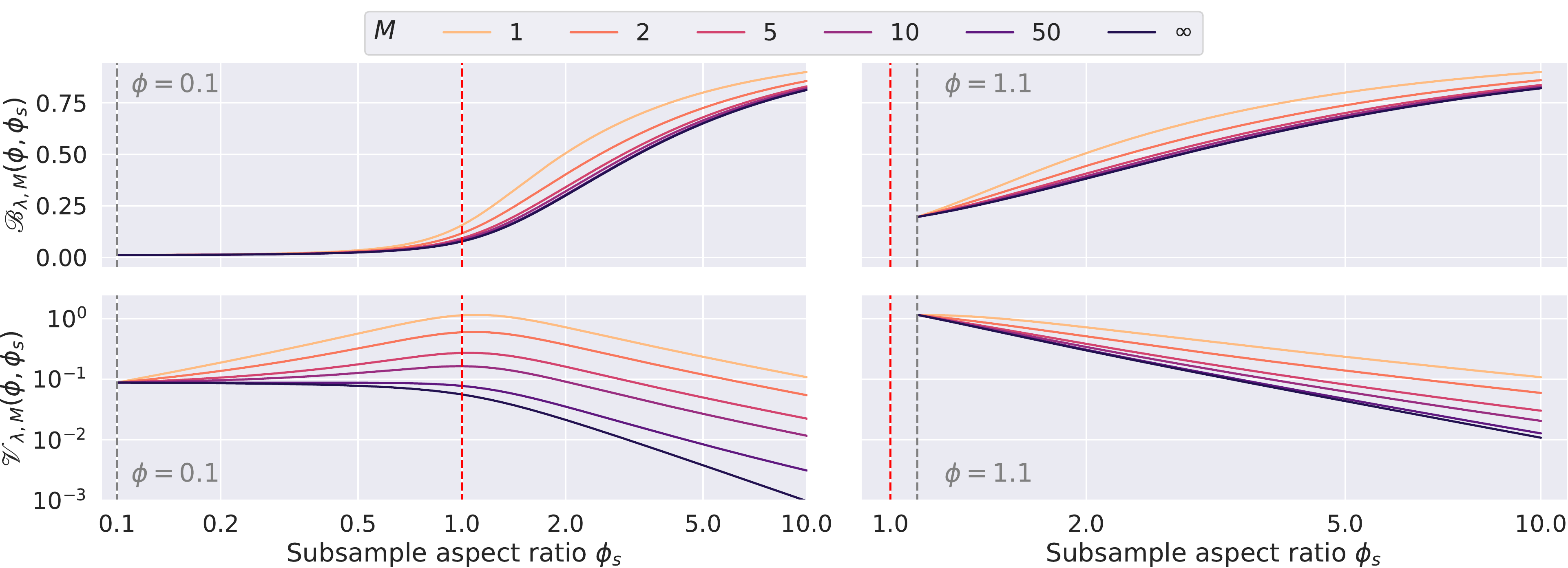}
        \caption{Asymptotic bias and variance curves in \eqref{eq:Blam_V_lam} for subagged ridge predictors ($\lambda=0.1$), under model \eqref{eq:model} when $\rho^2=1$ and $\sigma^2=1$ for varying bag size $k=\lfloor p/\phi_s\rfloor$ and number of bags $M$.
        The left and the right panels correspond to the cases when $p<n$ ($\phi=0.1$) and $p>n$ ($\phi=1.1$), respectively.
        The values of $\VzeroM{M}{\phi}$ are shown in $\log$-10 scale.
        }
        \label{fig:ridge-bias-var-with-replacement-iso}
    \end{figure}

    \begin{figure}[!ht]
        \centering
        \includegraphics[width=0.90\textwidth]{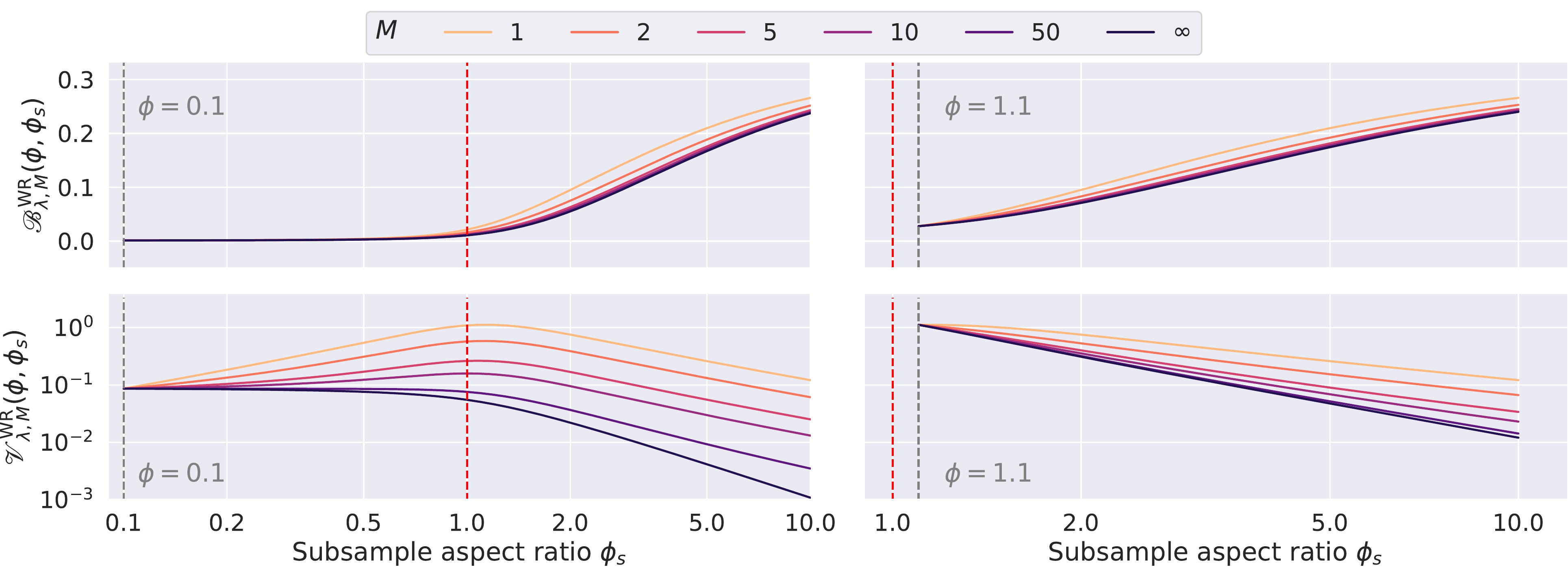}
        \caption{Asymptotic bias and variance curves in \eqref{eq:Blam_V_lam} for subagged ridge predictors ($\lambda=0.1$), under model \eqref{eq:model-ar1} when $\rho^2=1$ and $\sigma^2=1$ for varying bag size $k=\lfloor p/\phi_s\rfloor$ and number of bags $M$.
        The left and the right panels correspond to the cases when $p<n$ ($\phi=0.1$) and $p>n$ ($\phi=1.1$), respectively.
        The values of $\VzeroM{M}{\phi}$ are shown in $\log$-10 scale.
        }
        \label{fig:ridge-bias-var-with-replacement}
    \end{figure}

\clearpage    
\subsubsection{Bias-variance curves for subagged ridge predictors with varying $\lambda$ ($M = 1$)}

    \begin{figure}[!ht]
        \centering
        \includegraphics[width=0.90\textwidth]{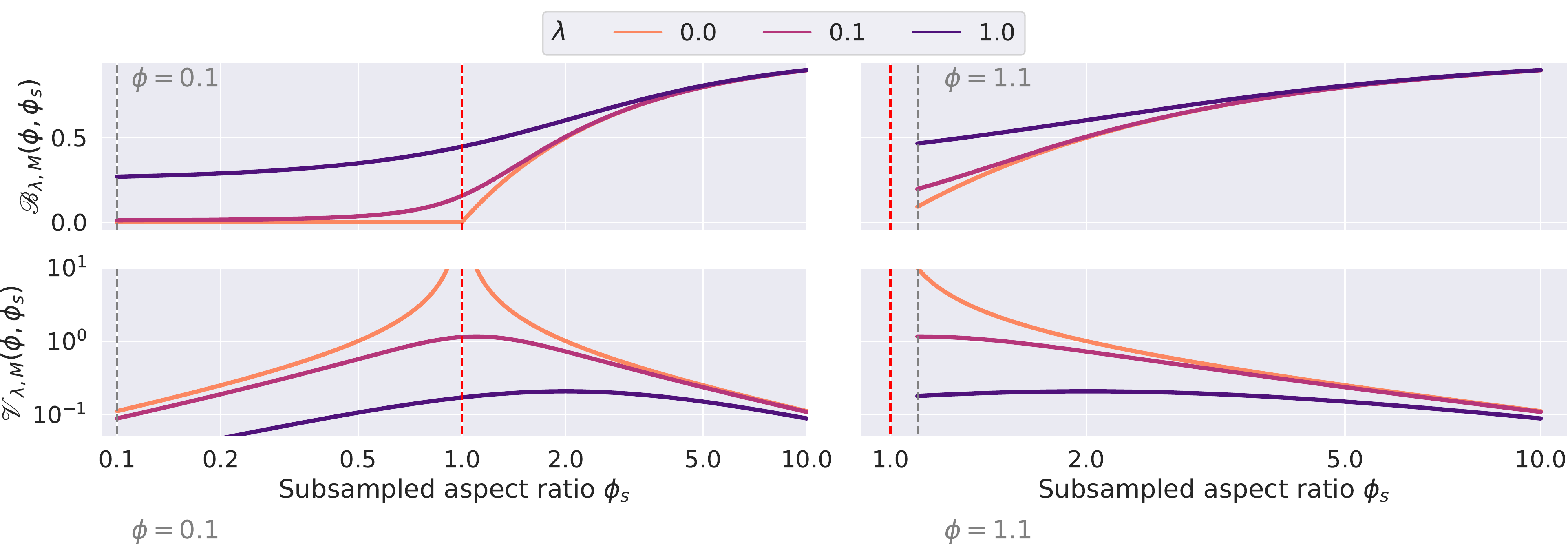}
        \caption{Asymptotic bias and variance curves in \eqref{eq:Blam_V_lam} for subagged ridge and ridgeless predictors with number of bags $M=1$, under model \eqref{eq:model} when $\rho^2=1$ and $\sigma^2=1$ for varying regularization parameter $\lambda$.
        The left and the right panels correspond to the cases when $p<n$ ($\phi=0.1$) and $p>n$ ($\phi=1.1$), respectively.
        The values of $\VzeroM{M}{\phi}$ are shown in $\log$-10 scale.
        }
        \label{fig:bias-var-with-replacement-varying_lam_M1}
    \end{figure}

\subsubsection{Bias-variance curves for subagged ridge predictors with varying $\lambda$ ($M = \infty$)} 

    \begin{figure}[!ht]
        \centering
        \includegraphics[width=0.90\textwidth]{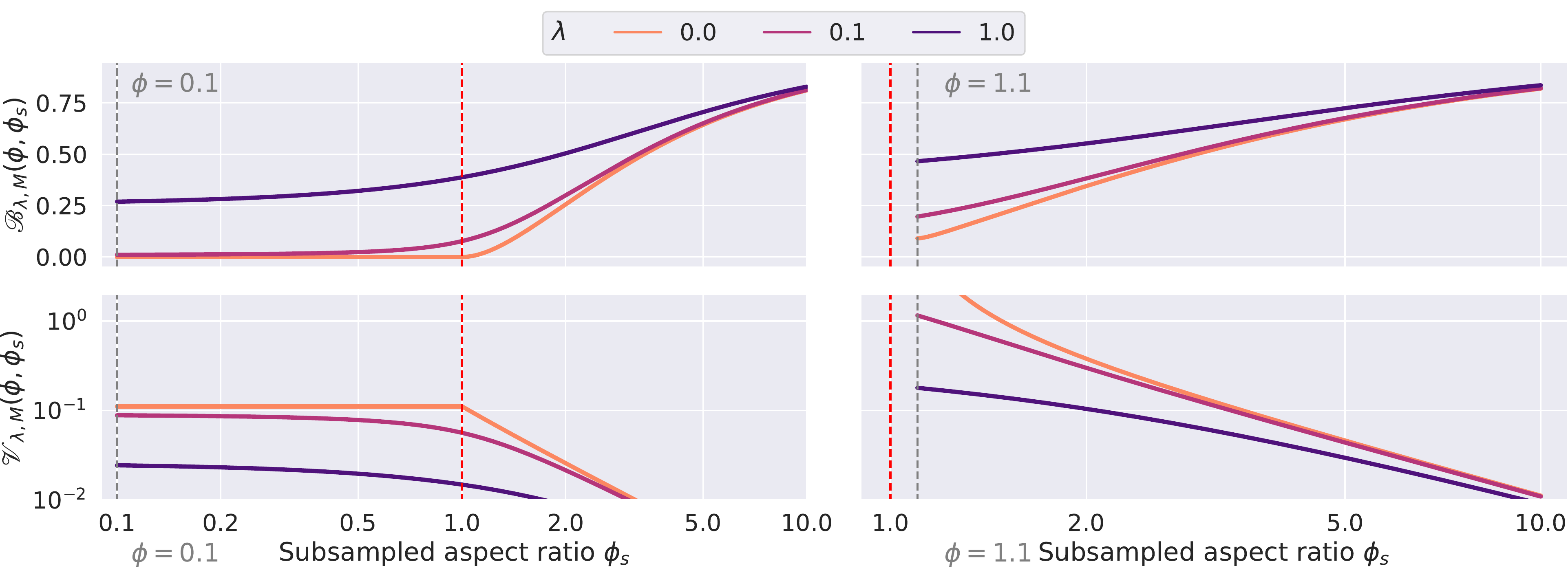}
        \caption{Asymptotic bias and variance curves in \eqref{eq:Blam_V_lam} for subagged ridge and ridgeless predictors with number of bags $M=\infty$,
        under model \eqref{eq:model} when $\rho^2=1$ and $\sigma^2=1$ for varying regularization parameter $\lambda$.
        The left and the right panels correspond to the cases when $p<n$ ($\phi=0.1$) and $p>n$ ($\phi=1.1$), respectively.
        The values of $\VzeroM{M}{\phi}$ are shown in $\log$-10 scale.
        }
        \label{fig:bias-var-with-replacement-varying_lam_Minf}
    \end{figure}

\clearpage
\subsection{Additional illustrations for \Cref{thm:ver-without-replacement}}    

\subsubsection{Prediction risk curves for \splagged ridgeless and ridge predictors with varying $M$}
    \begin{figure}[!ht]
        \centering
        \includegraphics[width=0.90\textwidth]{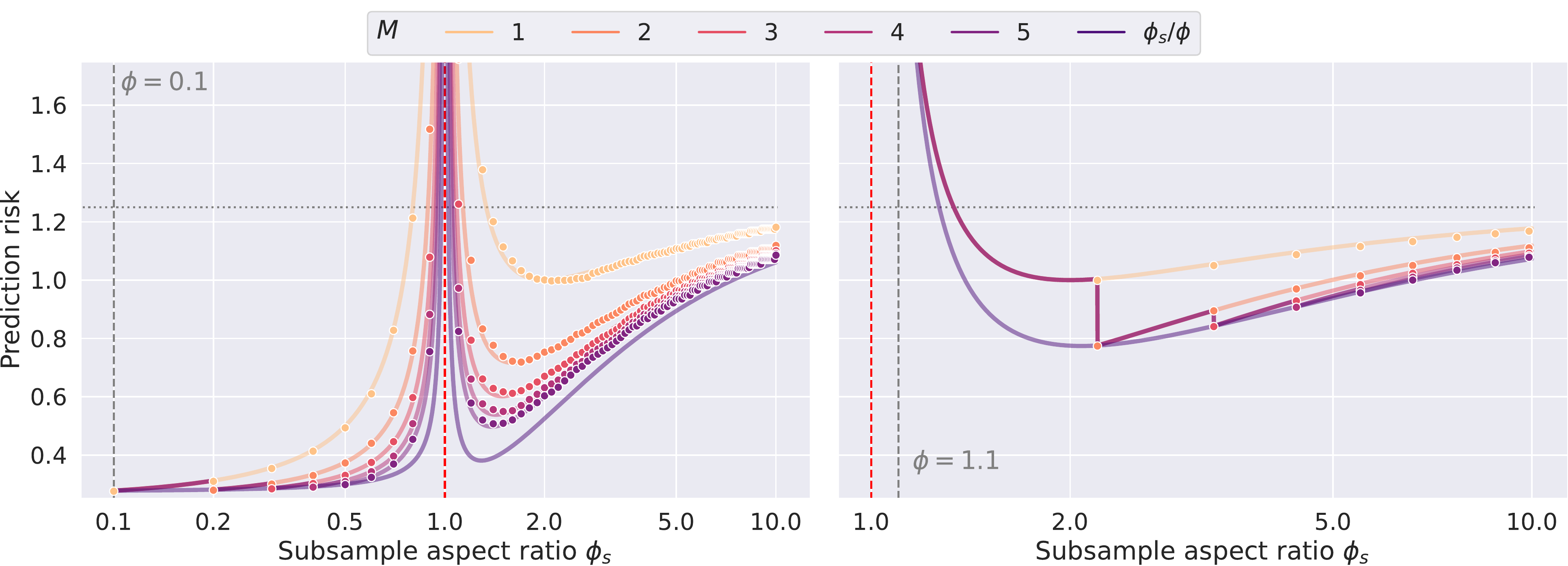}
        \caption{Asymptotic prediction risk curves in \eqref{eq:risk-det-without-replacement} for \splagged ridgeless predictors ($\lambda=0$), under model \eqref{eq:model} when $\rho^2=1$ and $\sigma^2=0.25$ for varying bag size $k=\lfloor p/\phi_s\rfloor$ and number of bags $M$ without replacement.
        The left and the right panels correspond to the cases when $p<n$ ($\phi=0.1$) and $p>n$ ($\phi=1.1$), respectively.
        The null risk is marked as a dotted line.
        For each value of $M$, the points denote finite-sample risks averaged over 100 dataset repetitions, with $n=1000$ and $p=\lfloor n\phi\rfloor$.}
        \label{fig:ridgeless-without-replacement-varing-M-iso}
    \end{figure}

    \begin{figure}[!ht]
        \centering
        \includegraphics[width=0.90\textwidth]{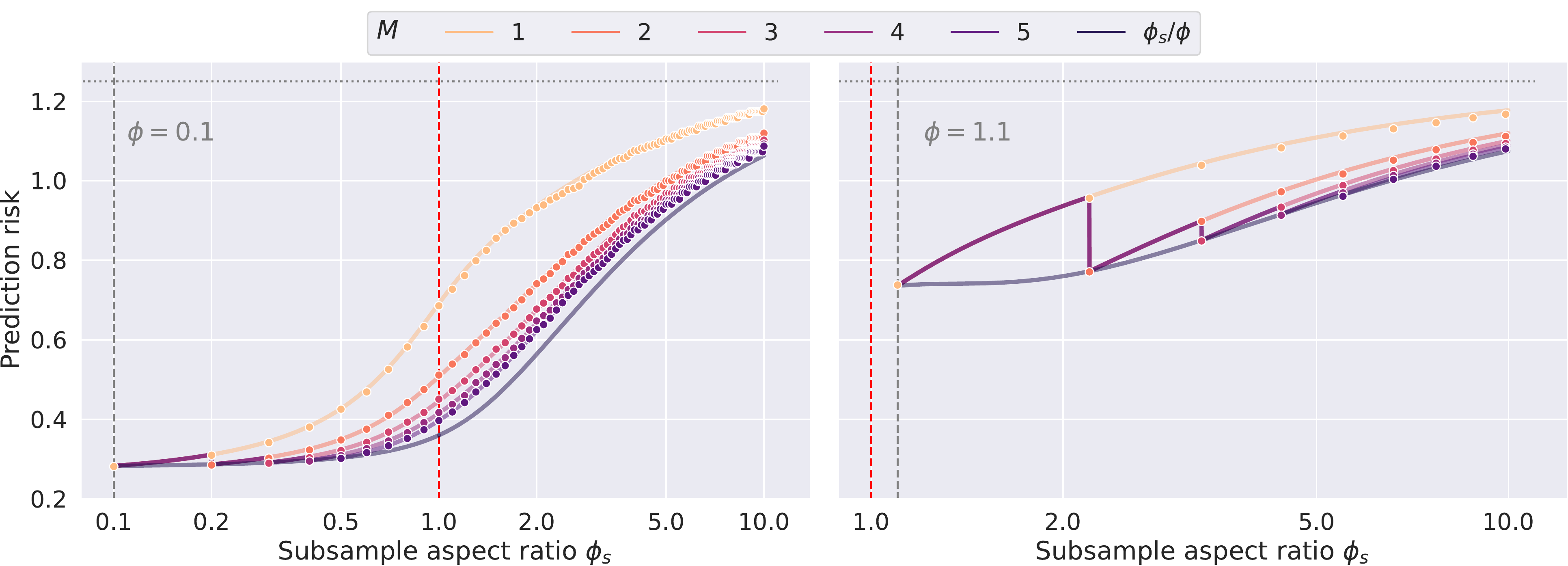}
        \caption{Asymptotic prediction risk curves in \eqref{eq:risk-det-without-replacement} for \splagged ridge predictors ($\lambda=0.1$), under model \eqref{eq:model} when $\rho^2=1$ and $\sigma^2=0.25$ for varying bag size $k=\lfloor p/\phi_s\rfloor$ and number of bags $M$ without replacement.
        The left and the right panels correspond to the cases when $p<n$ ($\phi=0.1$) and $p>n$ ($\phi=1.1$), respectively.
         The null risk is marked as a dotted line. For each value of $M$, the points denote finite-sample risks averaged over 100 dataset repetitions, with $n=1000$ and $p=\lfloor n\phi\rfloor$.}
        \label{fig:ridge-without-replacement-varing-M-iso}
\end{figure}

\clearpage
\subsubsection{Bias-variance curves for ridgeless and ridge predictors with varying $M$}

    \begin{figure}[!ht]
        \centering
        \includegraphics[width=0.90\textwidth]{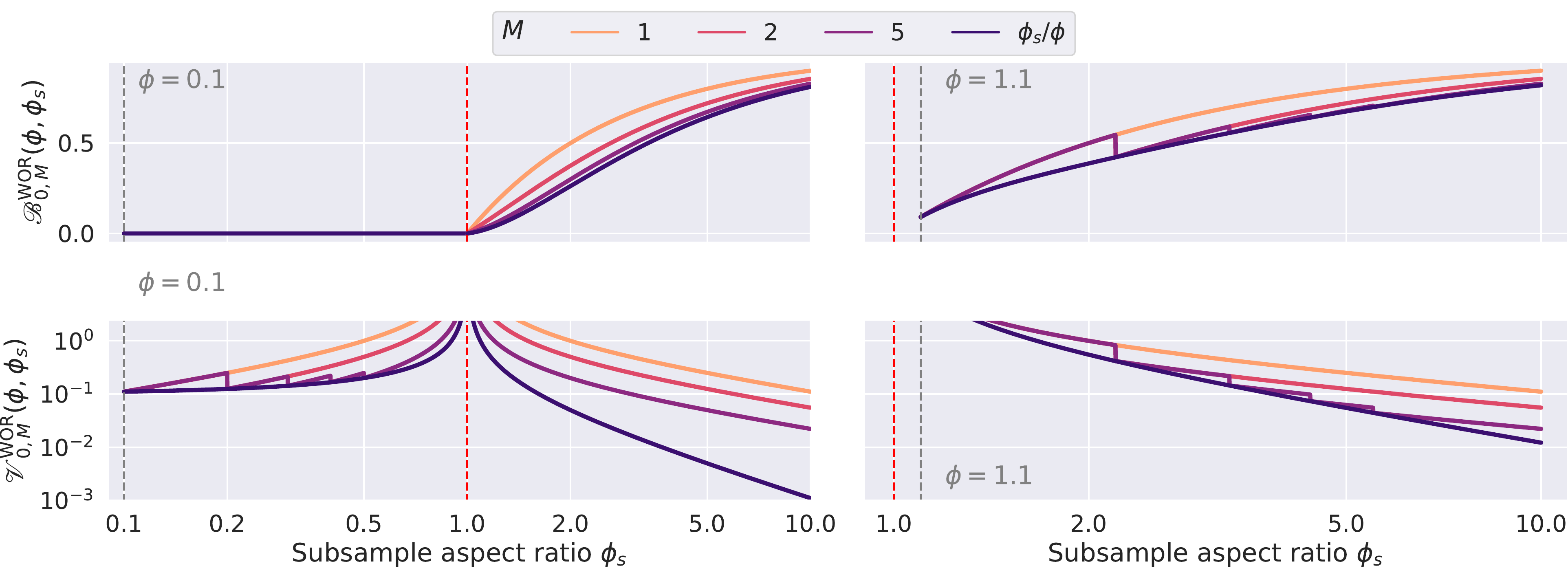}
        \caption{Asymptotic bias and variance curves in \eqref{eq:Blam_V_lam} for \splagged ridgeless predictors ($\lambda=0$), under model \eqref{eq:model} when $\rho^2=1$ and $\sigma^2=1$ for varying bag size $k=\lfloor p/\phi_s\rfloor$ and number of bags $M$ without replacement.
        The left and the right panels correspond to the cases when $p<n$ ($\phi=0.1$) and $p>n$ ($\phi=1.1$), respectively.
        The values of $\bVzeroM{M}{\phi}$ are shown in $\log$-10 scale.
        }
        \label{fig:ridgeless-bias-var-without-replacement-iso}
    \end{figure}
    
    \begin{figure}[!ht]
        \centering
        \includegraphics[width=0.90\textwidth]{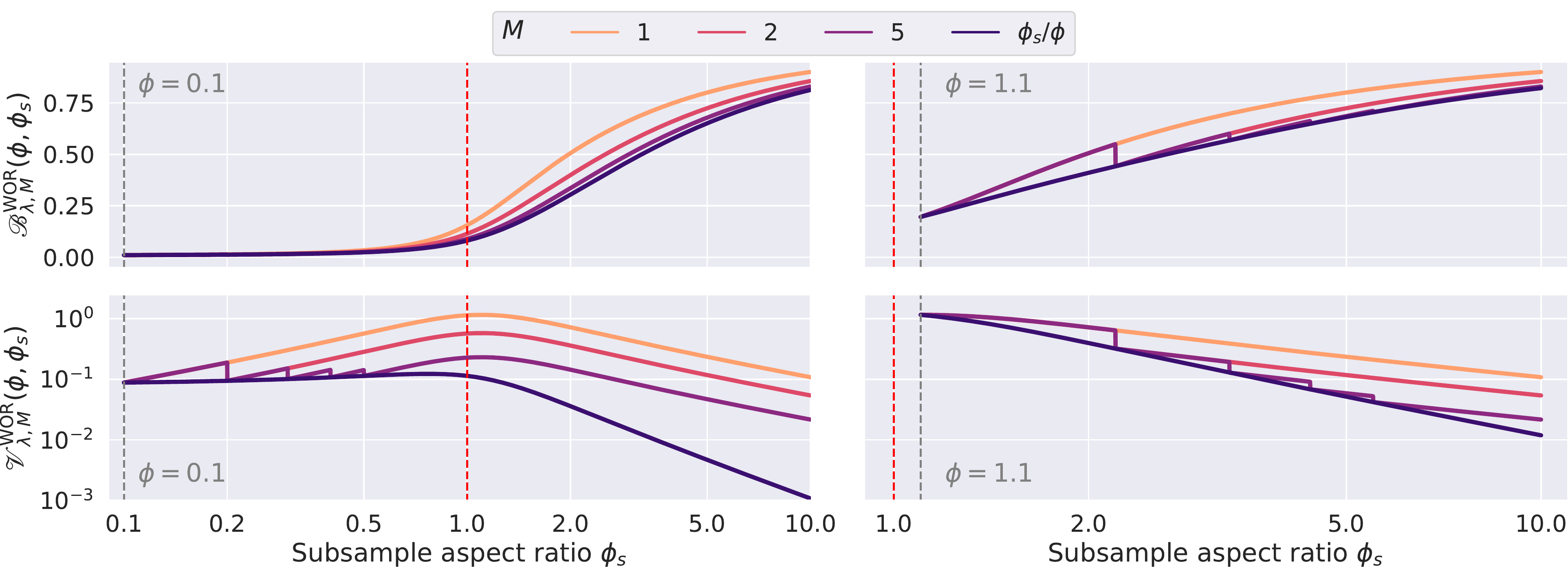}
        \caption{Asymptotic bias and variance curves in \eqref{eq:Blam_V_lam} for \splagged ridge predictors ($\lambda=0.1$), under model \eqref{eq:model} when $\rho^2=1$ and $\sigma^2=1$ for varying bag size $k=\lfloor p/\phi_s\rfloor$ and number of bags $M$ without replacement.
        The left and the right panels correspond to the cases when $p<n$ ($\phi=0.1$) and $p>n$ ($\phi=1.1$), respectively.
        The values of $\bVlamM{M}{\phi}$ are shown in $\log$-10 scale.
        }
        \label{fig:ridge-bias-var-without-replacement-iso}
    \end{figure}
    
    \begin{figure}[!ht]
        \centering
        \includegraphics[width=0.90\textwidth]{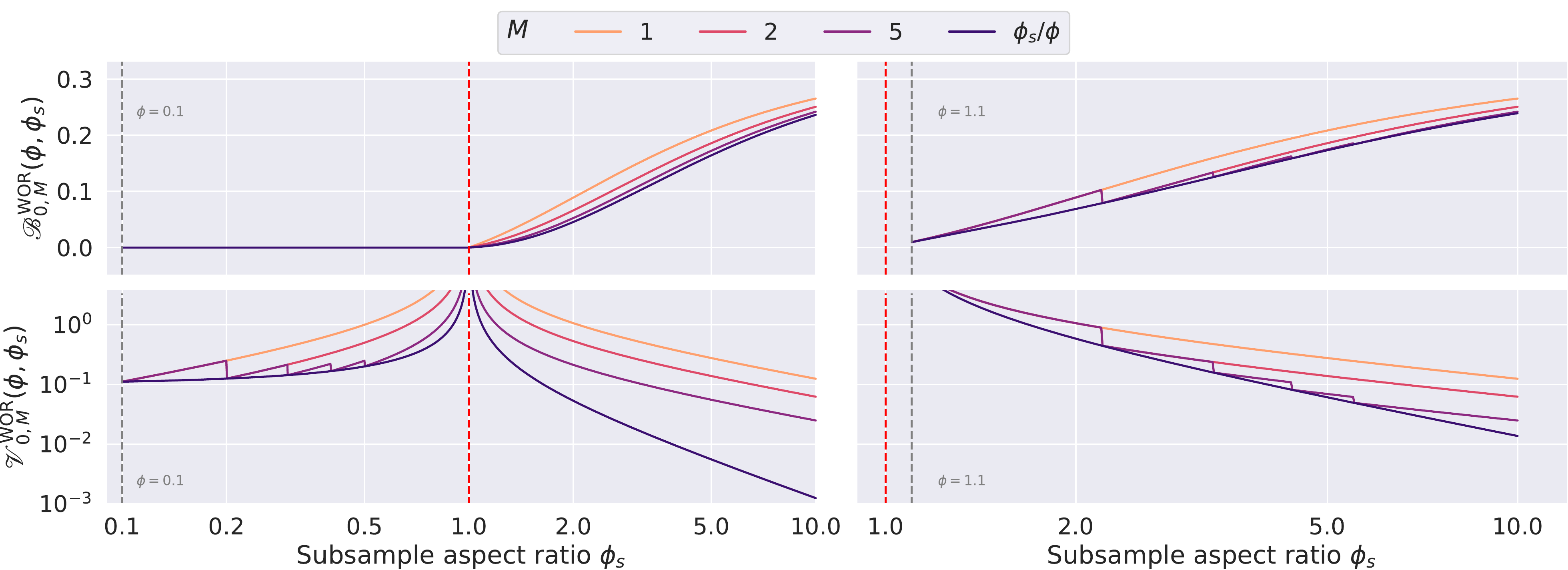}
        \caption{Asymptotic bias and variance curves in \eqref{eq:Blam_V_lam} for \splagged ridgeless predictors ($\lambda=0$), under model \eqref{eq:model-ar1} when $\rhoar=0.25$ and $\sigma^2=1$ for varying bag size $k=\lfloor p/\phi_s\rfloor$ and number of bags $M$ without replacement.
        The left and the right panels correspond to the cases when $p<n$ ($\phi=0.1$) and $p>n$ ($\phi=1.1$), respectively.
        The values of $\bVzeroM{M}{\phi}$ are shown in $\log$-10 scale.
        }
        \label{fig:ridgeless-bias-var-without-replacement}
    \end{figure}
    
    \begin{figure}[!ht]
        \centering
        \includegraphics[width=0.90\textwidth]{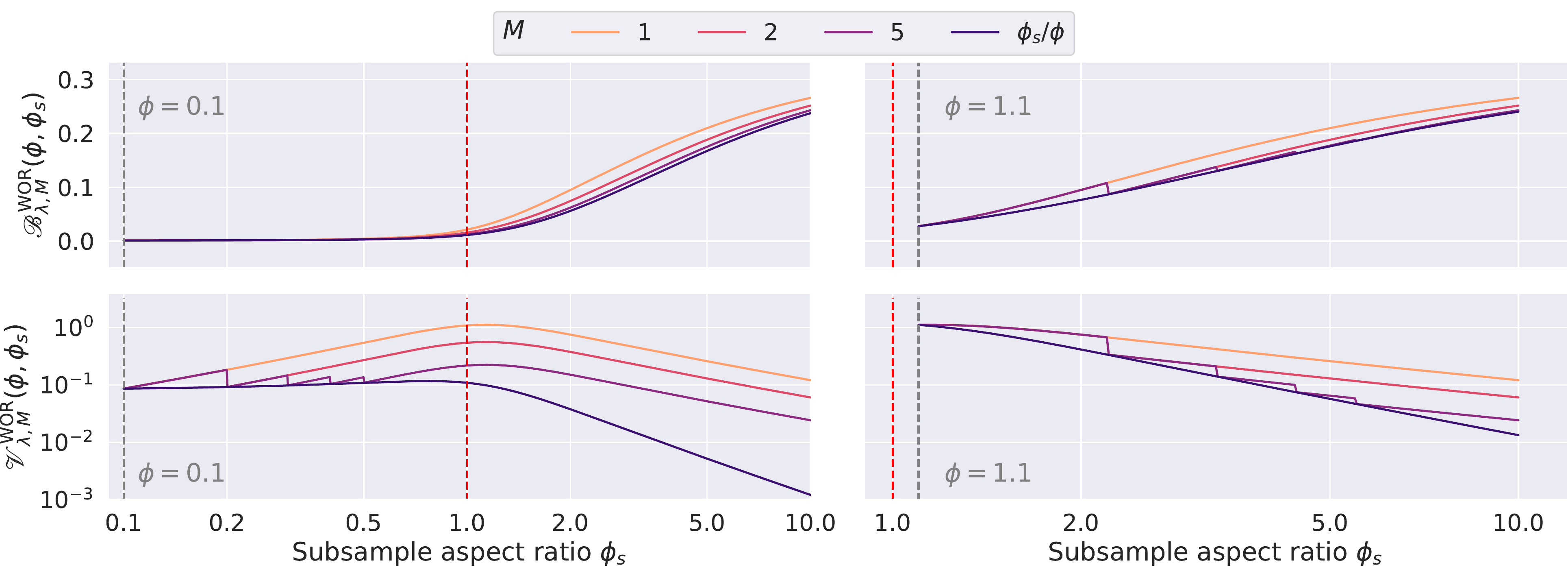}
        \caption{Asymptotic bias and variance curves in \eqref{eq:Blam_V_lam} for \splagged ridge predictors ($\lambda=0.1$), under model \eqref{eq:model-ar1} when $\rhoar=0.25$ and $\sigma^2=1$ for varying bag size $k=\lfloor p/\phi_s\rfloor$ and number of bags $M$ without replacement.
        The left and the right panels correspond to the cases when $p<n$ ($\phi=0.1$) and $p>n$ ($\phi=1.1$), respectively.
        The values of $\bVzeroM{M}{\phi}$ are shown in $\log$-10 scale.
        }
        \label{fig:ridge-bias-var-without-replacement}
    \end{figure}

\clearpage
\subsection{Additional illustrations for \Cref{thm:cv_general}}

\subsubsection{Risk monotonization for subagged ridgeless and ridge predictors 
}

    \begin{figure}[H]
        \centering
            \includegraphics[width=0.90\textwidth]{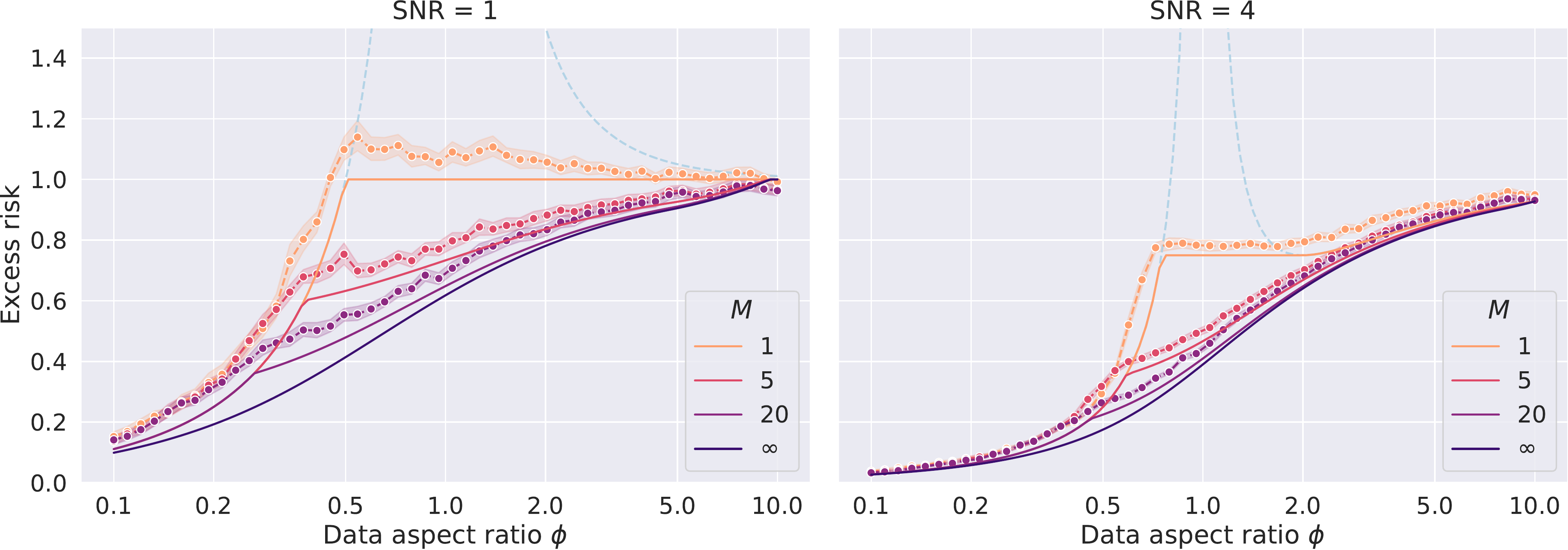}
        \caption{Asymptotic excess risk curves for cross-validated subagged ridgeless predictors ($\lambda=0$), under model \eqref{eq:model} when $\rho^2=1$ for varying 
        $\SNR$,
        subsample sizes $k=\lfloor p/\phi_s\rfloor$, and numbers of bags $M$ with replacement.
        The left and the right panels correspond to the cases when $\SNR = 1$ and $4$, respectively.
        The null risk is marked as a dotted line, and the risk for the unbagged ridgeless predictor is denoted by the dashed line. For each value of $M$, the points denote finite-sample risks, and the shaded regions denote the values within one standard deviation, with $n=1000$, $n_{\test}=63$, and $p=\lfloor n\phi\rfloor$.}
        \label{fig:ridgeless_cv_with_replacement-iso}
    \end{figure}

    \begin{figure}[H]
        \centering
        \includegraphics[width=0.90\textwidth]{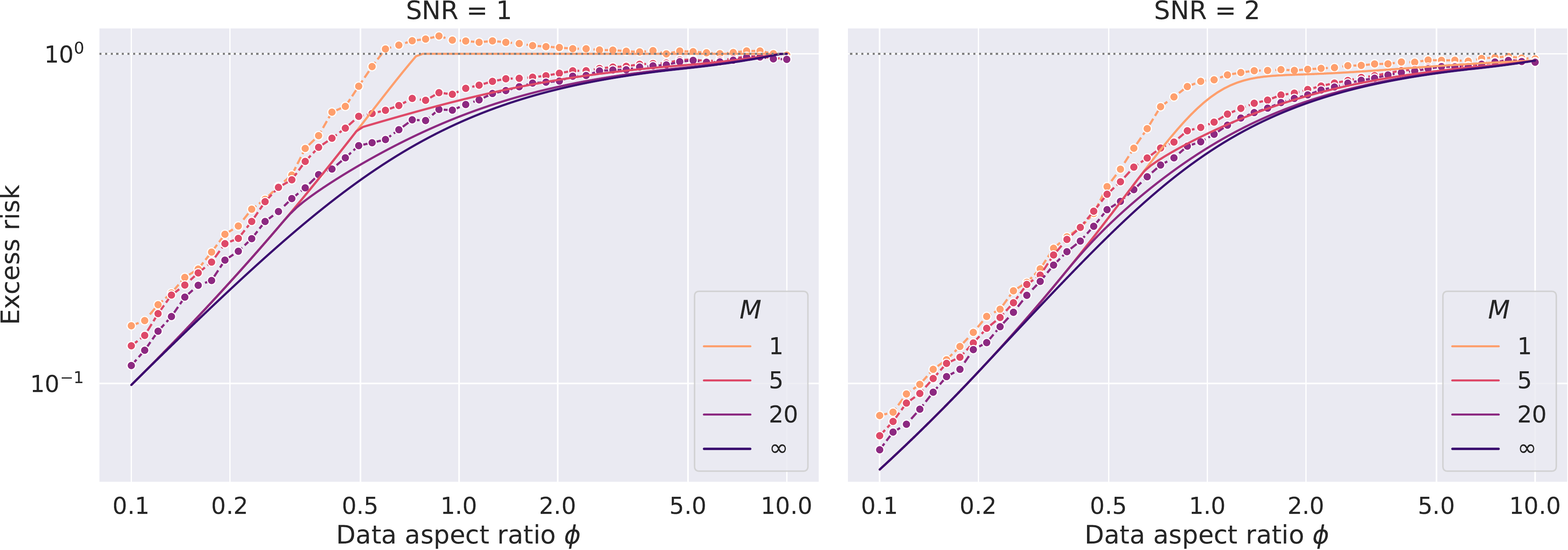}
        \caption{Asymptotic prediction risk curves for cross-validated subagged ridge predictors ($\lambda=0.1$), under model \eqref{eq:model} when $\rho^2=1$ for varying 
        $\SNR$,
        subsample sizes $k=\lfloor p/\phi_s\rfloor$ and numbers of bags $M$ with replacement.
        The left and the right panels correspond to the cases when $\SNR = 1$ and $2$, respectively.
        The null risk is marked as a dotted line. For each value of $M$, the points denote finite-sample risks averaged over 100 dataset repetitions, with $n=1000$, $n_{\test}=63$, and $p=\lfloor n\phi\rfloor$.}
        \label{fig:ridge_cv_with_replacement-iso}
    \end{figure}

    \begin{figure}[H]
        \centering
        \includegraphics[width=0.90\textwidth]{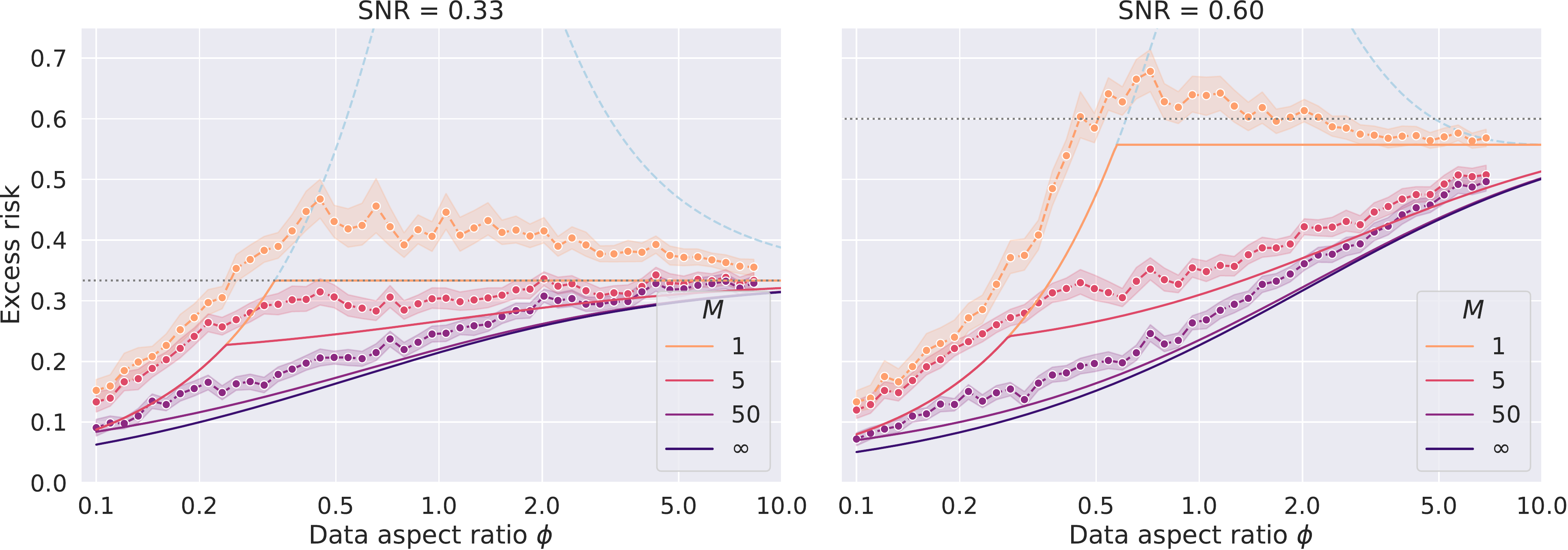}
        \caption{Asymptotic excess risk curves for cross-validated subagged ridge predictors ($\lambda=0.1$),
        under model \eqref{eq:model-ar1} when $\sigma^2=1$ for varying 
        $\SNR$,
        subsample sizes $k=\lfloor p/\phi_s\rfloor$ and numbers of bags $M$.
        The left and the right panels correspond to the cases when $\SNR = 0.33$ ($\rhoar=0.25$) and 0.6 ($\rhoar=0.5$), respectively.
        The excess null risk is marked as a dotted line, and the risk for the unbagged ridgeless predictor is denoted by the dashed line. For each value of $M$, the points denote finite-sample risks averaged over 100 dataset repetitions, and the shaded regions denote the values within one standard deviation, with $n=1000$, $n_{\test}=63$, and $p=\lfloor n\phi\rfloor$.}
    \end{figure}

\vspace*{\fill}
\subsubsection{Risk monotonization for \splagged ridgeless and ridge predictors 
}

    \begin{figure}[!ht]
        \centering
        \includegraphics[width=0.90\textwidth]{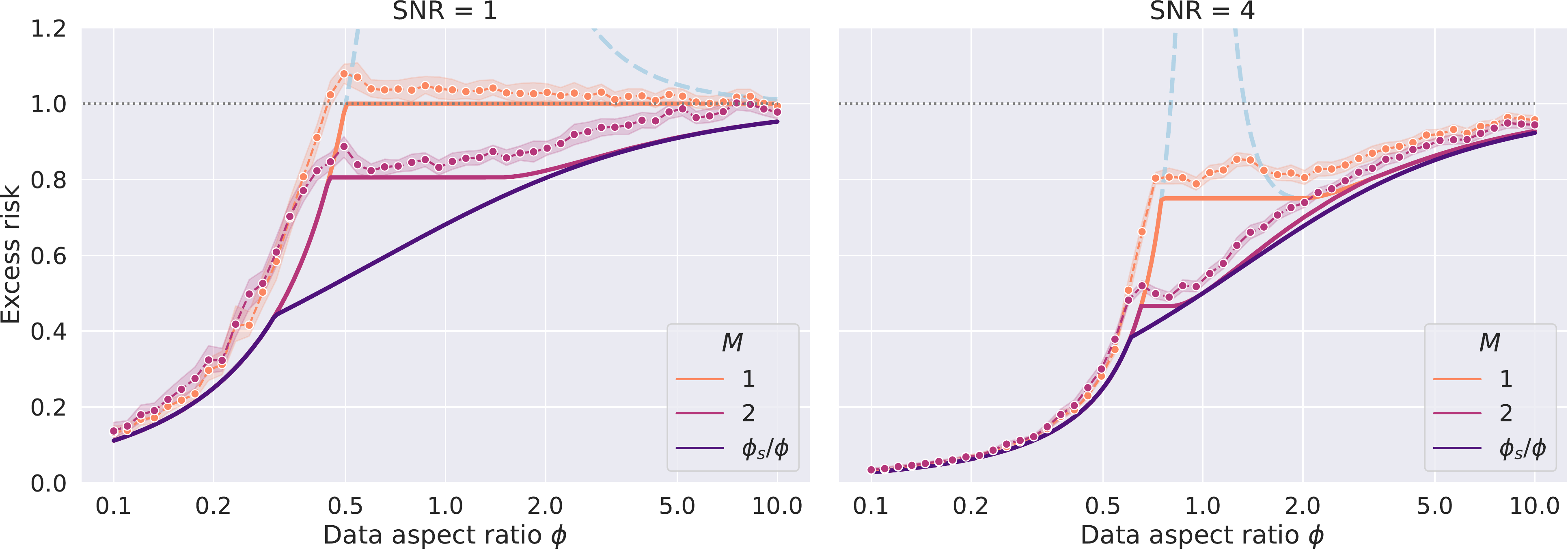}
        \caption{Asymptotic excess risk curves for cross-validated \splagged ridgeless predictors ($\lambda=0$), under model \eqref{eq:model} when $\rho^2=1$ for varying 
        $\SNR$,
        subsample sizes $k=\lfloor p/\phi_s\rfloor$, and numbers of bags $M$ without replacement.
        The left and the right panels correspond to the cases when $\SNR = 1$ and $4$, respectively.
        The null risk is marked as a dotted line, and the risk for the unbagged ridgeless predictor is denoted by the dashed line. For each value of $M$, the points denote finite-sample risks averaged over 100 dataset repetitions, and the shaded regions denote the values within one standard deviation, with $n=1000$, $n_{\test}=63$, and $p=\lfloor n\phi\rfloor$.}
        \label{fig:ridgeless_cv_without_replacement-iso}
    \end{figure}

    \begin{figure}[!ht]
        \centering
        \includegraphics[width=0.90\textwidth]{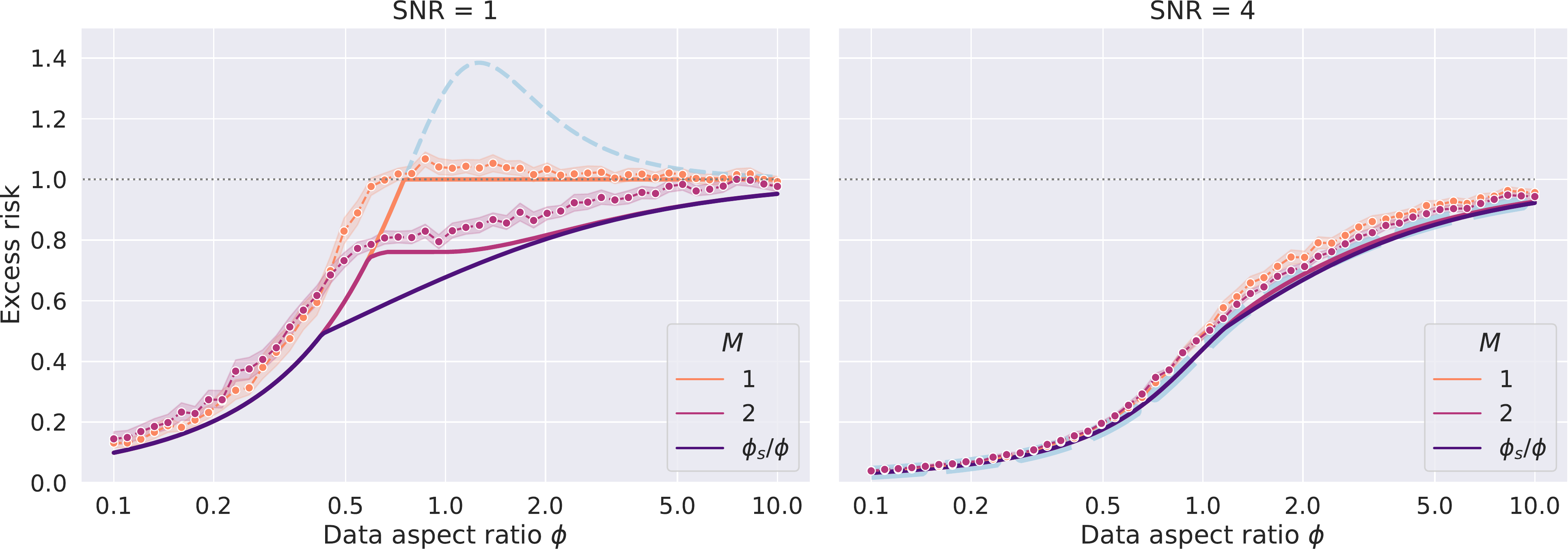}
        \caption{Asymptotic prediction risk curves for cross-validated \splagged ridge predictors ($\lambda=0.1$), under model \eqref{eq:model} when $\rho^2=1$ for varying 
        $\SNR$,
        subsample sizes $k=\lfloor p/\phi_s\rfloor$, and numbers of bags $M$ without replacement.
        The left and the right panels correspond to the cases when $\SNR = 1$ and $4$, respectively.
        The null risk is marked as a dotted line. For each value of $M$, the points denote finite-sample risks averaged over 100 dataset repetitions, with $n=1000$, $n_{\test}=63$, and $p=\lfloor n\phi\rfloor$.}
        \label{fig:ridge_cv_without_replacement-iso}
    \end{figure}
    
    \begin{figure}[H]
        \centering
        \includegraphics[width=0.90\textwidth]{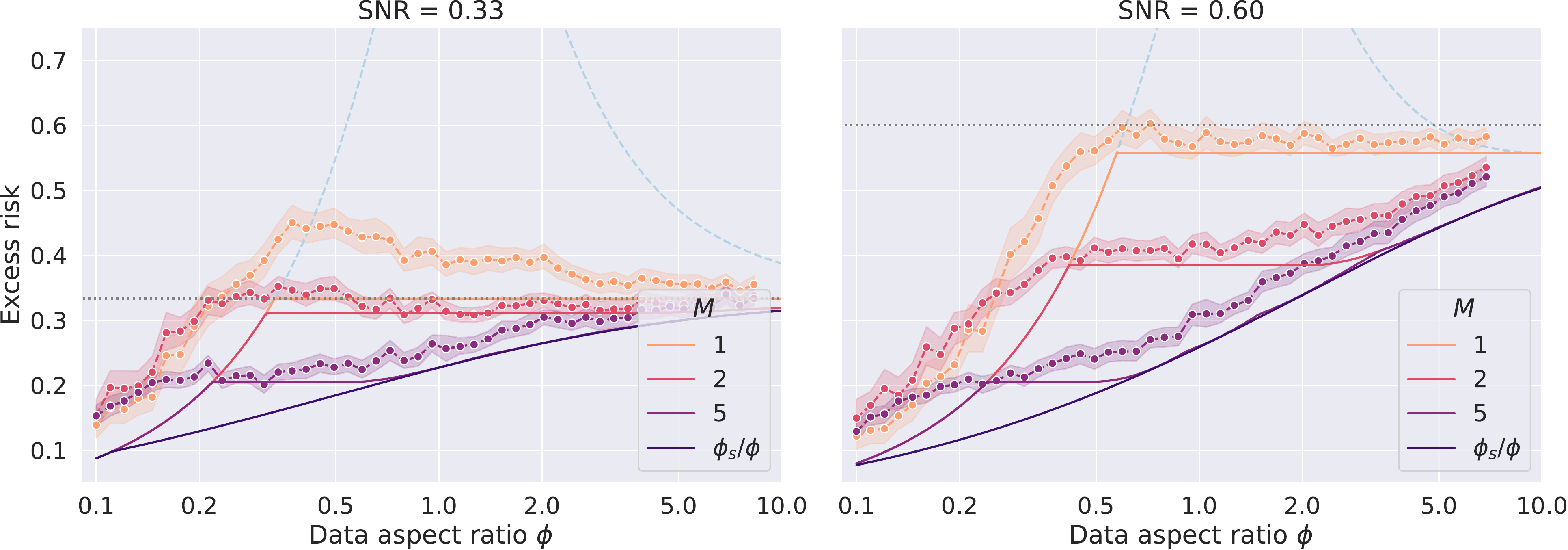}
        \caption{Asymptotic excess risk curves for cross-validated \splagged ridge predictors ($\lambda=0.1$),
        under model \eqref{eq:model-ar1} when $\sigma^2=1$ for varying 
        $\SNR$,
        subsample sizes $k=\lfloor p/\phi_s\rfloor$ and numbers of bags $M$.
        The left and the right panels correspond to the cases when $\SNR = 0.33$ ($\rhoar=0.25$) and 0.6 ($\rhoar=0.5$), respectively.
        The excess null risk is marked as a dotted line, and risk for the unbagged ridgeless predictor is denoted by the dashed line. For each value of $M$, the points denote finite-sample risks averaged over 100 dataset repetitions and the shaded regions denote the values within one standard deviation, with $n=1000$, $n_{\test}=63$, and $p=\lfloor n\phi\rfloor$.}
    \end{figure}

\clearpage
\subsection{Additional illustrations in \Cref{sec:isotropic_features}}

\subsubsection{Subagging with replacement and splagging without replacement}

    \begin{figure}[!ht]
        \centering
        \includegraphics[width=0.85\textwidth]{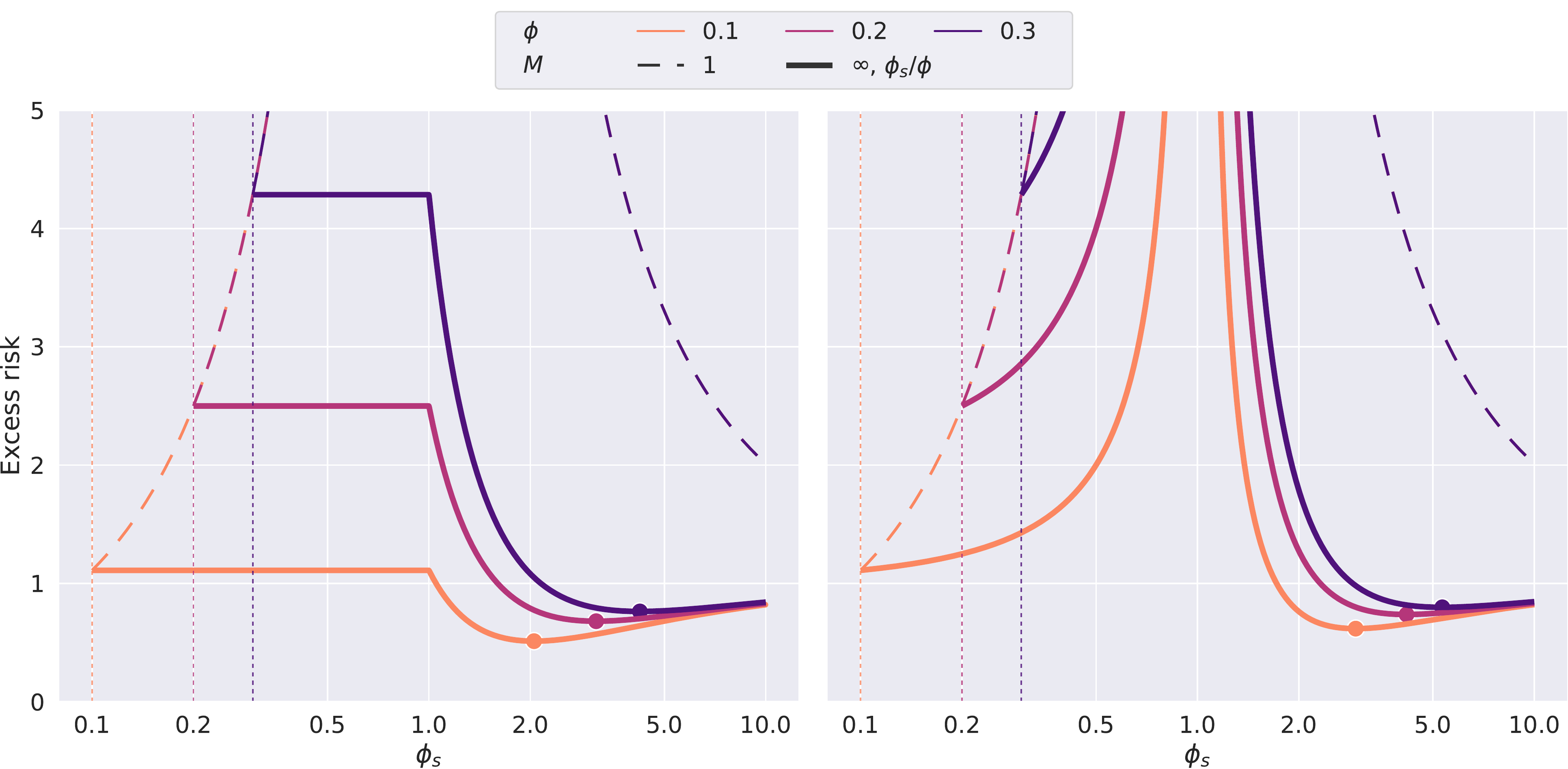}
        \caption{Asymptotic excess risk (the difference between the prediction risk and the noise level $\sigma^2$) curves of bagged ridgeless predictors ($\lambda=0$) for subagging (left panel) and \splagging (right panel), under model \eqref{eq:model} when $\rho^2=1$ and $\SNR=0.1$, for varying $\phi$ ($p<n$), bag size $k=\lfloor p/\phi_s\rfloor$ and number of bags $M$.
        The solid lines represent the optimal risks with respect to $M$ for either with replacement ($M=\infty$) or without replacement ($M=\phi_s/\phi$); 
        the dashed lines represent the risks for $M=1$;
        the dotted lines indicates the aspect ratio $\phi$ of the full dataset;
        the solid dots represent the optimal risk with respect to both $M$ and $\phi_s$.}
        \label{fig:risk_varying_phi}
    \end{figure}

    \begin{figure}[!ht]
        \centering
        \includegraphics[width=0.85\textwidth]{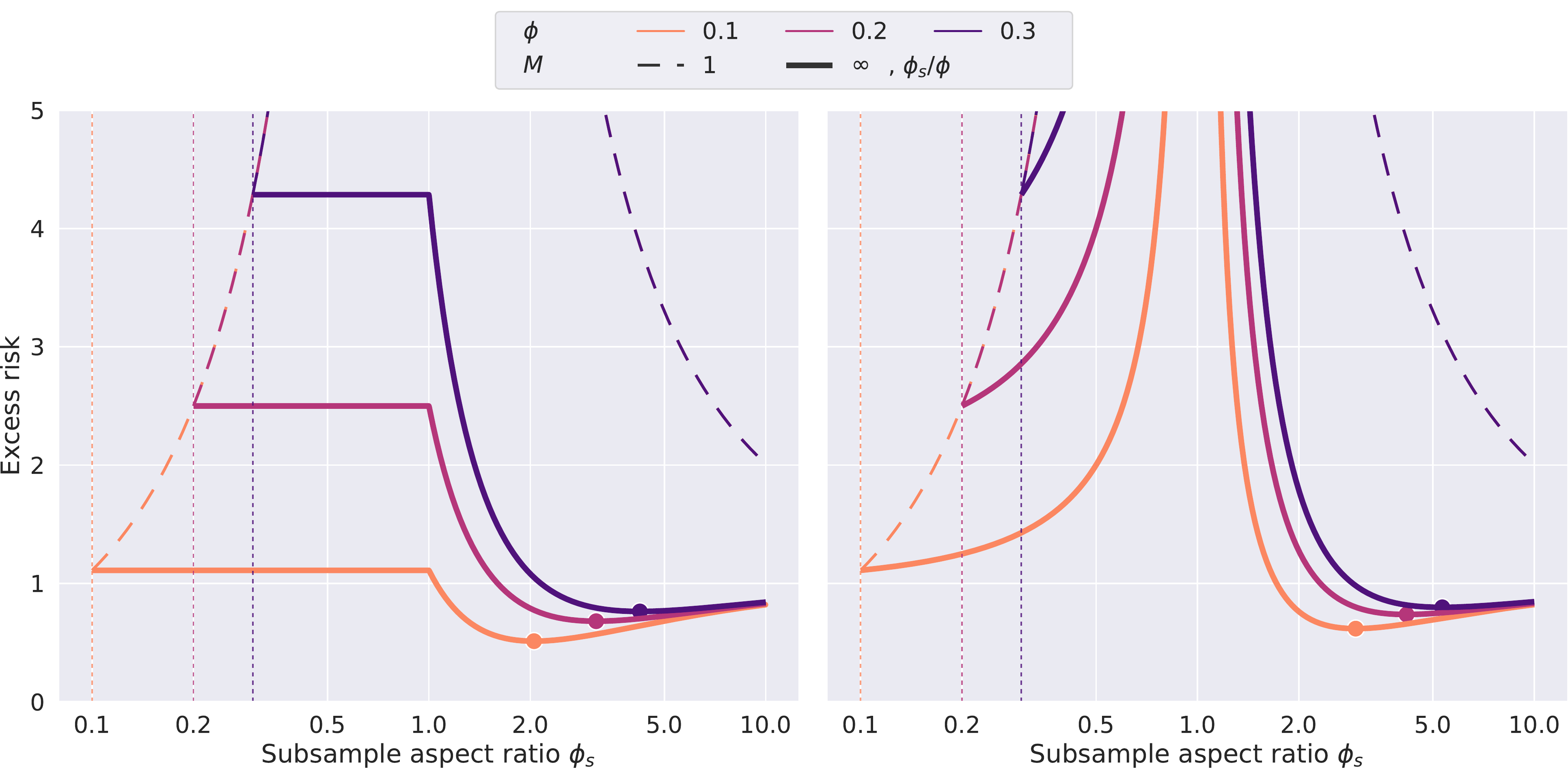}
        \caption{Asymptotic excess risk (the difference between the prediction risk and the noise level $\sigma^2$) curves of bagged ridgeless predictors ($\lambda=0$) for subagging (left panel) and \splagging (right panel), under model \eqref{eq:model} when $\rho^2=1$ and $\SNR=0.5$, for varying $\phi$ ($p\geq n$), bag size $k=\lfloor p/\phi_s\rfloor$ and number of bags $M$.
        The solid lines represent the optimal risks with respect to $M$ for either with replacement ($M=\infty$) or without replacement ($M=\phi_s/\phi$); 
        the dashed lines represent the risks for $M=1$;
        the dotted lines indicates the aspect ratio $\phi$ of the full dataset;
        the solid dots represent the optimal risk with respect to both $M$ and $\phi_s$.}
        \label{fig:risk_varying_phi_gt_1}
    \end{figure}
\end{document}